\documentclass[10pt]{amsart}

\usepackage{pdfpages}
\usepackage[margin=1.16in]{geometry}
\usepackage{setspace}
\usepackage{array}
\usepackage{graphicx}
\graphicspath{{images/}}
\usepackage{amsmath,amssymb,amsfonts,amsthm,mathtools}
\numberwithin{equation}{section}
\usepackage{cite}
\usepackage{listings}
\usepackage{multicol}
\usepackage{xcolor}
\usepackage{titlesec}
\usepackage{fancyhdr}
\usepackage{tikz-cd}
\usepackage{amscd}
\usepackage{xspace}
\usepackage[hidelinks]{hyperref}
\usepackage{booktabs}
\usepackage{longtable}
\usepackage{enumitem}
\usepackage[T1]{fontenc}
\usepackage[utf8]{inputenc}
\usepackage{microtype}
\usepackage{mathrsfs}
\usepackage{etoolbox}
\usepackage{bm}
\AtBeginEnvironment{tabular}{\footnotesize}

\titleformat{\section}
  {\normalfont\Large\bfseries}
  {\thesection}
  {1em}
  {}

\titlespacing*{\section}
  {0pt}
  {3.5ex plus 1ex minus .2ex}
  {2.3ex plus .2ex}

\titleformat{\subsection}
  {\normalfont\large\bfseries}
  {\thesubsection}
  {1em}
  {}
  
\titlespacing*{\subsection}
  {0pt}
  {3.5ex plus 1ex minus .2ex}
  {2.3ex plus .2ex}

\titleformat{\subsubsection}
  {\normalfont\fontsize{9.6pt}{11.5pt}\selectfont\bfseries}
  {\thesubsubsection}
  {1em}
  {}

\titlespacing*{\subsubsection}
  {0pt}
  {2.5ex plus .5ex minus .2ex}
  {1ex plus .2ex}
  
\makeatletter
\def\l@section{\@tocline{1}{1.0em}{0pt}{1.5em}{\bfseries}}
\def\l@subsection{\@tocline{2}{0pt}{1.5em}{2.3em}{}}
\makeatother

\newcommand{\currentsectiontitle}{}

\fancypagestyle{mainstyle}{
  \fancyhf{}
  \fancyhead[C]{\footnotesize\currentsectiontitle}
  \fancyfoot[C]{\small\thepage}

}

\newcommand{\mysection}[1]{%
  \section{#1}%
  \gdef\currentsectiontitle{#1}%
  \thispagestyle{mainstyle}%
}

\makeatletter
\renewenvironment{abstract}{%
  \ifx\maketitle\relax
    \ClassWarning{\@classname}{Abstract should precede
      \protect\maketitle\space in AMS document classes; reported}%
  \fi
  \global\setbox\abstractbox=\vtop \bgroup
    \normalfont\normalsize
    \list{}{\labelwidth\z@
      \leftmargin3pc \rightmargin\leftmargin
      \listparindent\normalparindent \itemindent\z@
      \parsep\z@ \@plus\p@
      
    }%
    \item[\hskip\labelsep\scshape\abstractname.]%
}{%
  \endlist\egroup
  \ifx\@setabstract\relax \@setabstracta \fi
}
\makeatother

\newcommand{\ttt}{two-by-two\xspace}
\newcommand{\bbb}{ura\xspace}

\newcommand{\C}{\mathbb C}
\newcommand{\Z}{\mathbb Z}

\newcommand{\OO}{\mathcal O}

\newcommand{\M}{\mathcal M}

\newcommand{\Jac}{\operatorname{Jac}}
\newcommand{\Pic}{\operatorname{Pic}}
\newcommand{\Prym}{\operatorname{Prym}}
\newcommand{\Nm}{\operatorname{Nm}}
\newcommand{\tr}{\operatorname{tr}}
\newcommand{\rank}{\operatorname{rank}}
\newcommand{\id}{\operatorname{id}}

\newcommand{\Tot}{\operatorname{Tot}}

\newcommand{\Aut}{\operatorname{Aut}}

\newcommand{\PT}{\operatorname{PT}}
\newcommand{\Res}{\operatorname{Resultant}}

\theoremstyle{plain}
\newtheorem{theorem}{Theorem}[section]
\newtheorem{proposition}[theorem]{Proposition}
\newtheorem{lemma}[theorem]{Lemma}
\newtheorem{corollary}[theorem]{Corollary}

\theoremstyle{definition}
\newtheorem{definition}[theorem]{Definition}
\newtheorem{example}[theorem]{Example}

\theoremstyle{remark}
\newtheorem{remark}[theorem]{Remark}

\numberwithin{equation}{section}
\numberwithin{figure}{section}
\numberwithin{table}{section}

\makeatletter
\def\@settitle{%
  \begin{center}%
    \normalfont\huge \@title
  \end{center}%
}
\makeatother

\title{Elementary Moves on Ideal Webs, Local Monodromies on Hitchin Fibrations, and Tropical Mutations on $\mathrm{K_0}$ Groups}
\author{Yuhan Liao}

\begin{document}

\pagestyle{empty}
\pagenumbering{gobble}

\begin{abstract}
This paper builds on the work of Goncharov by studying the correspondence between topological spectral covers constructed from generic $A_{m-1}$-webs with rank $m\ge2$, Hitchin spectral curves over good Hitchin loci, and the Grothendieck groups $K_0(\mathcal{D}_{\Gamma})$ associated with the quivers with potential $(Q_{\Gamma},W_{\Gamma})$ assigned to corresponding bipartite $A_{m-1}$-graphs.

\smallskip
We show that the automorphisms on the non-frozen quiver lattices $\Lambda_{\Gamma}^{\circ}$ induced by compositions of elementary moves on ideal webs/bipartite graphs are the same as the Gauss--Manin monodromy actions on the $SL_m$ Hitchin fiber lattices $H_1(\operatorname{Prym}(\Sigma);\mathbb{Z})$ locally around discriminant loci. And the Gauss--Manin monodromy actions agree with the automorphisms on the Grothendieck groups $K_0(\mathcal{D}_{\Gamma}^{\circ})$ induced by compositions of mutations on the quivers with potential $(Q_{\Gamma},W_{\Gamma})$.
\end{abstract}

\maketitle
\thispagestyle{empty}

\clearpage

\newgeometry{
  top=1.1in,
  bottom=1.1in,
  left=1.16in,
  right=1.16in
}

\tableofcontents
\thispagestyle{empty}

\clearpage

\restoregeometry

\pagenumbering{arabic}
\pagestyle{mainstyle}

\newpage
\mysection{Introduction}
The present work is motivated by the interplay among three structures associated with decorated surfaces:
\begin{itemize}
    \item the combinatorial data encoded by ideal webs;
    \item the geometric data arising from Hitchin fibrations;
    \item the algebraic and categorical data provided by cluster structures on moduli spaces of local systems and by quivers with potential.
\end{itemize}
A central theme is that elementary moves on ideal webs are reflected in the geometry of spectral covers and in the mutation theory of the corresponding cluster and 3-Calabi--Yau categorical structures.

\bigskip
Let $\mathring{C}=C\setminus P$ be a smooth complex genus-$g$ curve with a set $P$ of $n$ punctures satisfying $2g-2+n>0$.
In rank $m=2$, the relation between the categorical and spectral lattices is well-established.
Goncharov's $A_1$-webs/bipartite $A_1$-graphs are in bijection with ideal triangulations, and two-by-two moves correspond to flips.
If $\Gamma$ is the bipartite $A_1$-graph associated with an ideal triangulation $T$, then Goncharov's non-frozen quiver lattice
$\Lambda_\Gamma^{\circ}
=\mathbb Z[\{\text{internal faces of }\Gamma\}]$
is indexed by the edges of $T$ \cite{Goncharov}.

\medskip
Let $W$ be the $A_1$-web corresponding to $\Gamma$. Denote by $\pi:\Sigma\to \mathring{C}$ its spectral cover.
One has
$\Lambda_{\Gamma}^{\circ}
\cong\ker(
H_1(\Sigma;\mathbb Z)
\xrightarrow{\pi_*}
H_1(\mathring{C};\mathbb Z))$ \cite{Goncharov}.
On the other hand, Bridgeland--Smith define
the hat-homology lattice by
$\widehat H(\phi)
=H_1(\Sigma;\mathbb Z)^{-}$,
the anti-invariant part for the covering involution \cite{BS15}.

\medskip
For a saddle-free quadratic differential $\phi$, the WKB triangulation
$T(\phi)$ is an ideal triangulation, and Bridgeland--Smith associate to
each edge $e\in T(\phi)$ a spectral cycle
$\gamma_e\in \widehat H(\phi)$.
They identify the edge lattice of $T(\phi)$ with
$\widehat H(\phi)$, and
$K_0(\mathcal D(T,\epsilon))
\xrightarrow{\sim}
\widehat H(\phi)$,
$[S_e]\longmapsto \gamma_e$.
Thus
$\Lambda_\Gamma^\circ
\cong
K_0(\mathcal D(T,\epsilon))
\cong
\widehat H(\phi)$,
under which the basis elements correspond as
$e_G
\leftrightarrow
[S_e]
\leftrightarrow
\gamma_e$, where $G$ denotes an internal face of $\Gamma$.

\medskip
Moreover, this identification is compatible with wall-crossing.
If $\Gamma_-$ and $\Gamma_+$ are related by the flip corresponding to
crossing a simple wall, Proposition~10.7 of
\cite{BS15} gives an equivalence
$\Psi:
\mathcal D_{\Gamma_-}^{\circ}
\xrightarrow{\sim}
\mathcal D_{\Gamma_+}^{\circ}$
such that the induced isomorphism on Grothendieck groups and the isomorphism given by the Gauss--Manin connection fit into the following commutative diagram, under above spectral identifications \cite{BS15}:
\[
\begin{CD}
K_0(\mathcal D_{\Gamma_-}^{\circ})
@>{\Psi_*}>>
K_0(\mathcal D_{\Gamma_+}^{\circ})
\\
@V{\nu_-}V{\sim}V
@V{\nu_+}V{\sim}V
\\
\widehat H(\phi_-)
@>{\mathrm{GM}}>>
\widehat H(\phi_+),
\end{CD}
\]
and the Euler form is identified with the intersection form.
Thus, in rank two, mutation of the Goncharov's quiver lattice and the induced mutation on $K_0$ are compatible with the Gauss--Manin connection of the corresponding spectral cycles.

\medskip
Goncharov introduces $A_m$-webs/bipartite $A_m$-graphs as higher-rank analogues of ideal triangulations. An $A_m$-web $W$ and its corresponding bipartite $A_m$-graph $\Gamma$ determine a spectral cover
$\pi_W:\Sigma_W\to \mathring{C}$
and the non-frozen quiver lattice
$\Lambda_\Gamma^\circ=\mathbb Z[\{\text{internal faces of }\Gamma\}]
\cong\ker(
H_1(\Sigma_W;\mathbb Z)\to H_1(\mathring{C};\mathbb Z))$ \cite{Goncharov}.
For the Hitchin system, Goncharov proposes these quiver lattices and the associated $K_0$ groups as higher-rank counterparts of the spectral homology and categorical structures appearing in rank two.

\newpage
This paper builds on the framework introduced in \cite{Goncharov}. Section 2 recalls the background and notation for bipartite graphs and webs on decorated surfaces. At the beginning of Sections 3, 6 and 7, we recall additional results from \cite{Goncharov} that are specific to the corresponding arguments.

\medskip
For rank $m\ge 2$, we first show the correspondence between topological spectral covers constructed from generic $A_{m-1}$-webs and Hitchin spectral curves over good Hitchin loci, and the correspondence between the mutations on non-frozen quiver lattices $\Lambda_\Gamma^\circ$ induced by elementary moves on ideal webs and the Gauss--Manin monodromy actions on $SL_m$ Hitchin fiber lattices $H_1(\Prym(\Sigma);\mathbb Z)$.

\medskip
We show that the automorphisms on $\Lambda_\Gamma^\circ$ induced by two successive \ttt moves in the same orientation are the same as the Gauss--Manin monodromy actions on $H_1(\Prym(\Sigma);\mathbb Z)$ along small meridians in good Hitchin loci around discriminant loci $\mathcal{D}_{bdry,C,P}$ (Proposition~\ref{prop:hittoweb} and Proposition~\ref{prop:match}). 

\medskip
The automorphisms on $\Lambda_\Gamma^\circ$ induced by other types of elementary moves and the Gauss--Manin monodromy actions on $H_1(\Prym(\Sigma);\mathbb Z)$ along small meridians in good Hitchin loci around other types of discriminant loci are trivial.
Hence we conclude the correspondence between the web side and the Hitchin side in Table~\ref{tab:compare} and Table~\ref{table:3compare} which we reproduce below.

\medskip
We prove that the Gauss--Manin monodromy actions on $H_1(\Prym(\Sigma);\mathbb Z)$ along small meridians in good Hitchin loci around discriminant loci $\mathcal{D}_{bdry,C,P}$, and the automorphisms on $K_0(\mathcal{D}_{\Gamma}^{\circ})$ induced by two successive \ttt moves in the same orientation on generic $A_{m-1}$-webs, fit into the commutative diagram of isomorphisms in Theorem~\ref{thm:GM}, which is reproduced below as Theorem~\ref{thm:intro}, and the intersection form is identified with the Euler form.

\begin{table}[htbp]
\label{tab:1}
\centering
\caption{Correspondence between the web perspective and the Hitchin perspective.}
\begin{tabular}{c|c|c}
&Web Perspective & Hitchin Perspective \\
\hline
$\text{Parametrization}$
&$\begin{array}{ccc}
\text{Two successive two-by-two moves}  \\
\text{A singular rank-$m$ spectral web}   \\
\text{Return to the initial generic $A_{m-1}$-web}   \\
\end{array}
$
& $\begin{array}{ccc}
\text{A small meridian around } \mathcal{D}_{bdry,C,P}  \\
\text{A point in } \mathcal{D}_{bdry,C,P}^{\circ}\\
\text{Return to the initial point in }B_{SL_m,C,P}^{\mathrm{good}}\\
\end{array}$ \\
\hline
$\begin{array}{cc}
\mathring{S}=S\setminus P \\
\mathring{C}=C\setminus P  
\end{array}$
&$\begin{array}{ccc}
\text{Half-twist of simple branch points }q_+,q_-  \\
\text{Simple branch points $q_+,q_-$ collide}  \\
q_+ \text{ and }q_- \text{ braid and switch positions}  \\
\end{array}$ 
& $\begin{array}{ccc}
\text{Half-twist of simple zeros $q_+,q_-$ of }\Delta  \\
\text{Simple zeros $q_+,q_-$ collide into a double zero}  \\
q_+ \text{ and }q_- \text{ braid and switch positions}  \\
\end{array}$ \\
\hline
$\text{Spectral Cover}$
&$\begin{array}{ccc}
\text{Dehn twist along the face loop lift}  \\
\text{Nodal singularity at the degenerated loop}  \\
\text{Automorphism on the quiver lattice}\\
\end{array}$ 
& $\begin{array}{ccc}
\text{Dehn twist along the vanishing cycle}  \\
\text{Nodal singularity at the "vanished" cycle}  \\
\text{Monodromy on the Hitchin fiber lattice}\\
\end{array}$
\end{tabular}
\end{table}

\begin{table}[htbp]
\label{tab:2}
\centering
\caption{Elementary moves and disc areas around discriminant loci.}
\begin{tabular}{c|c|c}
Two \ttt moves & Two shrink \& expand moves & Isotopy of webs \\
\hline
Around $\mathcal{D}_{bdry,C,P}$
& Around $\mathcal{D}_{cau,C,P}$
& Around $\mathcal{D}_{Max,C,P}$\\
\hline
2 simple branch points braid/collide
& 2 simple branch points braid/coalesce
& 2 simple branch points braid/coincide\\
\hline
Dehn twist/nodal singularity
& 2 ramification points rotate/coalesce
& 2 ramification points over close values\\
\hline
Picard-Lefschetz transformation
& Trivial automorphism or monodromy
& Trivial automorphism or monodromy
\end{tabular}
\end{table}

\newpage
\begin{theorem}
\label{thm:intro}
Let $\mathring{C}=C\setminus P$
be a $n$-punctured smooth complex curve of genus-$g$ with $2g-2+n>0$. Let $G=SL_m$, $m\ge2$.
Let $a\in B_{G,C,P}^{\mathrm{good}}$ be close to a point $b_0\in \mathcal{D}_{bdry,C,P}^\circ$.

Then for any sufficiently small meridian in $B_{G,C,P}^{\mathrm{good}}$ around $\mathcal{D}_{bdry,C,P}$ starting and ending at $a$, there is a commutative diagram of isomorphisms:
\begin{equation}
\begin{tikzcd}
H_1(\Prym(\Sigma_a);\mathbb Z)
\arrow[r, "GM"]
\arrow[d, "k_c"]
&
H_1(\Prym(\Sigma_a);\mathbb Z)
\arrow[d, "k_c"]
\\
K_0(\mathcal{D}_{\Gamma}^{\circ})
\arrow[r, "(\Phi'^\pm)_\ast\circ(\Phi^\pm)_\ast"]
&
K_0(\mathcal{D}_{\Gamma}^{\circ}),
\end{tikzcd}
\end{equation} 
where the top automorphism is given by the Gauss--Manin monodromy, and $K_0(\mathcal{D}_{\Gamma}^{\circ})$ is the Grothendieck group associated with the quiver with potential assigned to certain generic $A_{m-1}$-web $W$ and its corresponding bipartite $A_{m-1}$-graph $\Gamma$ on $C$.

$(\Phi'^\pm)_\ast\circ(\Phi^\pm)_\ast$ is the composition of two successive mutations on $K_0(\mathcal{D}_{\Gamma}^{\circ})$-basis induced by two successive \ttt moves in the same orientation at certain internal face of $W$ in the same orientation.

The skew bilinear intersection form on the Hitchin fiber lattices and the skew Euler form on the $K_0$ groups are compatible with the isomorphisms in the commutative diagram.
\end{theorem}

\bigskip
\noindent
\textsc{Section 2. }
We show the \emph{connectedness} of $A_m$ or $A_m^*$-webs/bipartite $A_m$ or $A_m^*$-graphs under finite sequences of elementary moves (Proposition~\ref{prop:connect}).
We define \emph{singular spectral} webs and \emph{generic ideal} webs for use in the correspondence developed in Section 4 and 5. Then we show the associated spectral covers of generic $A_{m-1}$-webs have \emph{full monodromy groups} $S_m$ (Corollary~\ref{cor:fullmono}).

\bigskip
\noindent
\textsc{Section 3. }
We introduce a local operation, called a ura operation, on spectral covers that reflects the change of spectral covers constructed from ideal webs under a \ttt move on the ideal webs/bipartite graphs (Definition~\ref{def:Ura}).
The composition of two successive ura operations in the same orientation gives a \emph{Dehn twist} on spectral covers (Proposition~\ref{prop:dehn}).

The ura operation gives a \emph{discontinuous bijection} between spectral covers and induces an isomorphism on the quiver lattices assiganed to the bipartite graphs (Theorem~\ref{thm:ttt}).
The \emph{two possible orientations} of the ura operation match the two signed mutations between the relevant $K_0$ groups, which we discuss in Section 7.

We then discuss the change of the spectral covers under "shrink \& expand" moves and isotopy of webs, and conclude that the compositions of such moves give \emph{trivial automorphisms} on quiver lattices.

\bigskip
\noindent
\textsc{Section 4. }
As a preliminary to Section 5, we recall definitions and basic results for Hitchin fibrations and other relevant concepts.
For a fixed complex curve $C$ and $G=GL_m$ or $SL_m$,
we define the \emph{good Hitchin locus} $B_{G,C,P}^{\mathrm{good}}$ in the Hitchin base $B_{G,C,P}$ and take the restrictions of the three types of the discriminant loci, $\mathcal{D}_{bdry}$, $\mathcal{D}_{Max}$ and $\mathcal{D}_{cau}$, in universal Hitchin base introduced in \cite{Basok}.

We show that spectral curves over good Hitchin loci have \emph{full monodromy groups} $S_m$ (Corollary~\ref{cor:good}). Then we prove that for any spectral curve over a point in the good Hitchin locus and any spectral cover constructed from a given isotopy class of generic $A_{m-1}$-webs, we can find an isotopy web representative such that the two spectral covers are \emph{Hurwitz isomorphic} (Corollary~\ref{cor:iso}).
This result lays the groundwork for the subsequent discussions in Section 5.

\newpage
\noindent
\textsc{Section 5. }
We establish two complementary propositions:

Proposition~\ref{prop:hittoweb} states that given any small meridian $m$ in $B_{SL_m,C,P}^{\mathrm{good}}$ around $\mathcal{D}_{bdry,C,P}$, we can find a generic $A_{m-1}$-web $W$ on $C$, such that the monodromy action along $m$ on the $SL_m$ Hitchin fiber lattice can be realized by the automorphism on the quiver lattice $\Lambda_{\Gamma}^{\circ}$ resulting from two successive \ttt moves in the same orientation at some internal face of $W$.

Proposition~\ref{prop:match} states that given any generic $A_{m-1}$-web $W$ on $C$, the automorphism on the quiver lattice $\Lambda_{\Gamma}^{\circ}$ resulting from two successive \ttt moves in the same orientation at an internal face can be realized by the monodromy action on the $SL_m$ Hitchin fiber lattice along some small meridian in $B_{SL_m,C,P}^{\mathrm{good}}$ around $\mathcal{D}_{bdry,C,P}$.

We discuss the small deformations on spectral curves around $\mathcal{D}_{cau,C,P}$ and $\mathcal{D}_{Max,C,P}$, and show that they correspond to the changes of spectral covers constructed from ideal webs on $C$ produced by the composition of two successive "shrink \& expand" moves and isotopy of webs respectively.

We state the correspondence between the web side and the Hitchin side in Table~\ref{tab:compare} and Table~\ref{table:3compare}.

\bigskip
\noindent
\textsc{Section 6. }
For $G=GL_m$ or $SL_m$, we recall the definitions of moduli spaces of decorated/$*$-decorated local systems and regular functions produced by $A_{m-1}$/$A_{m-1}^*$-webs introduced in \cite{Goncharov}. We show that a \ttt move centred at an internal face $F$ of an ideal web $W$ induces a \emph{cluster $\mathcal{A}$-mutation} at $F$ (Proposition~\ref{prop:22}).

We show that an $A_{m-1}$-web on a decorated surface $S$ provides a \emph{cluster $\mathcal{A}$-coordinate chart} on the moduli space of decorated $SL_m$-local systems $\mathcal{A}_{SL_m,S}$; an $A_{m-1}^*$-web on $S$ provides a \emph{cluster $\mathcal{A}$-coordinate chart} on the moduli space of $*$-decorated $SL_m$-local systems $\mathcal{A}_{SL_m,S}^*$, and the collection of such cluster $\mathcal{A}$-coordinate charts gives a \emph{cluster atlas} (Theorem~\ref{thm:general}).

\bigskip
\noindent
\textsc{Section 7. }
Denote by $K_0(\mathcal{D}_{fd}(\Gamma(Q,W)))$ the Grothendieck group associated with the quiver with potential $(Q,W)$ assigned to a bipartite $A_{m-1}$ or $A_{m-1}^*$-graph on a decorated surface.

We recall the isomorphisms on $K_0(\mathcal{D}_{fd}(\Gamma(Q,W)))$ induced by mutations on $(Q,W)$, and that the two signed dual isomorphisms on $K_0(\mathcal{D}_{fd}(\Gamma(Q,W)))^\vee$ induced by mutations on quivers with potential match the two linear pieces in tropicalizations of cluster $\mathcal{A}$-mutations induced by the same mutations.

For mutations given by \ttt moves, the two signed induced isomorphisms on $K_0$ groups coincide with the two signed isomorphisms on the quiver lattices induced by ura bijections in two orientations, giving the commutative diagram in Theorem~\ref{thm:7com}.

We conclude in Theorem~\ref{thm:GM} that the automorphisms on the $SL_m$ Hitchin fiber lattices given by Gauss--Manin monodromy actions agree with the induced automorphisms on the Grothendieck groups $K_0(\mathcal{D}_{\Gamma}^{\circ})$ associated with certain generic $A_{m-1}$-webs on smooth complex genus-$g$ curves with $n$ punctures, satisfying $2g-2+n>0$.

\subsection*{Acknowledgements}
This work was submitted as a Master's thesis to the School of Mathematics and Statistics at The University of Melbourne.
I would like to thank my supervisors, Gufang Zhao and Paul Norbury, for their guidance, encouragement, and many helpful discussions throughout this work. I would also like to thank Yuxin Chen for drawing all the figures for the paper models, and Brice Arrigo for proofreading the manuscript.

\newpage
\mysection{Ideal Webs and Spectral Covers}

This section recalls the parts of the framework of \cite{Goncharov} that will be used throughout the paper, together with some additional material needed later. It is included to fix notation and to make the subsequent arguments reasonably self-contained.

\subsection{\texorpdfstring{$\mathrm{A_m}$}{A}-Webs and \texorpdfstring{$\mathrm{A_m^*}$}{A}-Webs on Decorated Surfaces}

A \emph{decorated surface S} is a connected oriented surface with a nonempty finite set of special points, which is the union of the sets of punctures and marked points on the boundary. Each boundary component is assumed to have one special point.

In this paper, we will focus on compact genus-$g$ decorated surfaces without boundary, and with a finite set $P$ of $n>0$ special points satisfying $2g-2+n>0$.

\subsubsection{Bipartite Graphs and Webs}
An \emph{ideal triangulation} $T$ of $S$ is a triangulation with vertices at the special points. An ideal triangulation is called regular if it has no
self-folded triangles, i.e., every triangle has three distinct
sides. Vertices of a regular triangle are allowed to be identified. An m-ideal triangulation is obtained by subdividing an ideal triangulation into $(m+1)^2$ small triangles by drawing three families of m equidistant lines, parallel to the sides of the triangle.

A \emph{standard bipartite $A_m$-graph $\Gamma_{A_m}$(T)} associated with an m-ideal triangulation T is a bipartite graph on an m-ideal triangulation, with $\bullet$-vertices to the centers of all looking-down subtriangles of one triangle T of the m-triangulation, and $\circ$-vertices to the centers of all internal looking-up subtriangles, to the centers of middle m-1 little segments on each side of T which do not contain the special points, and to the special points, edges connecting adjacent vertices of different kinds.\cite{Goncharov}

A \emph{standard bipartite $A_m^*$-graph $\Gamma_{A_m^*}$(T)} associated with an m-ideal triangulation T is a bipartite graph on an m-ideal triangulation, with $\circ$-vertices to the centers of all looking-up subtriangles of one triangle T of the m-triangulation, and the $\bullet$-vertices to the centers of all looking-down subtriangles, as well as to the centers of m+1 little segments on each side of T, edges connecting adjacent vertices of different kinds.\cite{Goncharov}

\begin{remark}
The $A_m$ here is the Dynkin type of the semisimple part of the corresponding reductive Lie algebras that we will introduce later.
\end{remark}

\begin{remark}
Shrinking each 2-valent $\bullet$-vertex of the bipartite $A_m^*$-graph $\Gamma_{A_m^*}$(T) connecting two $\circ$-vertices to a single $\circ$-vertex, we get the bipartite $A_m$-graph $\Gamma_{A_m}$(T).\cite{Goncharov}
\end{remark}

A face F of a bipartite graph $\Gamma$ on $S$ is \emph{external} if it intersects the boundary of S, or its interior contains a special point. It is called \emph{internal} otherwise. A face F of $\Gamma$ is \emph{non-special} if it does not contain a special point of $S$.\cite{Goncharov}

A collection of \emph{zig-zag strands} associated with a bipartite graph $\Gamma$ is the set of paths on $\Gamma$ which go left at the $\bullet$-vertices, and right at the $\circ$-vertices. A strand on $S$ is either an oriented loop, or an oriented path connecting two boundary points.\cite{Goncharov}

\begin{definition}[{\cite{Goncharov}}]
A \emph{web} on a decorated surface $S$ is a connected collection of oriented strands in generic position, considered modulo isotopy, satisfying the following two conditions:
\begin{itemize}
      \item The directions of the strands intersecting any given strand alternate.
      \item The connected components of the complement to the union of strands are discs.
  \end{itemize}
\end{definition}

\begin{definition}[\cite{Goncharov}]
A \emph{web} associated with a bipartite graph $\Gamma$ is the collection of all zig-zag strands associated with $\Gamma$.
\end{definition}

Let $p:\widetilde{S} \to S$ be the universal cover of $S$. The special points on $\widetilde{S}$ are the preimage of the special points on $S$. So $\widetilde{S}$ is a decorated surface. Let $W$ be a web on a decorated surface $S$. Then $\widetilde{W} := p^{-1}(W)$ is a $\pi_1(S)$-equivariant web on $\widetilde{S}$.\cite{Goncharov}

\begin{figure}[htbp]
\centering
\includegraphics[width=0.2\textwidth]{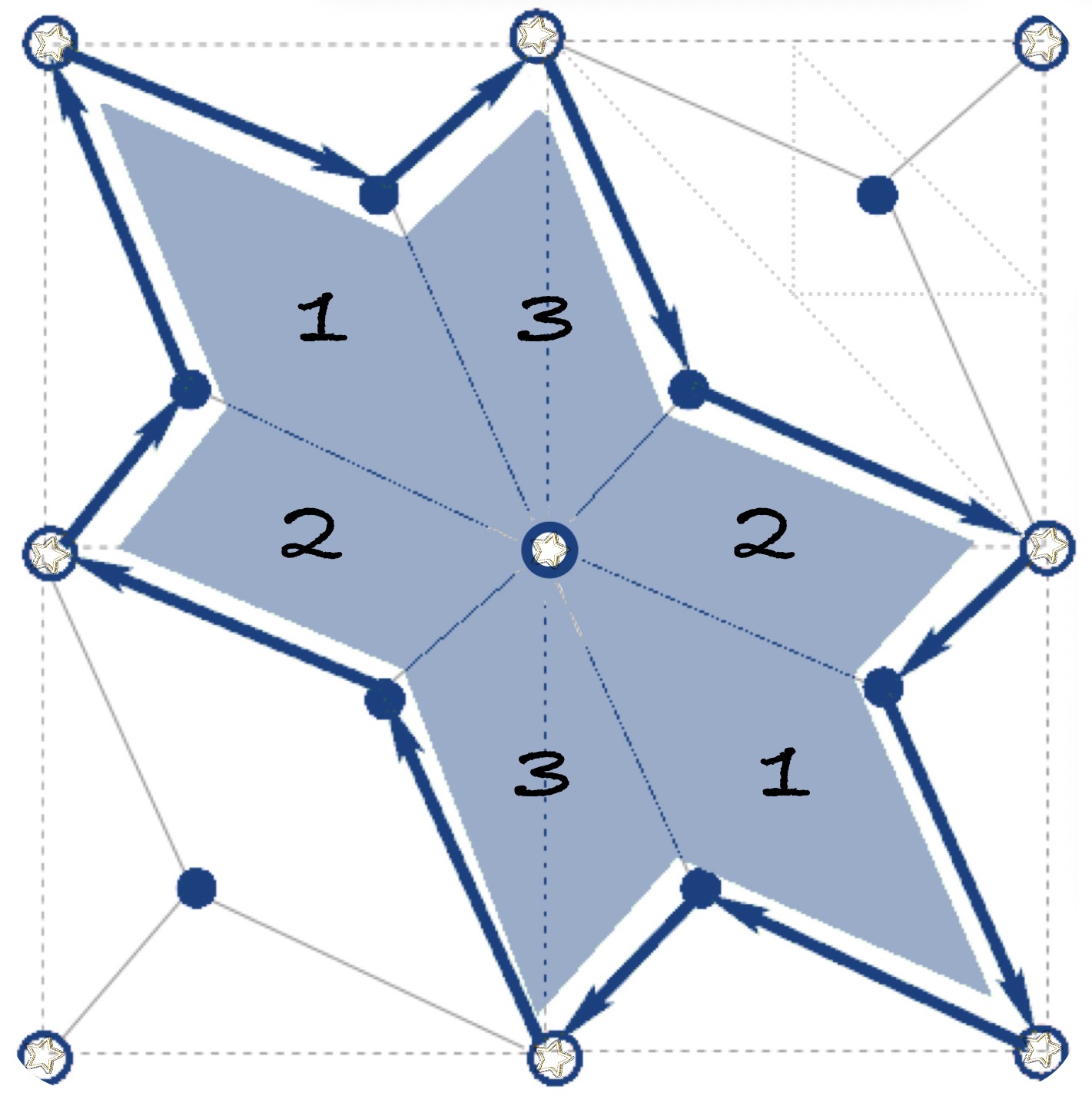}
\caption{Lift of an $A_1$-web on the universal cover of a torus with one special point.\cite{Goncharov}}
\label{fig:A1web}
\end{figure}

\begin{figure}[htbp]
\centering
\includegraphics[width=0.35\textwidth]{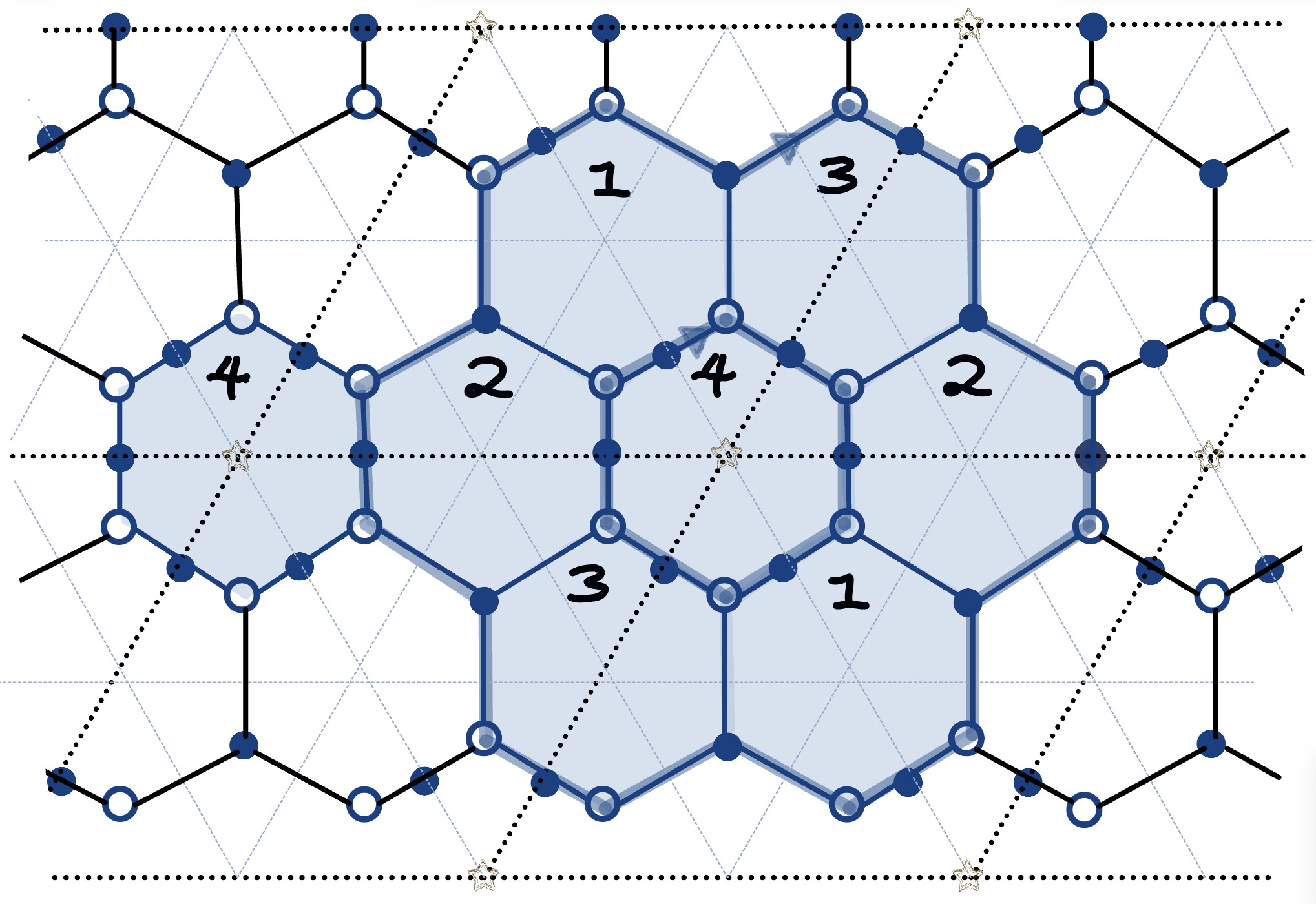}
\caption{Lift of an $A_1^*$-web on the universal cover of a torus with one special point.}
\label{fig:A1star}
\end{figure}

\subsubsection{Ideal Webs and Spectral Webs}

\begin{definition}[\cite{Goncharov}]
A web $W$ on a decorated surface $S$ is \emph{minimal} if its lift $\widetilde{W}$ on $\widetilde{S}$ consists of strands with no selfcrossings, and there are no parallel bigons. A bipartite graph on a surface $S$ is \emph{minimal} if the corresponding web is minimal.
\end{definition}

\begin{definition}[{\cite{Goncharov}}]
\label{def:ideal}
A \emph{strict ideal web} on a decorated surface $S$ is a minimal web satisfying the additional condition:
 \begin{itemize}
     \item The part of the surface on the right of each strand is a disc containing a single special point.
 \end{itemize}
An \emph{ideal web} is a web that becomes strict ideal on the universal cover of $S$.
\end{definition}

\begin{definition}[\cite{Goncharov}]
An ideal web $W$ on $S$ is an $A_m$ (resp. $A_m^*$)-web if its lift $\widetilde{W}$ on the universal cover $\widetilde{S}$ has m (resp. $m+1$) strands around each special point, and its rank is $m+1$.
\end{definition}
\noindent The collection of strands associated with a bipartite $A_m$ (resp. $A_m^*$)-graph forms an $A_m$ (resp. $A_m^*$)-web.

\begin{definition}[\cite{Goncharov}]
A web $W$ on a decorated surface $S$ is called \emph{spectral} if each strand $\gamma$ of $\widetilde{W}$ cuts the surface $\widetilde{S}$ into two components:
\begin{equation}
\label{eq:2.1}
\widetilde{S} - \gamma = \widetilde{S}_{\gamma} \cup \widetilde{S}_{\gamma}'.
\end{equation}
\end{definition}

We have the following relation of webs on a decorated surface:
\begin{equation}
\{A_m\text{-webs and }A_m^*\text{-webs}\}=\{\text{rank $m+1$ ideal webs}\}\subset\{\text{spectral webs}\}.
\end{equation}

\begin{definition}
A \emph{standard} $A_m$ or $A_m^*$-web is a web associated with a standard bipartite $A_m$ or $A_m^*$-graph.
\end{definition}

\begin{remark}
We may call a bipartite $A_m$ or $A_m^*$-graph an \emph{ideal} bipartite graph as its corresponding web is an ideal web.
\end{remark}

\begin{lemma}[{\cite{Goncharov}}]
On a decorated surface, there is a bijection between the isotopy classes of webs and connected bipartite graphs, whose faces are discs.
\end{lemma}

\begin{lemma}
\label{le:connectedness}
On a decorated surface, there is a bijection between isotopy classes of $A_m$ $($resp. $A_m^*)$-webs and bipartite $A_m$ $($resp. $A_m^*)$-graphs.
\end{lemma}

\begin{definition}[{\cite{Goncharov}}]
Given an ideal web $W$ on a decorated surface S, the \emph{codistance}
$\langle F,s\rangle$ between a face $F$ of the web $\widetilde{W}$ and
a special point s on $\widetilde{S}$ is the number of strands associated with s containing $F$.
\end{definition}

\begin{corollary}[\cite{Goncharov}]
Given an ideal web  $W$ on a decorated surface $S$, for any face $F$ of the web
$\widetilde{W}$ on $\widetilde{S}$,
\begin{equation}
\sum_s \langle F,s\rangle = \operatorname{rank}(W).
\end{equation}
\end{corollary}

\begin{example}
Let us consider bipartite $A_1$-graph $\Gamma$ and bipartite $A_1^*$-graph $\Gamma^*$ and the corresponding $A_1$-web $W$ and $A_1^*$-web $W^*$ on a torus T with one special point. Let $\widetilde{W}$ and $\widetilde{W^*}$ be the lifts of $W$ and $W^*$ respectively on the universal cover of T.

The lifts of the graphs and webs on the universal cover of the torus, with the special points the lifts of the special point on the torus, are shown in Figure~\ref{fig:A1web} and Figure~\ref{fig:A1star}, where the special points are the little stars on the corner of the fundamental domain. Face 4 in Figure~\ref{fig:A1star} is an external face containing a special point. Contracting face 4 we get Figure~\ref{fig:A1web}.

Here, $rank(W) = rank(W^*) = 2$. There is one ideal strand of $\widetilde{W}$ associated with a special point, bounding two copies of three internal faces; there are two ideal strands of $\widetilde{W^*}$ associated with a special point, one bounding two copies of three internal faces and an external face, another bounding an external face.
\end{example}

\subsection{Three Elementary Moves}

\textbf{Shrinking a 2-valent vertex} of a bipartite graph on a decorated surface $S$. It amounts to resolving a bigon move of a web.\cite{Goncharov}

\noindent\textbf{Expanding a k-valent vertex ($k\ge 4$)} of a bipartite graph on a decorated surface $S$. It amounts to creating a bigon move of a web.\cite{Goncharov}

\noindent A \textbf{\ttt move} of a bipartite graph on a decorated surface $S$ is a local transformation affecting four vertices of one color, and two of the other. The middle face F is called the face that the \ttt move is performed on.

A \ttt move amounts to a local transformation of a web, involving four zig-zag strands, also called a \ttt move. It \emph{decreases by one the codistances of F to two special points and increases by one the codistances of F to other two ones}.\cite{Goncharov}

\begin{figure}[htbp]
\centering
\includegraphics[width=0.4\textwidth]{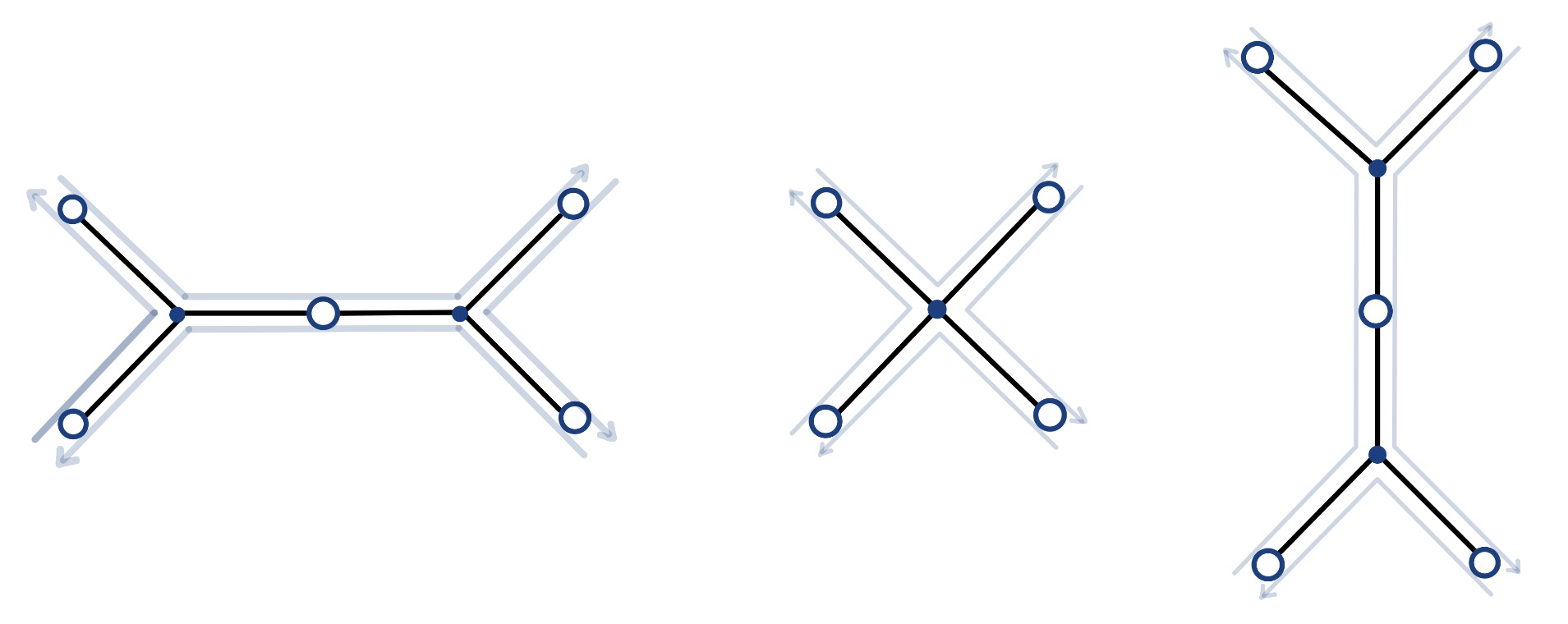}
\caption{Shrinking a 2-valent $\circ$-vertex, and then expanding a 4-valent $\bullet$-vertex\cite{Goncharov}.}
\label{fig:se}
\end{figure}

\begin{figure}[htbp]
\centering
\includegraphics[width=0.45\textwidth]{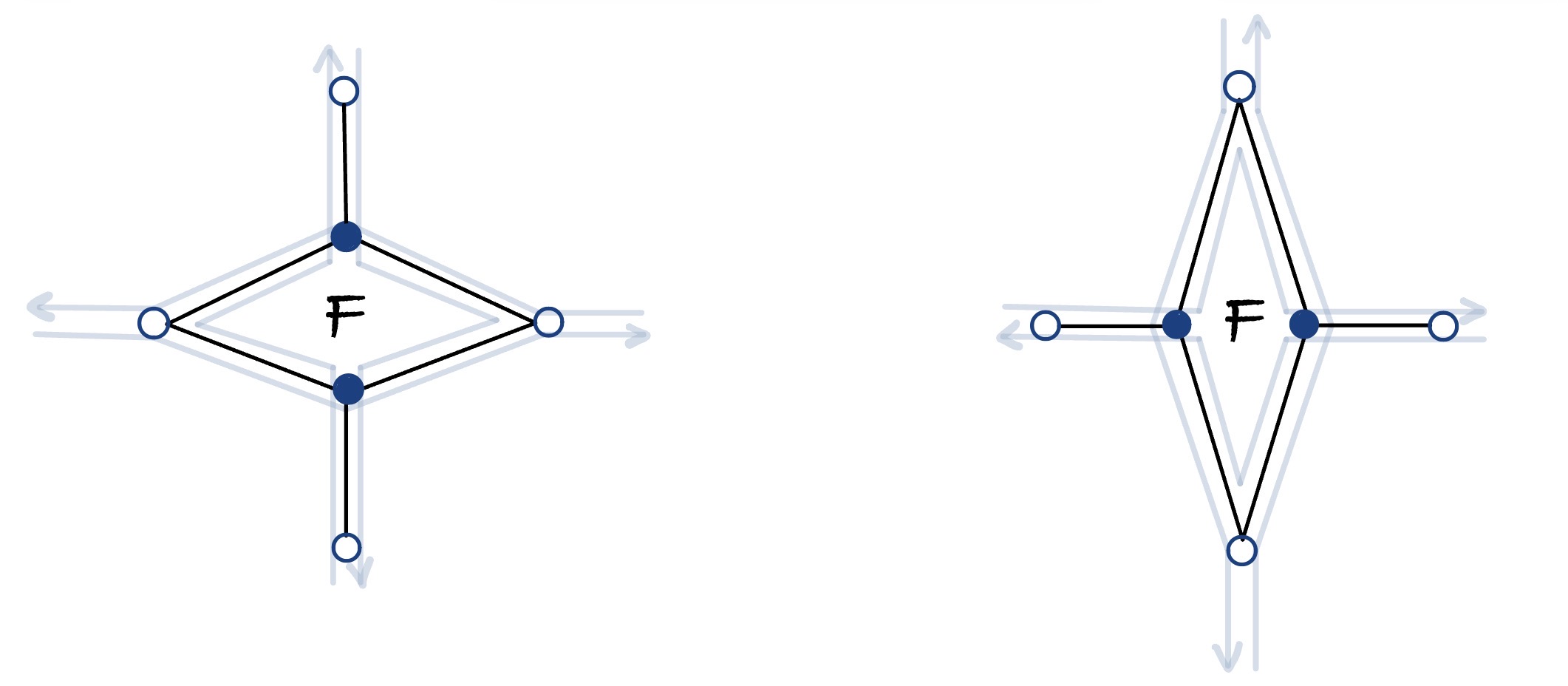}
\caption{A Two-by-two move.}
\label{fig:22move}
\end{figure}

\begin{remark}
Elementary moves send an ideal web to an ideal web and preserve the rank.
\end{remark}

\begin{remark}
Shrink/expand moves change only the bipartite graph while leaving the associated web isotopic;
whereas \ttt moves change the isotopy class of the web.

The \ttt move is the only type of elementary move that changes the codistances, hence, as we will see in Section 6, a \ttt move amounts to a cluster coordinate mutation on the corresponding moduli space; while shrink and expand moves do not change the coordinates.
\end{remark}

\begin{remark}
Let $\Gamma$ be a bipartite $A_m$ or $A_m^*$-graph on $S$. Let $\Gamma'$ be the bipartite graph after performing a \ttt move on $\Gamma$. Notice that $\Gamma$ and $\Gamma'$ are not isomorphic as graphs in general, but they are homotopy equivalent. Therefore, the spectral covers associated with $\Gamma$ and $\Gamma'$ are homotopy equivalent.
\end{remark}

\begin{remark}
For the bipartite graphs associated with a fixed ideal triangulation $T$, the graph $\Gamma_{A_m^*}(T)$ is reduced to $\Gamma_{A_m}(T)$
by shrinking all of its $2$-valent $\bullet$-vertices. Hence the $A_m^*$-web associated with $T$ reduces to the corresponding $A_m$-web by resolving all the bigons.

Accordingly, whenever the $A_m^*$-web differs from an $A_m$-web only by these removable bigons, we may work with the reduced $A_m$-web, and the results can be transferred to the $A_m^*$-web by the inverse bigon-creation moves.
\end{remark}

Any two ideal triangulations without self-folded triangles can be connected by a finite sequence of flips through ideal triangulations without self-folded triangles\cite{Goncharov}.
From now on, we will only consider triangulations without self-folded
triangles.

\begin{lemma}[{\cite{Goncharov}}]
There is a bijection between $A_1$-webs/bipartite $A_1$-graphs and ideal triangulations of S, such
that \ttt moves of ideal webs correspond to flips of triangulations.
Each face of the bipartite $A_1$-graph corresponds to an edge of the ideal triangulation.
\end{lemma}

\begin{figure}[htbp]
\centering
\includegraphics[width=0.45\textwidth]{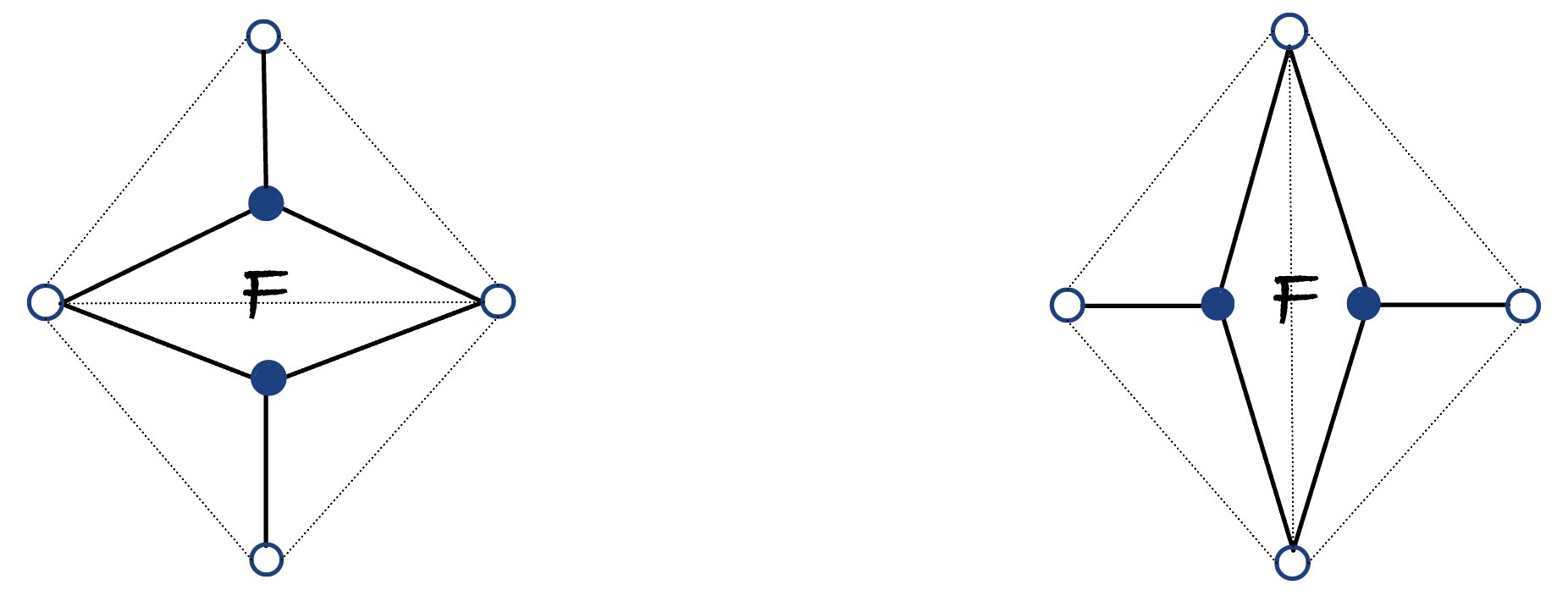}
\caption{An $A_1$-flip is the same thing as a \ttt move.\cite{Goncharov}}
\label{fig:A1flip}
\end{figure}

\begin{proposition}[{\cite{Goncharov}}]
\label{prop:am-flip-decomposition}
There is a decomposition of an $A_m$-flip into a composition of elementary
moves, including $\binom{m+2}{3}$ \ttt moves.
\end{proposition}

\begin{proposition}
\label{prop:connect}
Any two $A_m$-webs on a decorated surface $S$ are connected by a finite sequence of \ttt and shrink/expand moves.
\end{proposition}

\begin{proof}
By Lemma~\ref{le:connectedness}, the connectedness of $A_m$-webs under \ttt moves may be studied equivalently as connectedness of the corresponding $A_m$-bipartite graphs.

Given an ideal triangulation $T$ of $S$, let $\Gamma_{A_m}(T)$ be the standard $A_m$-bipartite graph associated with $T$. Suppose $\Gamma_1,\Gamma_2$ are $A_m$-bipartite graphs on $S$. We first show that they can be related to standard graphs associated with triangulations: $\Gamma_1
\sim
\Gamma_{A_m}(T_1)$, and $\Gamma_2
\sim
\Gamma_{A_m}(T_2),$
where $\sim$ denotes a finite sequence of the three elementary moves. Since the elementary moves are local changes, the local connectedness problem may be reduced to the case of a single ideal triangle and the matter of gluing.

An ideal triangle $t$ of an ideal triangulation $T$ on $S$ is called internal if it is contained in the interior of $S$; it is called external otherwise.
Let $t$ be an ideal triangle of a $m$-ideal triangulation $T$. We define the boundary and interior of $t$ as following:

Color all the sub-triangulations inside the $m$-ideal triangulation which has at least one boundary on the boundary of $t$ that is not on the boundary of $S$. The complement of the colored area is defined to be the interior of $t$.

\begin{figure}[htbp]
\centering
\includegraphics[width=0.2\textwidth]{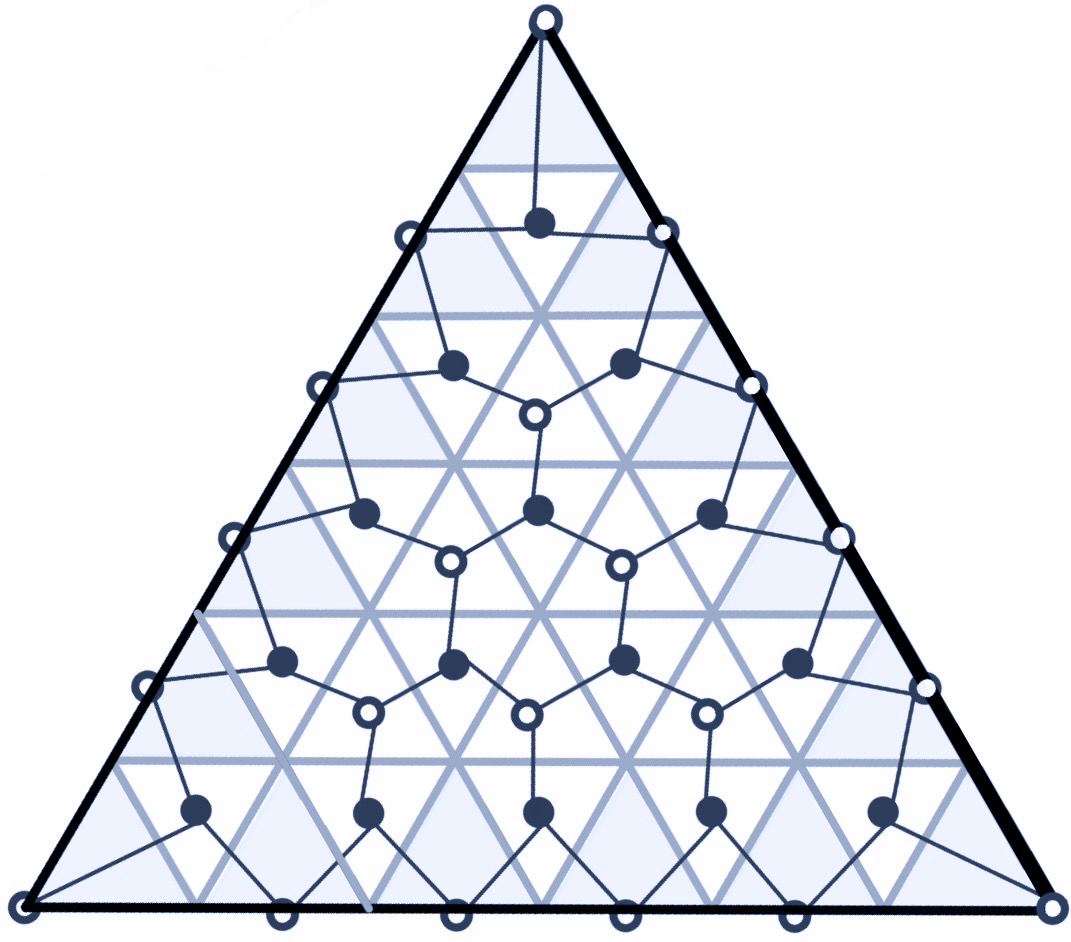}
\caption{The boundary and interior of an internal ideal triangle $t$.}
\label{fig:bdry}
\end{figure}

Any ideal triangle $t$ of $T$ has the same interior as the polygon version triangle, i.e., triangle whose boundaries are the boundary of the geometric object.

For a polygon, every ideal web is strict ideal. For a fixed rank, it is proven that any two $A_m$-webs on a polygon are connected by finitely many elementary moves~\cite{FG06}. Equivalently, any two $A_m$-bipartite graphs on a polygon are connected by finitely many elementary moves.

Therefore, given an $A_m$-bipartite graph $\Gamma$ on $S$, we can find an ideal triangulation $T$, such that on the interior of each ideal triangle $t$ of $T$, $\Gamma|_{\mathring{t}}
\sim
\Gamma_{A_m}(T)|_{\mathring{t}}$. Performing the elementary moves on $\Gamma$ on each each $\mathring{t}$, we obtain a new $A_m$-bipartite graph $\Gamma'$.

\begin{figure}[htbp]
\centering
\includegraphics[width=0.35\textwidth]{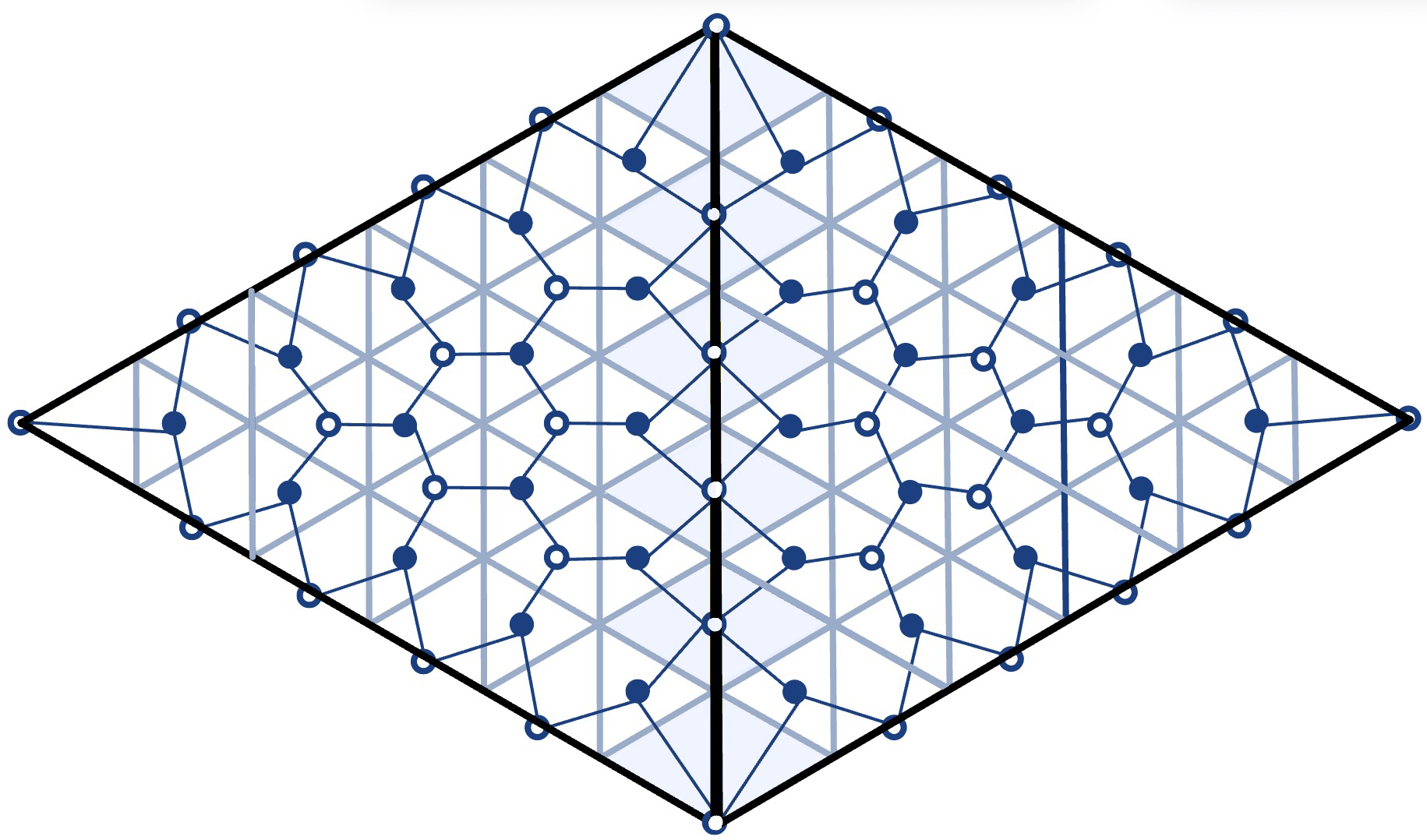}
\caption{Glue two ideal triangles.}
\label{fig:glue}
\end{figure}

When we glue two internal sides of two ideal triangles $t_a$ and $t_b$ of $T$, pairs of sub-triangles on the boundaries combines into the part of an ideal triangulation of an internal quadrilateral that is on the common edge of the two triangles, as shown in Figure~\ref{fig:glue}. We can see that having $\Gamma_{A_m}(T)|_{\mathring{t}}$ on the interior of each ideal triangle $t$, the parts of any $A_m$-bipartite graph $\Gamma$ on the common edge of the two ideal triangles $t_a$ and $t_b$ has the same structure as in the Figure~\ref{fig:22move}. Thus $\Gamma'$ and $\Gamma_{A_m}(T)$ only differ by finitely many elementary moves on the common edges of the ideal triangles, as shown in Figure~\ref{fig:6valent}. Hence $\Gamma\sim\Gamma_{A_m}(T)$.

\begin{figure}[htbp]
\centering
\includegraphics[width=1.0\textwidth]{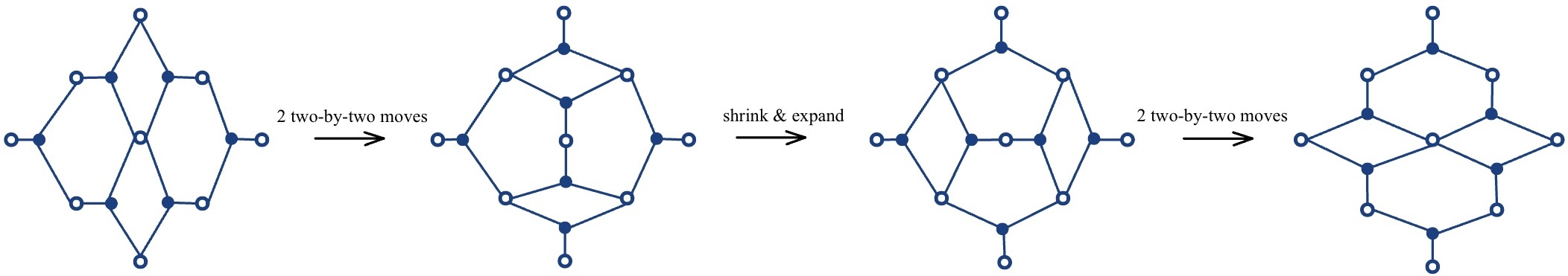}
\caption{Elementary moves on common edges of ideal triangles.}
\label{fig:6valent}
\end{figure}

Any two ideal triangulations $T_1$ and $T_2$ of $S$ are connected by finitely many flips\cite{Goncharov}. Each flip is supported on an ideal quadrilateral, hence locally on two triangles, and is realized by a finite sequence of elementary moves of the associated $A_m$-graphs. Therefore, $\Gamma_{A_m}(T_1)\sim\Gamma_{A_m}(T_2)$. Combining the three steps gives
$\Gamma_1
\sim
\Gamma_{A_m}(T_1)
\sim
\Gamma_{A_m}(T_2)
\sim
\Gamma_2.$

Hence connectedness of the $A_m$-bipartite graphs implies connectedness of the $A_m$-webs:
$\Gamma_1\sim\Gamma_2
\to
W_{\Gamma_1}\sim W_{\Gamma_2}.$
\end{proof}

\subsection{Conjugate Ribbon Graphs and Spectral Covers}

\subsubsection{Construction of Spectral Covers}
\label{sssec:231}

We recall the construction of the topological spectral cover from a spectral web introduced in \cite{Goncharov}.

Let $W$ be a \emph{spectral web} on a decorated surface $S$, and let $\widetilde{\Gamma}$ be the corresponding bipartite graph on the universal cover $\widetilde{S}$ associated with the web $\widetilde{W}$. Each zig-zag strand $\gamma$
on $\widetilde{\Gamma}$ bounds the unique component $\widetilde{S}_{\gamma}$
on $\widetilde{S}$, as in (\ref{eq:2.1}). We attach the component $\widetilde{S}_{\gamma}$ to the graph $\widetilde{\Gamma}$, viewed as an abstract ribbon bipartite graph, along the strand $\gamma$. Performing this procedure for
all zig-zag strands, we get a surface $\widetilde{\Sigma}_{\widetilde{W}}$.
By construction, each component $\widetilde{S}_{\gamma}$ becomes a part of
$\widetilde{\Sigma}_{\widetilde{W}}$, denoted by $F_{\gamma}$:
\begin{equation}
\widetilde{\Sigma}_{\widetilde{W}}
=
\widetilde{\Gamma}\cup_{\gamma_i} F_{\gamma_i}.
\end{equation} 

Using the tautological isomorphisms $F_{\gamma}\xrightarrow{\sim}\widetilde{S}_{\gamma}$,
we get a map
$\widetilde{\pi}:
\widetilde{\Sigma}_{\widetilde{W}}
=\widetilde{\Gamma}\cup_{\gamma_i}F_{\gamma_i}
\to
\widetilde{S}$.

The fundamental group $\pi_1(S)$ action
commutes with the projections on $\widetilde{S}$ and $\widetilde{\Sigma}_{\widetilde{W}}$. So the group $\pi_1(S)$ acts by automorphisms
of the cover.
The \emph{spectral cover} is $\Sigma_{W}
:=\widetilde{\Sigma}_{\widetilde{W}}/\pi_1(S)$.

Then we have a commutative diagram, with the horizontal maps given by the
quotients of the free action of $\pi_1(S)$:
\begin{equation}
\begin{tikzcd}
\widetilde{\Sigma}_{\widetilde{W}}
\arrow[r]
\arrow[d, "\widetilde{\pi}"']
&
\Sigma_{W}
:=
\widetilde{\Sigma}_{\widetilde{W}}/\pi_1(S)
\arrow[d, "\pi"]
\\
\widetilde{S}
\arrow[r]
&
S:=\widetilde{S}/\pi_1(S).
\end{tikzcd}
\end{equation}

Therefore we get the \emph{spectral cover map}:
$\pi:\Sigma_{W}\to S$.
The special points on the spectral cover $\Sigma_{W}$ are the preimages of the special points on $S$.

\begin{lemma}[\cite{Goncharov}]
Given any ideal web $W$ on a decorated surface $S$, $\deg(\widetilde{\pi})=\deg(\pi)$.
\end{lemma}

\begin{definition}[\cite{Goncharov}]
Given any spectral web $W$ on a decorated surface $S$, $\rank(W):=\deg(\pi)$.
\end{definition}

\begin{definition}
A spectral web $W$ on a decorated surface $S$ is called \emph{singular} if the associated spectral cover $\Sigma_W$ has nodal singularities.
\end{definition}

\begin{remark}
Ideal webs are non-singular.
\end{remark}

\begin{figure}[htbp]
\centering
\includegraphics[width=1.0\textwidth]{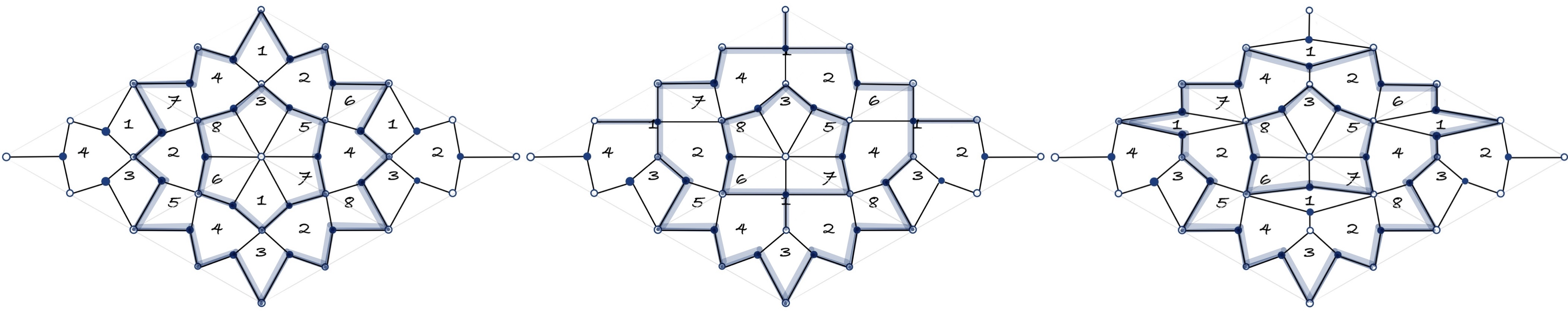}
\caption{An example of a singular web formed from collision of two 3-valent $\bullet$-vertices. The 4-valent $\bullet$-vertex in face 1 corresponds to the nodal singularity. The strand $\gamma$ of $\widetilde{W}$ that passes the 4-valent $\bullet$-vertex in face 1 on the universal cover $\widetilde{S}$ is an infinite line.
$\widetilde{S}_{\gamma}$ and $\widetilde{S}_{\gamma}'$ in Equation~\ref{eq:2.1}: $\widetilde{S} - \gamma = \widetilde{S}_{\gamma} \cup \widetilde{S}_{\gamma}'$ are two infinite areas.}
\label{fig:singular}
\end{figure}

Let
$f:X\to Y$
be a finite morphism of smooth curves.
A point $p\in X$ is a \emph{ramification point} if its \emph{ramification index} satisfies
$e_p\ge 2.$
Equivalently, in suitable local coordinates near $p$ and $f(p)$,
$f(t)=t^{e_p}.$
A point $q\in Y$ is a \emph{branch point} if
$q=f(p)$
for some ramification point $p\in X$.

A point $p\in X$ is \emph{simply ramified} if
$
e_p=2.
$
Equivalently, locally,
$f(t)=t^2.$
A cover is \emph{simply branched} if, over every branch point
$q\in Y$, there is exactly one ramification point
$p\in f^{-1}(q)$, and that point is simply ramified.

Hence a fibre over a \emph{simple branch point} has multiplicities
$2,1,\ldots,1,$
and the corresponding local monodromy is a single transposition. If two distinct simply ramified points lie over the same
$q\in Y$, then the individual points are simply ramified, but
the cover is not simply branched at $q$.

\begin{theorem}[{\cite{Goncharov}}]
Let $W$ be a non-singular spectral web on a decorated surface $S$. The spectral cover map $\pi$ is ramified only at the $\bullet$-vertices of the
associated bipartite graph $\Gamma$. The ramification index at a $\bullet$-vertex is its valency minus one.
\end{theorem}

\begin{remark}
In particular, the spectral cover map is unramified over the special points.
\end{remark}

\begin{remark}
Let $W$ be an $A_m$-web constructed from an $A_m$-sub-triangulation on an decorated surface $S$. Let $\Gamma$ be the associated bipartite $A_m$-graph. Then $\bullet$-vertices in $\Gamma$ are 3-valent, so the spectral cover map is simply branched at the $\bullet$-vertices of $\Gamma$.
\end{remark}

\begin{definition}
An $A_{m}$-web $W$ is called \emph{generic} if every
$\bullet$-vertex of its associated bipartite graph is 3-valent.
Equivalently, the associated spectral cover $\pi_W:\Sigma_W\to S$ is simply branched.
\end{definition}

\begin{lemma}
Every $A_{m}$-web is related by a finite sequence of shrink/expand moves to a generic $A_{m}$-web.
More precisely, a $k$-valent $\bullet$-vertex, $k\ge 4$, can be replaced by $k-2$ 3-valent $\bullet$-vertices by $k-3$
successive expansion moves.
\end{lemma}

\begin{proof}
An expansion may split a $k$-valent $\bullet$-vertex into a 3-valent $\bullet$-vertex and a $(k-1)$-valent
$\bullet$-vertex, separated by a $2$-valent $\circ$-vertex. Repeating this operation gives
\begin{equation}
k\mapsto 3+(k-1)
\mapsto 3+3+(k-2)
\mapsto\cdots
\mapsto
\underbrace{3+\cdots+3}_{k-2}.
\end{equation}
Thus after $k-3$ expansions all resulting $\bullet$-vertices
are 3-valent.

Moreover, the ramification contribution is preserved:
$k-2=\sum_{i=1}^{k-2}(3-2)$.
Hence the operation resolves one higher ramification point of
order $k-2$ into $k-2$ simple ramification points.
\end{proof}

\begin{figure}[htbp]
\centering
\includegraphics[width=1.0\textwidth]{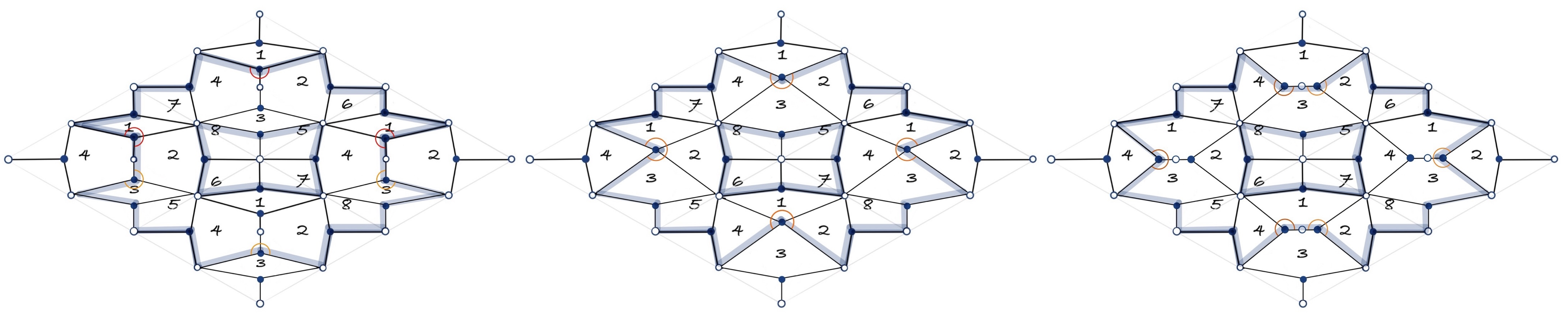}
\caption{An example of two simple ramification points coalescing into a double ramification point (corresponding to the $4$-valent $\bullet$-vertex) through shrinking, and a double ramification point splitting to two simple ramification points through expanding.}
\label{fig:3rami}
\end{figure}

Let $W$ be a generic $A_m$-web on a decorated surface $S$, and let $\Gamma$ be the corresponding bipartite $A_m$-graph. If we allow distinct vertices of $\Gamma$ to be placed at the same point on $S$ during an isotopy, then the associated spectral cover $\pi_W:\Sigma_W\to S$ is simply ramified but not simply branched, since two distinct simple ramification points can have the same branch value. Such webs are not generic.

\begin{remark}
Two-by-two moves have orientations at the level of the spectral cover. The orientation of a \ttt move and the orientation of the corresponding tropical mutation on the $K_0$ groups constructed from the associated quivers with potential are chosen in accordance with each other. We will formally define the orientation in Section 3 Definition~\ref{def:orittt}.
\end{remark}

\begin{remark}
Two-by-two moves can be performed when rank $m\ge 2$.
Shrinking a 2-valent $\circ$-vertex connecting two 3-valent $\bullet$-vertices and then expanding a 4-valent $\bullet$-vertex into another two 3-valent $\bullet$-vertices can be performed when rank $m\ge 3$.
\end{remark}

\begin{remark}
Two-by-two moves have orientations at the level of the spectral cover. This means \ttt moves are not their own inverse in that case, as opposed to \ttt moves on webs/bipartite graphs. So performing two \ttt moves in the same orientation results in a non-trivial automorphism on the lattice $\Lambda_\Gamma^\circ=\ker(H_1(\mathring{\Sigma};\mathbb Z)\to H_1(\mathring{S};\mathbb Z))$, which we will discuss in Section 3. This also corresponds to a non-trivial monodromy around $\mathcal D_{bdry,C,P}$ on the Hitchin fibration, which we will discuss in Section 5. It also results in a non-trivial cluster mutation on associated regular functions, which we will discuss in Section 6.

In contrast, the whole process of  "shrinking a 2-valent $\circ$-vertex connecting two 3-valent $\bullet$-vertices and then expanding a 4-valent $\bullet$-vertex into another two 3-valent $\bullet$-vertices" do not have orientations at the level of the spectral cover. So shrink \& expand moves are their own inverses. This means performing two successive shrink \& expand moves and eventually come back to the initial web induce trivial automorphism on the the lattice $\Lambda_\Gamma^\circ$, which corresponds to a trivial monodromy around $\mathcal D_{cau,C,P}$ on the Hitchin fibration, and a trivial cluster mutation on associated regular functions on the moduli space of decorated local systems.

If we allow overlapping of $\bullet$-vertices during isotopy of initial $A_m$-webs. The non-generic middle stage web corresponds to the locus $\mathcal D_{Max,C,P}$ in the Hitchin base, where two distinct simple ramification points have the same branch value. The monodromy around $\mathcal D_{Max,C,P}$ is trivial since the isotopy class of the web remains unchanged during the process. We will discuss these further in Section 5.
\end{remark}

\subsubsection{Bipartite Ribbon Graphs}
A \emph{ribbon graph} is a graph
with a cyclic order of the edges at every vertex. A \emph{face path}
on a ribbon graph $\Gamma$ is an oriented path on the graph which turns
right at every vertex. A closed face path is called a \emph{face loop}.\cite{Goncharov}

A graph embedded into an oriented surface $S$ has an induced ribbon structure. Its faces $F_{\varphi_i}$ are the areas on the surface bounded by the union of the boundary of $S$ and face paths $\varphi_i$. We may assume that its external vertices are on the boundary of $S$.\cite{Goncharov}

The \emph{conjugate bipartite ribbon graph} $\Gamma^\vee$ is given by reversing the cyclic
orders at all $\bullet$-vertices of a bipartite ribbon graph $\Gamma$.\cite{Goncharov}

$\Gamma$ and $\Gamma^\vee$
are isomorphic as graphs, differing only in their ribbon structures.
There is a bijection:
\begin{equation}
\label{eq:conj}
\{\text{zig-zag strands } \gamma_i \text{ on } \Gamma\}
\longleftrightarrow
\{\text{face paths on } \Gamma^\vee\}.
\end{equation}
The orientations of zig-zag strands match the orientations of the face
paths on $\Gamma^\vee$.\cite{Goncharov}

A \emph{completed conjugate surface $\mathbf{\Sigma}$} of the bipartite ribbon graph $\Gamma$ is constructed by attaching faces $F_{\varphi_i}$ of the conjugate ribbon graph $\Gamma^\vee$ to the union of the boundary of $S$ and face paths $\varphi_i$:
$\mathbf{\Sigma}
=\Gamma^\vee\cup_{\varphi_i}F_{\varphi_i}$.\cite{Goncharov}

\begin{proposition}[{\cite{Goncharov}}]
Given a spectral web $W$ on a decorated surface $S$, the spectral cover $\Sigma_{W}$ is canonically identified with the completed conjugate surface $\mathbf{\Sigma}$ of the bipartite ribbon graph $\Gamma$ associated with $W$:
$\mathbf{\Sigma}=\Sigma_{W}$.
\end{proposition}

\begin{figure}[htbp]
\centering
\includegraphics[width=0.25\textwidth]{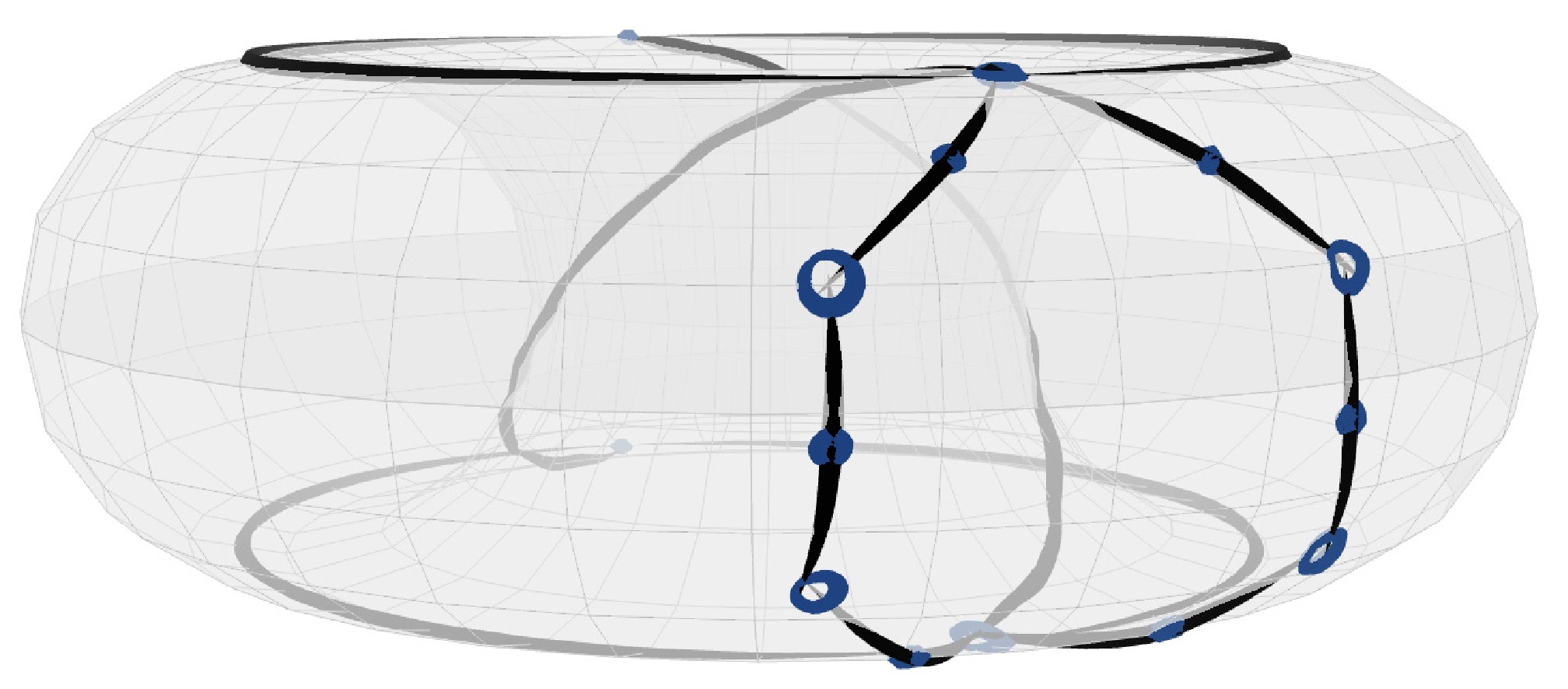}
\caption{Bipartite $A_1^*$-graph on a torus with one special point.}
\label{fig:genus1}
\end{figure}

\begin{figure}[htbp]
\centering
\includegraphics[width=0.45\textwidth]{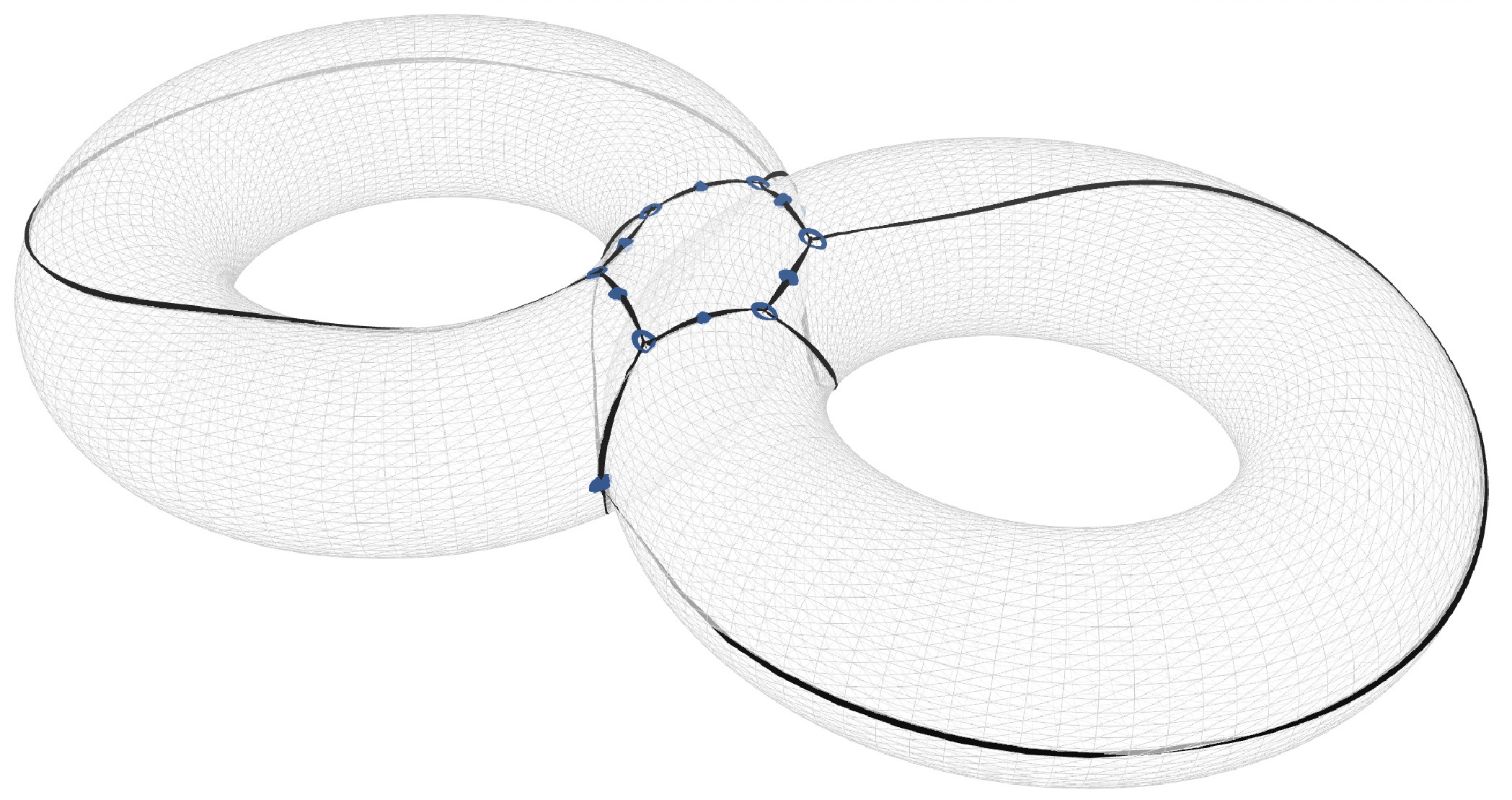}
\caption{Lift of the $A_1^*$-web on a spectral cover of a torus with one special point.}
\label{fig:genus2}
\end{figure}

\medskip
Let $W$ be an ideal web on a decorated surface $S$. Let $\Gamma$ be the bipartite graph associated with $W$. Let $\varphi$ be a face path on the conjugate bipartite ribbon graph $\Gamma^\vee$. Denote by $\Sigma_{\varphi}$ the unique domain obtained by cutting $\Sigma_{W}$ along
$\varphi$, such that the orientation of the path $\varphi$ coincides with  the orientation of the boundary of $\Sigma_{\varphi}$. According to Definition~\ref{def:ideal}, Equation~\ref{eq:conj} and the construction of the spectral cover $\Sigma_{W}$, each face path on
$\Gamma^\vee$ surrounds one special point on $\Sigma$, thus there is a bijection\cite{Goncharov}
\begin{equation}
\begin{aligned}
\{\text{face paths on }\Gamma^\vee\}
&\longleftrightarrow
\{\text{special points of }\Sigma_{W}\}.
\\
\{\text{face path }\varphi\}
&\longleftrightarrow
\{\text{the unique special point }s_{\varphi}
\text{ inside of the domain }\Sigma_{\varphi}\}.
\end{aligned}
\end{equation}

Let $\Gamma$ be a bipartite $A_m$ or $A_m^*$-graph on a decorated surface $S$.
Since the abstract graphs $\Gamma$ and $\Gamma^\vee$ coincide, and the spectral cover $\Sigma$ is homotopy equivalent to $\Gamma^\vee$, $\Sigma$ is hence homotopy equivalent to $\Gamma$. The \emph{completed spectral cover} $\overline{\Sigma}$ is given by filling the faces on $\Sigma$ by discs.\cite{Goncharov}

Let $\Gamma_{A_{m}^{*}}$ be an $A_{m}^{*}$-bipartite graph on a decorated surface $S$, and let $\Gamma_{A_{m}}$ be the associated $A_{m}$-bipartite graph on $S$.
The special points of $S$ form a subset of the set of $\circ$-vertices
of $\Gamma_{A_{m}}$, while the $\Gamma_{A_{m}^{*}}$ is a graph on the punctured surface $\mathring{S} := S \setminus \{\text{special points}\}$. Moreover,\cite{Goncharov}
\begin{equation}
\{\text{Faces of the graph }\Gamma_{A_{m}}\}
=
\{\text{Non-special faces of the graph }\Gamma_{A_{m}^{*}}\}.
\end{equation}

\begin{remark}
Let $\Gamma$ be a bipartite graph and $W$ be the associated web on a decorated surface $S$. Let $\widetilde{\Gamma}$ be the lift of $\Gamma$ and $\widetilde{W}$ be the lift of $W$ on the universal cover $\widetilde{S}$ of $S$. The spectral cover $\Sigma$ is constructed by attaching discs on $\widetilde{S}$ to the strands of $\widetilde{W}$, and then take the quotient.

Thus we consider the lift of $\Gamma$ to each disc of $\Sigma$ as the corresponding part of $\widetilde{\Gamma}$ on $\widetilde{S}$, and hence the lift of $\Gamma$ to $\Sigma$ is the connected graph we get after attaching discs to strands and taking the quotient. The lift of $W$ to $\Sigma$ is considered to be the quotient of $\widetilde{W}$, which is the conjugate ribbon graph $\Gamma^\vee$.
\end{remark}

\begin{example}
Let us continue to consider bipartite $A_1$-graph $\Gamma$ and bipartite $A_1^*$-graph $\Gamma^*$ and the corresponding $A_1$-web $W$ and $A_1^*$-web $W^*$ on a torus T with one special point. Let $\widetilde{W}$ and $\widetilde{W^*}$ be the lifts of $W$ and $W^*$ respectively on the universal cover of T. Let $\Sigma$ and $\Sigma^*$ be the corresponding spectral covers.

Here, $\rank(W) = \rank(W^*) =$ degree of the spectral cover maps $=$ number of copies of the faces bounded by all strands of $\widetilde{W}$ or $\widetilde{W^*}$ associated with any special point $=$ 2.

The spectral covers $\Sigma$ and $\Sigma^*$ are genus 2 surfaces with two special points. This can be obtained by the Riemann-Hurwitz formula, Euler's formula for graphs, or cutting-and-pasting. 

Figure~\ref{fig:genus1} is the torus T with one special point, with $\Gamma$ on it. Figure~\ref{fig:genus2} is $\Sigma^*$ with the lift of $W^*$, which is the conjugate ribbon graph $\Gamma^\vee$. The graphs in Figure~\ref{fig:genus1} and Figure~\ref{fig:genus2} are isomorphic. There are three internal faces and one external face in Figure~\ref{fig:genus1}, while there are only two external faces in Figure~\ref{fig:genus2}. If we consider the special points in Figure~\ref{fig:genus2} as punctures, then the surface deformation retracts to the graph.
\end{example}

\subsection{Properties of Topological Spectral Covers}

\subsubsection{Monodromy Representations}
Let $S$ be a closed genus-$g$ decorated surface with a finite set $P$ of $n$ special points, and let $\mathring{S}= S \setminus P$. Let
$\pi:\Sigma\to S$
be a degree-$m$ ramified cover with a finite set of distinct branch points $B=\{q_1,\dots,q_r\}$, $B\cap P=\varnothing$. Assume $2g-2+n>0$.

\begin{definition}[\cite{Hatcher}]
Choose a base point $x_0\in \mathring{S}\setminus B$
and label the fiber $\pi^{-1}(x_0)=\{1,\dots,m\}.$
Path lifting defines the monodromy representation
\begin{equation}
\rho_\pi:
\pi_1(\mathring{S}\setminus B,x_0)\to S_m.
\end{equation}
Its image $\operatorname{Mon}(\pi):=\operatorname{im}\rho_\pi$
is the monodromy group of the covering. Changing the labeling of the fiber conjugates $\rho_\pi$ by an
element of $S_m$.
\end{definition}

\begin{lemma}[\cite{Hatcher}]
The covering $\Sigma$ is connected if and only if $\operatorname{im}\rho_\pi$ acts transitively on $\{1,\dots,m\}$.
\end{lemma}

\begin{lemma}[\cite{EisenbudHarris}]
If $\pi$ is simply branched, then the monodromy of a sufficiently small
positively oriented loop $\gamma_i$ around $q_i$ is a transposition: $\rho_\pi(\gamma_i)=\tau_i\in S_m.$
\end{lemma}

If $ S$ is a closed genus-$g$ surface and $B=\{q_1,\dots,q_r\},$
then
\begin{equation}
\pi_1( S\setminus B)
=
\left\langle
a_1,b_1,\dots,a_g,B_{G,C,P},\gamma_1,\dots,\gamma_r
\ \middle|\
\prod_{i=1}^g[a_i,b_i]\gamma_1\cdots\gamma_r=1
\right\rangle.
\end{equation}

So a connected simply branched degree-$m$ cover is represented by $(A_1,B_1,\dots,A_g,B_{G,C,P},\tau_1,\dots,\tau_r)$
with $A_i,B_i\in S_m$, $\tau_j$ a transposition,
such that
\begin{equation}
\prod_{i=1}^g[A_i,B_i]\tau_1\cdots\tau_r=1,
\end{equation}
and the generated subgroup acts transitively. For a fixed branch configuration, covering isomorphism corresponds to simultaneous conjugation by elements in $S_m$.
Allowing the branch points to move introduces the surface braid action.

\begin{proposition}
Let $t$ be one ideal triangle in an ideal triangulation $T$ of $S$. Let $W|_t$ be a generic $A_{m-1}$-web $W$ compatible with $T$ on $S$ restricted to $t$, and $\Gamma|_t$ be the corresponding bipartite $A_{m-1}$-graph $\Gamma$ on $S$ restricted in $t$. Then the spectral cover
$\pi|_t:\Sigma|_t\to t$
is connected, and the monodromy group of $\pi|_t:\Sigma|_t\to t$
is $S_m$.
\end{proposition}

\begin{proof}
$\Gamma|_t$ is connected. Passing to the conjugate ribbon structure does not change its
underlying abstract graph, hence the conjugate graph remains connected.
The spectral cover is obtained by attaching the face discs to this
connected graph. Attaching discs cannot disconnect it, so
$\Sigma|_t$ is connected.

Because $t$ is a disc, the fundamental group of $t\setminus \text{\{branch points\}}$ is generated by meridians around the branch points. Since every branch point is simple, their monodromies
$\tau_1,\dots,\tau_s$ are transpositions.
If $G\subseteq S_m$
is generated by transpositions and acts transitively on $\{1,\dots,m\}$, then $G=S_m.$ Since $\Sigma|_t$ is connected, the monodromy action on the $m$ sheets is transitive, so $\operatorname{Mon}(\pi|_t) = S_m$.
\end{proof}

\begin{corollary}
\label{cor:fullmono}
Let $W$ be a generic $A_{m-1}$-web on $S$, and let $\pi_W:\Sigma_W\to S$ be the spectral cover. Then
$\operatorname{Mon}(\pi_{W})=S_m.$
\end{corollary}

\begin{proof}
Let $T$ be an ideal triangulation of $S$ compatible with $W$. Choose one ideal triangle $t$ of $T$. The corresponding local monodromy is a subgroup of the monodromy of the global cover. So
$S_m=
\operatorname{Mon}(\pi_{W|_t})
\subseteq
\operatorname{Mon}(\pi_{W})
\subseteq
S_m.
$
Hence $
\operatorname{Mon}(\pi_{W})=S_m.
$
\end{proof}

\subsubsection{Spectral Covers as Connected Sums}
Denote by $S$ a genus-$g_S$ oriented topological surface with $n$ special points. Assume $2g_S-2+n>0$. Let $W$ be a generic $A_{m-1}$-web corresponding to an ideal
triangulation $T$ of $S$. Denote by $\Gamma$ the corresponding bipartite $A_{m-1}$-graph on $S$, and by $\Sigma$ the corresponding spectral cover. Let $g_\Sigma$ be the genus of $\Sigma$.

The spectral cover map $\pi$ is a ramified $m:1$ map. It is unramified over special points. The spectral cover $\Sigma$ related to the bipartite graph $\Gamma$ is a oriented surface with $mn$ special points.

Recall that the branch points of the spectral cover map $\pi$ are the $\bullet$-vertices of a spectral web $W$ on $S$ of valency $\geq 3$: a $\bullet$-vertex of valency $v$ has ramification index $v-1$.

The maximal valency of the $\bullet$-vertices of $\Gamma$ is 3 by the construction of $\Gamma$. The number of 3-valent $\bullet$-vertices in a single triangle is $\binom{m}{2}$, while the number of ideal triangles is $2(2g_S-2+n)$. So the number of 3-valent $\bullet$-vertices is $\binom{m}{2}2(2g_S-2+n) = m(m-1)(2g_S-2+n)$. Therefore, there are $m(m-1)(2g_S-2+n)$ branch points, each having ramification index 2\cite{Goncharov}.

By the Riemann-Hurwitz formula, we have
\begin{align}
\label{eq:RH}
\chi(\Sigma) &= \deg(\pi)\chi(S) - m(m-1)(2g_S-2+n)
\\
&= m^2(2-2g_S)-m^2n+mn.
\notag
\end{align}
Hence
\begin{equation}
\label{eq:genus}
g_\Sigma = 1 - \frac{m}{2}(m(2-2g_S)-mn+n).
\end{equation}

\begin{lemma}
The spectral cover $\Sigma$ can be equivalently constructed by making branch cuts and gluing copies of $S$. Namely, $\Sigma$ can be constructed as a connected sum of m copies of S, with $2g_S-2+n$ connected tubes in between each pair of copies of $S$. So the number of total connected tubes is $\binom{m}{2}(2g_S-2+n)$.
\end{lemma}

\begin{remark}
The symmetric group of a such spectral cover viewed as a connected sum of the original surface is $S_m$, which is the Weyl group of $Gl_m$ or $SL_m$.
\end{remark}

We verify that the Euler characteristic of $\Sigma$ as a connected sum is the same as Equation~\ref{eq:RH}
\begin{equation}
\chi(\Sigma) = m \cdot \chi(S) - 2 \cdot \binom{m}{2}(2g_S-2+n).
\end{equation}

\begin{figure}[htbp]
\centering
\includegraphics[width=0.4\textwidth]{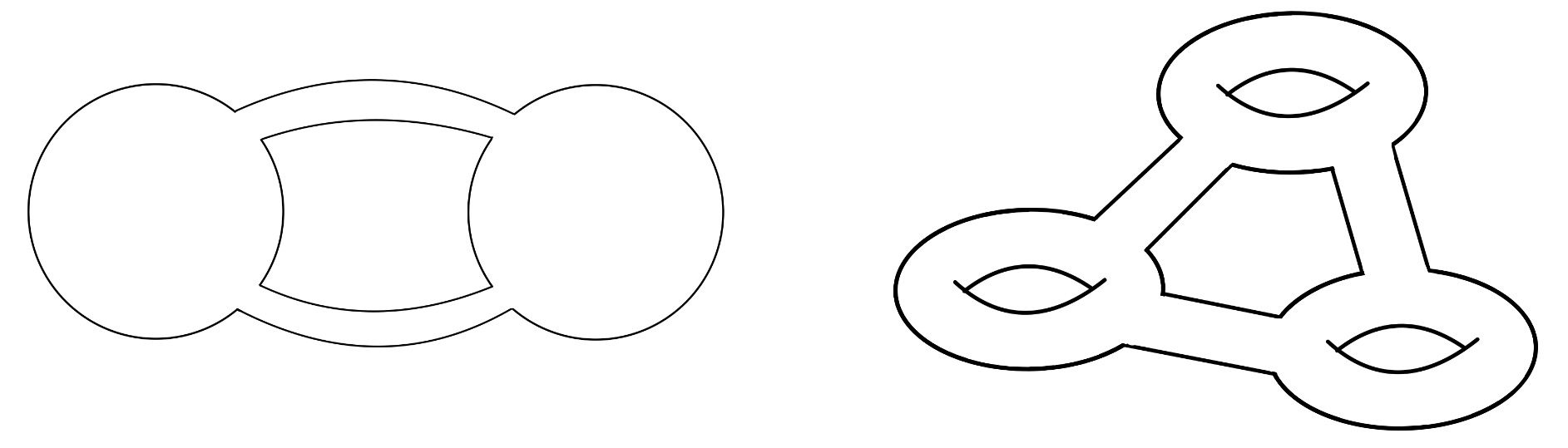}
\caption{The spectral cover of a 4-punctured sphere with an $A_1$-web (left) and the spectral cover a 1-punctured torus with an $A_2$-web (right).}
\label{fig:consum}
\end{figure}

\begin{example}
Consider an $A_2^*$-web on a torus with one special point. The lift of an $A_2^*$-web on its universal cover is shown in Figure~\ref{fig:A2star}. The three highlighted blue strands are the strands associated with the middle lift of the special point, which is the yellow star in the middle of the figure. The spectral cover map is 3:1, with 6 simple branch points. The spectral cover has genus 4, as shown in Figure~\ref{fig:consum}.
\end{example}

\begin{figure}[htbp]
\centering
\includegraphics[width=0.4\textwidth]{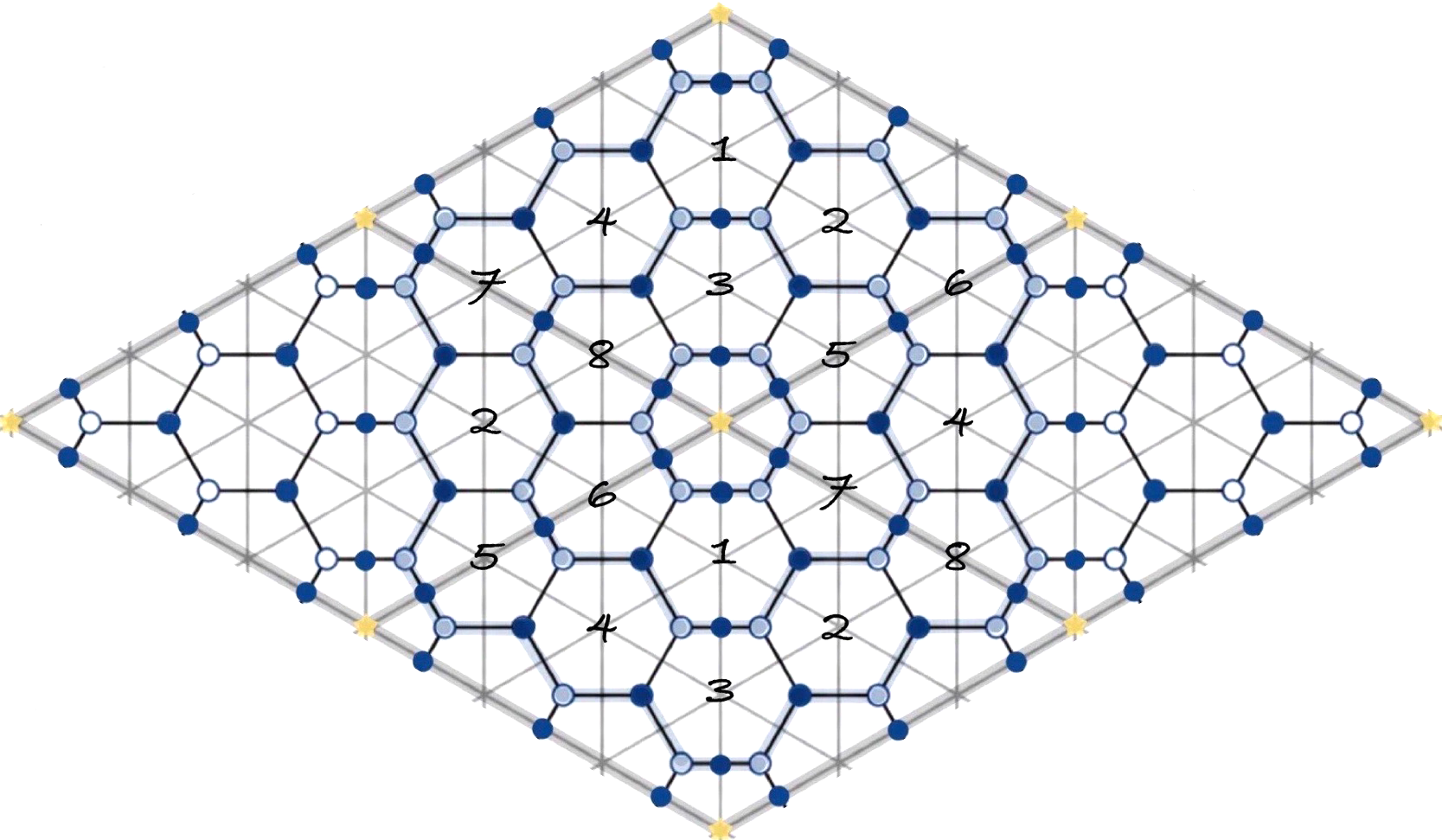}
\caption{Lift of an $A_2^*$-web on the universal cover of a torus with one special point.}
\label{fig:A2star}
\end{figure}

\begin{example}
The spectral cover of a four-punctured sphere with an $A_1$-web has genus one. The spectral cover map is 2:1 and simply branched.
\end{example}

\newpage
\mysection{Local Operations on Spectral Covers that Respect Elementary Moves}

In this section, we will first recall the quivers associated with bipartite graphs on decorated surfaces introduced in \cite{Goncharov}.
Then we will find a local operation on the spectral cover that gives a discontinuous bijection on the spectral cover, which realizes the change of the spectral cover under a \ttt move.
We will discuss three elementary moves at the level of the spectral cover and find the automorphisms on quiver lattices induced by compositions of elementary moves on ideal webs.

\medskip
From this section, we consider only decorated surfaces without boundary, i.e., connected closed oriented genus-$g$ surfaces with a non-empty finite set $P$ of $n$ special points. Denote such a surface by $S$, and denote the punctured surface by $\mathring{S} = S \setminus P$. Assume $2g-2+n>0$.

\subsection{Quivers Associated with Bipartite Graphs on Decorated Surfaces}

\subsubsection{Quiver Lattices}

We first recall the quivers associated with bipartite graphs on decorated surfaces introduced in \cite{Goncharov}.

Let $\Gamma$ be a bipartite $A_m^*$-graph on $S$. Then $\Gamma \subset \mathring{S}$. The quiver lattice $\Lambda_{\Gamma}$ and non-frozen quiver lattice $\Lambda_{\Gamma}^{\circ}$ are given by \cite{Goncharov}
\begin{equation}
\Lambda_{\Gamma}
:=
\operatorname{ker}
\bigl(
H_{1}(\Gamma;\mathbb Z)
\to
H_{1}(S;\mathbb Z)
\bigr).
\end{equation}
\begin{equation}
\Lambda_{\Gamma}^{\circ}
:=
\operatorname{ker}
\bigl(
H_{1}(\Gamma;\mathbb Z)
\to
H_{1}(\mathring{S};\mathbb Z)
\bigr).
\end{equation}
There is a canonical embedding $\Lambda_{\Gamma}^{\circ} \hookrightarrow \Lambda_{\Gamma}$.

Recall that the spectral cover $\Sigma$ of $\Gamma$ is homotopy equivalent to $\Gamma$.
\begin{lemma}[{\cite{Goncharov}}]
There are isomorphisms of lattices, where $\pi_{*}$ is induced by the spectral cover map:
\begin{equation}
\operatorname{ker}
\bigl(
H_{1}(\Sigma;\mathbb Z)
\xrightarrow{\pi_{*}}
H_{1}(S;\mathbb Z)
\bigr)
\cong
\operatorname{ker}
\bigl(
H_{1}(\Gamma;\mathbb Z)
\rightarrow
H_{1}(S;\mathbb Z)
\bigr) = \Lambda_{\Gamma}.
\end{equation}
\begin{equation}
\operatorname{ker}
\bigl(
H_{1}(\Sigma;\mathbb Z)
\xrightarrow{\pi_{*}}
H_{1}(\mathring{S};\mathbb Z)
\bigr)
\cong
\operatorname{ker}
\bigl(
H_{1}(\Gamma;\mathbb Z)
\rightarrow
H_{1}(\mathring{S};\mathbb Z)
\bigr) = \Lambda_{\Gamma}^{\circ}.
\end{equation}
\end{lemma}

The completed spectral cover $\overline{\Sigma}$ is obtained by filling the faces on $\Sigma$ by discs. The inclusion $\Sigma \hookrightarrow \overline{\Sigma}$ gives a short exact sequence
\begin{equation}
0
\to
Z_{\Gamma}
\to
H_{1}(\Sigma;\mathbb Z)
\to
H_{1}(\overline{\Sigma};\mathbb Z)
\to
0.
\end{equation}

The subgroup $Z_{\Gamma}$ is generated by the zig-zag loops of $\Gamma$,
which are the same as the face loops of $\Gamma^\vee$, with a single relation: the sum of all generators equals to zero. The \emph{intersection
form} on $H_{1}(\Sigma;\mathbb Z)$ gives a skew-symmetric bilinear
form with kernel $Z_{\Gamma}$:
\begin{equation}
H_{1}(\Sigma;\mathbb Z)
\wedge
H_{1}(\Sigma;\mathbb Z)
\to
\mathbb Z.
\end{equation}

$\Gamma$ is homotopy equivalent to $\Sigma$, so we get a skew-symmetric bilinear form on
$H_{1}(\Gamma;\mathbb Z)$. The skew-symmetric bilinear form on $\Lambda_{\Gamma}$ is given by restriction:
\begin{equation}
(*,*)_{\Gamma}
:
\Lambda_{\Gamma}
\wedge
\Lambda_{\Gamma}
\to
\mathbb Z.
\end{equation}

Denote the free abelian group generated by a set X by $\mathbb{Z}[X]$. Assuming that the faces of $\Gamma$ are discs on S, $S\cong \Gamma \cup (\bigcup F_i
)$, $\Gamma \cap (\bigcup F_i)\cong \bigcup S^1_i$, there are Mayer-Vietoris sequences
\begin{equation}
0
\to
H_{2}(S;\mathbb Z)
\xrightarrow{r}
\mathbb Z[\{\text{Faces of }\Gamma\}]
\to
H_{1}(\Gamma;\mathbb Z)
\to
H_{1}(S;\mathbb Z)
\to
0.
\end{equation}
\begin{equation}
0
\to
\mathbb Z[\{\text{Internal faces of }\Gamma\}]
\to
H_{1}(\Gamma;\mathbb Z)
\to
H_{1}(\mathring{S};\mathbb Z)
\to
0.
\end{equation}

Hence there are isomorphisms
\begin{equation}
\Lambda_{\Gamma}^{\circ}
\cong
\mathbb Z[\{\text{Internal faces of }\Gamma\}],
\qquad
\Lambda_{\Gamma}
\cong
\frac{\mathbb Z[\{\text{Faces of }\Gamma\}]}{r(H_{2}(S;\mathbb Z))}.
\end{equation}

There is a basis $ \{e_G\} $ of the lattice $ \Lambda_\Gamma^\circ $ given by the internal faces $G$ of $ \Gamma $, and a collection of vectors $ \{e_F\} $ given by all faces $ F $ of $ \Gamma $, satisfying a single relation
\begin{equation}
\sum_{F: \text{faces of } \Gamma} e_F = 0.
\end{equation}

Let $ E $ be an edge of the graph $ \Gamma $. $ \langle e_F, e_G \rangle^{(E)} = 0 $ if $ E $ is not a common edge of $ F $ and $G$, 
$ \langle e_F, e_G \rangle^{(E)} = +1 $ if the anticlockwise orientation of $F$ induces $\bullet \to \circ$ orientation of $E$, and $ \langle e_F, e_G \rangle^{(E)} = -1 $ otherwise. \cite{Goncharov}

\begin{lemma}[{\cite{Goncharov}}]
\label{lemma:form}
The bilinear form on $\Lambda_{\Gamma}$ satisfies
\begin{equation}
\label{eq:arrow}
\langle e_F, e_G\rangle_{\Gamma}
=
\sum_{E:\ \text{edges of }\Gamma}
\langle e_F, e_G\rangle^{(E)}.
\end{equation}
\end{lemma}

\begin{proposition}
\label{prop:3.3}
Let $S$ be a decorated surface without boundary.

\begin{enumerate}
    \item An $A_m$-bipartite graph $\Gamma$ on $S$ has only internal faces.
    Hence the only quiver lattice associated with $\Gamma$ is 
    \[
    \Lambda_\Gamma^\circ
    \cong
    \mathbb Z[\{\text{Internal faces of }\Gamma\}],
    \]
    with no relation between basis vectors.

    \item An $A_m^*$-bipartite graph $\Gamma^*$ on $S$ has both internal faces and external faces.
    Hence the full quiver lattice associated with $\Gamma^*$ is
    \[
    \Lambda_{\Gamma^*}
    \cong
    \frac{\mathbb Z[\{\text{Faces of }\Gamma^*\}]}{r(H_{2}(S;\mathbb Z))},
    \]
    with one single relation $\sum_{F: \text{faces of } \Gamma^*} e_F = 0$.
    And its non-frozen sublattice is
    \[
    \Lambda_{\Gamma^*}^\circ
    \cong
    \mathbb Z[\{\text{Internal faces of }\Gamma^*\}].
    \]
    The basis in $\Lambda_{\Gamma^*} \setminus \Lambda_{\Gamma^*}^\circ$ corresponds to external faces of $\Gamma^*$.
    For any basis $e_F$ in $\Lambda_{\Gamma^*}\setminus\Lambda_{\Gamma^*}^\circ$ and any basis $e_G$ in $\Lambda_{\Gamma^*}$,
    \[
    \langle e_F, e_G\rangle_{\Gamma}
    =\sum_{E:\ \text{edges of }\Gamma^*}
    \langle e_F, e_G\rangle^{(E)}=0.
    \]

    \item The $A_m$-bipartite graph $\Gamma_*$ obtained by shrinking each 2-valent $\bullet$-vertex of $\Gamma^*$ connecting two $\circ$-vertices to a single $\circ$-vertex has only internal faces, which are internal faces of $\Gamma^*$.
    All the external faces of $\Gamma^*$ are contracted to $\circ$-vertices of $\Gamma_*$ on special points.   
    The only quiver lattice associated with $\Gamma_*$ is
    \[   \Lambda_{\Gamma_*}^{\circ}=\Lambda_{\Gamma^*}^\circ
    \cong
    \mathbb Z[\{\text{Internal faces of }\Gamma^*\}],
    \]
    with no relation between basis vectors. 
\end{enumerate}
\end{proposition}

\begin{proof}
An $A_m$-bipartite graph $\Gamma$ on $S$ has only internal faces since all the special points lie on its $\circ$-vertices and $S$ has no boundary.

$\bullet$-vertices on the boundary of external faces of $\Gamma^*$ are 2-valent.
And the two edges incident to the 2-valent $\bullet$-vertex has opposite orientation induced from the orientation on the boundary of the external face. Thus the bilinear form vanishes for external face basis vectors with other face basis vectors.

Shrinking each 2-valent $\bullet$-vertex of $\Gamma^*$ connecting two $\circ$-vertices to a single $\circ$-vertex of $\Gamma^*$ contracts every external face of $\Gamma^*$ to the special point inside it, which becomes a $\circ$-vertex of $\Gamma_*$.
\end{proof}

\subsubsection{Quasi-Quiver}
We follow \cite{Goncharov} and restate the definition of a quasi-quiver, where the vectors $\{e_i\}$ generate the lattice $\Lambda$. $\{e_i\}$ are called \emph{basic vectors}. Each basic vector is classified as either \emph{frozen} or
\emph{non-frozen}.

\begin{definition}[{\cite{Goncharov}}]
\label{def:quasi}
A quasi-quiver is a data $\bigl(\Lambda,\{e_i\},(*,*)\bigr)$,
where $\Lambda$ is a lattice, $\{e_i\}$ is a finite set of vectors of two types: non-frozen and frozen, generating the lattice $\Lambda$, and $(*,*)$ is a skew-symmetric $\mathbb{Z}$-valued bilinear form on $\Lambda$ such that $(e_i,e_j)\in \mathbb{Z}$ unless both $e_i$ and $e_j$ are frozen. The basic vectors $e_i$ satisfy a single relation: $\sum_i e_i=0$.
\end{definition}

\begin{definition}[{\cite{Goncharov}}]
The \emph{quiver assigned to a bipartite graph $\Gamma$ on S} is given by
$Q_{\Gamma}
:=
\bigl(\Lambda_{\Gamma},\{e_F\},(*,*)_{\Gamma}\bigr).$
The vectors $\{e_G\}$ assigned to internal faces $G$ of $\Gamma$ are the
non-frozen basic vectors. The vectors $\{e_F\}$ assigned to external faces $F$ of $\Gamma$ are the frozen basic vectors.

The non-frozen subquiver is given by the sublattice
$\Lambda_{\Gamma}^{\circ}$ with a basis $\{e_G\}$ and the restricted form:
$Q_{\Gamma}^{\circ}
:=\bigl(
\Lambda_{\Gamma}^{\circ},
\{e_G\},
\operatorname{Res}_{\Lambda_{\Gamma}^{\circ}}(*,*)_{\Gamma}
\bigr).$
\end{definition}

\begin{figure}[htbp]
\centering
\includegraphics[width=0.5\textwidth]{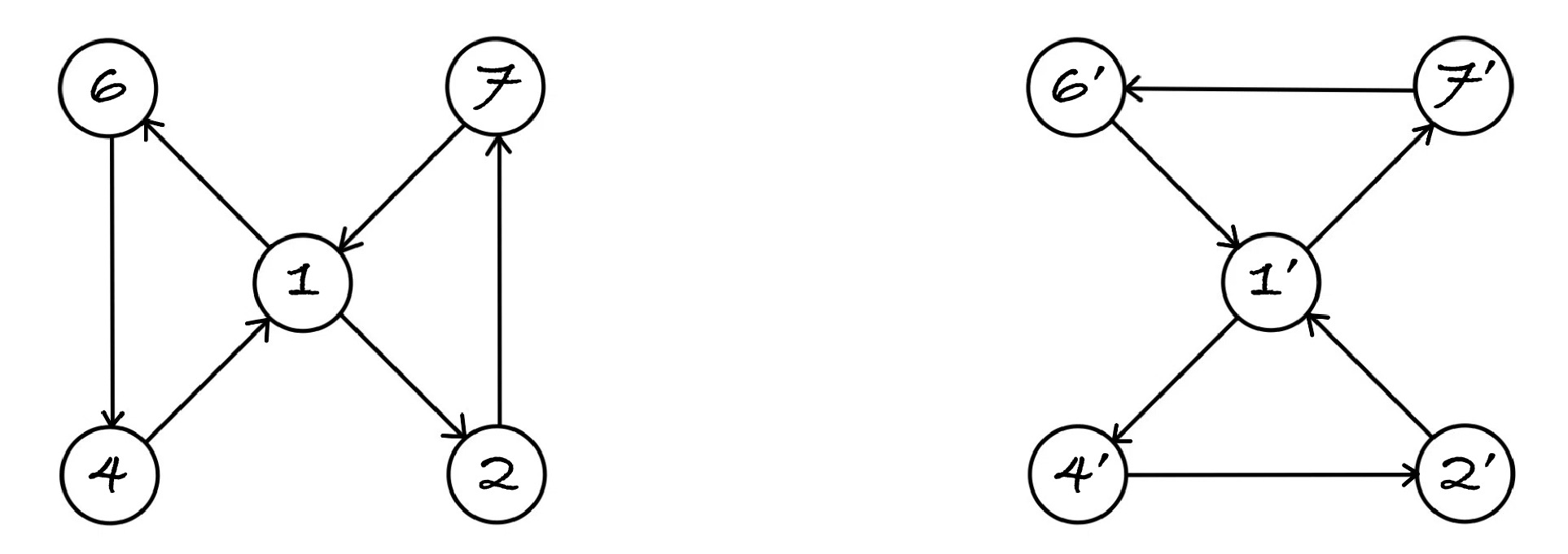}
\caption{The local change of the quiver under a \ttt move at face 1.\\ (This figure only shows the part of the quiver that changes under the mutation.\\ The rest of the quiver remains unchanged.)}
\label{fig:localquiver3}
\end{figure}

\begin{lemma}
\label{le:quiverchange}
The Two-by-two move on $\Gamma$ is the only type of elementary move that changes the quiver $Q_\Gamma$, as shown in Figure~\ref{fig:localquiver3}.
Shrink/expand moves on $\Gamma$ leave the quiver $Q_\Gamma$ unchanged.
\end{lemma}

\begin{remark}
Two-by-two moves can only act on internal faces that correspond to 4-valent vertices in the quiver $Q_{\Gamma}^{\circ}
:=\bigl(
\Lambda_{\Gamma}^{\circ},
\{e_G\},
\operatorname{Res}_{\Lambda_{\Gamma}^{\circ}}(*,*)_{\Gamma}
\bigr)$.
\end{remark}

\begin{figure}[htbp]
\centering
\includegraphics[width=0.5\textwidth]{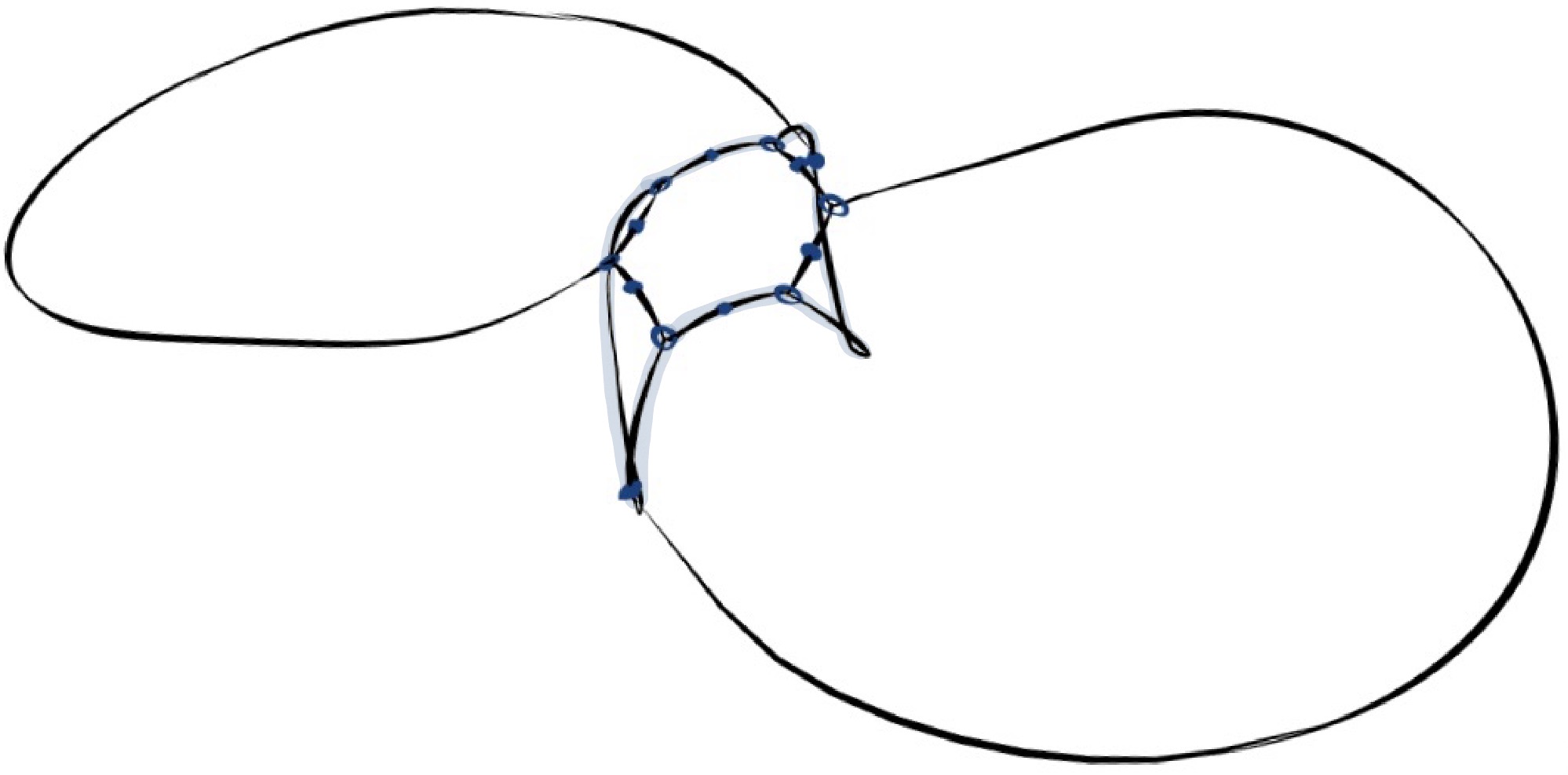}
\caption{Lift of the bipartite $A_1^*$-graph on a spectral cover of a torus with one special point. The highlighted loop lies in $\Lambda_{\Gamma}^{\circ}$.}
\label{fig:kernal}
\end{figure}

\begin{figure}[htbp]
\centering
\includegraphics[width=0.15\textwidth]{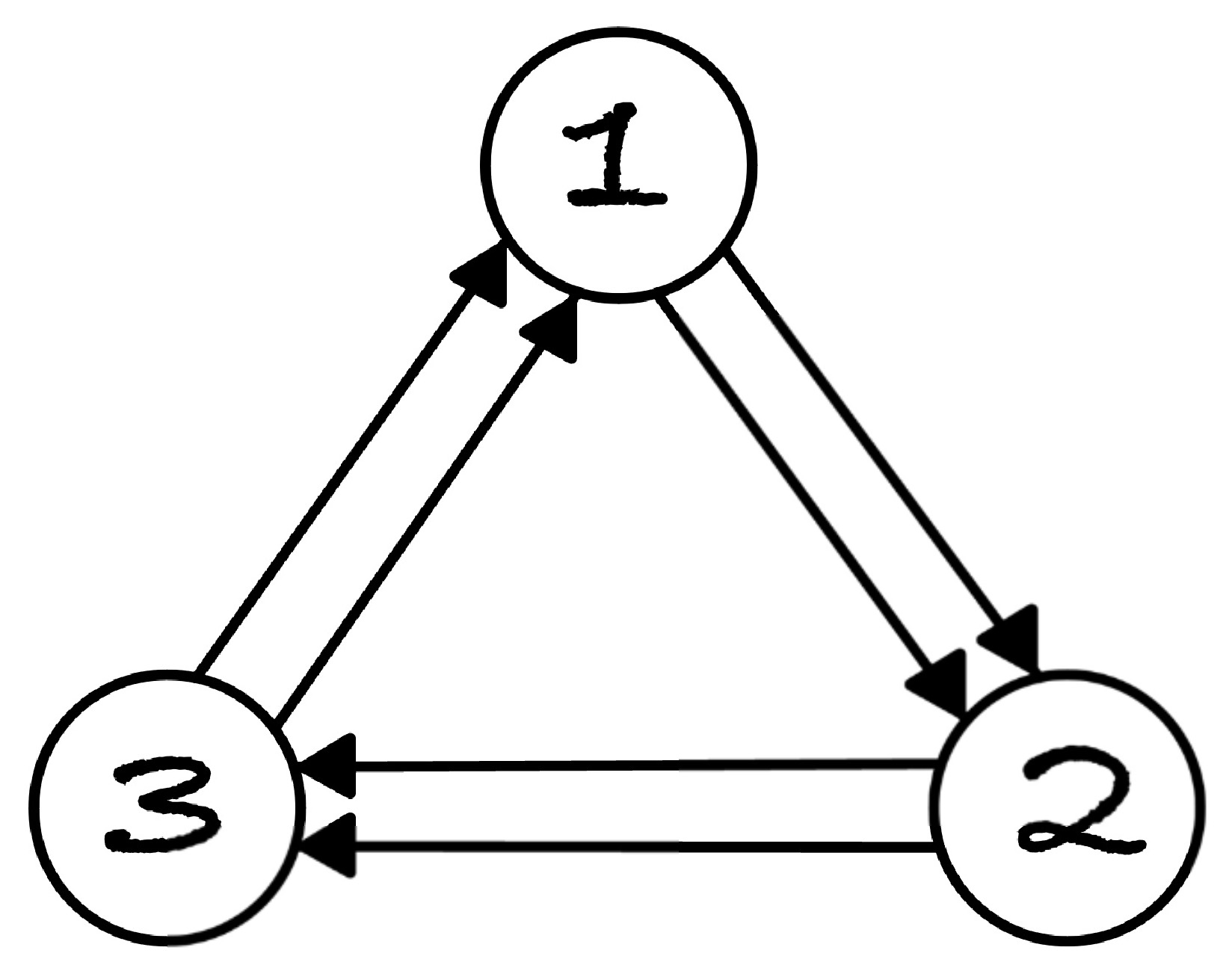}
\caption{Non-frozen quiver assigned to a bipartite $A_1$ or $A_1^*$-graph on a torus with one special point.}
\label{fig:A1quiver}
\end{figure}

\begin{example}
Let us consider a bipartite $A_1^*$-graph $\Gamma$ and the corresponding $A_1^*$-web $W$ on a torus $S$ with one special point. Let $\widetilde{W}$ be the lift of $W$ on the universal cover of $S$. Let $\Sigma$ be the corresponding spectral cover.

There are three non-trivial elements in $H_1(\Sigma; \mathbb Z)$ that maps to internal face loops on $\mathring{S}$, which are 0 in $H_1(\mathring{S}; \mathbb Z)$. The Highlighted Loop in Figure~\ref{fig:kernal} is one of those elements. The non-frozen quiver assigned to $\Gamma$ is shown in Figure~\ref{fig:A1quiver}.
\end{example}

\begin{example}
The quiver assigned to a bipartite $A_2$ or $A_2^*$-graph on a torus with one special point has one frozen vertex and eight non-frozen vertices. The non-frozen sub-quiver is shown in Figure~\ref{fig:A2quiver}.
\end{example}

\begin{figure}[htbp]
\centering
\includegraphics[width=0.2\textwidth]{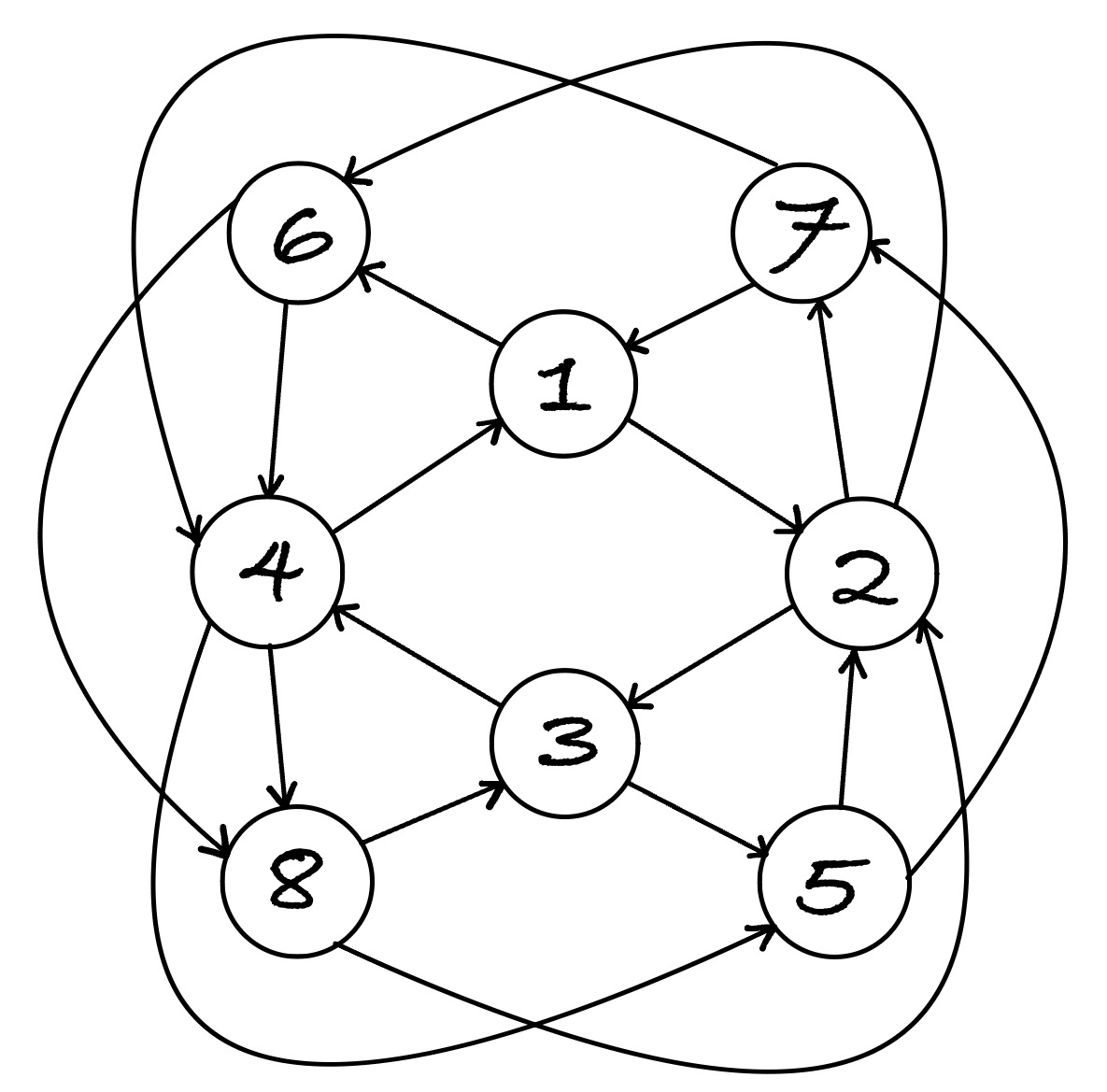}
\caption{Non-frozen quiver assigned to a bipartite $A_2$ or $A_2^*$-graph on a torus with one special point.}
\label{fig:A2quiver}
\end{figure}

\begin{example}
The quiver assigned to a bipartite $A_1$ or $A_1^*$-graph on a sphere with four special points has four frozen vertices and six non-frozen vertices. The non-frozen sub-quiver is shown in Figure~\ref{fig:4Squiver}.
\end{example}

\begin{figure}[htbp]
\centering
\includegraphics[width=0.2\textwidth]{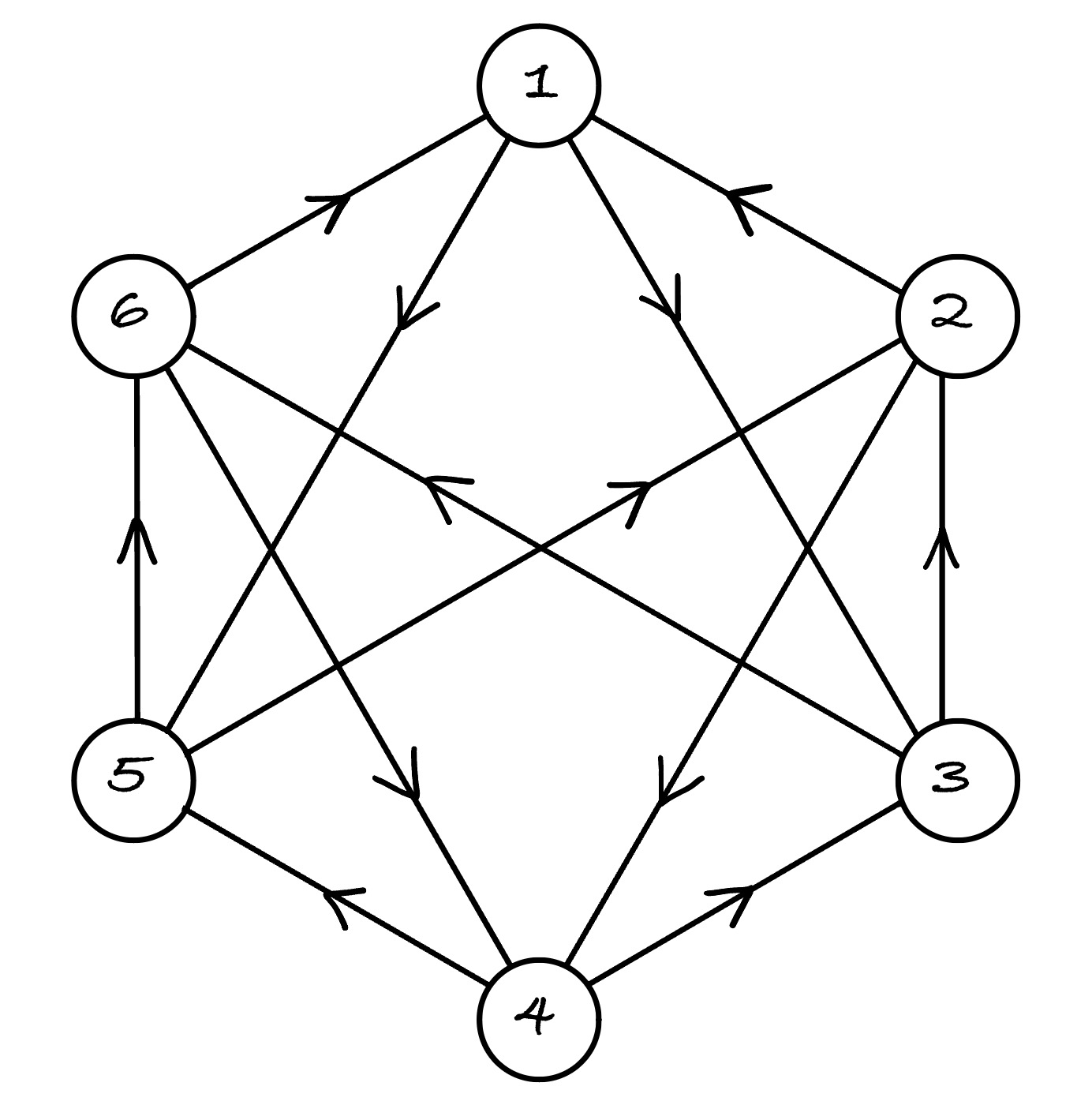}
\caption{Non-frozen quiver assigned to a bipartite $A_1$ or $A_1^*$-graph on a sphere with four special point.}
\label{fig:4Squiver}
\end{figure}

\subsection{A Local Operation on Spectral Covers that Respects Two-by-two Moves}

In this subsection, we will study the local changes of topological spectral covers under \ttt moves on $A_{m-1}$-webs with $m\ge 2$. A \ttt move amounts to a local discrete change on the spectral cover, inducing a discontinuous bijection between spectral covers before and after the move.

\subsubsection{Local Quiver}
Let $W$ be a $A_{m-1}$-web on a decorated surface $S$ of genus-$g_S$ and $\Gamma$ the corresponding bipartite graph. Let $\Sigma$ be the corresponding spectral cover, $\widetilde{W}$ be the lift of $W$ on $\Sigma$ and $\widetilde{\Gamma}$ be the lift of $\Gamma$ on $\Sigma$.

Label all the faces of $\Gamma$. Each face of $\Gamma$ has $m$ lifts on $\Sigma$. We hence label all the lifts of each face of label $i$: $i_1,\cdots, i_m$ on $\Sigma$. The set of faces of $\widetilde{\Gamma}$ consists of the lifts of faces of $\Gamma$.

\begin{figure}[htbp]
\centering
\includegraphics[width=0.4\textwidth]{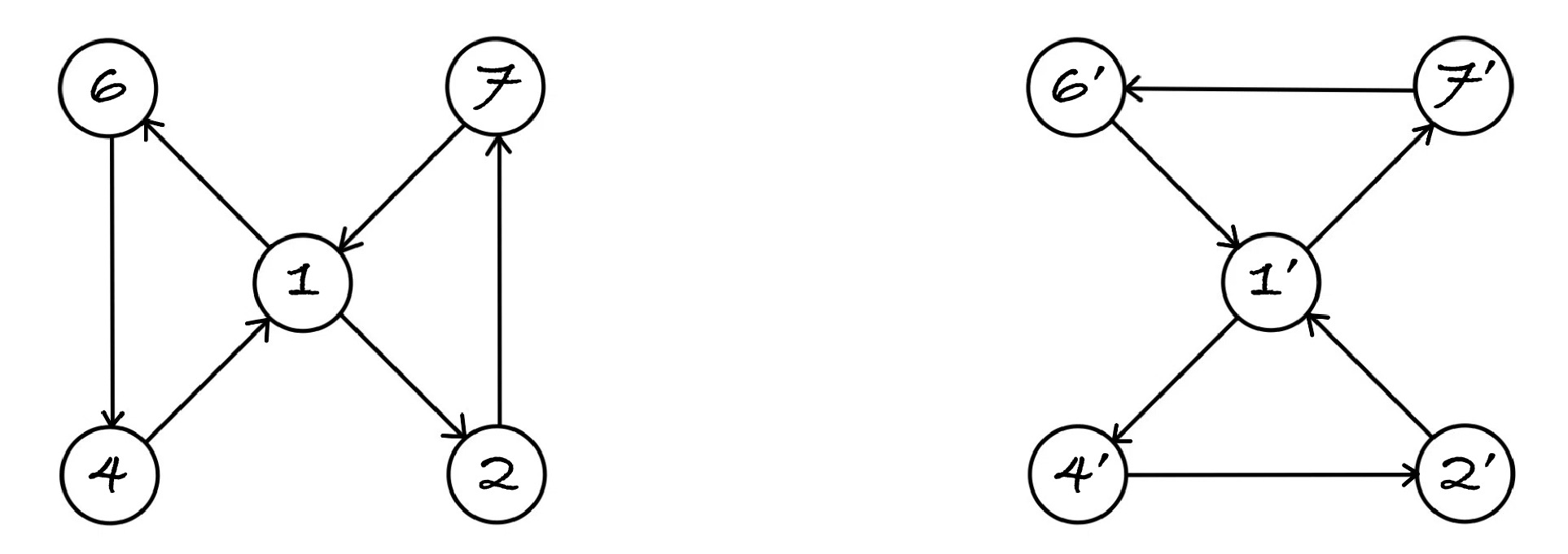}
\caption{The local change of the quiver under a \ttt move at face 1.\\ (This figure only shows the part of the quiver that changes under the mutation.\\ The rest of the quiver remains unchanged.)}
\label{fig:localquiver32}
\end{figure}

Recall the quiver $Q_{\Gamma}$ assigned to $\Gamma$, and the fact that a \ttt move of $\Gamma$ at an internal face amounts to a mutation of $Q_{\Gamma}$ at a 4-valent vertex. Consider the local part of $Q_{\Gamma}$ around the 4-valent vertex that a \ttt move acts on, as shown in Figure~\ref{fig:localquiver32}. Each vertex corresponds to an internal face of $\Gamma$. The two anti-clockwise cycles bound two 3-valent $\bullet$-vertices in $\Gamma$.

Notice that every quiver assigned to an $A_{m-1}$ or $A_{m-1}^*$-bipartite graph has local parts like this around each 4-valent vertex, where face 2 and face 6, face 4 and face 7, need not be distinct, like the one in Figure~\ref{fig:A1quiver}.

We call the local part around the mutated vertex \emph{degenerated} if it is the same as the part shown in Figure~\ref{fig:A1quiver}, and call the local part around the mutated vertex \emph{non-degenerated} if it is the same as the part shown in Figure~\ref{fig:localquiver32}.

\subsubsection{Local Topology}
There are $m$ disjoint lifts of the interior of face 1 on $\Sigma$. Now take an open set $U\subset \mathring{S}$ slightly larger then face 1. The preimage of $U$ on $\Sigma$ has $m-1$ connected components, with $m-2$ of them containing only one lift of face 1. The other one containing two lifts of face 1 and two lifts of the branch points on the boundary on face 1 -- this is the local $2:1$ ramified component of the preimage of $U$, whereas the other $m-2$ components are unramified. Only the ramified component will get involved in the process of the \ttt move.

Denote the local $2:1$ ramified component of the preimage of $U$ on $\Sigma$ by $U_r$. A \ttt move at face 1 only changes the topology in $U_r$ on $\Sigma$. By the construction of the spectral cover, $U_r$ is topologically a cylinder, with labels of the lifts of faces. When mapping to $U \subset S$, the labels on the model are all facing up, hence there will be no "fold lines" on the actual cover, as shown in Figure~\ref{fig:Ura}. The reason that there are fold lines on the model is that we "untwisted" the actual ramified shape of the mapping and only kept the topology of the relevant local open set $U_r$ on $\Sigma$.

\subsubsection{Ura Operation}

\begin{definition}
\label{def:Ura}
The clockwise \emph{\bbb operation} on a topological surface along loop $1$ is a local operation performed on a local cylinder area as the following procedure in Figure~\ref{fig:Ura}.

\begin{figure}[htbp]
\centering
\includegraphics[width=0.9\textwidth]{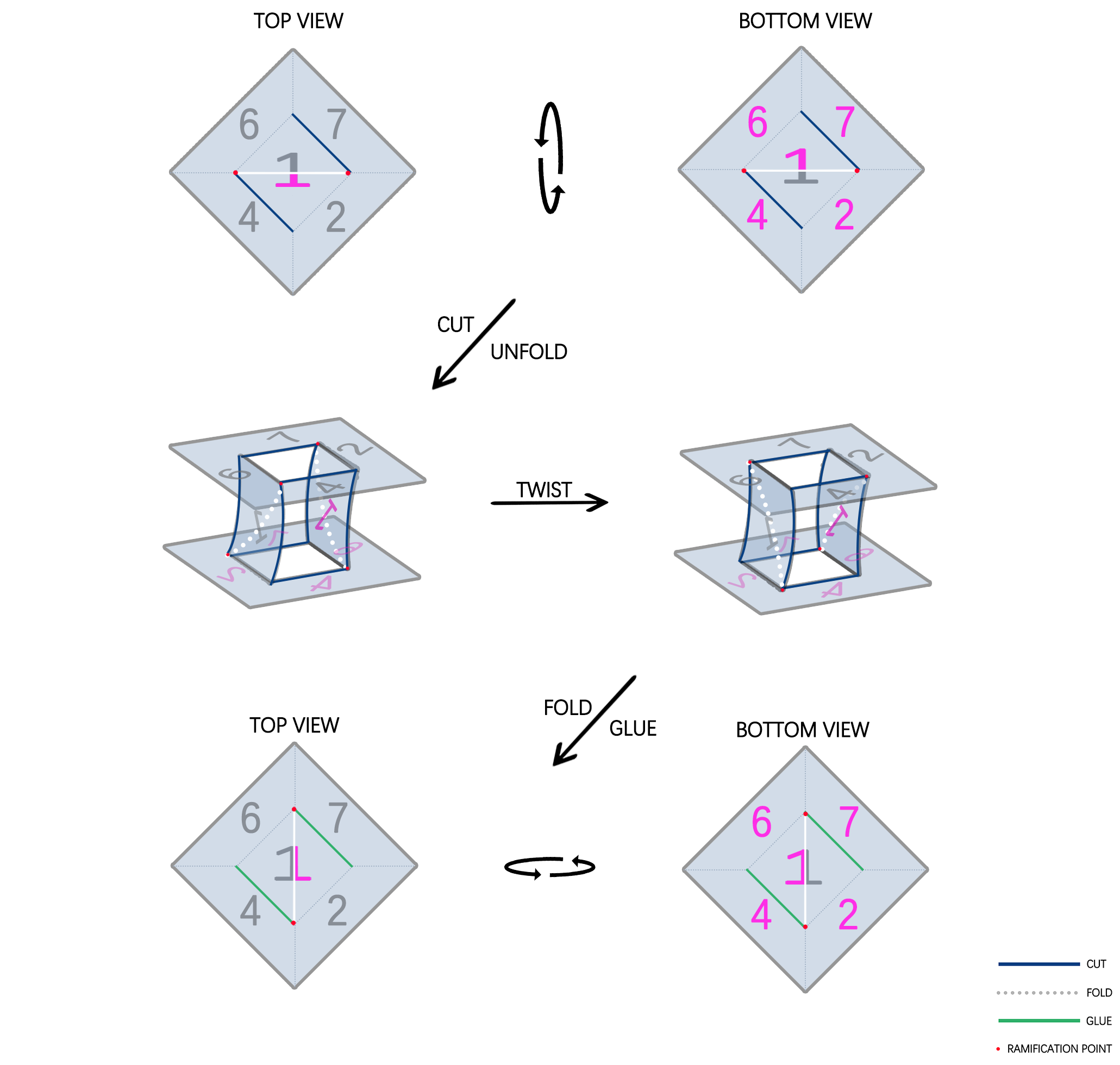}
\caption{Clockwise ura operation. \\ (Instructions for making the paper model are provided in the appendix.)}
\label{fig:Ura}
\end{figure}

The \bbb operation applies to general topological cylinders, and the labels here are not necessarily associated with faces of a bipartite graph. The labels are only used to indicate the relative positions on the cylinder.

The \bbb operation gives a discontinuous bijection, called \emph{\bbb bijection}, from the spectral cover before the operation to the one after.
\end{definition}

\begin{remark}
The top row and the bottom row in Figure~\ref{fig:Ura} are both topologically cylinders.
\end{remark}

\begin{proposition}
Given an ideal web $W$ and its associated bipartite graph $\Gamma$ on a decorated surface $S$, the topological change of the spectral cover $\pi_W:\Sigma_W\to S$ under a \ttt move at an internal face $i$ of $\Gamma$ is given by a ura operation along the lift of face loop $i$.
\end{proposition}

\begin{lemma}
The transposition monodromies of the two simple ramification points affected by a \ttt move are both $(ij)\in S_m$, $m\ge 2$.
\end{lemma}

\begin{definition}
\label{def:orittt}
The \emph{orientation of a ura operation} is defined to be the orientation of the "Twist" operation in the definition of the \bbb operation.

More precisely, for \emph{clockwise} orientation, we start with the fold lines as shown on the top row and the left figure on the second row, and choose to make the cuts between 1 and 4, and 1 and 7 at first, then after unfolding, the orientation of the "Twist" is constrained to be rotating the top square clockwise by $\pi$ relative to the bottom square.

For \emph{anticlockwise} orientation, we start with the fold lines as shown on the top row, make cuts between 1 and 2, and 1 and 6; or start with the fold lines as shown on the bottom row, make cuts between 1 and 4, and 1 and 7. Then the orientation of the "Twist" is constrained to be rotating the top square anticlockwise by $\pi$ relative to the bottom square.

At the level of the spectral cover, the \emph{orientation of a \ttt move} is defined to be the orientation of the \bbb operation.
\end{definition}

\subsubsection{Properties of the Ura Operation}
Let us observe two properties of the \bbb operation regarding our framework. We refer to the lift of face i when we say face i on the spectral cover.

Let $W'$ be the $A_{m-1}$-web after performing a \ttt move at face 1. Denote the local $2:1$ ramified component of the preimage of $U$ on $\Sigma'$ by $U_r'$. Similarly, $U_r'$ is topologically a cylinder.

We know from Section 3.1 that the lattice $\Lambda_{\Gamma}^\circ$ has a basis indexed by internal faces of $\Gamma$. On the level of the loops and intersection form on $\Sigma$, the basis becomes the set of the equivalent classes of lifts of the internal face loops in $H_1(\Sigma;\mathbb Z)$, and the bilinear skew-symmetric form is the intersection form.

When performing a \ttt move at face 1, the only change on $\Sigma$ happens locally on $U_r\to U_r'$, which is the \bbb operation along the lift of face loop 1.

\medskip
\noindent
\textbf{Property 1.} Before the \bbb operation, two lifts of face 1 form the "neck" of the model. The fold lines on two lifts of face 1 are the two horizontal diagonal lines. The two ramification points on the two 3-valent $\bullet$-vertices of $\Gamma$ are the two end points of the fold lines (they share the same end points). When performing the \bbb operation, we cut the two edges between two pairs of face 1 and face 4, and the two edges between two pairs of face 1 and face 7.

Then we twist around the neck and glue the cut line on one face 4 to the other cut line on face 1 that was originally glued to the other face 4 -- we perform similar gluing to all four cuts simultaneously. So after the \bbb operation, two lifts of face 4 / 7 switched places, while the lifts of other faces stay unchanged. The two ramification points on the two 3-valent $\bullet$-vertices of $\Gamma'$ are the two ends of the vertical fold lines.

\medskip
\noindent
\textbf{Property 2.} Now we draw four circles around four corners on each sides of the model, as shown in Figure~\ref{fig:circlesquare}. Notice that there are two disjoint pink / purple circles, while the two yellow / green circles on different sides are connected through the middle fold line. We make four cuts, so there are four pairs of cut lines in total. After the cut, there are two hollow squares on the side, whose boundaries are the cut lines. The twisting process between cutting and gluing on the hollow squares is shown in Figure~\ref{fig:hollow}.

Notice that the transformations on the two hollow squares happen at the same time, but they don't interact with each other. After gluing, there are two disjoint yellow / green circles, while the two pink / purple circles on different sides are connected through the middle fold line. The two lifts of face 1 still form the "neck" of the model. The new fold lines on two lifts of face 1 are now the two vertical diagonal lines. The two new ramification points are the two end points of the new fold lines. The circles after the move are shown in Figure~\ref{fig:circlesquare}.

\begin{figure}[htbp]
\centering
\includegraphics[width=0.5\textwidth]{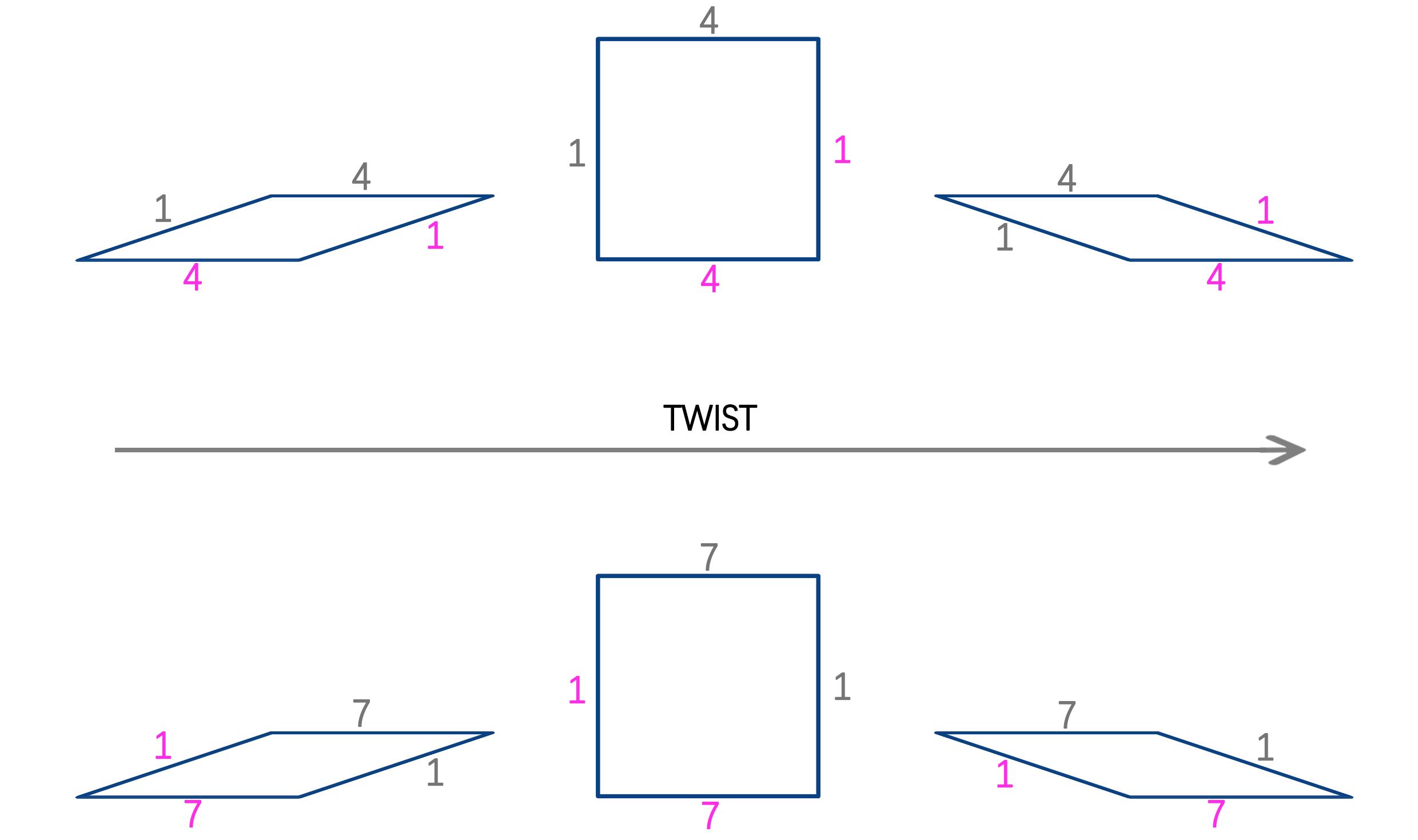}
\caption{Property 2: Changes of two hollow squares after two cuts when twisting.}
\label{fig:hollow}
\end{figure}

\begin{figure}[htbp]
\centering
\includegraphics[width=0.6\textwidth]{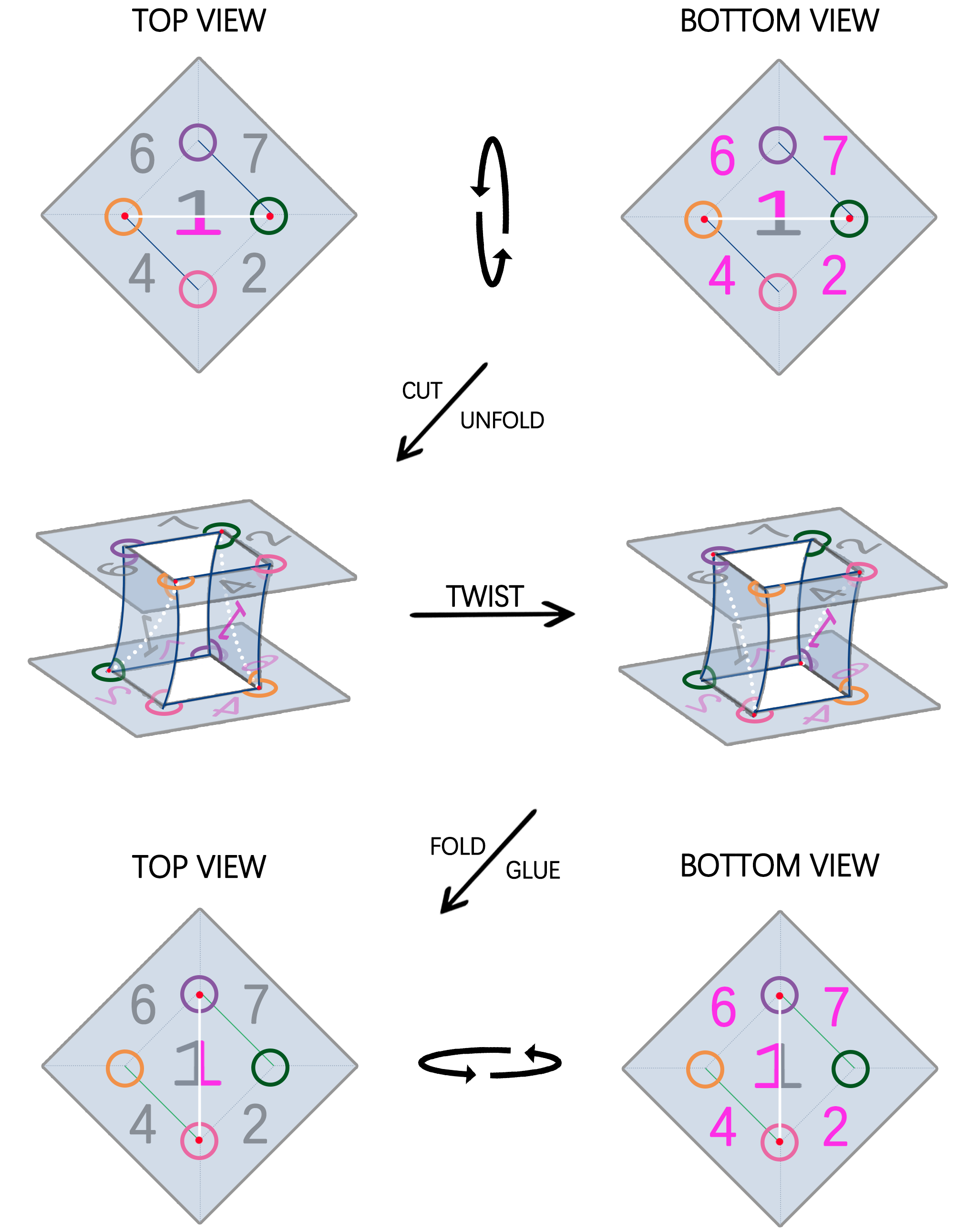}
\caption{Property 2: Four circles capture the change of two ramification points.}
\label{fig:circlesquare}
\end{figure}

\begin{figure}[htbp]
\centering
\includegraphics[width=0.83\textwidth]{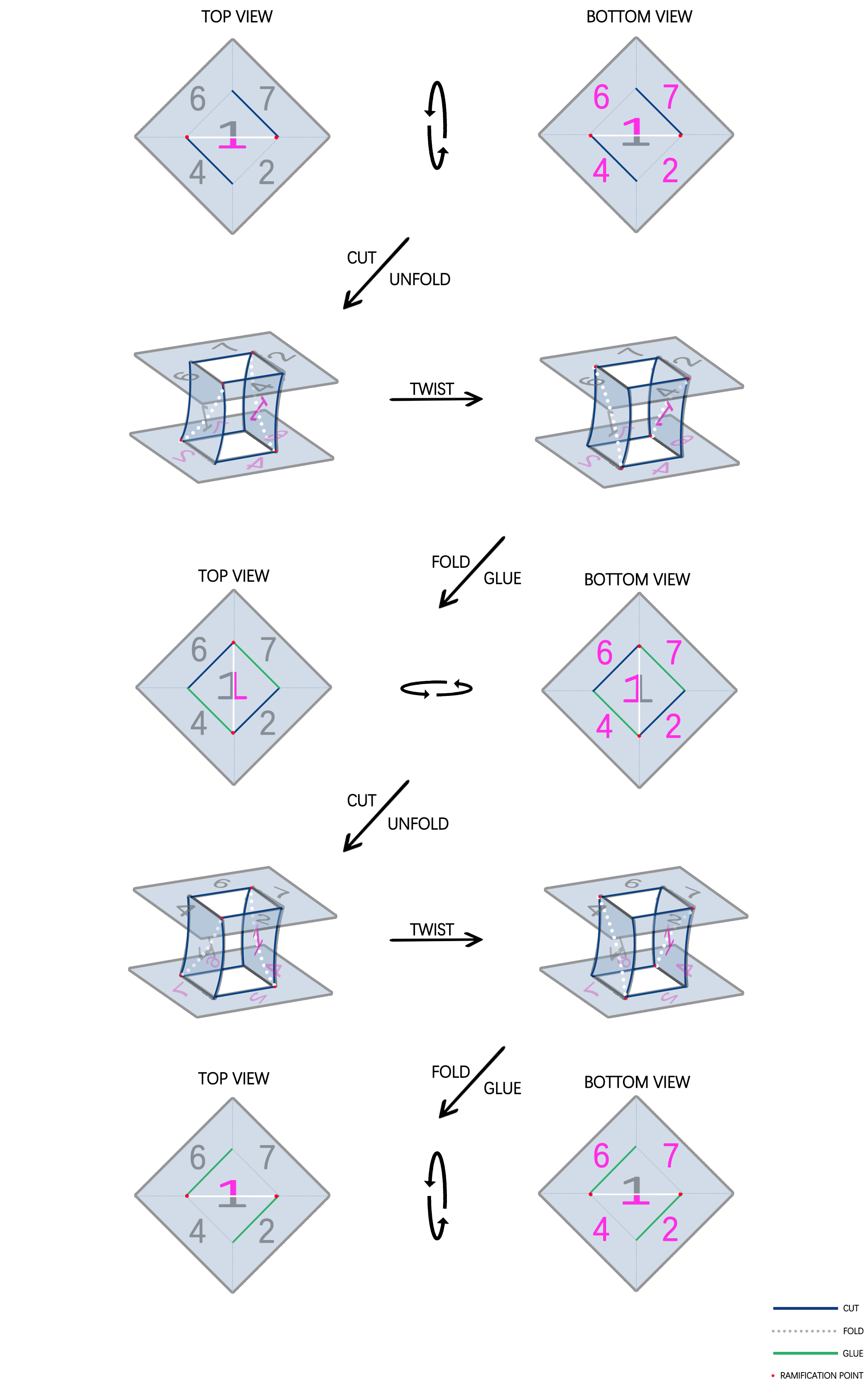}
\caption{Composition of two successive ura operations in the same orientation forms a Dehn twist.}
\label{fig:Dehn}
\end{figure}

\subsection{Two Successive Ura Operations and Automorphisms on Lattices}

\subsubsection{Two Successive Ura Operations in the Same Orientation}

When performing one clockwise \bbb operation, we make two cuts on the opposite pair of edges on the top and two cuts on the other opposite pair of edges on the bottom, and get two independent hollow squares on the "neck" with boundaries the cut lines. And then we glue different pairs of the edges of two hollow squares. Now let us perform another clockwise \bbb operation on top of this. This procedure amounts to a second mutation at vertex 1 / a second \ttt move at face 1 in the same orientation.

As shown in Figure~\ref{fig:Dehn}, when performing the second clockwise \bbb operation, we make the other two cuts on the top and the other two cuts on the bottom, where we didn't cut in the first \bbb operation. The second \bbb operation makes two new hollow squares on the neck with boundaries the cut lines. Then we glue the different pairs of boundaries. Observing the overall change of the model, we see that the loop 1 got cut completely during the process. The pair of the cut lines we got from cutting loop 1 was the boundary of the middle-stage surface. Then the cut lines were relatively twisted by $2\pi$ before glued. Therefore, we conclude the following statement.

\begin{proposition}
\label{prop:dehn}
Given a non-separating loop on a topological surface, the composition of two successive \bbb operations in the same orientation along that same loop forms a Dehn twist.
\end{proposition}

\subsubsection{Induced Quiver Lattice Isomorphism}
Choose the clockwise orientation as the positive orientation for face loops when all labels are facing-up (as mapping to $U \subset S$). Let us now draw the relevant local part of the initial-state non-trivial lifts of face loops 1, 2, 4, 6 and 7 on the model, as in the top row of Figure~\ref{fig:latticeiso}. We could say that the intersection form restricting to these loops provides the same structure as the left image in Figure~\ref{fig:localquiver3}.

Let us now draw the relevant local part of the final-state non-trivial lifts of face loops 1, 2, 4, 6 and 7 on the model after the change, as in the bottom row of Figure~\ref{fig:latticeiso}. We notice that the intersection form restricting to these loops provides the same structure as the right image in Figure~\ref{fig:localquiver3}.

\medskip
Now observe the change of the non-trivial lifts of face loops between the quiver lattices $\Lambda_{\Gamma}$ and $\Lambda_{\Gamma'}$. In Figure~\ref{fig:latticeiso}, the green string (string 6) is the the local part of the lift of face loop 6, as it eventually bounds the lift of face 6. The purple string (string 4) is the local part of loop 4, the yellow string (string 2) is the local part of loop 2, and the pink string (string 7) is the local part of loop 7.

The initial string 2 / 6 and the initial line 7 / 4 intersect respectively before the change, while there is no intersection between initial string 2 / 6 and initial line 4 / 7. The initial loop 1 will lie on the position of the final-state lift of face loop 1 with opposite orientation. During the change, the ends of the initial string 2 / 6 wrap around the initial loop 1 once in the opposite direction, while the ends of the initial line 4 / 7 stay unchanged.

Therefore, we arrive at an isomorphism between the bases of the initial and final quiver lattices $\phi^\pm : \Lambda_{\Gamma'} \to \Lambda_{\Gamma}$, with the restriction $\phi^\pm|_{\Lambda_{\Gamma'}^\circ} : \Lambda_{\Gamma'}^\circ \to \Lambda_{\Gamma}^{\circ}$.

\begin{figure}[htbp]
\centering
\includegraphics[width=0.8\textwidth]{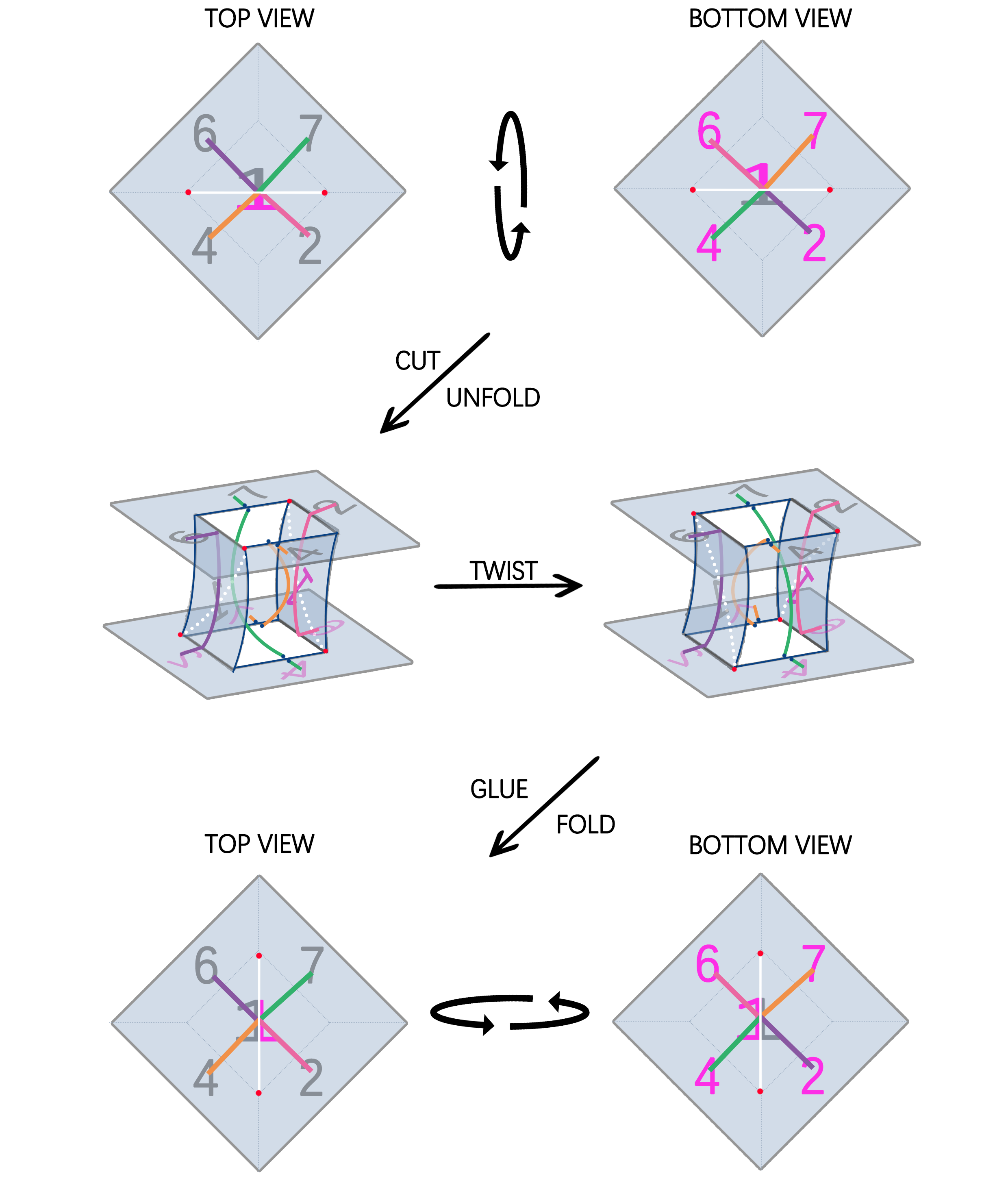}
\caption{Induced isomorphism between quiver lattices.}
\label{fig:latticeiso}
\end{figure}

\begin{theorem}
\label{thm:ttt}
The clockwise \bbb bijection induces a quiver lattice isomorphism $\phi^-: \Lambda_{\Gamma'}\to\Lambda_{\Gamma}$, with restriction $\phi^-|_{\Lambda_{\Gamma'}^\circ} : \Lambda_{\Gamma'}^{\circ}\to\Lambda_{\Gamma}^{\circ}$. The map between two bases is given by:
\begin{equation}
\label{eq:basis}
\begin{array}{cc}
   l_1'\mapsto -l_1,\qquad 
\begin{array}{l}
l_7'\mapsto l_7,\\[1pt]
l_4'\mapsto l_4,
\end{array}
\qquad
\begin{array}{l}
l_6'\mapsto l_6+l_1,\\[1pt]
l_2'\mapsto l_2+l_1.
\end{array}
\end{array}
\end{equation}
The rest of the basis is mapped to the corresponding ones with the same labels.

Equivalently,
$l_1'\mapsto -l_1,\quad l_j'\mapsto l_j+\epsilon_{1j} l_1, \quad j\neq 1$,
where $\epsilon_{1j}=\#\{arrows: 1\to j\}$ in $Q_{\Gamma}$.

The anticlockwise \bbb bijection induces a quiver lattice isomorphism $\phi^+: \Lambda_{\Gamma'}\to\Lambda_{\Gamma}$, with restriction $\phi^+|_{\Lambda_{\Gamma'}^\circ} : \Lambda_{\Gamma'}^{\circ}\to\Lambda_{\Gamma}^{\circ}$. The map between two bases is given by:
\begin{equation}
\label{eq:basis1}
\begin{array}{cc}
   l_1'\mapsto -l_1,\qquad 
\begin{array}{l}
l_7'\mapsto l_7+l_1,\\[1pt]
l_4'\mapsto l_4+l_1,
\end{array}
\qquad
\begin{array}{l}
l_6'\mapsto l_6,\\[1pt]
l_2'\mapsto l_2.
\end{array}
\end{array}
\end{equation}
The rest of the basis is mapped to the corresponding ones with the same labels.

Equivalently,
$l_1'\mapsto -l_1,\quad l_j'\mapsto l_j+\epsilon_{j1} l_1, \quad j\neq 1$,
where $\epsilon_{j1}=\#\{arrows: j\to 1\}$ in $Q_{\Gamma}$.
\end{theorem}

We will investigate this further in in section~\ref{sec:7.4}.

\medskip
After two successive \ttt moves in the same orientation, the bipartite graph $\Gamma''$ returns back to the initial $\Gamma$.
Then following Theorem~\ref{thm:ttt}, we check the compositions of the two induced quiver lattice isomorphisms
$\phi'^\pm\circ\phi^\pm:\Lambda_{\Gamma''}=\Lambda_{\Gamma} \to \Lambda_{\Gamma'} \to \Lambda_{\Gamma}$, with the restriction $\phi'^\pm\circ\phi^\pm|_{\Lambda_{\Gamma}^{\circ}}:$ $\Lambda_{\Gamma''}^{\circ}=\Lambda_{\Gamma}^{\circ} \to \Lambda_{\Gamma'}^{\circ} \to \Lambda_{\Gamma}^{\circ}$.

\begin{corollary}
The composition of two successive ura bijections in the same orientation induces an automorphism $\phi'^\pm\circ\phi^\pm$ on $\Lambda_{\Gamma}$, with restriction $\phi'^\pm\circ\phi^\pm|_{\Lambda_{\Gamma}^{\circ}}$ an automorphism on $\Lambda_{\Gamma}^{\circ}$.

For the clockwise orientation, the map $\phi'^{-}\circ\phi^{-}$ on the basis is given by:
\begin{equation}
\begin{array}{cc}
   l_1\mapsto l_1,\qquad 
\begin{array}{l}
l_7\mapsto l_7+l_1,\\[1pt]
l_4\mapsto l_4+l_1,
\end{array}
\qquad
\begin{array}{l}
l_6\mapsto l_6-l_1,\\[1pt]
l_2\mapsto l_2-l_1,
\end{array}
\end{array}
\end{equation}
and the rest of the basis remains unchanged.

For the anticlockwise orientation, the map $\phi'^{+}\circ\phi^{+}$ on the basis is given by:
\begin{equation}
\begin{array}{cc}
   l_1\mapsto l_1,\qquad 
\begin{array}{l}
l_7\mapsto l_7-l_1,\\[1pt]
l_4\mapsto l_4-l_1,
\end{array}
\qquad
\begin{array}{l}
l_6\mapsto l_6+l_1,\\[1pt]
l_2\mapsto l_2+l_1,
\end{array}
\end{array}
\end{equation}
and the rest of the basis remains unchanged.
\end{corollary}

This is exactly the induced isomorphism of a Dehn twist $T_\delta$ along the loop 1 denoted by $\delta$, i.e., $(T_\delta)_*:
\Lambda_{\Gamma}
\to
\Lambda_{\Gamma},$
$(T_\delta)_*(\alpha)
=\alpha+\langle\alpha,\delta\rangle\delta$,
where $\langle\alpha,\delta\rangle$ is the intersection form, up to a chosen sign convention. The pair of branch points braids a half turn and switches positions.

\medskip
Let $\mathring{S}=S\setminus P$, where $S$ is a closed genus-g decorated surface with a set $P$ of $n$ punctures. Let $W$ be a generic $A_{m-1}$-web on $S$ and $\Gamma$ be the corresponding bipartite $A_{m-1}$-graph. Let $\mathring{\Sigma}=\Sigma_{g_{\Sigma},nm}\setminus \widetilde{P}$ be the topological spectral cover associated with $W$ with $nm$ punctures.

The \emph{pure mapping class group} of $S$ is
\begin{equation}
\operatorname{PMCG}(S)
:=
\operatorname{Diff}^{+}(S;p_1,\ldots,p_n)
\big/
\operatorname{Diff}_{0}(S;p_1,\ldots,p_n)
\end{equation}
That is, the group of isotopy classes of orientation-preserving
diffeomorphisms of $S$ fixing each puncture individually, where the
isotopies also fix the punctures. \cite{FarbMargalit}

Loops in $\Lambda_\Gamma\setminus\Lambda_\Gamma^\circ$ are lifts of external face loops, so they are separating loops on $\mathring{\Sigma}$, hence Dehn twists along them are isotopic to the identity.
Dehn twists along loops in $\Lambda_\Gamma^{\circ}$ give non-trivial diffeomorphisms in $\operatorname{PMCG}(\mathring{\Sigma})$, while induce the trivial diffeomorphism of
$\operatorname{PMCG}(\mathring{S}).$
Dehn twists along loops in $H_1(\mathring{\Sigma};\mathbb Z)\setminus\Lambda_\Gamma$ induce non-trivial diffeomorphisms in
$\operatorname{PMCG}(\mathring{S})$.
Thus such a change generally changes the marking of the base surface $\mathring{S}$.

\medskip
\begin{example}
Consider an $A_2$-web on a torus with one special point, as shown in Figure~\ref{fig:A2}. We perform a \ttt move at face 1. The resulting web is shown in Figure~\ref{fig:22A2}. The change on the genus four spectral cover amounts to an \bbb operation along loop 1.
\end{example}

\begin{figure}[htbp]
\centering
\includegraphics[width=0.5\textwidth]{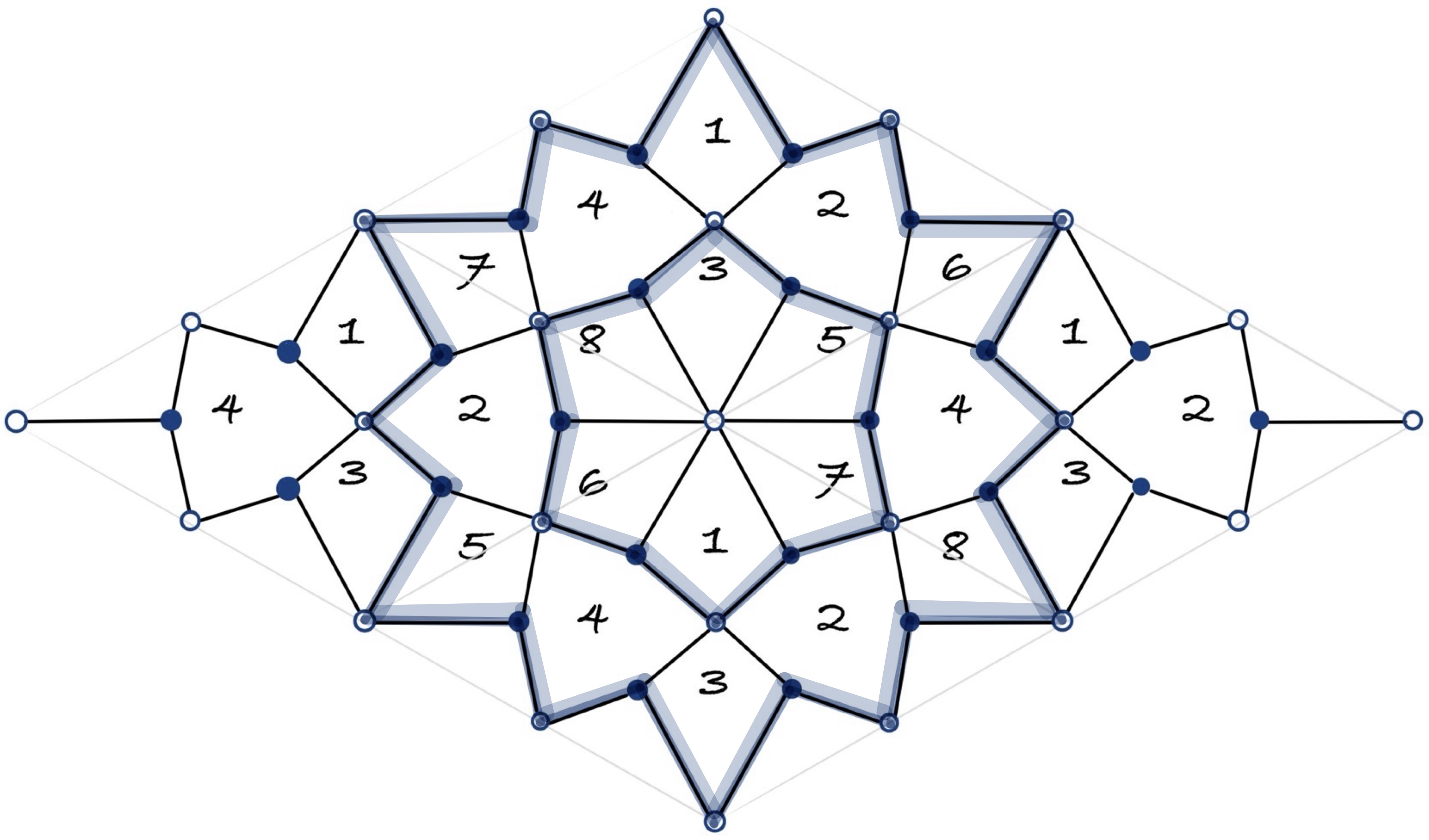}
\caption{Lift of an $A_2$-web 
$W$ on the universal cover of a torus with one special point. The two strands associated with the middle lift of the special point are the two blue highlighted loops.}
\label{fig:A2}
\end{figure}

\begin{figure}[htbp]
\centering
\includegraphics[width=0.5\textwidth]{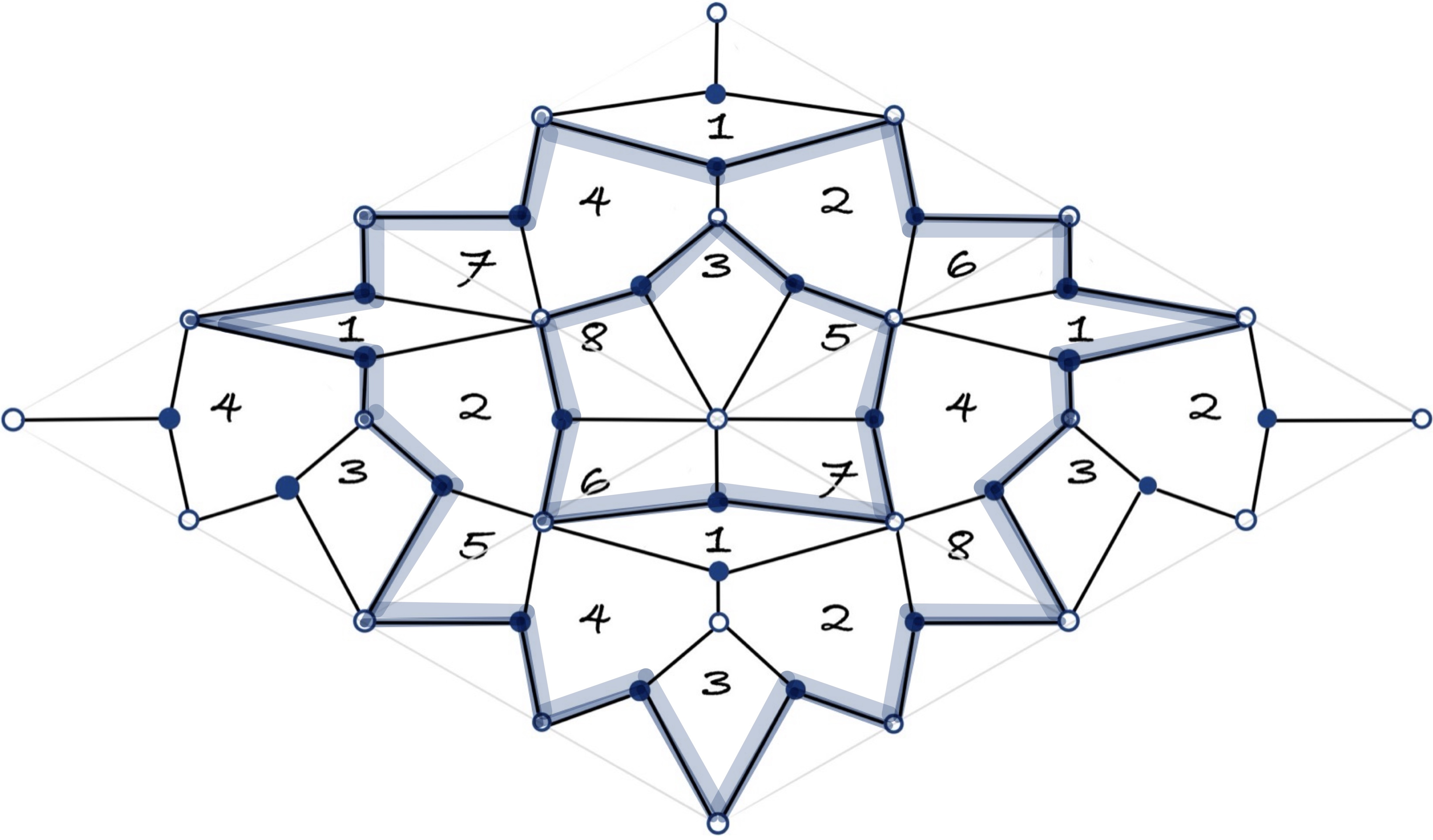}
\caption{Lift of the $A_2$-web $W$ in Figure~\ref{fig:A2} after performing a \ttt move. The two strands associated with the middle lift of the special point are the two blue highlighted loops.}
\label{fig:22A2}
\end{figure}

\newpage
\subsubsection{Degenerated Case}
Let us now consider the degenerated case where in the quiver the vertex 2 and vertex 6 are the same vertex, denoted by 2; and the vertex 4 and vertex 7 are the same vertex, denoted by 3. This is the case for a one-punctured torus with $A_1$-web. The quiver is shown in Figure~\ref{fig:A1quiver}. Now we mutate at face 1.

The relevant area on the spectral cover that is affected by the mutation is now no longer a cylinder. It is the whole two-punctured spectral cover whose compactification has genus 2, as shown in Figure~\ref{fig:genus2}. The lifts of all three face loops on the spectral cover are shown in Figure~\ref{fig:A1web}.

If we compare the faces with the labelled faces on the \bbb operation model, the two pairs of face 2 and face 6 connecting to the same face 1 in the model corresponds to the pair of face 2 in the degenerated case. And likewise, the two pairs of face 4 and face 7 connecting to the same face 1 in the model corresponds to the pair of face 3 in the degenerated case.

The two-punctured spectral cover can be deformation retracted to the graph skeleton in Figure~\ref{fig:kernal}. We can see that the face loop 1, the highlighted loop in Figure~\ref{fig:kernal}, is now a separating loop on the two-punctured spectral cover.

If we were to perform a degenerated version of the \bbb operation along loop 1 here, the two edges between two pairs of face 1 and face 3 get cut, and are glued to the cut line of the other pair. Thus performing one degenerated \bbb operation along face loop 1 is a reflection along the middle line through face 2, face 1, and face 2, which is orientation-reversing, while the \bbb operation is orientation-preserving for generic cases. Performing two successive degenerated \bbb operations along face loop 1 in the same orientation is the identity map on the spectral cover, which is a Dehn twist along a separating loop.

\subsection{A Local Spectral Cover Model that Respects Shrink \& Expand Moves}

In this subsection, we will study the local changes of the topological spectral covers under the whole process of "shrinking a 2-valent $\circ$-vertex connecting two 3-valent $\bullet$-vertices and then expanding a 4-valent $\bullet$-vertex into another two 3-valent $\bullet$-vertices" on $A_{m-1}$-webs with $m\ge 3$. We call this process a "shrink \& expand" move.

Performing the shrink \& expand move twice successively on an $A_{m-1}$-web, as shown in Figure~\ref{fig:trimon}, the web returns back to the initial web.

\begin{figure}[htbp]
\centering
\includegraphics[width=1.0\textwidth]{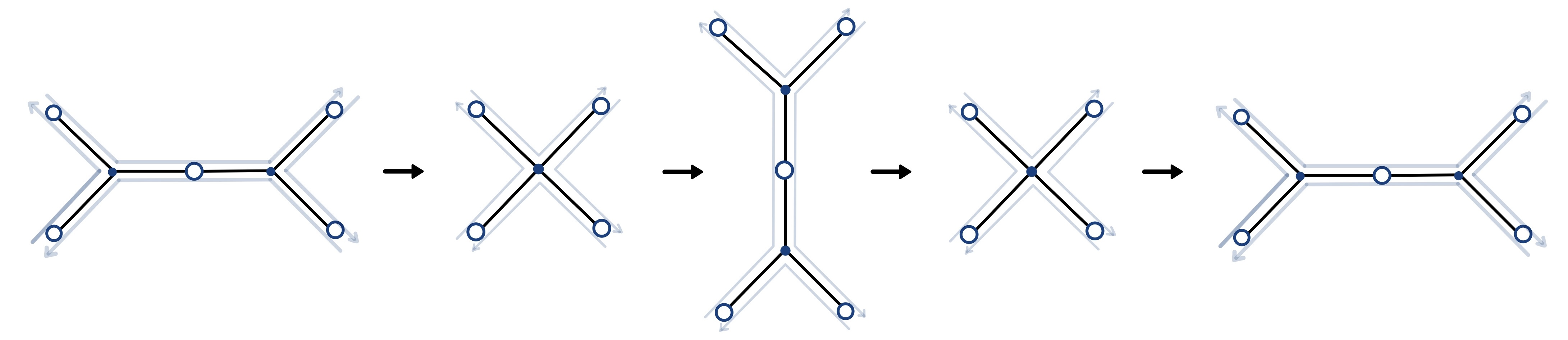}
\caption{Performing "shrinking a 2-valent $\circ$-vertex connecting two 3-valent $\bullet$-vertices and then expanding a 4-valent $\bullet$-vertex into two 3-valent $\bullet$-vertices" twice.}
\label{fig:trimon}
\end{figure}

At the level of the spectral cover, the shrink \& expand move coalesces two simple ramification points on the spectral cover into a double ramification point, and then splits the double ramification point to two simple ramification points. The spectral cover remains non-singular through the process. The local behaviour of the spectral cover is shown in Figure~\ref{fig:shrinkexpand}. A specific example is shown in Figure~\ref{fig:3rami}.

\begin{figure}[htbp]
\centering
\includegraphics[width=1.0\textwidth]{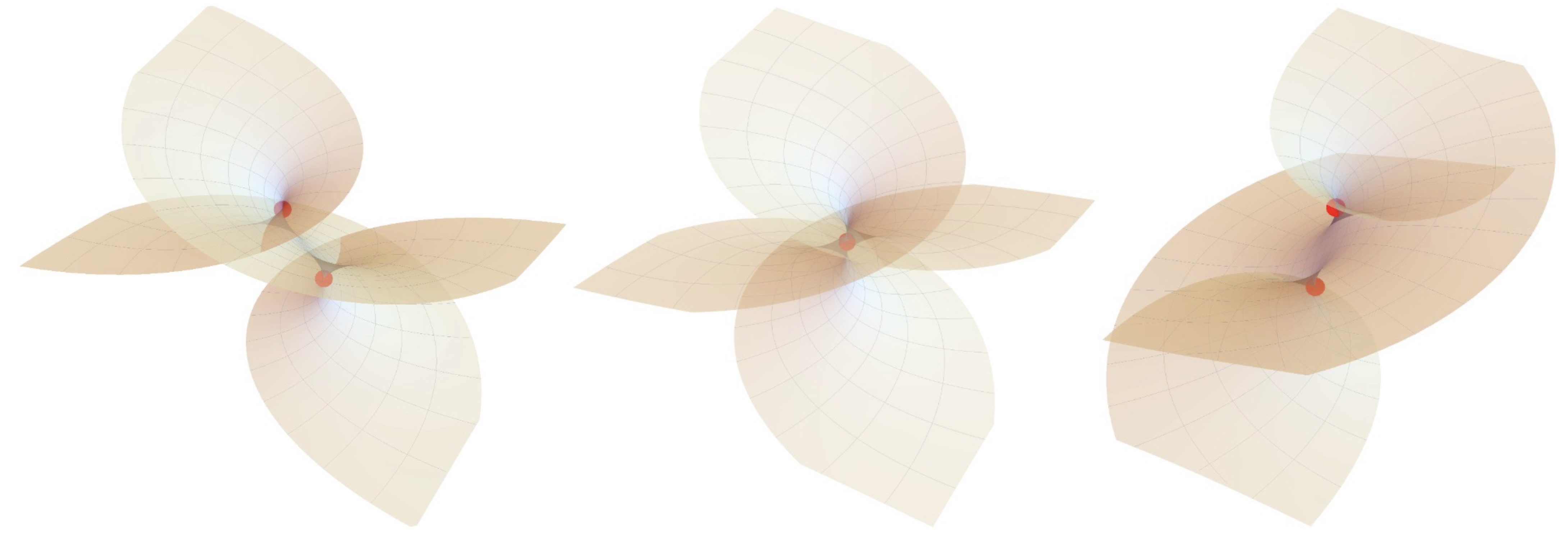}
\caption{Local change of the spectral cover under a shrink \& expand move. (Mathematica animation code is provided in the appendix.)}
\label{fig:shrinkexpand}
\end{figure}

The shrink \& expand move is still its own inverse at the level of the spectral cover. Equivalently, performing the shrink \& expand move does not involve a choice of orientation, as opposed to performing a \ttt move on the spectral cover.
So performing the shrink \& expand move twice and returning to the initial spectral cover results in a trivial automorphism on the quiver lattice $\Lambda_{\Gamma}$.

\begin{lemma}
Compositions of two successive shrink $\&$ expand moves produce trivial automorphisms on the quiver lattice $\Lambda_{\Gamma}$.
\end{lemma}

\begin{lemma}
The transposition monodromies of the two simple ramification points affected by the shrink $\&$ expand move are $(ij)$ and $(jk)$ in $S_m$, $m\ge 3$. The double ramification point formed by the coalescing two simple ramification points has transposition monodromy $(ij)(jk)=(ijk)\in S_m$.
\end{lemma}




\newpage
\mysection{Hitchin Fibrations and Spectral Curves}

This section serves as a preliminary to Section 5. Readers already familiar with the basic material may wish to focus only on subsection~\ref{sec:4.1.4} and subsection~\ref{sec:4.3}, which contain material beyond the basics.

\medskip
Let $\mathring{C}=C\setminus P$ be a $n$-punctured complex curve, where $C$ is a smooth compact complex curve of genus-$g$ and $P=\{p_1,\cdots,p_n\}$ is a finite set of punctures, $2g-2+n>0$.

Let $K_{C}=\Omega_{C}^1$ be the canonical line bundle, and let $K_{C}(P)=K_{C}\otimes \OO_{C}(P)$ be the meromorphic/logarithmic twist. $\deg K_{C}(P)=2g-2+n$. $K_{C}(P)^{\otimes i}\simeq K_{C}^{\otimes i}(iP)$.

Let $G$ be $GL_m$ or $SL_m$, $m\ge 2$. In general, $G$ can be any reductive Lie group whose derived subgroup has Dynkin type $A_{m-1}$.

\subsection{Hitchin Fibrations and Good Hitchin Loci}

\subsubsection{Higgs Bundles and Hitchin Bases}

For $G=GL_m$, a Higgs bundle is a pair $(E,\phi)$, where $E$ is a rank $m$ holomorphic vector bundle on $C$, and $\phi\in H^0(C,\operatorname{End}(E)\otimes K_C(P))$ is the Higgs field. Locally, if $z$ is a coordinate on $C$, then $\phi$ looks like a matrix of meromorphic one-forms: $\phi=A(z)\,dz$. For $G=SL_m$, one imposes $\det E\simeq L_0$ for a fixed determinant line bundle $L_0$, and $\tr(\phi)=0$.

\begin{definition}[\cite{Hitchin}]
Let $\mathcal M_{\mathrm{Higgs}}(GL_m)$ be the moduli space of
$GL_m$-Higgs bundles $(E,\phi)$ on $C$.
The $GL_m$ Hitchin base is
\begin{equation}
B_{GL_m,C,P}
:=
\bigoplus_{i=1}^m H^0(C,K_C(P)^{\otimes i}).
\end{equation}
The Hitchin map is
\begin{equation}
h_{GL_m}:
\mathcal M_{\mathrm{Higgs}}(GL_m)
\to
B_{GL_m,C,P},
\qquad
(E,\phi)\mapsto
\bigl(
\operatorname{tr}\phi,
\operatorname{tr}\phi^2,
\ldots,
\operatorname{tr}\phi^m
\bigr).
\end{equation}
\end{definition}

\begin{definition}[[\cite{Hitchin}]]
Let $\mathcal M_{\mathrm{Higgs}}(SL_m)$ be the moduli space of
$SL_m$-Higgs bundles $(E,\phi)$ on $C$, with $\det E\simeq L_0$ for a fixed determinant line bundle $L_0$, and $\tr \phi=0$.
The $SL_m$ Hitchin base is
\begin{equation}
B_{SL_m,C,P}
:=
\bigoplus_{i=2}^m H^0(C,K_C(P)^{\otimes i}).
\end{equation}
The Hitchin map is
\begin{equation}
h_{SL_m}:
\mathcal M_{\mathrm{Higgs}}(SL_m)
\to
B_{SL_m,C,P},
\qquad
(E,\phi)\mapsto
\bigl(
\operatorname{tr}\phi^2,
\ldots,
\operatorname{tr}\phi^m
\bigr).
\end{equation}
\end{definition}

\begin{lemma}[\cite{Goncharov}]
Let $G$ be $GL_m$. The two descriptions of the Hitchin base
\begin{equation}
\begin{array}{cc}
&B_{G,C,P}=\bigoplus_{i=1}^m H^0(C,K_C(P)^{\otimes i})\\[2pt]
\text{and}\quad
&B_{G,C,P}
=
H^0\!\left(C,(K_C(P)\otimes \mathfrak h)/W\right)
\end{array}
\end{equation}
are equivalent, where $\mathfrak h$ is a Cartan subalgebra of
$\mathfrak{gl}_m$ and $W\cong S_m$ is its Weyl group.
\end{lemma}

\begin{proof}
By Chevalley restriction, $\mathbb C[\mathfrak g]^G
\cong
\mathbb C[\mathfrak h]^W$. So $G$-conjugation-invariant functions on the full Lie algebra are the same as Weyl-invariant functions on the Cartan.

Now choose homogeneous generators $p_1,\dots,p_r\in \mathbb C[\mathfrak h]^W$ with degrees $\deg p_i=d_i$. Take $\mathfrak h
=
\{(\lambda_1,\cdots,\lambda_m)\}$, with $W=S_m$, and choose
\begin{equation}
\begin{array}{l}
p_1(\lambda_1,\cdots,\lambda_m)
=(-1)^1\Sigma_{i=1}^{m}\lambda_m,\\[2pt]
p_2(\lambda_1,\cdots,\lambda_m)
=(-1)^2\Sigma_{1\leq i<j\leq m}\lambda_i\lambda_j,\\[2pt]
\cdots\\[2pt]
p_m(\lambda_1,\cdots,\lambda_m)
=(-1)^m\prod_{i=1}^{m}\lambda_i.
\end{array}
\end{equation}
Then the invariant ring is freely generated by homogeneous invariants:
\begin{equation}
\mathbb C[\mathfrak h]^{S_m}
=
\mathbb C[p_1,p_2,\cdots,p_m].
\end{equation}

\medskip
Now define $\Theta:
H^0\!\left(C,(K_C(P)\otimes \mathfrak h)/W\right)
\to
\bigoplus_{i=1}^m H^0(C,K_C(P)^{\otimes i})$ as follows: Take a section $s\in H^0\!\left(C,(K_C(P)\otimes \mathfrak h)/W\right)$. Locally on an open set $U\subset C$, choose a trivialization of $K_C(P)$. Then $s$ is locally represented by a $W$-orbit of $K_C(P)\otimes\mathfrak h$-valued functions $s|_U:U\to \mathfrak h/W$. Let $x_U=s|_U(U)=(\lambda_1,\dots,\lambda_m)$, where the $\lambda_i$ are local $K_C(P)$-valued eigenvalue functions. Then set
\begin{equation}
\Theta(s|_U)
=
\bigl(p_1(x_U),\dots,p_m(x_U)\bigr).
\end{equation}

Because $p_i$ is homogeneous of degree $d_i$, the local expression $p_i(x_U)$ transforms as a section of $K_C(P)^{\otimes d_i}$. So globally $p_i(s)\in H^0(C,K_C(P)^{\otimes d_i})$. Thus $\Theta(s) \in \bigoplus_{i=1}^m H^0(C,K_C(P)^{\otimes i})$.

The map is $[\lambda_1,\lambda_2,\dots,\lambda_m]
\mapsto
(a_1,a_2,\dots,a_m)$, where
\begin{equation}
\prod_{i=1}^m(\lambda-\lambda_i)
=
\lambda^m+a_1\lambda^{m-1}+\cdots+a_m.
\end{equation}

\medskip
Now we construct the inverse map, take $a=(a_1,\cdots ,a_m)\in \bigoplus_{i=1}^m H^0(C,K_C(P)^{\otimes i})$. Locally trivialize $K_C(P)$ by a frame $e$ on an open set $U\subset C$. Write $a_1=a_1|_U\cdot e^1$, $\cdots$, $a_m=a_m|_U\cdot e^m$. Then define a local section $s|_U \in H^0\!\left(U,(K_C(P)\otimes \mathfrak h)/W\right)$ by unordered roots $(t_1,\cdots,t_m)$ of $t^m+a_1t^{m-1}+\cdots+a_m=0$. Multiplying back by $e$, the eigenvalues are $\lambda_i=t_i e$.

If we change the trivialization $e\mapsto f e$, then $a_i\mapsto f^{-i}a_i$, and the roots transform by $t_i\mapsto f^{-1}t_i$. Therefore $t_i e$ is unchanged. So the construction patches globally. Thus we defined a section $s\in H^0\!\left(C,(K_C(P)\otimes \mathfrak h)/W\right)$. At each point $x\in C$, $s(x)$ is the unordered tuple of m roots of $P_a(x)=\lambda^m+a_1(x)\lambda^{m-1}+\cdots+a_m(x)=0$.

There is no global choice of the roots. The roots may permute locally around branch points, which is why the target is the quotient by $W$.

The inverse map sends $(a_1,a_2,\dots,a_m)$ to the unordered roots of $\lambda^m+a_1\lambda^{m-1}+\cdots+a_m=0$.
\end{proof}

\begin{lemma}[\cite{Goncharov}]
Let $G$ be $SL_m$. The two descriptions of the Hitchin base
\begin{equation}
\begin{array}{cc}
&B_{SL_m,C,P}
=\bigoplus_{i=2}^m H^0(C,K_C(P)^{\otimes i})\\[2pt]
\text{and}\quad
&B_{SL_m,C,P}
=
H^0\!\left(C,(K_C(P)\otimes\mathfrak h)/W\right)
\end{array}
\end{equation}
are equivalent, where $\mathfrak h$ is a Cartan subalgebra of $\mathfrak{sl}_m$ and
$W\cong S_m$ is its Weyl group.
\end{lemma}

\subsubsection{Dimensions of Hitchin Bases}
For $i\ge 1$, Riemann--Roch gives
\begin{equation}
h^0(C,K_C(P)^{\otimes i})
=h^1(C,K_C(P)^{\otimes i})+
\deg K_C(P)^{\otimes i}+1-g.
\end{equation}
By Serre duality,
\begin{equation}
h^1(C,K_C(P)^{\otimes i})
=
h^0(C,K_{C}\otimes (K_C(P)^{\otimes i})^{-1}).
\end{equation}
So
\begin{equation}
h^0(C,K_C(P)^{\otimes i})
=
h^0(C,K_{C}\otimes (K_C(P)^{\otimes i})^{-1})+
\deg K_C(P)^{\otimes i}+1-g.
\end{equation}
Because $\deg K_{C}\otimes (K_C(P)^{\otimes i})^{-1}=2g-2-\deg K_C(P)^{\otimes i}=2g-2-i(2g-2+n)<0$,
we have $h^0(C,K_{C}\otimes (K_C(P)^{\otimes i})^{-1})=0$,
and hence
\begin{equation}
h^0(C,K_C(P)^{\otimes i})
=
i(2g-2+n)+1-g.
\end{equation}

For $GL_m$,
$B_{GL_m,C,P}
=
\bigoplus_{i=1}^m H^0(C,K_C(P)^{\otimes i})$.
Hence
\begin{equation}
\dim_{\mathbb C}B_{GL_m,C,P}
=
\sum_{i=1}^m
(i(2g-2+n)+1-g)
=
m^2(g-1)
+
\frac{m(m+1)}2 n.
\end{equation}

For $SL_m$, the trace term is removed,
$B_{SL_m,C,P}
=
\bigoplus_{i=2}^m H^0(C,K_C(P)^{\otimes i}).$
Since
$B_{GL_m,C,P}
=
H^0(C,K_C(P))\oplus B_{SL_m,C,P},$
and
$h^0(C,K_C(P))=g-1+n,$
we obtain
\begin{equation}
\dim_{\mathbb C}B_{SL_m,C,P}
=
\dim_{\mathbb C}B_{GL_m,C,P}-(g-1+n)
=
(m^2-1)(g-1)
+
\left(\frac{m(m+1)}2-1\right)n.
\end{equation}

\subsubsection{Spectral Curves}

\begin{definition}[\cite{GRR14}]
Let $p:\Tot(K_C(P))\to C$ be the bundle projection, and let $\lambda\in
H^0\!\left(
\Tot(K_C(P)),p^*K_C(P)
\right)$
be the tautological section. The spectral curve of a Higgs bundle
$(E,\phi)$ is defined by the characteristic equation
\begin{equation}
\det(\lambda I-p^*\phi)=0.
\end{equation}
Therefore the equation defines a curve
$\Sigma\subset\Tot(K_C(P))$.
Geometrically, $\Sigma$ records the eigenvalues of the Higgs field.
\end{definition}

A point $a=(a_1,\ldots,a_m)\in B_{G,C,P}$ determines a spectral curve $\Sigma_a\subset \operatorname{Tot}(K_C(P))$ by
\begin{equation}
\Sigma_a
:=
\left\{
\lambda\in \operatorname{Tot}(K_C(P))
\;\middle|\;
P_a=\lambda^m+a_1\lambda^{m-1}+\cdots+a_m=0
\right\}.
\end{equation}
Therefore the Hitchin base parametrizes spectral curves.

The projection $\operatorname{Tot}(K_C(P))\to C$ induces the spectral cover map
\begin{equation}
\pi_a:\Sigma_a\to C.
\end{equation}
It is a $m:1$ ramified cover of $C$. For generic $a\in B_{G,C,P}$, $\pi_a$ has only simple
ramification points, whose projections, i.e., the branch points of the
spectral cover, are the simple zeros of the discriminant
\begin{equation}
\Delta_a
=
\prod_{1\leq i<j\leq m}(\lambda_i-\lambda_j)^2
\in
H^0(C,K_{C}(P)^{\otimes m(m-1)}).\end{equation}
of equation $P_a=0$, where $\{\lambda_i\}_{i=1,\ldots,m}$ are the
roots. Each eigenvalue differential $\lambda_i\in H^0(C,K_{C}(P))$ has at most a simple pole at every puncture. $\Delta_a$ may
have pole order $m(m-1)$ at each puncture if viewed as an ordinary meromorphic $m(m-1)$-differential. The discriminant divisor has degree
\begin{equation}
\deg \operatorname{div}_{K_C(P)^{\otimes m(m-1)}}(\Delta_a)
=
m(m-1)(2g-2+n).
\end{equation}

For generic $a\in B_{G,C,P}$, $\Delta_a$ has only simple zeros, each corresponds to the collision of exactly two eigenvalues and hence to
a simple branch point of the spectral cover $\pi_b:\Sigma_a\to C$. Hence the total number of branch points is $m(m-1)(2g-2+n)$. \cite{Goncharov}

For $GL_m$, the characteristic polynomial has the form
\begin{equation}
\lambda^m+a_1\lambda^{m-1}+a_2\lambda^{m-2}+\cdots+a_m=0,
\end{equation}
where $a_i\in H^0(C,K_C(P)^{\otimes i})$.

For $SL_m$, the trace term $a_1=0$, so the spectral equation has no $\lambda^{m-1}$ term:
\begin{equation}
\lambda^m+a_2\lambda^{m-2}+\cdots+a_m=0.
\end{equation}

If $L$ is the line bundle on $C$ in whose total space the spectral curve lies, and if $\Sigma_a$ is smooth and compact, then
\begin{equation}
g(\Sigma_a)
=1+m(g-1)+\frac{m(m-1)}{2}\deg L.
\end{equation}
For the case we are considering, $L = K_{C}(P=-\Sigma_{i=1}^{n}p_i)$, $\deg L = 2g-2-n$, this becomes
\begin{equation}
g(\Sigma_a)=1 - \frac{m}{2}(m(2-2g)-mn+n).
\end{equation}
This is the same as Equation~\ref{eq:genus}.

\subsubsection{Good Hitchin Loci and Discriminant Loci}
\label{sec:4.1.4}

\begin{definition}
Define the good locus in the Hitchin base as
\begin{equation}
B_{G,C,P}^{\mathrm{good}}
:=
\left\{
a\in B_{G,C,P}:
\Sigma_a \text{ is smooth, }
\pi_a \text{ is simply branched and } \operatorname{Br}(\pi_a)\cap P=\varnothing
\right\}\subset B_{G}.
\end{equation}
Define further
$B_{G,C,P}^{\mathrm{good},S_m}
:=
\left\{
x\in B_{G,C,P}^{\mathrm{good}}:
\operatorname{Mon}(\pi_x)=S_m
\right\}$.
\end{definition}

Assume that $\Sigma_a$ is smooth. Then, in characteristic zero, the discriminant divisor is the branch divisor with multiplicities:
$\operatorname{ord}_q(\Delta_a)
=
\sum_{p\in\pi_a^{-1}(q)}(e_p-1).$ \cite{Basok}

\begin{lemma}[\cite{Basok}]
The discriminant $\Delta_a$ has only simple zeros on
$\mathring C=C\setminus P$ and is nonzero at every marked point
$p_i\in P$ if and only if $\Sigma_a$ is smooth, $\pi_a$ is simply
branched over $\mathring C$, and $\pi_a$ is unramified over $P$.
\end{lemma}

Therefore,
$B_{G,C,P}^{\mathrm{good}}
=
\left\{
a\in B_{G,C,P}:
\Delta_a \text{ has only simple zeros}
\right\}$.
Consequently,
$\deg\operatorname{div}(\Delta_a)
=
m(m-1)(2g-2+n).$
Hence for $a\in B_{G,C,P}^{\mathrm{good}}$,
$|\operatorname{Br}(\pi_a)|
=m(m-1)(2g-2+n).$

\begin{remark}
$\operatorname{ord}_q(\Delta_a)=2$
may describe two different phenomena:
triple ramification: $e_p=3,$ $\Sigma_a$ smooth; collision of simple branches: $\Sigma_a$ nodal.
The distinction is detected by the spectral polynomial: ramification: $P=\partial_\lambda P=0$; singularity of $\Sigma$: $P=\partial_\lambda P=\partial_zP=0$.
\end{remark}

Let $B_{G}$ be the universal Hitchin base, and denote the multiple-zero discriminant locus in the universal Hitchin base by $\mathcal D
:=
\left\{
b\in B_{G}:
\Delta_b
\text{ has at least one zero of multiplicity }\ge2
\right\}$ \cite{Basok}.
At a generic point of $\mathcal D$, the discriminant $\Delta_b$ has
exactly one double zero and all its other zeros are simple.
$\mathcal D$
is a complex divisor in the universal Hitchin base $B_G$. In particular, $\operatorname{codim}_{\mathbb C}\mathcal D=1$. \cite{Basok}

At a generic point of $\mathcal D$ where the discriminant has a double
zero, there are up to three possible local geometries: the boundary,
Maxwell, and caustic types. For $m=2$ only the boundary
type occurs; for $m=3$ the boundary and caustic types occur; and for
$m\ge4$ all three types may occur. \cite{Basok}

\begin{definition}[\cite{Basok}]
The boundary stratum
$\mathcal D_{\mathrm{bdry}}\subset B_G$
is generically the locus where the spectral curve develops a nodal
singularity over a double zero of the discriminant. Thus the spectral
curve is singular.
\end{definition}

\begin{definition}[\cite{Basok}]
The Maxwell stratum
$\mathcal D_{\mathrm{Max}}\subset B_G$
is generically the locus where two distinct simple ramification points
$p_1\neq p_2$ have the same branch value
$\pi_b(p_1)=\pi_b(p_2)=q.$
Then
$\operatorname{ord}_q(\Delta_b)=2,$
and the spectral curve is generically smooth.
\end{definition}

\begin{definition}[\cite{Basok}]
The caustic
$\mathcal D_{\mathrm{cau}}\subset B_G$
is generically the locus where two simple ramification points coalesce
into a single ramification point $p$ of ramification index $e_p=3$.
Then
$\operatorname{ord}_{\pi_b(p)}(\Delta_b)=2,$
and the spectral curve is generically smooth.
\end{definition}

Therefore, for $m\ge4$, the full discriminant divisor decomposes as
$\mathcal D
=
\mathcal D_{\mathrm{bdry}}
\cup
\mathcal D_{\mathrm{Max}}
\cup
\mathcal D_{\mathrm{cau}}$
\cite{Basok}.

\begin{remark}
The boundary--Maxwell--caustic decomposition is established on the
universal Hitchin base.
For a fixed curve $(C,P)$, restricting $\mathcal D
=
\mathcal D_{\mathrm{bdry}}
\cup
\mathcal D_{\mathrm{Max}}
\cup
\mathcal D_{\mathrm{cau}}$ to the particular
Hitchin base $B_{G,C,P}$ gives the corresponding geometric strata.
\end{remark}

Define
\begin{equation}
\mathcal D_{C,P}
:=
\mathcal D\cap B_{G,C,P}
=
\left\{
a\in B_{G,C,P}:
\Delta_a\text{ has at least one multiple zero}
\right\}.
\end{equation}
The locus $\mathcal D_{C,P}$ is a complex codimension-one locus in
$B_{G,C,P}$ \cite{Basok}.
Since $B_{G,C,P}^{\mathrm{good}}$ also requires the spectral cover to
be unramified over the marked points, its complement is
\begin{equation}
B_{G,C,P}^{\mathrm{bad}}
=
\mathcal D_{C,P}
\cup
\left\{
a\in B_{G,C,P}:
\Delta_a(p_i)=0
\text{ for some }p_i\in P
\right\}.
\end{equation}
This is a proper complex analytic locus of complex codimension-one in the complex vector space $B_{G,C,P}$. Consequently,
$B_{G,C,P}^{\mathrm{good}}$
is path-connected.

For later use, we define
\begingroup
\small
\begin{equation}
\mathcal D_{bdry,C,P}^\circ
:=
\left\{
b\in \mathcal D_{bdry,C,P}
\;\middle|\;
\begin{array}{l}
\mathcal D_{bdry,C,P} \text{ is smooth at } b,\\
\Delta_b \text{ has exactly one double zero},\\
\text{all other zeros of } \Delta_b \text{ are simple},\\
\pi_b^{-1}(q)
\text{ contains exactly one nodal singularity},\\
b \text{ lies on no other discriminant component}
\end{array}
\right\}.
\end{equation}

\begin{equation}
\mathcal D_{cau,C,P}^\circ
:=
\left\{
b\in \mathcal D_{cau,C,P}
\;\middle|\;
\begin{array}{l}
\mathcal D_{cau,C,P} \text{ is smooth at } b,\\
\Delta_b \text{ has exactly one double zero } q,\\
\text{all other zeros of } \Delta_b \text{ are simple},\\
\pi_b^{-1}(q) \text{ contains exactly one double ramification point},\\
b \text{ lies on no other discriminant component}
\end{array}
\right\}.
\end{equation}

\begin{equation}
\mathcal D_{Max,C,P}^\circ
:=
\left\{
b\in \mathcal D_{Max,C,P}
\;\middle|\;
\begin{array}{l}
\mathcal D_{Max,C,P} \text{ is smooth at } b,\\
\Delta_b \text{ has exactly one double zero } q,\\
\text{all other zeros of } \Delta_b \text{ are simple},\\
\pi_b^{-1}(q) \text{ contains exactly two distinct simple ramification points},\\
b \text{ lies on no other discriminant component}
\end{array}
\right\}.
\end{equation}
\endgroup

\subsubsection{Spectral Correspondence}

For $a\in B_{G,C,P}$, the Hitchin fiber
$h^{-1}(a)$
consists of Higgs bundles with fixed characteristic polynomial, and hence
with fixed spectral curve $\Sigma_a$.

A point $a\in B_{G,C,P}$ determines a spectral curve $\Sigma_a$, but not the
full Higgs bundle. By the spectral correspondence, if $\Sigma_a$ is smooth,
a Higgs bundle in the fiber determines a spectral line bundle $\mathcal L$ on
$\Sigma_a$, and conversely the Higgs bundle is recovered by
$E=\pi_{a*}\mathcal L,$
with Higgs field induced by multiplication by the tautological section.
For $G=SL_m$, the spectral line bundle is additionally required to satisfy
the fixed-determinant condition
$\det(\pi_{a*}\mathcal L)\simeq\mathcal O_C$
\cite{BNR}.

More precisely, let
$p:\Tot(K_C(P))\to C$
be the bundle projection. The tautological section
$\lambda\in
H^0\!\left(\Tot(K_C(P)),p^*K_C(P)\right)$
restricts to
$\lambda\in
H^0\!\left(\Sigma_a,\pi_a^*K_C(P)\right).$
Multiplication by $\lambda$ defines a morphism
\begin{equation}
m_\lambda:
\mathcal L
\to
\mathcal L\otimes\pi_a^*K_C(P),
\qquad
s\mapsto\lambda s,
\end{equation}
for every local section $s$ of $\mathcal L$.
Pushing forward and applying the projection formula gives
\begin{equation}
\phi:
\pi_{a*}\mathcal L
\to
\pi_{a*}\!\left(
\mathcal L\otimes\pi_a^*K_C(P)
\right)
\cong
\pi_{a*}\mathcal L\otimes K_C(P).
\end{equation}

Thus, morally, the spectral correspondence is
$(C,E,\phi)
\Longleftrightarrow
(\Sigma_a,\mathcal L,\lambda).$

Fiberwise, for $x\in C$ away from the branch locus, write
$\pi_a^{-1}(x)=\{q_1,\ldots,q_m\}.$
Then
$E_x
=
\bigoplus_{i=1}^m\mathcal L_{q_i},$
and the Higgs field acts diagonally:
\begin{equation}
\phi_x
=
\begin{pmatrix}
\lambda(q_1) & & 0\\
& \ddots &\\
0 & & \lambda(q_m)
\end{pmatrix}.
\end{equation}
Thus the values of $\lambda$ on the sheets are precisely the eigenvalues
of $\phi$.

\subsection{Jacobian and Prym Varieties}

\subsubsection{Jacobian}
The sheaf $\OO_C^*$ is the sheaf of nowhere-zero holomorphic functions: $\OO_C^*(U)=\{f:U\to \C^*\mid f\text{ holomorphic}\}$. Line bundles are glued using transition functions valued in $\OO_C^*$.

Suppose $L$ is a line bundle on $C$ and $\{U_i\}$ is a trivializing open cover. On overlaps $U_i\cap U_j$, the two trivializations differ by functions $g_{ij}\in \OO_C^*(U_i\cap U_j)$. These satisfy the cocycle condition $g_{ij}g_{jk}g_{ki}=1$ on triple overlaps. Thus line bundles are classified by the cohomology group $H^1(C,\OO_C^*)$.

The Picard group is the set of isomorphism classes of holomorphic line bundles on C. The group operation is tensor product: $[L_1]+[L_2]=[L_1\otimes L_2]$. The identity is $\OO_C$, and the inverse of $L$ is the dual line bundle $L^\vee$. Therefore $\Pic(C)\simeq H^1(C,\OO_C^*)$ as abelian groups.

Every line bundle on a compact smooth curve has an integer degree. This gives a homomorphism $\deg:\Pic(C)\to \Z$. The Jacobian is defined as the kernal of deg:
\begin{equation}
\Jac(C)=\Pic^0(C)\simeq \operatorname{Div}^0(C)/\operatorname{Prin}(C).
\end{equation}
If $C$ has genus $g_C$, then $\Jac(C)$ is a complex torus of dimension $g_C$:
\begin{equation}
\Jac(C)\cong H^1(C,\mathcal O_C)/H^1(C;\mathbb Z)\cong H^0(C,K_C)^\vee/H_1(C;\Z),
\end{equation}
where $H_1(C;\Z)$ embeds into $H^0(C,K_C)^\vee$ by integration:
\begin{equation}
\gamma\mapsto\left(\omega\mapsto \int_\gamma \omega\right).
\end{equation}

Let $a\in B_{GL_m,C,P}^{\mathrm{good}}$. The $GL_m$ Hitchin fiber over
$a$ is a Picard variety of the spectral curve $\Sigma_a$:
$h_{GL_m}^{-1}(a)\cong \Pic^{d'}(\Sigma_a)$
for a suitable degree $d'$ \cite{BNR}. Since
$\Pic^{d'}(\Sigma_a)$ is a torsor for
$\Pic^0(\Sigma_a)=\Jac(\Sigma_a),$
the $GL_m$ Hitchin fiber is a torsor for the Jacobian of $\Sigma_a$.
Its first homology lattice is
\begin{equation}
H_1\!\left(h_{GL_m}^{-1}(a);\mathbb Z\right)
\cong
H_1\!\left(\Jac(\Sigma_a);\mathbb Z\right)
\cong
H_1(\Sigma_a;\mathbb Z).
\end{equation}

\subsubsection{Prym}
\begin{definition}[\cite{DonagiPantev}]
Let $\pi:\Sigma_a\to C$ be a finite spectral cover of degree $m$. The norm map is a homomorphism
\begin{equation}
\operatorname{Nm}_{\Sigma_a/C}:
\operatorname{Jac}(\Sigma_a)
\to
\operatorname{Jac}(C).
\end{equation}
If $L\simeq\mathcal O_{\Sigma_a}(D)$, $D=\sum_{p\in\Sigma_a}n_p p$, then
\begin{equation}
\operatorname{Nm}_{\Sigma_a/C}(L)
=
\mathcal O_C(\pi_*D),
\end{equation}
where $\pi_*D
=
\sum_{p\in\Sigma_a}n_p\,\pi(p)$. Hence $\operatorname{deg}(\operatorname{Nm}_{\Sigma_a/C}(L))=\operatorname{deg}(L)$.

Equivalently, the norm map is characterized by
\begin{equation}
\operatorname{Nm}_{\Sigma_a/C}(L)
\simeq
\det(\pi_*L)\otimes
\det(\pi_*\mathcal O_{\Sigma_a})^{-1}.
\end{equation}
Thus
\begin{equation}
\det(\pi_*L)
\simeq
\operatorname{Nm}_{\Sigma_a/C}(L)
\otimes
\det(\pi_*\mathcal O_{\Sigma_a}).
\end{equation}
\end{definition}

One has \cite{BNR}
\begin{equation}
\deg(\pi_{a*}\mathcal O_{\Sigma_a})
=
-\frac{m(m-1)}{2}\deg K_C(P)
=
-\frac{m(m-1)}{2}
\bigl(2g_C-2+\deg P\bigr).
\end{equation}
Therefore, for $E=\pi_{a*}\mathcal L$,
\begin{equation}
\deg E
=
\deg\mathcal L+\deg(\pi_{a*}\mathcal O_{\Sigma_a})
=
\deg\mathcal L
-
\frac{m(m-1)}{2}
\bigl(2g_C-2+\deg P\bigr).
\end{equation}
If $\deg P=n$, then
$\deg\mathcal L
=
\deg E
+
\frac{m(m-1)}{2}(2g_C-2+n)$.

Away from the branch locus, if
$\pi_a^{-1}(x)=\{p_1,\dots,p_m\}$,
then fiberwise \cite{HauselThaddeus}
\begin{equation}
\operatorname{Nm}_{\Sigma_a/C}(\mathcal L)_x
\simeq
\mathcal L_{p_1}\otimes\cdots\otimes\mathcal L_{p_m}.
\end{equation}

\begin{definition}[\cite{DonagiPantev}]
The Prym variety of the cover $\pi:\Sigma_a\to C$ is
\begin{equation}
\operatorname{Prym}(\Sigma_a/C)
:=
\ker\!\left(
\operatorname{Nm}_{\Sigma_a/C}:
\operatorname{Jac}(\Sigma_a)\to\operatorname{Jac}(C)
\right)^0.
\end{equation}
Here the superscript $0$ denotes the connected component of the identity.
\end{definition}

Thus the Prym variety parametrizes degree-zero line bundles
$L\in\operatorname{Jac}(\Sigma_a)$ satisfying $\operatorname{Nm}_{\Sigma_a/C}(L)
\simeq
\mathcal O_C$, and lying in the identity component of the kernel.

For every compact Riemann surface $X$, $H_1(\operatorname{Jac}(X);\mathbb Z)
\cong
H_1(X;\mathbb Z)$. Therefore, the map induced by the norm map on homology is
the pushforward
\begin{equation}
\pi_*:
H_1(\Sigma_a;\mathbb Z)
\to
H_1(C;\mathbb Z).
\end{equation}

\begin{lemma}[\cite{Goncharov}]
Let $\pi_a:\Sigma_a\to C$ be a smooth finite spectral cover, and let
\begin{equation}
P:=\operatorname{Prym}(\Sigma_a/C)
=
\ker\!\left(
\operatorname{Nm}:J\Sigma_a\to JC
\right)^0.
\end{equation}
Then there is a canonical isomorphism
\begin{equation}
H_1(P;\mathbb Z)
\cong
\ker\!\left(
\pi_*:H_1(\Sigma_a;\mathbb Z)\to H_1(C;\mathbb Z)
\right).
\end{equation}
\end{lemma}

\begin{proof}
Choose Abel--Jacobi maps $a_{\Sigma_a}:\Sigma_a\to J\Sigma_a,$ $a_C:C\to JC.$ They induce canonical isomorphisms
\begin{equation}
(a_{\Sigma_a})_*:
H_1(\Sigma_a;\mathbb Z)
\xrightarrow{\sim}
H_1(J\Sigma_a;\mathbb Z),\qquad
(a_C)_*:
H_1(C;\mathbb Z)
\xrightarrow{\sim}
H_1(JC;\mathbb Z).
\end{equation}
The identity $\operatorname{Nm}\circ a_{\Sigma_a}
=
a_C\circ\pi$
implies that the following diagram commutes:
\begin{equation}
\begin{array}{ccc}
H_1(\Sigma_a;\mathbb Z)& \xrightarrow{(a_{\Sigma_a})_*}& H_1(J\Sigma_a;\mathbb Z) \\
\downarrow\pi_* && \downarrow(\operatorname{Nm})_* \\
H_1(C;\mathbb Z)& \xrightarrow{(a_C)_*}& H_1(JC;\mathbb Z)
\end{array}
\end{equation}
Since the horizontal maps are isomorphisms, $\ker(\operatorname{Nm})_*
\cong
\ker\pi_*.$
Let $\iota:P\hookrightarrow J\Sigma_a$
be the inclusion. Since $P$ is the identity component of the kernel
of the norm map, $\iota_*$ identifies $H_1(P;\mathbb Z)$
with $\ker\!\left(
(\operatorname{Nm})_*:
H_1(J\Sigma_a;\mathbb Z)\to H_1(JC;\mathbb Z)
\right)$.
Consequently,
\begin{equation}
H_1(P;\mathbb Z)
\cong
\ker\!\left(
H_1(\Sigma_a;\mathbb Z)
\xrightarrow{\pi_*}
H_1(C;\mathbb Z)
\right),
\end{equation}
where the isomorphism is $\Phi
=
(a_{\Sigma_a})_*^{-1}\circ\iota_*.$
\end{proof}

In other words, the Prym lattice consists of cycles on the spectral curve
whose projection to $C$ is homologically trivial.

Let $a\in B_{SL_m,C,P}^{good}$. The $SL_m$ Hitchin fiber over $a$ is the Prym-type norm fiber. Fix a determinant line bundle $L_0\simeq \det(\pi_*\mathcal L)$ on $C$. The spectral line bundle $\mathcal L$ satisfies the norm condition:
\begin{equation}
\Nm_{\Sigma_a/C}(\mathcal L)\simeq L_0\otimes \det(\pi_*\OO_{\Sigma_a})^{-1}.
\end{equation}
Therefore
\begin{equation}
h_{SL_m}^{-1}(a)\simeq\Nm^{-1}\left(L_0\otimes \det(\pi_*\OO_{\Sigma_a})^{-1}\right),
\end{equation}
inside a suitable component of $\Pic(\Sigma_a)$. This fiber is not canonically equal to a group unless one chooses a base point. It is a torsor for the Prym variety $\Prym(\Sigma_a/C)$.

Its lattice is
\begin{equation}
H_1(h_{SL_m}^{-1}(a);\Z)\cong H_1(\Prym(\Sigma_a/C);\Z)\cong\ker\left(H_1(\Sigma_a;\Z)\xrightarrow{\pi_*} H_1(C;\Z)\right).
\end{equation}

\subsection{Properties of Geometric Spectral Covers}
\label{sec:4.3}

Let $\mathring{C}=C\setminus P$
be a $n$-punctured smooth complex curve of genus-$g$ satisfying $2g-2+n>0.$
Let $G$ be $GL_m$ or $SL_m$, $m \ge 2$.
Set $L:=K_{C}(P)$, $d:=\deg L=2g-2+n>0.$

\subsubsection{Monodromy Groups of Hitchin Spectral Curves}
Let $x$ be a point in the Hitchin base $B_{G,C,P}$, and $\pi_x:\Sigma_x\to C$ be the corresponding spectral curve.

\begin{lemma}[\cite{BNR}]
For any spectral curve
$\pi_x:\Sigma_x\to C,$
defined by
$P_x(\lambda)=\lambda^m+a_1\lambda^{m-1}+\cdots+a_m=0,$
one has
$(\pi_x)_*\mathcal O_{\Sigma_x}
\simeq
\mathcal O_{C}
\oplus L^{-1}\oplus L^{-2}
\oplus\cdots\oplus L^{-(m-1)}$, as $\mathcal O_{C}$-modules.
\end{lemma}

\begin{proof}
Locally on an open subset $U\subset C$, after trivializing $L$, the spectral algebra is
$
\mathcal O_U[\lambda]/(P_x(\lambda)).
$
Since $P_x$ is monic of degree $m$,
then $1,\lambda,\dots,\lambda^{m-1}$
form an $\mathcal O_U$-basis. Globally $\lambda^i$ transforms as $L^{-i}$ on the function-algebra
side, giving the claimed splitting.
\end{proof}

\begin{corollary}[\cite{BNR}]
Suppose
$\deg L>0.$
Then
$H^0(\Sigma_x,\mathcal O_{\Sigma_x})
=
\mathbb C.$ So $\Sigma_x$ is connected.
In particular, if $\Sigma_x$ is smooth, then it is irreducible.
\end{corollary}

\begin{proof}
Since
$\deg L^{-i}<0$ $(i>0)$,
$H^0(C,L^{-i})=0.$
Therefore
$H^0(\Sigma_x,\mathcal O_{\Sigma_x})
=
H^0(C,(\pi_x)_*\mathcal O_{\Sigma_x})
=
\bigoplus_{i=0}^{m-1}H^0(C,L^{-1})
=
H^0(C,\mathcal O_{C})
=
\mathbb C.$
A smooth proper curve with only constant global functions is
connected. If $\Sigma_x$ is additionally smooth, then its irreducible components cannot intersect. Thus $\Sigma_x$ has only one irreducible component.
\end{proof}

\begin{definition}[\cite{BKP}]
A finite morphism $f:Y\to X$
between smooth irreducible projective complex curves is called \emph{genuinely ramified} if it does not factor as
$Y\to X'\to X,$
with $X'\to X$ a non-trivial finite unramified holomorphic covering.
Equivalently, $f$ does not factor through any non-trivial finite-sheeted
topological covering space of $X$.
\end{definition}

\begin{proposition}[\cite{BKP}]
Let $f:Y\to X$ be a finite morphism between smooth irreducible projective complex curves.
Then $f$ is genuinely ramified if and only if $\mathcal O_X$ is the maximal semistable subbundle of $f_*\mathcal O_Y$.
\end{proposition}

\begin{corollary}[\cite{BKP}]
If $\deg L>0$ and $\Sigma_x$ is smooth.
Then $\pi_x:\Sigma_x\to C$ is genuinely ramified.
\end{corollary}

\begin{proof}
By the spectral-algebra splitting,
$(\pi_x)_*\mathcal O_{\Sigma_x}
\simeq
\mathcal O_{C}
\oplus L^{-1}\oplus\cdots\oplus L^{-(m-1)}.$
The first summand has degree $0$, whereas every other summand has
negative degree. Hence
$\mathcal O_{C}$
is the maximal semistable subbundle of
$(\pi_x)_*\mathcal O_{\Sigma_x}$.
The criterion of Biswas-Kumar-Parameswaran identifies this condition with genuine ramification.
\end{proof}

\begin{theorem}[\cite{BKP}]
Let
$f:Y\to X$
be a genuinely ramified simply branched morphism of degree-$m$ between smooth
irreducible projective curves.
Then the Galois group of its Galois closure, equivalently its
geometric monodromy group, is the full symmetric group $S_m$.
\end{theorem}

\begin{corollary}
\label{cor:good}
For any
$x\in B_{G,C,P}^{\mathrm{good}}$, one has $\operatorname{Mon}(\pi_x)=S_m.$
Therefore
$B_{G,C,P}^{\mathrm{good}}
=
B_{G,C,P}^{\mathrm{good},S_m}$.
\end{corollary}

\subsubsection{Hurwitz Equivalence}

Let $\mathring{S}=S\setminus P$ be a decorated surface with a finite puncture set $P$. Let $\pi_i:\Sigma_i\to S$, $i=1,2$, be ramified covers whose ramification sets avoid $P$.

$\pi_1$ and $\pi_2$ are \emph{Hurwitz equivalent relative to $P$} if there exist
orientation-preserving homeomorphisms
$F:\Sigma_1\to\Sigma_2$, $h: S\to S$
such that
$\pi_2\circ F=h\circ\pi_1,$
where
$h|_P=id_P$
and $h$ is isotopic to the identity relative to $P$. Thus the topological surface $S$ is unchanged, but its branch points
are allowed to move. We denote the corresponding class by
$[\pi]_{\mathrm{Hur}}.$

\begin{definition}
Let $r=m(m-1)(2g-2+n)$ be the number of distinct branch points. Define $\mathfrak C_{g,n,m}$ to be the set of Hurwitz-equivalence classes of connected degree-$m$
covers $\pi:\Sigma\to S$, such that
\begin{enumerate}
    \item $\pi$ is simply branched over $r$ distinct points;
    \item $\operatorname{Br}(\pi)\cap P=\varnothing$.
\end{enumerate}
\end{definition}

Fix a complex curve $\mathring{C}=C\setminus P$. Let $\pi_i:\Sigma_i\to C$, $i=1,2$, be ramified covers that are unramified over $P$.
$\pi_1$ and $\pi_2$ are \emph{Hurwitz isomorphic} if there exists an orientation-preserving homeomorphism
$F:\Sigma_1\to \Sigma_2$,
such that
$\pi_2\circ F=\pi_1.$
We denote the corresponding class by
$[\pi]$.

\begin{definition}
Let $H^0_{m,r}(C)$
denote the Hurwitz space of isomorphism classes of connected smooth degree-$m$ covers
$\pi:\Sigma \to C$,
such that
\begin{enumerate}
    \item $\pi$ is simply branched over $r$ distinct points;
    \item
    $
    \operatorname{Mon}(\pi)=S_m.
    $
\end{enumerate}
In addition, define
$H^{0,P}_{m,r}(C)
:=
\left\{
[\pi]\in H^0_{m,r}(C):
\operatorname{Br}(\pi)\cap P=\varnothing
\right\}.$
\end{definition}

\begin{remark}
$H^0_{m,r}(C)$ fixes the target complex curve
\emph{pointwise}, whereas $\mathfrak C_{g,n,m}$ forgets the complex
structure and allows isotopies of the base moving branch values.
\end{remark}

\begin{lemma}
\label{le:hurhit}
Each $x\in B_{G,C,P}^{\mathrm{good}}$ determines a Hurwitz isomorphism class $[\pi_x]\in H^{0,P}_{m,r}(C)$.
\end{lemma}

\begin{lemma}[Riemann Existence Principle\cite{CavalieriMiles}]
Let
$f:Y\to C$
be a finite orientation-preserving topological ramified cover of a
Riemann surface. There is a unique complex structure on $Y$ for which
$f:Y\to C$
is holomorphic.
\end{lemma}

\begin{corollary}
Let $W$ be an $A_{m-1}$-web on $\mathring{C}=C\setminus P$. Let $\pi_W:\Sigma_W\to C$ be its associated with topological spectral cover, where $\Sigma$ is a compact topological surface after filling in the $m$ punctures over each point of $P$.
By the Riemann existence principle, there is a unique complex structure on
$\Sigma_W$ for which $\pi_W$ is holomorphic.
\end{corollary}

\begin{lemma}
\label{le:hurweb}
Every spectral cover $\pi_W:\Sigma_W\to C$ constructed from a generic $A_{m-1}$-web $W$ on $C$ determines a Hurwitz isomorphism class $[\pi_W]\in H^{0,P}_{m,r}(C)$.
\end{lemma}

\begin{proof}
$\pi_W:\Sigma_W\to C$
is a connected degree-$m$ cover, simply branched at $r=m(m-1)(2g-2+n)$
points, unramified over $P$, and
$\operatorname{Mon}(\pi_W)=S_m.$ After compactification and application of the Riemann existence principle, it determines a point $[\pi_W]\in H^{0,P}_{m,r}(C).$
\end{proof}

\begin{lemma}
\label{le:irr}
Suppose
$H^0_{m,r}(C)$
is irreducible and
$H^{0,P}_{m,r}(C)\neq\varnothing.$
Then
$H^{0,P}_{m,r}(C)$ is irreducible.
\end{lemma}

\begin{proof}
A branch point in the finite set $P$ is
algebraically closed. Its complement is therefore Zariski open. A nonempty open subset of an irreducible variety is irreducible.
\end{proof}

\begin{theorem}[\cite{Kanev}]
\label{thm:irr}
For a fixed positive-genus target $C$, $H^0_{m,r}(C)$ is irreducible in the range $r\ge\max\{2,2m-4\}$. For genus zero the corresponding full-symmetric simply branched
Hurwitz space is irreducible.
\end{theorem}

\begin{remark}
For the logarithmic Hitchin number
$r=m(m-1)(2g-2+n)$, $d\ge1$,
one has
$m(m-1)\ge2m-4$
for every $m\ge2$. Therefore the required Hurwitz irreducibility bound is automatically
satisfied.
\end{remark}

\begin{remark}
Over $\mathbb C$, irreducibility implies connectedness, and hence
path-connectedness.
Thus any two covers in $H^0_{m,r}(C)$ can be joined by a continuous
path in the Hurwitz space. Along such a path, the branch points move
without colliding and the local ramification types remain fixed.
\end{remark}

\begin{lemma}
\label{le:hur}
If $[f_0]$ and $[f_1]$ lie in the same connected component in
$H^{0,P}_{m,r}(C)$, then
$[f_0]_{\mathrm{Hur}}
=[f_1]_{\mathrm{Hur}}.$
\end{lemma}

\begin{proof}
Choose a path
$[f_t]$, $t\in [0,1]$,
joining the two points. The $r$ branch values move in
$\mathring{C}$
without collisions. By isotopy extension, their trajectories is represented by an isotopy
$h_t:C\to C$, $h_t|_P=id.$
Parallel transport of the finite covers along the path gives a
homeomorphism
$F:X_0\to X_1$
such that
$f_1\circ F
=
h_1\circ f_0.$ This is Hurwitz equivalence relative to $P$.
\end{proof}

\begin{corollary}
\label{cor:equi}
For any $x\in B_{G,C,P}^{\mathrm{good}}$
and any generic $A_{m-1}$-web
$W$ on $C$, the spectral covers determine the same Hurwitz equivalence class in $\mathfrak C_{g,n,m}$.
\end{corollary}

\begin{corollary}
\label{cor:iso}
Let $\pi_0:\Sigma_0\to C$ be either a Hitchin spectral cover corresponding to a point in $B_{G,C,P}^{\mathrm{good}}$ or a spectral cover associated with a generic $A_{m-1}$-web.
Let $\pi_W:\Sigma_W\to C$
be a spectral cover associated with a generic $A_{m-1}$-web $W$.
After replacing $W$ by a proper isotopic representative, Hurwitz equivalence implies
that the two spectral covers are isomorphic over $\id_C$, i.e., they determine the same Hurwitz isomorphism class in $H_{m,r}^{0,P}(C)$.
\end{corollary}

\begin{proof}
Since the two spectral covers lie in the same Hurwitz equivalence class in $\mathfrak C_{g,n,m}$, by definition there exist
orientation-preserving homeomorphisms $F:\Sigma_0\to\Sigma_W$ and 
$h:C\to C$, 
$h\simeq \id_C$,
such that
$\pi_W\circ F=h\circ\pi_0.$

$W$ is defined only up to isotopy, then the web $W':=h^{-1}(W)$ represents the same isotopy class as $W$. For the corresponding spectral cover $\pi_{W'}:\Sigma_{W'}\to C,$
there exists an orientation-preserving homeomorphism
$F':\Sigma_0\to\Sigma_{W'}$
such that
$\pi_{W'}\circ F'=\pi_0.$
\end{proof}

\begin{corollary}
\label{cor:iso1}
For any point $x_0\in B_{G,C,P}^{\mathrm{good}}$, we can find a generic $A_{m-1}$-web $W$, such that $\pi_0:\Sigma_0\to C$ and $\pi_W:\Sigma_W\to C$ are Hurwitz isomorphic, i.e., there is an orientation-preserving homeomorphism $f:\Sigma_W\to \Sigma_0$, such that $\pi_0\circ f=\pi_W$.

Then there is an induced isomorphism $f_*:H_1(\Sigma_W;\mathbb Z)\to H_1(\Sigma_0;\mathbb Z)$, and its restriction $f_*|_{\Lambda_{\Gamma}^{\circ}}:\ker (H_1(\Sigma_W;\mathbb Z)\to H_1(C;\mathbb Z))\to \ker(H_1(\Sigma_0;\mathbb Z)\to H_1(C;\mathbb Z))$ gives a basis for the $SL_m$ Hitchin fiber lattice at $x_0$.
\end{corollary}

\subsection{Pullbacks and Descents of Higgs Bundles}

\subsubsection{Markings}
Let $S^{\mathrm{top}}$ be a fixed oriented reference surface. A
\emph{marking} of a complex curve $C$ is an isotopy class of orientation-preserving homeomorphisms $\phi:S^{\mathrm{top}}\xrightarrow{\sim}C$.
A topological marking induces an isomorphism $\phi_*:
H_1(S^{\mathrm{top}};\mathbb Z)
\xrightarrow{\sim}
H_1(C;\mathbb Z)$.
This induced map is called the \emph{homological marking}.

\begin{definition}
Fix an oriented topological ramified cover
$\pi^{\mathrm{top}}:
\Sigma^{\mathrm{top}}
\to C^{\mathrm{top}}$.
A \emph{compatible marking} of a ramified cover
$\pi:\Sigma\to C$
is a pair of orientation-preserving homeomorphisms
$\phi:\Sigma^{\mathrm{top}}\to\Sigma$,
$\psi:C^{\mathrm{top}}\to C$,
such that
$\pi\circ\phi
=\psi\circ\pi^{\mathrm{top}}$.
\end{definition}

\begin{lemma}[\cite{Hatcher}]
Let $\psi:C^{\mathrm{top}}\to C$
be an orientation-preserving homeomorphism carrying the branch points and punctures of $\pi^{\mathrm{top}}$ to those of $\pi$.
A compatible upstairs lift
$\phi:\Sigma^{\mathrm{top}}\to\Sigma$
exists if and only if the two monodromy representations $\rho^\mathrm{top}$ and $\rho$ agree up to sheet
relabelling:
$\rho\circ\psi
=\sigma\,\rho^{\mathrm{top}}\,\sigma^{-1}$
for some
$\sigma\in S_m.$
When such lifts exist, any two compatible lifts of the same $\psi$
differ by a deck transformation.
\end{lemma}

Let us assign a marking $f:S_{g,n}=(S;p_1,\ldots,p_n)\to(C;q_1,\ldots,q_n)$, where $f(p_i)=q_i$ and $S_{g,n}$ is a given topological surface. The complex structure on $C$ pulls back to
$J_S=f^*J_C.$ We assign to $(C,P,f)$ a generic $A_{m-1}$-web, and hence obtain a topological spectral cover $\pi_{W}:\Sigma_{W}\to S_{g,n}$. Thus there is a unique marking $f_\Sigma:\Sigma_{W}\to \Sigma$ satisfying the following commutative diagram:
\begin{equation}
\begin{tikzcd}
\Sigma_{W}
\arrow[r, "f_\Sigma"]
\arrow[d, "\pi_{W}"]
&
\Sigma
\arrow[d, "\pi"]
\\
S_{g.n}
\arrow[r, "f"]
&
C.
\end{tikzcd}
\end{equation}

Let $p:\widetilde{S_{g,n}}\to S_{g,n}$ be the topological universal cover, then $\widetilde{S_{g,n}}$ carries the unique complex
structure
$J_{\widetilde{S_{g,n}}}=p^*J_S$. Denote $\widetilde{S_{g,n}}$ with $J_{\widetilde{S_{g,n}}}$ by $\widetilde{C}$, and the lifted marking by $\widetilde{f}:\widetilde{S_{g,n}}\to \widetilde{C}$. $f\circ p:\widetilde{C}\to C$ is holomorphic and locally biholomorphic.
Every deck transformation
$\gamma\in\Gamma:=\pi_1(S)$
acts biholomorphically on $\widetilde{C}$, and
$C\cong \widetilde{C}/\Gamma.$
We could obtain the following commutative diagram:
\begin{equation}
\begin{tikzcd}
\widetilde{\Sigma}_{\widetilde{W}}
\arrow[r, "\widetilde{f_\Sigma}"]
\arrow[d, "\widetilde{\pi}_W"]
&
\widetilde{\Sigma}
\arrow[d, "\widetilde{\pi}"]
\\
\widetilde{S_{g,n}}
\arrow[r, "\widetilde{f}"]
&
\widetilde{C}.
\end{tikzcd}
\end{equation}

\subsubsection{Pullbacks of Higgs Bundles}
Let $E\to C$ be a rank-$r$ holomorphic vector bundle on $C$. Its pullback to the universal cover $\widetilde{C}$ is
$\widetilde E:=p^*E,$
with fibers
$\widetilde E_u=E_{p(u)}$, $u\in \widetilde{C}$.
Intrinsically,
$p^*E
=
\{(u,v):v\in E_{p(u)}\}$ \cite{GRR14}.

For every deck transformation $\gamma\in\Gamma$,
$\gamma^*(p^*E)\cong p^*E.$
If $E$ is flat and
$\rho:\Gamma\to GL_r(\mathbb C)$
is its monodromy representation, then after choosing a parallel trivialization,
$p^*E\cong \widetilde{C}\times\mathbb C^r,$
and
$\gamma(u,v)
=
(\gamma u,\rho(\gamma)v).$
Thus
$E\cong(\widetilde{C}\times\mathbb C^r)/\Gamma.$

Let $\phi:E\to E\otimes M$
be an $M$-valued Higgs field, where $M=K_C(P)$.
Its pullback is \cite{GRR14}
\begin{equation}
p^*\phi:
p^*E\to p^*E\otimes p^*M.
\end{equation}
If $P=0$, then $p$ is locally biholomorphic,
$p^*K_C\simeq K_{\widetilde{C}}.$
In the logarithmic case,
$p^*K_C(P)\simeq K_{\widetilde{C}}(p^{-1}P).$
Locally, if
$\phi=A(z)\,dz$
and $z=p(w)$, then
$p^*\phi=A(p(w))\,p'(w)\,dw.$
The pulled-back Higgs bundle
$(\widetilde E,\widetilde\phi)
:=
(p^*E,p^*\phi)$
is $\Gamma$-equivariant.

\begin{lemma}
Let $\pi:\Sigma\to C$ be the spectral curve of $(E,\phi)$, with spectral line bundle $\mathcal L\to\Sigma$
and tautological section
$\lambda\in H^0(\Sigma,\pi^*M).$
The total space satisfies
$\operatorname{Tot}(p^*M)
\simeq
\operatorname{Tot}(M)\times_C \widetilde{C}.$
Since determinants commute with pullback,
$\det(\lambda-p^*\phi)
=
p^*\det(\lambda-\phi).$
Hence the spectral data of the pulled-back Higgs bundle $(p^*E,p^*\phi)$ are
\begin{equation}
\widetilde\Sigma=\Sigma\times_C \widetilde{C},
\qquad
\widetilde{\mathcal L}=q^*\mathcal L,
\qquad
\widetilde\lambda=q^*\lambda,
\end{equation}
where
\begin{equation}
\label{eq:car}
\begin{tikzcd}
\widetilde\Sigma
\arrow[r, "q"]
\arrow[d, "\widetilde\pi"]
&
\Sigma
\arrow[d, "\pi"]
\\
\widetilde{C}
\arrow[r, "p"]
&
C
\end{tikzcd}
\end{equation}
is Cartesian.
Moreover,
$\widetilde\pi_*\widetilde{\mathcal L}\simeq p^*E,$
and
$p^*\phi
=
\widetilde\pi_*
(\text{multiplication by }\widetilde\lambda).$
\end{lemma}

\begin{remark}
We can see from the Diagram~\ref{eq:car} and the construction of $\widetilde{\Sigma}_W$ in section~\ref{sssec:231}:
\begin{equation}
\label{eq:con}
\begin{tikzcd}
\widetilde{\Sigma}_{\widetilde{W}}
\arrow[r, "q"]
\arrow[d, "\widetilde{\pi}_W"']
&
\Sigma_{W}
\arrow[d, "\pi_W"]
\\
\widetilde{S_{g,n}}
\arrow[r, "p"]
&
S_{g,n},
\end{tikzcd}
\end{equation}
that the topological spaces in Diagram~\ref{eq:con} is the underlying topological surface of the complex curves in Diagram~\ref{eq:car}.
\end{remark}

\subsubsection{Descents of Higgs Bundles}
\begin{proposition}
\label{prop:descent}
Suppose
$C=\widetilde{C}/\Gamma,$
where
$\Gamma\subset\operatorname{Aut}_{\mathrm{hol}}(\widetilde{C})$
acts freely and properly discontinuously. Let
$\Gamma'\subset\operatorname{Aut}_{\mathrm{hol}}(\widetilde{C})$
be another such group, and set
$C':=\widetilde{C}/\Gamma'.$
Then the pulled-back Higgs bundle
$(\widetilde E,\widetilde\phi)$
descends to a Higgs bundle
$(E',\phi')$
on $C'$ if and only if it admits a compatible
$\Gamma'$-equivariant structure.
\end{proposition}

If the condition in Proposition~\ref{prop:descent} hold, then $E'\simeq \widetilde E/\Gamma'$, $\phi'\simeq \widetilde\phi/\Gamma'$.
Suppose moreover that the spectral curve of
$(\widetilde E,\widetilde\phi)$ is $\widetilde\pi:
\widetilde\Sigma\to \widetilde{C}$. 
Then the equivariant structure induces a compatible action of
$\Gamma'$ on $\widetilde\Sigma$, and $\Sigma':=\widetilde\Sigma/\Gamma'$
is the spectral curve of $(E',\phi')$. Thus there is a commutative
diagram
\begin{equation}
\begin{tikzcd}
\widetilde{\Sigma}
\arrow[r, "\widetilde{q'}"]
\arrow[d, "\widetilde{\pi}"'] & \Sigma'
\arrow[d, "\pi'"] \\
\widetilde{C} \arrow[r, "q'"'] & C'
\end{tikzcd}
\end{equation}

If, in addition, the spectral line bundle $\widetilde{\mathcal L}\to\widetilde\Sigma$ admits a compatible $\Gamma'$-linearization, then it descends to a line
bundle $\mathcal L'\to\Sigma'$.
If the tautological section $\widetilde\lambda
\in
H^0\!\left(
\widetilde\Sigma,
\widetilde\pi^*\widetilde M
\right)$
is $\Gamma'$-invariant, then it descends to the tautological section $\lambda'
\in
H^0\!\left(
\Sigma',
(\pi')^*M'
\right)$,
where $M'\to C'$ is the line bundle descended from $\widetilde M$.

Consequently, the spectral correspondence gives
$E'\simeq\pi'_*\mathcal L'$,
$\phi'=\pi'_*
\bigl(\text{multiplication by}\lambda'
\bigr)$.

\begin{remark}[Special case one]
If
$\Gamma'\subset\Gamma,$
then equivariance is automatic, and
$C'=\widetilde{C}/\Gamma'\to C$
is a covering. In this case,
$(E',\phi')=q^*(E,\phi),$
and $\Sigma'=\Sigma\times_C C'$, $\mathcal L'=\operatorname{pr}_\Sigma^*\mathcal L$, $\lambda'=\operatorname{pr}_\Sigma^*\lambda$.
\end{remark}

\begin{remark}[Special case two]
If
$
\Gamma'=h\Gamma h^{-1}$, $h\in\operatorname{Aut}_{\mathrm{hol}}(\widetilde{C})$, then
$\widetilde{C}/\Gamma\cong \widetilde{C}/\Gamma'$
by
$[z]_\Gamma\mapsto[h(z)]_{\Gamma'}.$
Thus conjugate deck groups define biholomorphic unmarked curves.
\end{remark}

\begin{remark}
For $h\in\operatorname{Aut}_{\mathrm{hol}}(\widetilde{C})$, if $h\Gamma h^{-1}=\Gamma,$
then $h$ descends to a biholomorphic automorphism
$\overline h:C\to C,$
whose isotopy class defines an element in the mapping class group
$[\overline h]\in\operatorname{MCG}(S).$
In the hyperbolic case,
$\operatorname{Aut}(C)
\cong
N_{\operatorname{Aut}_{\mathrm{hol}}(\widetilde{C})}(\Gamma)/\Gamma,$
which is generally only a finite subgroup of
$\operatorname{MCG}(S).$

If
$h\Gamma h^{-1}=\Gamma'\neq\Gamma,$
then $h$ gives a biholomorphism between two quotient curves $\overline h:C=\widetilde{C}/\Gamma\to C'=\widetilde{C}/\Gamma'$. After choosing markings
$f:S\to C$, and $f':S\to C',$
one obtains the mapping class
$\left[
(f')^{-1}\circ\overline h\circ f
\right]
\in\operatorname{MCG}(S).$
\end{remark}

\subsection{Ehresmann Connection and Gauss--Manin Transport}

Let $p:E\to B$ be a topological fiber bundle with typical fiber $F$.
Given a path $\gamma:[0,1]\to B$,
the pullback bundle is
$\gamma^*E
=
\left\{
(s,e)\in[0,1]\times E:
p(e)=\gamma(s)
\right\},$
with projection
$\gamma^*E\to[0,1],$ $(s,e)\mapsto s.$
Its fiber over $s$ is $(\gamma^*E)_s\cong E_{\gamma(s)}.$

\begin{definition}
Let $p:E\to B$
be a smooth map between smooth manifolds.
The map $p$ is called \emph{smooth} if it is a submersion, i.e. if for every $x\in E$ the differential
$dp_x:T_xE\to T_{p(x)}B$
is surjective.
\end{definition}

\begin{definition}
Let $p:E\to B$
be a smooth map between smooth manifolds.
The map $p$ is called \emph{proper} if the inverse image of every
compact subset of $B$ is compact.
\end{definition}

\begin{definition}
A \emph{smooth proper family} is a proper smooth submersion
$p:E\to B$.
For each $b\in B$, the fiber
$E_b:=p^{-1}(b)$
is a compact smooth manifold.
\end{definition}

\subsubsection{Ehresmann Connection}

\begin{definition}
Let $p:E\to B$
be a smooth submersion. Its vertical tangent bundle is
$V:=\ker(dp)\subset TE$.
An \emph{Ehresmann connection} is a choice of a horizontal subbundle
$H\subset TE$
such that
$TE=V\oplus H$.
A path
$\widetilde\gamma:[0,1]\to E$
is horizontal if
$\dot{\widetilde\gamma}(s)\in H_{\widetilde\gamma(s)}$
for all $s$.
\end{definition}

\begin{definition}
Let $H$ be an Ehresmann connection and let $\gamma:[0,1]\to B$
be a smooth path. Assume that the horizontal lift through every
$x\in E_{\gamma(0)}$ exists on the whole interval $[0,1]$.
Let
$\widetilde\gamma_x$
be the unique horizontal lift satisfying
$p\circ\widetilde\gamma_x=\gamma$,
$\widetilde\gamma_x(0)=x$.

The \emph{Ehresmann parallel transport} is
$T_\gamma^H:
E_{\gamma(0)}
\to
E_{\gamma(1)}$,
$T_\gamma^H(x)
=
\widetilde\gamma_x(1)$.
It is a diffeomorphism, with inverse given by parallel transport along
the reversed path. In particular, if $\gamma(0)=\gamma(1)$, then
$T_\gamma^H$ is a self-diffeomorphism of the fiber.
\end{definition}

\begin{theorem}[Ehresmann Fibration Theorem]
Let $p:E\to B$
be a proper smooth submersion.
Then $p:E\to B$
is a locally trivial smooth fiber bundle.
\end{theorem}

\begin{lemma}
Let $p:\mathcal X\to I$
be a smooth submersion over an interval.
Giving an Ehresmann connection is equivalent to giving a smooth vector field
$X\in\Gamma(T\mathcal X)$
such that
$dp(X)=\partial_t.$
The horizontal bundle is then
$
H=\mathbb R X.
$
\end{lemma}

\begin{proof}
Since
$TI=\mathbb R\partial_t$
has rank one, the horizontal subspace at every point must also have rank one.
If $H$ is an Ehresmann connection, there is a unique horizontal lift
$X=\operatorname{Hor}(\partial_t)$
satisfying
$dp(X)=\partial_t$.

Conversely, if such a smooth $X$ is given, then $X$ is never vertical, so
$T_x\mathcal X
=
\ker(dp_x)\oplus\mathbb RX_x.$
Thus
$H=\mathbb RX$
is an Ehresmann connection.
\end{proof}

\begin{lemma}
Let
$\mathscr X_\gamma=\{P(\lambda,z,t)=0\}.$
In local complex coordinates, write a lift of $\partial_t$ as
$X
=\partial_t
+A\,\partial_\lambda
+B\,\partial_z.$
Then $X$ is tangent to $\mathscr X_\gamma$ if and only if
$\partial_tP
+A\,\partial_\lambda P
+B\,\partial_zP
=0$ $\text{on }\mathscr X_\gamma.$
\end{lemma}

A family of diffeomorphisms
$F_t:\Sigma_0\xrightarrow{\sim}\Sigma_t$ is \emph{smooth} if
$F_0=\id$
and the associated trivializing map
$\Phi:\Sigma_0\times I\to\mathcal X,
$ $\Phi(x,t)=F_t(x)$ is smooth.

\begin{lemma}[\cite{CavalieriMiles}]
\label{le:diff}
Let
$\pi_i:\Sigma_i\to C,
$ $ i=1,2,$
be smooth branched covers, and suppose
$F:\Sigma_1\to\Sigma_2$
is a homeomorphism satisfying
$\pi_2\circ F=\pi_1.$
Then $F$ is a diffeomorphism. In fact, locally near a ramification point it is
holomorphic in suitable branched-cover coordinates.
\end{lemma}

\subsubsection{Gauss--Manin Parallel Transport}

\begin{definition}
Let $p:E\to B$ be a smooth proper family.
Then $\mathcal H_k
:=\left\{
H_k(E_b;\mathbb Z)
\right\}_{b\in B}$
is a local system.
Its path transport
$\PT_\gamma^{GM}:
H_k(E_{\gamma(0)};\mathbb Z)
\to
H_k(E_{\gamma(1)};\mathbb Z)$
is called \emph{Gauss--Manin (parallel) transport}.
After tensoring with $\mathbb C$, this local system determines a flat connection
$\nabla^{GM},$
called the \emph{Gauss--Manin connection}.
\end{definition}

\begin{definition}
Fix $b_0\in B$. For a loop
$\gamma:[0,1]\to B,$
$\gamma(0)=\gamma(1)=b_0,$
Gauss--Manin transport gives an automorphism
$\PT_\gamma^{GM}
\in
\Aut H_k(E_{b_0};\mathbb Z).$
Thus one obtains the monodromy representation
$\rho:
\pi_1(B,b_0)
\to
\Aut H_k(E_{b_0};\mathbb Z).$
\end{definition}

\begin{proposition}
Let $p:E\to B$
be a smooth proper family and let
$\gamma:b_0\to b_1.$
Choose any topological parallel-transport representative
$T_\gamma:E_{b_0}\to E_{b_1}.$
Then
$(T_\gamma)_*
=\PT_\gamma^{GM}
:H_k(E_{b_0};\mathbb Z)
\to H_k(E_{b_1};\mathbb Z).$
In particular, the induced homology map is independent of the chosen
trivialization or Ehresmann connection.
\end{proposition}

\begin{proof}
Different trivializations of the family along the interval give endpoint maps
which differ by an isotopy, so they induce the same map on singular homology.
Thus all smooth trivializations and all Ehresmann connections produce the same map
$H_k(E_{b_0};\mathbb Z)
\to
H_k(E_{b_1};\mathbb Z).$
This is the transport of the homology local system associated with the smooth family, namely Gauss--Manin transport.
\end{proof}

\begin{definition}
Given a $n$-punctured smooth complex curve $\mathring{C}=C\setminus P$,
there are two distinct fibrations over $B_{G,C,P}^{\mathrm good}$.

\noindent
(1) The spectral-curve family
$q:\Sigma\to B_{G,C,P}^{\mathrm good},$
$q^{-1}(t)=\Sigma_t.$

Here
$\pi_t:\Sigma_t\to C$
is the spectral covering associated with the Hitchin parameter $t$.

\noindent
(2) The regular Hitchin fibration
$h:\M_{\mathrm{Higgs}}^{\mathrm good}\to B_{G,C,P}^{\mathrm good},$
$h^{-1}(t)=\Lambda_t.$

Here $\Lambda_t$ is a Jacobian or Prym-type abelian variety
(or a torsor for one) associated with $\Sigma_t$.
\end{definition}

\begin{corollary}
Suppose
$q:\Sigma\to B_{G,C,P}^{\mathrm good}$
is the smooth proper family of spectral curves.
Then
$\mathcal H^\Sigma_1
:=
\left\{
H_1(\Sigma_t;\mathbb Z)
\right\}_{t\in B_{G,C,P}^{\mathrm good}}$
is a local system.
For every path
$\gamma:t_0\to t_1$
there is canonical Gauss--Manin transport
$\PT_\gamma^{GM,\Sigma}:
H_1(\Sigma_{t_0};\mathbb Z)
\xrightarrow{\sim}
H_1(\Sigma_{t_1};\mathbb Z).$
\end{corollary}

\begin{proof}
Over $B_{G,C,P}^{\mathrm good}$, the spectral curves are smooth and the family is a smooth
proper submersion.
Ehresmann's theorem therefore gives local smooth triviality.
Consequently the groups
$
H_1(\Sigma_t;\mathbb Z)
$
are locally identified and form the homology local system of the family.
Its path transport is the Gauss--Manin transport.
\end{proof}

\begin{lemma}
For a fixed base curve $C$, the family of spectral curves
$\pi_t:\Sigma_t\to C$
defines a morphism of homology local systems
$\mathcal H^\Sigma_1
\to
H_1(C;\mathbb Z)_{B_{G,C,P}^{\mathrm good}},$
where the right-hand side is the constant local system with fiber
$H_1(C;\mathbb Z)$.

Consequently, for every path
$\gamma:t_0\to t_1,$
the following diagram commutes.
\begin{equation}
\begin{tikzcd}
H_1(\Sigma_{t_0};\mathbb Z)
\arrow[r, "\operatorname{PT}^{GM,\Sigma}_\gamma"]
\arrow[d, "{(\pi_{t_0})_*}"']
&
H_1(\Sigma_{t_1};\mathbb Z)
\arrow[d, "{(\pi_{t_1})_*}"]
\\
H_1(C;\mathbb Z)
\arrow[r, "\mathrm{id}"']
&
H_1(C;\mathbb Z)
\end{tikzcd}
\end{equation}
\end{lemma}

\begin{proof}
The spectral projections assemble into a map of families
$\Sigma
\to
C\times B_{G,C,P}^{\mathrm good}$
over $B_{G,C,P}^{\mathrm good}$.
Homology is functorial, so this map induces a morphism between the corresponding
Gauss--Manin local systems.
Path transport in a local system is with respect to morphisms of local
systems, which gives the commutative diagram.
\end{proof}

\begin{corollary}
Let
$\Lambda_t
=
\ker(
H_1(\Sigma_t;\mathbb Z)
\xrightarrow{(\pi_t)_*}
H_1(C;\mathbb Z)).$
Then
$\Lambda
=\{\Lambda_t\}_{t\in B_{G,C,P}^{\mathrm good}}$
is a sub-local system of
$\mathcal H^\Sigma_1.$
For every path
$\gamma:t_0\to t_1,$
one has
$\PT_\gamma^{GM,\Sigma}(\Lambda_{t_0})
=\Lambda_{t_1}.$
Hence the Gauss--Manin transport of the Prym lattice is
$\PT_\gamma^{GM,\Lambda}
=\left.
\PT_\gamma^{GM,\Sigma}
\right|_{\Lambda_{t_0}}.$
\end{corollary}

\newpage
\mysection{Monodromies around Discriminant Loci and Elementary Moves on Webs}

In this section, we will explore the correspondence between the compositions of elementary moves on generic $A_{m-1}$-webs and small disc areas around discriminant loci.

In subsection 5.1, we will recall the monodromies on the Hitchin fiber lattices around singular locus, and identify the monodromies with the automorphisms on the quiver lattices obtained by performing two successive \ttt moves in the same orientation on generic $A_{m-1}$-webs.

In subsection 5.2 and 5.3, we will first observe the local deformations of spectral covers when two affected simple branch points on $C$ braid or collide, as suggested by the ura operation and the Dehn twist. We will then find deformations of the Hitchin polynomials that realize the prescribed local motions, and hence determine small disc areas around $\mathcal{D}_{bdry,C,P}$.

We will also discuss, in subsection 5.4, the monodromies around the two other types of the discriminant loci: $\mathcal{D}_{cau,C,P}$ and $\mathcal{D}_{Max,C,P}$, and compare the results with the automorphisms on quiver lattices under two shrink \& expand moves and isotopy of webs.

We will reach to the conclusion at the end of Section 5 that local monodromies around three types of the discriminant loci are the same as the automorphisms on quiver lattices given by compositions of three elementary moves on generic $A_{m-1}$-webs.

In Section 6 and 7, we will see that the consistency between local monodromies on Hitchin fibrations and automorphisms on quiver lattices, can be seen at the level of cluster mutations on relevant regular functions, and tropical mutations on relevant $K_0$ groups.

\medskip
Let $G_*$ be the groupoid whose objects are $A_{m-1}$-webs on a given decorated surface $C$, and morphisms are the three types of elementary moves. 
Let $G_0$ be the underlying abstract oriented graph of $G_*$, whose vertices and edges are objects and morphisms of $G_*$ respectively.
$G_0$ is connected by the connectedness of $A_{m-1}$-webs.

Let an object $W_1$ in $G_*$ be the initial generic $A_{m-1}$-web, and let $W_2$ be the generic $A_{m-1}$-web after one \ttt move on $W_1$.
Let $v_{W_1}$ and $v_{W_2}$ be vertices in $G_0$ corresponding to $W_1$ and $W_2$ respectively. The two morphisms in $G_*$ corresponding to the two \ttt moves from $W_1$ to $W_2$ and back to $W_1$ form a 2-cycle in $G_0$ starting and ending at $v_{W_1}$, passing through $v_{W_2}$.

The desired correspondence between the web perspective and the Hitchin perspective comparing two successive \ttt moves and meridians around $\mathcal{D}_{bdry,C,P}$ has three aspects as following:

\begin{table}[htbp]
\centering
\caption{Correspondence between the web perspective and the Hitchin perspective.}
\label{tab:compare}
\begin{tabular}{c|c|c}
&Web Perspective & Hitchin Perspective \\
\hline
$\text{Parametrization}$
&$\begin{array}{ccc}
\text{A 2-cycle in $G_0$ for 2 two-by-two moves}  \\
\text{A singular rank-$m$ spectral web}   \\
\text{Return to the initial generic $A_{m-1}$-web}   \\
\end{array}
$
& $\begin{array}{ccc}
\text{A small meridian around } \mathcal{D}_{bdry,C,P}  \\
\text{A point in } \mathcal{D}_{bdry,C,P}^{\circ}\\
\text{Return to the initial point in }B_{SL_m,C,P}^{\mathrm{good}}\\
\end{array}$ \\
\hline
$\begin{array}{cc}
\mathring{S}=S\setminus P \\
\mathring{C}=C\setminus P  
\end{array}$
&$\begin{array}{ccc}
\text{Half-twist of simple branch points }q_+,q_-  \\
\text{Simple branch points $q_+,q_-$ collide}  \\
q_+ \text{ and }q_- \text{ switch positions}  \\
\end{array}$ 
& $\begin{array}{ccc}
\text{Half-twist of simple zeros $q_+,q_-$ of }\Delta  \\
\text{Simple zeros $q_+,q_-$ collide into a double zero}  \\
q_+ \text{ and }q_- \text{ switch positions}  \\
\end{array}$ \\
\hline
$\text{Spectral Cover}$
&$\begin{array}{ccc}
\text{Dehn twist along the face loop lift}  \\
\text{Nodal singularity at the degenerated loop}  \\
\text{Automorphism on the quiver lattice}\\
\end{array}$ 
& $\begin{array}{ccc}
\text{Dehn twist along the vanishing cycle}  \\
\text{Nodal singularity at the "vanished" cycle}  \\
\text{Monodromy on the Hitchin fiber lattice}\\
\end{array}$
\end{tabular}
\end{table}

\subsection{Local Monodromies around \texorpdfstring{$\mathcal{D}_{\mathrm{bdry,C,P}}$}{}}

Let $\mathring{C}=C\setminus P$ be a $n$-punctured smooth complex genus-$g$ curve satisfying $2g-2+n>0$.
Let $G=SL_m$, $m \ge 2$. Set $L:=K_{C}(P)$, $d:=\deg L=2g-2+n>0.$ 

Let $b_0$ be a point in $\mathcal{D}_{bdry,C,P}^{\circ}$.
Denote by $D=\{z\in\mathbb C:|z|\le 1\}
=\left\{
(1-s)e^{2\pi i t}
:s,t\in[0,1]
\right\}$ the closed unit disc.
Let $\gamma:U\hookrightarrow B_{G,C,P}$ be a holomorphic embedding
defined on an open neighborhood
$U\supset D$,
with
$\gamma(0)=b_0\in \mathcal D_{bdry,C,P}^\circ,$
$\mathcal D_{bdry,C,P}\cap \gamma(D)=b_0$,
and $\gamma|_{\partial D}
: S^1\to B_{G,C,P}^{\mathrm{good}}$ a small meridian around
$\mathcal D_{bdry,C,P}$.

Let $a=(a_1,\ldots,a_m)\in B_{G,C,P}^{\mathrm{good}}$. The total family of spectral curves is
$\mathscr X
=\{(a,\lambda,z)\in
B_{G,C,P}^{\mathrm{good}}\times\operatorname{Tot}(L)
: P_a(\lambda,z)=0\}$.
The projection
$p:\mathscr X\to B_{G,C,P}^{\mathrm{good}},$
has fiber $p^{-1}(a)=\Sigma_a.$
$P_a$ depends holomorphically on $(a,\lambda,z)$, so
$\mathscr X$ is a holomorphic hypersurface in the 3 dimensional
complex manifold $B_{G,C,P}^{\mathrm{good}}\times\operatorname{Tot}(L)$.

The total family of spectral curves over the image of the holomorphic embedding $\gamma:D\hookrightarrow B_{G,C,P}$ is a 2 dimensional complex manifold $\mathscr X_D
=\{
(a,\lambda,z)\in D\times\operatorname{Tot}(L)
:P_a(\lambda,z)=0\}.$
The projection
$p|_{D}:\mathscr X_D\to D,$
is holomorphic with fiber $\Sigma_a$.
Let $r=1-s$.

\begin{theorem}[Picard--Lefschetz Theorem for Riemann Surfaces \cite{Lamotke}]
\label{thm:pic}
Let
$f:\mathscr X_D\to D$
be a proper holomorphic map from a 2 dimensional complex manifold $\mathscr X_D$ to a disc $D\subset\mathbb C$.
Assume:
\begin{enumerate}
    \item $f$ is a submersion over
    $D^\times=D\setminus\{0\};$
    \item the singular fiber $\Sigma_0=f^{-1}(0)$ contains exactly one critical point $p$;
    \item the critical point $p$ is holomorphically nondegenerate.
\end{enumerate}
Then, by the holomorphic Morse lemma, there are local holomorphic coordinates $(u,v)$ on $\mathscr X_D$ centred at $p$, such that $f(u,v)=uv=r\in D$.
Thus in a sufficiently small neighborhood $U$ of $p$, if $r=0$, $\Sigma_0\cap U=\{uv=0\}
=\{u=0\}\cup\{v=0\}$ consists of two smooth local branches meeting transversely at the node;
if $r\neq 0$, $\Sigma_r\cap U=\{uv=r\}$ contains a vanishing cycle $\delta\simeq S^1$ which collapses to $p$ as $r\to0$.

Let $\gamma(\theta)
=\varepsilon e^{i\theta}$,
$0\leq\theta\leq2\pi$,
be a positively oriented meridian around $0$.
Then the Gauss--Manin monodromy
$T_\gamma:
H_1(\Sigma_\varepsilon;\mathbb Z)
\to
H_1(\Sigma_\varepsilon;\mathbb Z)$
is
$T_\gamma(\alpha)
=\alpha+\langle\alpha,\delta\rangle\delta$,
where $\langle\alpha,\delta\rangle$ is the intersection form.
\end{theorem}

\begin{corollary}[\cite{FarbMargalit}]
For a nodal degeneration of Riemann surfaces, the geometric monodromy
around a positive meridian is represented, up to isotopy, by a Dehn twist along the vanishing cycle:
$[\tau_\delta]
\in\operatorname{MCG}(\Sigma)$.
Consequently
$(\tau_\delta)_*(\alpha)
=\alpha
+\langle\alpha,\delta\rangle\delta$,
up to the chosen sign convention.
\end{corollary}

\begin{lemma}
\label{le:picprym}
Let $a\in \gamma(D^{\times})$, and let $\pi_a:\Sigma_a\to C$ be the spectral curve.
Let $\delta\in H_1(\Sigma_a;\mathbb Z)$ be the vanishing cycle associated with a nodal degeneration
produced by the collision of two simple branch values.
Then
$(\pi_a)_*\delta=0$, i.e., $\delta\in H_1(\Prym(\Sigma_a);\mathbb Z)$.

Hence the Picard--Lefschetz transformation preserves the $SL_m$ Hitchin fiber lattice $H_1(\Prym(\Sigma_a);\mathbb Z)$
and acts by
$T(\alpha)
=\alpha+\langle\alpha,\delta\rangle\delta$, $\forall \alpha\in H_1(\Prym(\Sigma_a);\mathbb Z)$,
up to the chosen sign convention.
\end{lemma}

\begin{proposition}
\label{prop:hittoweb}
Given any sufficiently small oriented meridian $m$ in $B_{G,C,P}^{\mathrm{good}}$ around $\mathcal{D}_{bdry,C,P}$, we can find a generic $A_{m-1}$-web $W$, such that the monodromy action on the Hitchin fiber lattice along $m$ can be realized by the automorphism on the quiver lattice resulting from two successive \ttt moves at some internal face of $W$ in the same orientation.
\end{proposition}

\begin{proof}
By Corollary~\ref{cor:iso1}, given any point $a\in m \subset B_{G,C,P}^{\mathrm{good}}$ near a point $b_0\in \mathcal D_{bdry,C,P}^\circ$, we can find a generic $A_{m-1}$-web $W$, such that $\pi_a:\Sigma_a\to C$ and $\pi_W:\Sigma_W\to C$ are Hurwitz isomorphic. Hence a basis for $H_1(\Prym(\Sigma_a);\mathbb Z)$ can be imposed from the basis for $\Lambda_{\Gamma}^{\circ}$.

Let $q_+,q_-\in C$ be the two initial simple branch points of $\pi_a$ that move close to each other along a path from $a$ to $b_0$ and eventually collide at $b_0$. Since $q_+,q_-$ are also two simple branch points of $\pi_W$, they lie on two 3-valent $\bullet$-vertices on the boundary of an internal face $F$ of $W$ corresponding to a 4-valent vertex in the quiver.

The lift of the face loop $F$ is a simple closed loop $\delta\in \ker (H_1(\Sigma_W;\mathbb Z)\to H_1(\mathring{C};\mathbb Z)$ passing through the two simple ramification points over the two simple branch points on a cylinder area on the spectral cover $\Sigma_W$.

As $q_+,q_-$ move close to each other along a path from $a$ to $b_0$ and eventually collide at $b_0$, $\delta$ collapses to a nodal singularity, as shown in Figure~\ref{fig:schannelnew}. Then the vanishing cycle in the Picard--Lefschetz Theorem is isotopic to $\delta$.

By Proposition~\ref{prop:dehn} and Theorem~\ref{thm:ttt}, the automorphism on the quiver lattice is given by the Dehn twist along $\delta$, which provides the automorphism $(T_\delta)_*:
\Lambda_{\Gamma}
\to
\Lambda_{\Gamma},$
$(T_\delta)_*(\alpha)
=\alpha+\langle\alpha,\delta\rangle\delta$.
This is the same as the Picard--Lefschetz transformation on $H_1(\Prym(\Sigma_a);\mathbb Z)$ along $m$ in Theorem~\ref{thm:pic} and Lemma~\ref{le:picprym}.
\end{proof}

\subsection{Local Branch Motions}

Let us now consider the "converse" of Proposition~\ref{prop:hittoweb}. We start with an isotopy class of generic $A_{m-1}$-webs on $C$ and an internal face $F$ that we perform two successive \ttt moves on.

According to Corollary~\ref{cor:iso}, we pick two nearby points $x_0\in B_{G,C,P}^{\mathrm{good}}$ and $b_0\in \mathcal D_{bdry,C,P}^\circ$ such that $\pi_0:\Sigma_0\to C$ is Hurwitz isomorphic to $\pi_W:\Sigma_W\to C$ for some isotopy representative $W$,
and $\pi_{b_0}:\Sigma_{b_0}\to C$ has one nodal singularity formed by the collision of two simple branch points $q_+,q_-$ on the boundary of $F$.
A basis for $H_1(\Prym(\Sigma_0);\mathbb Z)$ can be imposed from the basis for $\Lambda_{\Gamma}^{\circ}$.

We will observe in the following two subsections that the local branch motion for $q_+,q_-\in C$ and the local spectral cover deformation produced by performing two successive \ttt moves on $F$ in the same orientation, determine a deformation of the Hitchin polynomial and hence a small oriented meridian $m$ around $\mathcal D_{bdry,C,P}$ starting and ending at $x_0$.

Thus the automorphism on $\Lambda_{\Gamma}^{\circ}$ resulting from two successive \ttt moves at $F$ in the same orientation can be given by the monodromy on $H_1(\Prym(\Sigma_0);\mathbb Z)$ along $m$.

\subsubsection{Continuous Variations of the Ura Operation}

We introduce two continuous variations of the \bbb operation on spectral covers defined in Definition~\ref{def:Ura}.

\medskip
\noindent
The \textbf{s-channel} variation of the \bbb operation is the process shown in Figure~\ref{fig:schannelnew}.

\begin{figure}[htbp]
\centering
\includegraphics[width=0.65\textwidth]{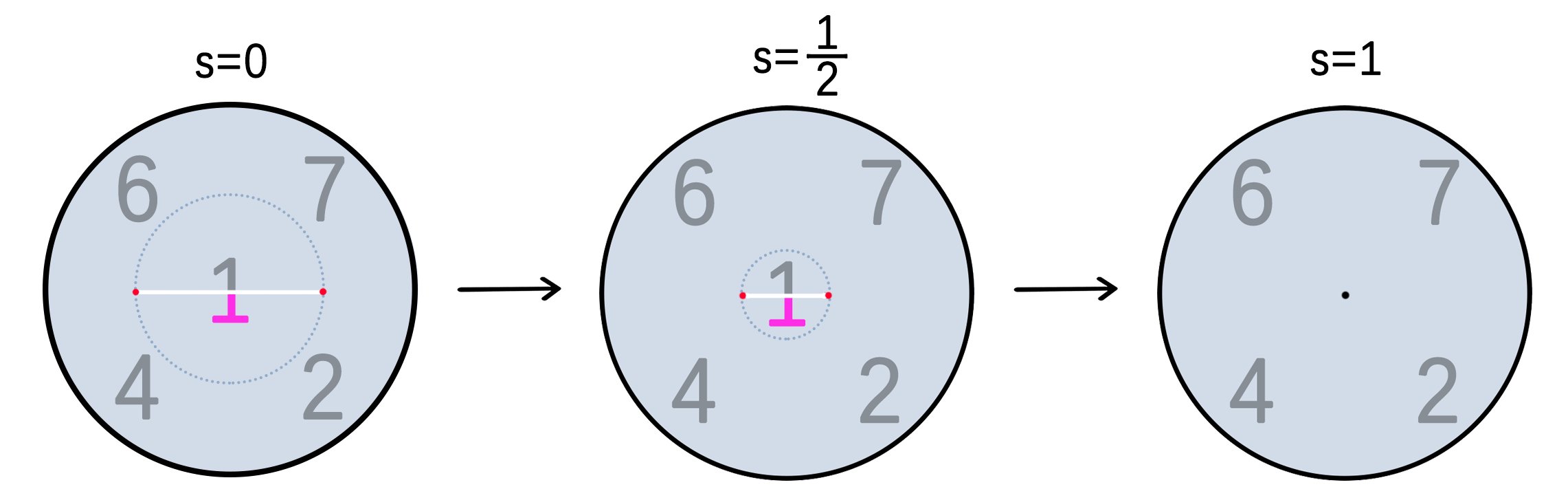}
\caption{s-channel.}
\label{fig:schannelnew}
\end{figure}

A local holomorphic expression on an open set $U\subset C$ for the discriminant of the varying Hitchin characteristic polynomial equation $P_t=0$ defining the spectral curves may have the form:
\begin{equation}
\Delta_t(z)=u(z,t)((z-z_0)^2-r(1-s)e^{2 \pi it})=0, \qquad s \in [0, 1],
\end{equation}
where $u(z,t)$ is nowhere zero and has order $m-2$, $z_0$ may be chosen as the middle point in face 1, $t=0$ and $r$ are constant. The two simple zeros $z=z_0 \pm \sqrt{r(1-s)}$ are the two simple branch points on the boundary of face 1. During the s-channel, the face 1 contracts to the middle point $z_0$ while the two simple zeros collide into one order-2 zero at $z_0$ when $s=1$. This means that an s-channel hit $\mathcal{D}_{bdry,C,P}$ when the two branch points collide at $z_0$ and the spectral curve $\Sigma$ develops a nodal singularity.

\begin{figure}[htbp]
\centering
\includegraphics[width=1.0\textwidth]{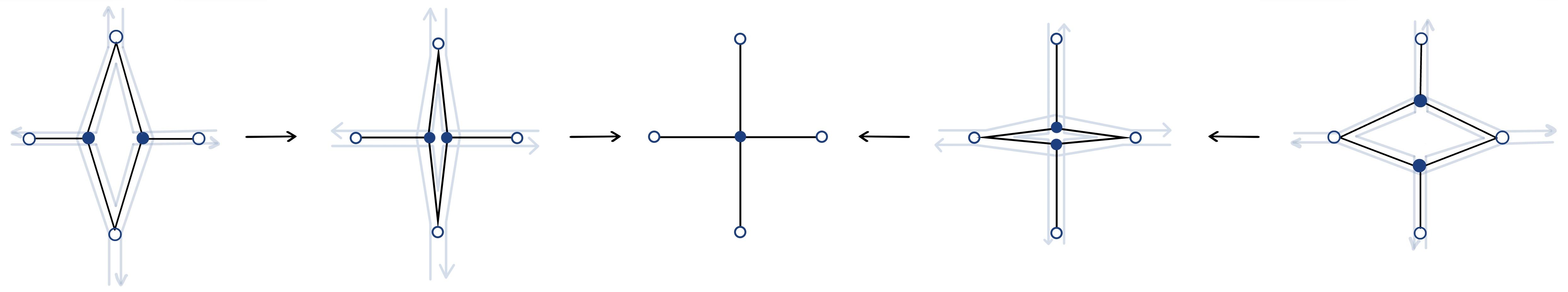}
\caption{Local motions of $A_{m-1}$-webs becoming singular through s-channel.}
\label{fig:singularweb}
\end{figure}

We can see from Figure~\ref{fig:schannelnew} and Figure~\ref{fig:singularweb} that when $s=1$, the web $W^{\mathrm{s=1}}$ is not an ideal web anymore, and the spectral cover $\Sigma_{W^{\mathrm{s=1}}}$ has a nodal singularity at the middle 4-valent $\bullet$-vertex. $W^{\mathrm{s=1}}$ is a singular web. An example of a singular web is given in Figure~\ref{fig:singular}.

\medskip
\noindent
The \textbf{t-channel} variation of the \bbb operation is the process shown in Figure~\ref{fig:tchannelnew}.


We can see the t-channel model is topologically the same as the model in Figure~\ref{fig:Ura}. This is simply an alternative way to view the \bbb operation as a continuous change of the discontinuous processes, namely infinitesimal cutting and gluing, on the spectral cover.

As discussed before, a local holomorphic expression on $U\subset C$ for the discriminant of the varying Hitchin characteristic polynomial equation $P_t=0$ defining the spectral curves may have the form:
\begin{equation}
\Delta_t(z)=u(z,t)((z-z_0)^2-re^{2 \pi it})=0, \qquad t \in [0, \frac{1}{2}],
\end{equation}
where $u(z,t)$ is nowhere zero and has order $m-2$, $z_0$ may be chosen as the middle point in face 1, $r$ is constant. The two simple zeros $z=z_0 \pm \sqrt r e^{\pi it}$ are the two simple branch points on the boundary of face 1. The two branch points braid around $z_0$ during the t-channel. At $t=1$, two simple branch points make a half braid and switch positions.

\begin{figure}[htbp]
\centering
\includegraphics[width=0.74\textwidth]{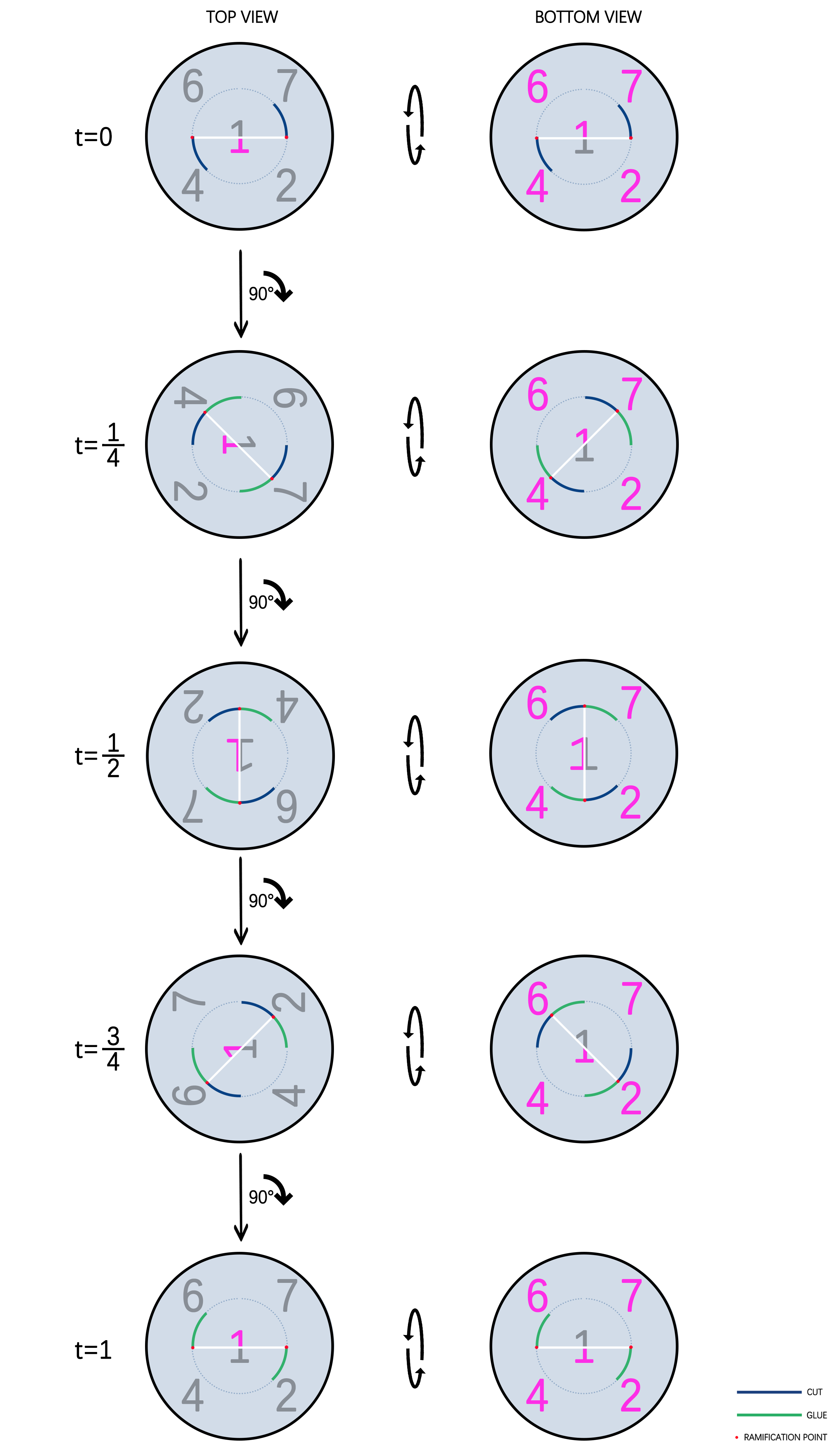}
\caption{t-channel. The $90^\circ$ clockwise rotation indicates rotation of the top layer relative to the bottom layer. (Instructions for making the paper model are provided in the appendix. For the t-channel, replace all the squares with circles.)}
\label{fig:tchannelnew}
\end{figure}

\subsubsection{Local Two-sheet Factorization}

Let $x_0\in B_{G,C,P}^{\mathrm{good}}$ be close to a point $b_0\in \mathcal D_{bdry,C,P}^\circ$.
Suppose $z_0$ is the double zero of the discriminant $\Delta_{b_0}$.
Let $P_0(\lambda,z)=0$ be the initial Hitchin characteristic polynomial equation over $x_0$. Let $q_+(0)$, $q_-(0)$ be the two simple branch points over $x_0$ that move close to each other along a path from $x_0$ to $b_0$ and eventually collide at $z_0$ over $b_0$.

Locally on an open set $U\subset C$ where the braiding/collision of two simple branch points $q_+$ and $q_-$ happens, write
$P_t(\lambda,z)=Q_t(\lambda,z)R_t(\lambda,z),$
where $Q_t=(\lambda-\lambda_i)(\lambda-\lambda_j)$ is the quadratic factor corresponding to the two participating sheets, and $R_t(\lambda,z)$ is the non-vanishing degree-$(m$-$2)$ factor corresponding to the $m-2$ unramified sheets over $U$.

Let $c(z)=\frac{\lambda_i(z)+\lambda_j(z)}2$ and $d_0(z)=
\left(
\frac{\lambda_i(z)-\lambda_j(z)}2
\right)^2.$
Then
$Q_0(\lambda,z)
=(\lambda-c(z))^2-d_0(z).$
Both $c$ and $d_0$ are invariant under the interchange
$\lambda_i\leftrightarrow\lambda_j.$
Thus $Q_0(\lambda,z)$ depends only on the unordered pair $\{\lambda_i,\lambda_j\}$.

Choose a local coordinate $z$ centred near the selected pair of simple branch points.
Set $d_t(z)
=\beta(z,t)^2
\left(
(z-z_0)^2-\alpha(t)
\right)$,
where $\alpha$, $\beta$ are holomorphic and $\beta$ is nowhere zero.
Then $Q_t(\lambda,z)=(\lambda-c(z))^2-\beta(z,t)^2
\left((z-z_0)^2-\alpha(t)\right)$.

Then the two eigenvalues are
$\lambda_{i,j}(z,t)
=c(z)\pm\beta(z,t)
\sqrt{(z-z_0)^2-\alpha(t)}$.
Since no other sheet collides in $U$, the discriminant satisfies
$\Delta_t(z)=u(z,t)
\left((z-z_0)^2-\alpha(t)\right)$,
for some non-vanishing $u(z,t)$.
Thus the two branch points are $q_\pm(t)=z_0\pm\sqrt{\alpha(t)}.$

\subsection{Global Hitchin Deformations}

In this subsection, we will find the deformations of the Hitchin characteristic polynomial $P_t$ which realize the local motions we described in Section 5.2, by controlling the local branch points motions parametrized by $D=\{(1-s)e^{2\pi i t}:s\in [0,1], t\in [0,1]\}$. Such deformations of $P_t$ are hence parametrized by $D$ as well.

We will see that such a deformation of $P_t$ gives a holomorphic embedding $\gamma:U\hookrightarrow B_{G,C,P},$ defined on an open neighborhood $U\supset D,$
with $\gamma(0)=b_0\in \mathcal D_{bdry,C,P}^\circ,$
$\mathcal D_{bdry,C,P}\cap \gamma(D)=b_0$,
and $\gamma|_{\partial D}
: S^1\to B_{G,C,P}^{\mathrm{good}}$ a small meridian around
$\mathcal D_{bdry,C,P}$.

The $P_t$, $t\in D$, realizes the braiding of two simple branch points.
$t=1$ corresponds to the half-twist.
$P_s$, $s\in D$ realizes the collision of two simple branch points.
$s=1$ corresponds to the collision which causes a nodal singularity on the spectral curve.

\subsubsection{Hitchin Polynomial Deformations Controlled by Determinants of Higgs Fields}

Let $x_0\in B_{G,C,P}^{\mathrm{good}}$ be close to a point $b_0\in \mathcal{D}_{bdry,C,P}^{\circ}$. Let $\pi_0:\Sigma_0\to C$
be the corresponding spectral cover, defined by $P_0(\lambda(z),z)=0$. Denote a pair of simple branch point downstairs by $q_\pm\in C$, and denote the corresponding pair of ramification point upstairs by
$p_\pm=(q_\pm,\mu_\pm)\in\Sigma_0\subset \Tot(L)$, where $\mu_\pm\in L_{q_\pm}$ is the repeated eigenvalue. Thus $\pi_0(p_\pm)=q_\pm.$

For a simple ramification point $p=(q,\mu)\in\Sigma_0$ on a smooth spectral curve,
one has $P_0(\mu,q)=\partial_\lambda P_0(\mu,q)=0$, and $\partial_\lambda^2 P_0(\mu,q)$, $\partial_zP_0(\mu,q)\neq0$ in a local coordinate $z$ on $C$.

Let $t\in [0, 1]$ be a parameter.
Write
$P_t(\lambda,z)
=G(\lambda,z)+a_m(z,t)$,
where
$G(\lambda,z)=\lambda^m+a_1(z)\lambda^{m-1}+\cdots+a_{m-1}(z)\lambda$
is held fixed. $a_m(z,t)=(-1)^m\det\phi_t$ is the determinant of the Higgs field up to a sign.
All lower invariant polynomials remain fixed. The individual eigenvalues
$\lambda_1,\ldots,\lambda_m$
generally all move.
For this family parametrized by t,
$\partial_\lambda P_t=\partial_\lambda G.$

The critical locus
$\Gamma_{\mathrm{crit}}
=\{(\lambda(z),z):\partial_\lambda P_t(\lambda,z)=\partial_\lambda G(\lambda,z)=0\}$
is independent of $t$.
For fixed $z$,
$\partial_\lambda G(\lambda,z)$
is a polynomial of degree $m-1$ in $\lambda$.
$\Gamma_{\mathrm{crit}}\to C$
is a degree-$(m$-$1)$ ramified cover.
Changing $a_m$ moves the spectral curve $\Sigma_t=\{P_t=0\}$
relative to this fixed critical locus.

\subsubsection{Local Branch Motions}

At a simple branch point $q_\pm(t)\in C$, two eigenvalues coincide:
$\lambda_{i_\pm}=\lambda_{j_\pm}=\mu_\pm(q_\pm(t))\in L_{q_\pm(t)}$.
The ramification point
$p_\pm=(q_\pm(t),\mu_\pm)\in  \Gamma_{\mathrm{crit}} \cap \Sigma_t$. When t varies, $\pi_t:\Sigma_t\to C$ varies relative to fixed $\pi_c:\Gamma_{\mathrm{crit}}\to C$, while their intersections are the ramification points.

\begin{lemma}
Let $p_\pm=(q_\pm,\mu_\pm)$
be an initial simple ramification point. On a small open neighbourhood $U_\pm$ of $q_\pm\in C$, there exists a unique holomorphic function
$\mu_\pm:U_\pm\to\mathbb C$ satisfying $\mu_\pm(q_\pm)=\mu_\pm$,
and $\partial_\lambda G(\mu_\pm(z),z)=0$, $z\in U_\pm$.
\end{lemma}

\begin{proof}
At $(\mu_\pm,q_\pm)$,
$\partial_\lambda G(\mu_\pm(q_\pm),q_\pm)=0$.
Moreover,
$\partial_\lambda^2G(\mu_\pm,q_\pm)
=\partial_\lambda^2P_0(\mu_\pm,q_\pm)
\neq0$.
The holomorphic implicit-function theorem therefore gives the required branch $\mu_\pm(z)$.
\end{proof}

\begin{lemma}
Let $q\in C$ be a point near $q_\pm$. Then $q$ is a branch point of $P_t(\lambda,z)
=G(\lambda,z)+a_m(z,t)$ if and only if
$a_m(q,t)=-G(\mu_\pm(q),q)$, i.e., $P_t(\mu_\pm(q),q)=0$.
\end{lemma}

\begin{corollary}
\label{cor:oncrit}
Suppose $q_\pm(t)\in U$ is prescribed nearby branch trajectory.
If a global section
$a_m(t)\in H^0(C,L^m)$,
satisfying
$a_m(q_+(t),t)=-G\bigl(\mu_+(q_+(t)),q_+(t)\bigr)$ and 
$a_m(q_-(t),t)=-G\bigl(\mu_-(q_-(t)),q_-(t)\bigr)$ is found, $P_t=P_0-a_{m,0}+a_m(t)$ is a deformation of the Hitchin polynomials that realizes the prescribed local branch motions.
\end{corollary}

\begin{example}
For $P_t(\lambda,z)=\lambda^3+a(z)\lambda+b_t(z)$,
one has
$G(\lambda,z)=\lambda^3+a(z)\lambda$.
The critical-root equation is
$\partial_\lambda G=3\mu^2(z)+a(z)=0$.
Thus
$G(\mu(z),z)=\mu^3+a(z)\mu
=\mu^3-3\mu^3
=-2\mu^3$.
Therefore the required condition is
$b_t(q)=-G(\mu(q),q)=2\mu^3(q)$.
\end{example}

\subsubsection{Globalizing the Branch Point Motions}

Recall that $L=K_C(P)$ has degree $2g-2+n> 0$, so $L$ is ample. $\deg (L^m)=m(2g-2+n)\ge 2g+1$ for $m\ge 3$, so $L^m$, $m\ge3$, is very ample. $\deg L^m=m(2g-2+n)\ge 2g$ for $m\ge 2$, so $L^m$, $m\ge 2$, is base-point free.

\begin{lemma}[\cite{EisenbudHarris}]
\label{le:veryample}
If a line bundle $M$ on a smooth projective curve $C$ is very ample, then for any distinct points
$q_+,q_-\in C$,
the evaluation map
$\operatorname{ev}_{q_+,q_-}:H^0(C,M)
\to M_{q_+}\oplus M_{q_-}$
is surjective.
\end{lemma}

\begin{lemma}[\cite{EisenbudHarris}]
\label{le:basepointfree}
If a line bundle $M$ on a smooth projective curve $C$ is base-point free, then for any point $q\in C$, the evaluation map $\operatorname{ev}_{q}:H^0(C,M)\to M_{q}$ is surjective.
\end{lemma}

\medskip
\noindent
\textbf{Case 1: $m=2$.}

Let $P_0(\lambda,z)
=\lambda^2+a_1(z)\lambda+a_{2,0}(z)$ be the initial Hitchin polynomial, where $a_1\in H^0(C,L)$, $a_{2,0}\in H^0(C,L^2)$. Then the critical locus $\Gamma_{\mathrm{crit}}
=\{(\lambda(z),z):\partial_\lambda P_0(\lambda,z)=2\lambda+a_1=0\}$ is the degree-one cover of $C$.
The initial branch points are the simple zeros of $\Delta_0=a_1^2-4a_{2,0}$.

Suppose $b_0\in\mathcal D_{bdry,C,P}^\circ$ corresponds to the collision of two
simple branch points at $z_0 \in C$. Then $\Delta_{b_0}$ has a double zero at $z_0$.
Locally, after choosing a coordinate $w$ centred at $z_0$, $\Delta_{b_0}=u(w)w^2$, $u(0)\neq0$.

During the deformation of $P_t$, $a_1$ remains fixed and $a_{2,0}$ varies.
By Lemma~\ref{le:basepointfree}, we can choose a fixed global section $s\in H^0(C,L^2)$
such that $s(z_0)\neq0$.
Define the one-parameter deformation $a_{2,t}=a_{2,0}+ts$, i.e., $\Delta_t=\Delta_0-4ts$.

Since $s(z_0)\neq0$, this deformation is transverse to the double-zero discriminant locus.
After a local holomorphic change of coordinate and parameter,
$\Delta_t=v(w,t)(w^2-\alpha(t))$, $v(0,0)\neq0$.
Thus the two nearby branch values are
$w_\pm(t)=\pm\sqrt{\alpha(t)}$.

Therefore the small meridian $\alpha(t)=\varepsilon e^{i\theta}$,
$0\le\theta\le2\pi$,
exchanges the two simple branch points $w_+$ and $w_-$ by a half-twist; the path $\alpha(t)=\varepsilon(1-r)$, $r\in [0,1]$, collides the two simple branch points $w_+$ and $w_-$.

\newpage
\noindent
\textbf{Case 2: $m\ge 3$.}

Let $U\subset C$ be a small open disc containing only two branch simple points $q_+,q_-$.
By Lemma~\ref{le:veryample}, we can choose fixed global sections
$s_+,s_-\in H^0(C,L^m)$ such that they distinguish the two simple branch points: after choosing local trivializations of $L^m$ on $U$, they may be normalized by
\begin{equation}
\begin{aligned}
s_+(q_+)&=1,
&
s_+(q_-)&=0,
\\
s_-(q_+)&=0,
&
s_-(q_-)&=1.
\end{aligned}
\end{equation}

Let $q_+(t),q_-(t)$
be prescribed nearby trajectories of two originally simple branch points within $U$.
We choose $s_+,s_-$ so that the evaluation matrix
\begin{equation}
M(t)=
\begin{pmatrix}
s_+(q_+(t)) & s_-(q_+(t))\\
s_+(q_-(t)) & s_-(q_-(t))
\end{pmatrix}
\end{equation}
is invertible along the deformation.

Choose holomorphic functions
$u_+(t)$ and $u_-(t)$.
Let $a_m(t)
=a_{m,0}+u_+(t)s_++u_-(t)s_-$, where $a_{m,0}$ is the determinant of the initial Higgs field, and keep $a_1,\dots,a_{m-1}$ fixed.
Then $P_t(\lambda,z)
=G(\lambda,z)+
a_{m,0}(z)+u_+(t)s_+(z)
+u_-(t)s_-(z)$
is the global Hitchin characteristic polynomial at $t$.

Since $a_m(q_\pm(t),t)-a_{m,0}(q_\pm(t))=u_+(t)s_+(q_\pm(t))+u_-(t)s_-(q_\pm(t))$. And by Corollary~\ref{cor:oncrit}, we would like to find $a_m(q_\pm(t),t)
=-G\bigl(\mu_\pm(q_\pm(t)),q_\pm(t)\bigr)$. Thus given fixed $s_\pm$, we may solve the equation
\begin{align}
\notag
P_t(\lambda,q_\pm(t))-P_0(\lambda,q_\pm(t))=u_+(t)s_+(q_\pm(t))+u_-(t)s_-(q_\pm(t))\\
=-G\bigl(\mu_\pm(q_\pm(t)),q_\pm(t)\bigr)-a_{m,0}(q_\pm(t))=-P_0(\mu_\pm(q_\pm(t)),q_\pm(t)),
\end{align}
to find desired $u_\pm(t)$.

\begin{lemma}
There are unique holomorphic coefficients $u_\pm(t)$ satisfying
\begin{equation}
M(t)
\begin{pmatrix}
u_+(t)\\
u_-(t)
\end{pmatrix}
=-
\begin{pmatrix}
P_0(\mu_+(q_+(t)),q_+(t))\\
P_0(\mu_-(q_-(t)),q_-(t))
\end{pmatrix}.
\end{equation}
The resulting Hitchin polynomial equation $P_t(\lambda,z)
=P_0(\lambda,z)+u_+(t)s_+(z)
+u_-(t)s_-(z)=0$ has
$q_+(t),q_-(t)$ as branch points.
\end{lemma}

\begin{proof}
Invertibility of $M(t)$ determines a unique pair $(u_+(t),u_-(t))\in \mathbb C^2$, hence a unique $a_m(t)
=a_{m,0}+u_+(t)s_++u_-(t)s_-\in H^0(C,L^m)$.
Since $P_t(\mu_\pm(q_\pm(t)),q_\pm(t))=0$,
then
$\left(
q_\pm(t),\mu_\pm(q_\pm(t))
\right)$
are ramification points of $\pi_t$, so $q_\pm(t)$ are the desired branch
values.
\end{proof}

\begin{example}
For
$P_0(\lambda,z)=\lambda^3+a(z)\lambda+b(z)$, take
$
P_t(\lambda,z)
=
\lambda^3+a(z)\lambda
+b(z)+u_+(t)s_+(z)+u_-(t)s_-(z),
$
where
$s_\pm\in H^0(C,L^3).$
Since
$-G\bigl(\mu_\pm(q_\pm(t)),q_\pm(t)\bigr)
=2\mu_\pm(q_\pm(t))^3$,
the coefficients satisfy
\begin{equation}
M(t)
\begin{pmatrix}
u_+(t)\\
u_-(t)
\end{pmatrix}
=
\begin{pmatrix}
2\mu_+(q_+(t))^3-b(q_+(t))\\
2\mu_-(q_-(t))^3-b(q_-(t))
\end{pmatrix}.
\end{equation}
\end{example}

\subsubsection{Simple Zeros of the Discriminants}
For a fixed $z\in C$, let
$\mu_1(z),\ldots,\mu_{m-1}(z)$
be the roots of
$\partial_\lambda G(\lambda,z)=0$.
For a monic degree-$m$ polynomial,
$\Delta_t(z)
=(-1)^{m(m-1)/2}
\operatorname{Resultant}(P_t,\partial_\lambda P_t)$.
Since
$\partial_\lambda P_t=\partial_\lambda G$
has leading coefficient $m$ and roots
$\mu_1,\ldots,\mu_{m-1},$
so $\operatorname{Resultant}(P_t,\partial_\lambda P_t)
=m^m
\prod_{j=1}^{m-1}P_t(\mu_j,z).$
Then
\begin{equation}
\Delta_t(z)
=
(-1)^{m(m-1)/2}m^m
\prod_{j=1}^{m-1}
P_t(\mu_j(z),z)
=
(-1)^{m(m-1)/2}m^m
\prod_{j=1}^{m-1}
\left[
G(\mu_j(z),z)+a_m(z)+\delta_t(z)
\right].
\end{equation}
For the prescribed branch points,
$\Delta_t(q_+(t))
=\Delta_t(q_-(t))
=0.$

Suppose the two prescribed branch values remain simple and no other branch value enters the chosen local disc. Then locally
$\Delta_t(z)
=u(z,t)
(z-q_+(t))(z-q_-(t))$,
where
$u(z,t)\neq0.$
If
$q_\pm(t)
=z_0\pm\sqrt{\alpha(t)},$
then locally
$\Delta_t(z)
=u(z,t)
\left((z-z_0)^2-\alpha(t)\right).$

\subsubsection{Motions of the Remaining Branch Points}

Let the close disc $D\subset C$
contain only two selected simple branch points, and let
$q_3,\dots,q_r$ be all the remaining simple branch points. Choose pairwise disjoint small closed discs
$D_j\ni q_j$
which are disjoint from $D$ and from the punctures, such that
$q_j$ is the unique zero of $\Delta_0$ in $D_j,$
and $\Delta_0|_{\partial D_j}\neq0$.

\begin{lemma}
For every sufficiently small deformation $P_t$,
the discriminant
$\Delta_t$ has exactly one zero
$q_j(t)\in D_j$,
counted with multiplicity.
This zero is simple and depends
holomorphically on $t$.
\end{lemma}

\begin{proof}
Since
$\Delta_0\neq0$ on $\partial D_j$
and $\partial D_j$ is compact,
$c_j
:=
\min_{z\in\partial D_j}|\Delta_0(z)|
>0.$
The discriminant depends holomorphically on the coefficients of the polynomial.
Thus, for sufficiently small deformation,
$\sup_{z\in\partial D_j}
|\Delta_t(z)-\Delta_0(z)|
<c_j.$
Hence
$|\Delta_t-\Delta_0|
<|\Delta_0|$ on $\partial D_j$.

By Rouch\'e's theorem, $\Delta_t$ and $\Delta_0$ have the same number of
zeros in $D_j$, counted with multiplicity.
As long as this zero remains simple, the implicit-function theorem gives its
continuous, locally analytic continuation in $t$. Hence
$q_j(t)\in D_j$
throughout the deformation.
\end{proof}

The initial Hitchin polynomial $P_0$ corresponds to $x_0\in B_{G,C,P}^{\mathrm{good}}$ that is sufficiently close to $b_0\in \mathcal D_{bdry,C,P}^\circ$. Hence the deformation of the Hitchin polynomial $P_t$ can be made sufficiently small, such that the corresponding deformation in the Hitchin base is constrained within a small disc $D$ that is transversal to $\mathcal D_{bdry,C,P}$ with $D\cap \mathcal D_{bdry,C,P}=b_0$.
Hence all the remaining branch points don't braid.

\subsubsection{s-channel}
If $q_+(t),q_-(t)\to z_0,$
then the two rows of the evaluation matrix $M(t)$ coincide at the collision. Hence at the limit $z=z_0$, $\Delta_t(z_0)=0$, 
$\partial_z\Delta_t(z_0)=0$, i.e., discriminant has a double zero at $z_0$.

\subsubsection{t-channel}
At $t=1$, the two branch points braid a half turn and switch positions: $q_+(1)=q_-(0)$, $q_-(1)=q_+(0)$, hence the unordered pair
$\{q_+(1),q_-(1)\}
=\{q_+(0),q_-(0)\}$. During the channel, the two chosen critical-root branches continue compatibly: $(q_+(1),\mu_+(1))
=(q_-(0),\mu_-(0))$,
$(q_-(1),\mu_-(1))=
(q_+(0),\mu_+(0))$. Assume we choose the $s_+,s_-$ such that the matrix $M(t)$ remains invertible.
Then $u_\pm(1)=u_\pm(0)=0$,
and hence $P_1=P_0$.

\begin{remark}
The motion of the two branch points alone does not determine $P_t$.
Indeed, if $v_t\in H^0(C,L^m)$
satisfies $v_t(q_\pm(t))=0$,
then replacing $a_m'(t)$ by $a_m(t)+v_t$ produces the same two prescribed branch points.

$P_1=P_0$ follows from the chosen normalization
$a_m(t)-a_{m,0}
\in \operatorname{span}\{s_+,s_-\}$:
setting $a_m(t)=a_m+u_+(t)s_++u_-(t)s_-$, two conditions for two branch motions uniquely determine $u_\pm(t)$ whenever the matrix $M(t)$ is invertible.
Therefore when $q_\pm(1)=q_\mp(0)$, uniqueness of $u_\pm(t)$ forces $a_m(1)=a_{m,0}$, hence $P_1=P_0$.
\end{remark}

\begin{remark}
The coefficients $(u_+,u_-)$ are independent coordinates on the chosen two-dimensional family $a_m+\operatorname{span}_{\mathbb C}\{s_+,s_-\}$.
However, when choosing a one-parameter deformation
$\gamma:D\to\mathbb C^2$,
$z\mapsto (u_+(z),u_-(z))$ to describe the motion of two simple branch points together, the functions $u_+(z)$ and $u_-(z)$ are no longer independent.
\end{remark}

\subsubsection{Conclusion}
We have found small deformations of the Hitchin polynomial $P_t$, during which two affected simple branch points braid a half turn and switch the positions and the remaining branch points only move within the disjoint discs.

Such small deformations of $P_t$ determine parametrization spaces in $B_{G,C,P}$:
Fix $a_1,\ldots,a_{m-1}$
and consider
$P_t=P_0+zw$, for some fixed $w\in H^0(C,L^m)$, $z\in D=
\{z\in\mathbb C:|z|\le 1\}
=\left\{
(1-s)e^{2\pi i t}
:s,t\in[0,1]
\right\}$.
The holomorphic embedding $\gamma:D\hookrightarrow B_{G,C,P}$, $z\mapsto\text{coefficients of }P_t=P_0+zw$ is transverse to $\mathcal{D}_{bdry,C,P}$ at $\gamma(s=1)=b_0\in \mathcal{D}_{bdry,C,P}^{\circ}$.

The parametrization line for s-channel $[0,1]\to B_{G,C,P}$,
$s\mapsto
(a_1,\ldots,a_{m-1},a_m+sw)$
meets $\mathcal{D}_{bdry,C,P}$ transversely. The small meridian for t-channel $[0,1]\to B_{G,C,P}^{\mathrm{good}}$, $t\mapsto
(a_1,\ldots,a_{m-1},a_m+tw)$ goes around $\mathcal{D}_{bdry,C,P}$, and the monodromies on the Hitchin fiber lattices is given by Picard--Lefschetz transformations and geometrically Dehn twists along vanishing cycles.

\medskip
We now state a counterpart to Proposition~\ref{prop:hittoweb}.

\begin{proposition}
\label{prop:match}
Given any generic $A_{m-1}$-web $W$ on $C$, the automorphism on the quiver lattice resulting from two successive \ttt moves at an internal face $F$ of $W$ in the same orientation, can be realized by the monodromy action on the Hitchin fiber lattice along some sufficiently small oriented meridian in $B_{G,C,P}^{\mathrm{good}}$ around $\mathcal{D}_{bdry,C,P}$.
\end{proposition}

\subsection{Local Monodromies around \texorpdfstring{$\mathcal{D}_{\mathrm{cau,C,P}}$}{} and \texorpdfstring{$\mathcal{D}_{\mathrm{Max,C,P}}$}{}}

In this subsection, we will see that local monodromies around $\mathcal D_{Max,C,P}$ and $\mathcal D_{cau,C,P}$ are trivial as suggested by isotopy deformations and the compositions of two shrink \& expand moves respectively on generic $A_{m-1}$-webs.

\subsubsection{Local Monodromies around \texorpdfstring{$\mathcal{D}_{\mathrm{cau,C,P}}$}{}}

Let $x_0\in B_{G,C,P}^{\mathrm{good}}$ be close to a point $b_0\in \mathcal{D}_{cau,C,P}^{\circ}$.
Let $\pi_t:\Sigma_t\to C$
be a family of smooth Hitchin spectral covers parametrized by a path in $B_{G,C,P}$ starting at $x_0$. Suppose two simple ramification points, whose local monodromies are adjacent transpositions $(ij)$ and $(jk)$,
coalesce and then split again. Denote this process by $A$.

Near the collision on an open set $U\subset C$, the three participating sheets can be isolated from the remaining sheets in a local coordinate:
$P_t(\lambda,z)
=Q_t(\lambda,z)\,R_t(\lambda,z),$
where $Q_t$ is a cubic factor and $R_t$ is non-vanishing at the collision.
After choosing local coordinates, the cubic factor may be put in the form
$Q_t(\lambda,z)=\lambda^3-3t\lambda-z.$

For $t\neq 0$, the equation
$\partial_\lambda Q_t=0$
has two distinct solutions $\lambda_\pm=\pm\sqrt{t}$, and hence $Q_t=0$ gives two simple branch points $z_\pm=\mp2t^{3/2}$. At $t=0$, $\partial_\lambda Q_0=0$ has one solution $\lambda=0$, the two simple ramification points coalescing into a single ramification point of index 3 over the double branch point $z=0$.
The local monodromy around both simple branch values is the product
$(ij)(jk)=(ijk)$, which is the monodromy of the double ramification point.

Locally on $U$, the discriminant satisfies
$\Delta_t(z)=\Delta(Q_t)\Delta(R_t)\Res(Q_t,R_t)^2.$
Since the last two factors are non-zero near the collision, the local behaviour is governed by the cubic factor.
Locally after choosing local coordinates, we may write $\Delta(Q_t)
=b(z)\cdot(z^2-c(z)\,t^3)$, where $b(z),c(z)$ are non-vanishing on $U$.
Hence, when $t\neq 0$,
$\Delta(P_t)$ has two simple zeros; when $t=0$, $\Delta(P_0)$ has one double zero.

\begin{remark}
Both a nodal singularity and a double ramification point produce
$\operatorname{ord}_q\Delta(P)=2$.
For an index-$3$ ramification point over $q\in C$, $P(q)=\partial_\lambda P(q)=\partial_\lambda^2P(q)=0$,
with $\partial_\lambda^3P(q)$, $\partial_zP(q)\neq0$.
The spectral curve is smooth over $q$.
By contrast, a nodal singularity satisfies
$P(q)=\partial_\lambda P(q)=\partial_zP(q)=0$, with $\partial_\lambda^2P(q)\neq 0$.
The spectral curve itself is singular.
\end{remark}

Let $B_{G,C,P}^{\mathrm{sm}}
=\{a\in B_{G,C,P}:\Sigma_a\text{ is smooth}\}$ be the smooth-spectral-curve locus. $B_{G,C,P}^{\mathrm{good}}\subset B_{G,C,P}^{\mathrm{sm}}$.
The index-$3$ ramification locus lies inside $\mathcal{D}_{cau,C,P}\subset B_{G,C,P}^{\mathrm{sm}}\setminus B_{G,C,P}^{\mathrm{good}}$.

\begin{lemma}
\label{le:contractible-monodromy}
Let
$q:\Sigma\to B_{G,C,P}^{\mathrm{sm}}$
be a smooth proper family of spectral curves,
and let
$\gamma:[0,1]\to B_{G,C,P}^{\mathrm{sm}}$
be a loop based at $a\in B_{G,C,P}^{\mathrm{sm}}$.
If $\gamma$ is contractible in
$B_{G,C,P}^{\mathrm{sm}}$, then the monodromy representation
$\rho:
\pi_1(B_{G,C,P}^{\mathrm{sm}},a)
\to
\operatorname{Aut}\bigl(H_1(\Sigma_a;\mathbb Z)\bigr)$
sends $[\gamma]$ to the identity. Hence its Gauss--Manin monodromy is trivial:
$\PT_{\gamma}^{GM}
=\operatorname{id}_{H_1(\Sigma_a;\mathbb Z)}$.
\end{lemma}

Take a sufficiently small meridian around $\mathcal{D}_{cau,C,P}$ in $B_{G,C,P}^{\mathrm{good}}$ obtained by performing the process $A$ twice.
This meridian is contractible in $B_{G,C,P}^{\mathrm{sm}}$. The family of spectral curves parametrized by the meridian extends smoothly
over the contraction.

Therefore its Gauss--Manin monodromy is trivial,
as suggested by the trivial automorphism on the lattices we got in Section 3.4 by performing shrink \& expand moves twice on $A_{m-1}$-webs.
The restricted monodromy on the $SL_m$ Hitchin fiber lattice $H_1(\Prym(\Sigma_0);\mathbb Z)
=\ker(
H_1(\Sigma_0,\mathbb Z)
\to
H_1(C,\mathbb Z))$ is also trivial.

\medskip
\subsubsection{Local Monodromies around \texorpdfstring{$\mathcal{D}_{\mathrm{Max,C,P}}$}{}}
Similarly, consider the locus $\mathcal{D}_{Max,C,P}\subset B_{G,C,P}^{\mathrm{sm}}\setminus B_{G,C,P}^{\mathrm{good}}$ where two distinct simple ramification points on the smooth spectral curve have the same branch value.
Since the family of spectral curves extends smoothly across
$\mathcal D_{Max,C,P}$,
a small meridian around $\mathcal D_{Max,C,P}$ is contractible in $B_{G,C,P}^{\mathrm{sm}}$. Therefore the corresponding Gauss--Manin monodromies on Hitchin-fiber lattices are trivial, consistent with the trivial automorphisms on quiver lattices produced by isotopy moves on generic $A_{m-1}$-webs.

\medskip
\begin{proposition}
\label{prop:trimono}
Sufficiently small oriented meridians in $B_{G,C,P}^{\mathrm{good}}$ around $\mathcal{D}_{cau,C,P}$ $($resp. $\mathcal{D}_{Max,C,P})$ and two successive shrink $\&$ expand moves $($resp. isotopy$)$ on generic $A_{m-1}$-webs both give the trivial monodromy/automorphism on corresponding lattices.
\end{proposition}

\newpage
\subsubsection{Conclusion}

Now we can finalize our summary on the correspondence between the compositions of three elementary moves on generic $A_{m-1}$-webs on $C$ and the small disc areas around three types of discriminant loci in the Hitchin base $B_{SL_m,C,P}$, for a $n$-punctured smooth complex genus-$g$ curve $\mathring{C}=C\setminus P$ satisfying $2g-2+n>0$, $G=SL_m$, $m \ge 2$.

\begin{table}[htbp]
\centering
\caption{Elementary moves and disc areas around discriminant loci.}
\label{table:3compare}
\begin{tabular}{c|c|c}
Two \ttt moves & Two shrink \& expand moves & Isotopy of webs \\
\hline
Around $\mathcal{D}_{bdry,C,P}$
& Around $\mathcal{D}_{cau,C,P}$
& Around $\mathcal{D}_{Max,C,P}$\\
\hline
2 simple zeros of $\Delta$ braid/coalesce
& 2 simple zeros of $\Delta$ braid/coalesce
& 2 simple zeros of $\Delta$ braid/coalesce\\
\hline
2 simple branch points braid/collide
& 2 simple branch points braid/coalesce
& 2 simple branch points braid/coincide\\
\hline
Dehn twist/nodal singularity
& 2 ramification points rotate/coalesce
& 2 ramification points over close values\\
\hline
Picard-Lefschetz transformation
& Trivial automorphism or monodromy
& Trivial automorphism or monodromy
\end{tabular}
\end{table}

\newpage
\mysection{Cluster Atlases on Moduli Spaces of Decorated Local Systems}

Let $G$ be a reductive Lie group whose derived subgroup has Dynkin type $A_{m-1}$, $m \ge 2$.
In this section, we focus on $GL_m$ and $Sl_m$. We will see how $A_{m-1}$-webs and $A_{m-1}^*$-webs, i.e., rank $m$ ideal webs, produce regular functions, and hence cluster atlases on moduli spaces of decorated $G$-local systems.

\medskip
\noindent
Let us first recall the framework introduced in \cite{Goncharov}.

We denote by $S$ a decorated surface without boundary. Removing a little disc around each special point of $S$, we get a surface $\mathring{S}$. Let $\partial_{\mathrm{th}} \mathring{S}$ be the thickened boundary of $\mathring{S}$, a small neighborhood of the boundary
of $\mathring{S}$.

Denote by $T'S$ the unit tangent bundle to $S$. A \emph{twisted $G$-local system} on $S$ is a $G$-local system on $T'S$ which has
monodromy $(-1)^{m-1}$ around a loop rotating a tangent vector at any point by
$2\pi$.

Let $W$ be an $A_{m-1}$-web on $S$. Denote by $\widetilde W:=p^{-1}(W)$ its lift to the universal cover $\widetilde S$. Define $\mathcal F(W)
:=
\{\text{faces of }\widetilde W\}/\pi_1(S)$. Thus, in the non-strict case, a face of $W$ means a
$\pi_1(S)$-orbit of faces of the lifted web $\widetilde W$. Denote by $d_i(F)$ the codistance of a face $F\in \mathcal F(W)$ and a special point $s_i$.


\subsection{Moduli Spaces of Decorated Local Systems}

Let $V_m$ be an $m$-dimensional vector space with volume form $\Omega_m\in \bigwedge^m V_m^*.$

\begin{definition}[{\cite{Goncharov}}]
A decorated flag in $(V_m,\Omega_m)$ is a datum
$(F_{\bullet},\alpha_{\bullet})$:
\begin{equation}
\label{eq:decorated-flag}
F_{\bullet}:\quad
0=F_0\subset F_1\subset \cdots \subset F_{m-1}\subset F_m=V_m,
\qquad
\dim F_p=p,
\qquad
\alpha_p\in \det F_p \setminus 0,
\end{equation}
consisting of a flag $F_{\bullet}$ and its degree-$p$ decorations: non-zero $p$-vectors
$\alpha_p\in \det F_p=\bigwedge^p F_p$ for $p=1,\ldots,m-1$.
\end{definition}

\begin{definition}[{\cite{Goncharov}}]
A $*$-decorated flag in $(V_m,\Omega_m)$ is a datum
$(F_{\bullet},\alpha_{\bullet})$:
\begin{equation}
F_{\bullet}:\quad
0=F_0\subset F_1\subset \cdots \subset F_{m-1}\subset F_m=V_m,
\qquad
\dim F_p=p,
\qquad
\alpha_p\in \det F_p.
\end{equation}
consisting of a flag $F_{\bullet}$ and its degree-$p$ $*$-decorations: $p$-vectors
$\alpha_p\in \det F_p=\bigwedge^p F_p$ for $p=1,\ldots,m$.
\end{definition}

The decorations determine non-zero $p$-vectors in the successive quotients
$\bar v_p\in F_p/F_{p-1}$,
such that intrinsically
$\det F_p \cong \det F_{p-1}\otimes(F_p/F_{p-1}),$ and, after choosing any lift $v_p\in F_p$, $\alpha_p=\alpha_{p-1}\wedge v_p.$

\medskip
Let $G$ be either $GL_m$ or $SL_m$. Let B be a Borel subgroup of G, U be a maximal unipotent subgroup of G. T $\cong$ B/U is the Cartan torus. For $GL_m$, $T \cong (G_m)^m$, for $SL_m$, $T \cong (G_m)^{m-1}$. G/B is the flag variety. G/U is the space of decorated flags.

Denote by $\mathcal A_{G}$ the moduli space of decorated flags for $G$ in
$(V_m,\Omega_m)$. We have
\begin{equation}
\mathcal A_{G}:=G/U
\end{equation}

\begin{definition}[{\cite{Goncharov}}]
Let $\mathcal L$ be a $G$-local system on $\mathring{S}$. Then $\mathcal L_{\mathcal A}:=\mathcal L/U$ is a decorated $G$-local system associated with $\mathcal L$, with fiber $G/U=\mathcal A_G$. 
A decoration on $\mathcal L_{\mathcal A}$ is a flat section of the
restriction of $\mathcal L_{\mathcal A}$ to $T'(\partial_{\mathrm{th}}\mathring{S})$.
\end{definition}

\begin{definition}[{\cite{Goncharov}}]
Let $\mathcal L$ be a $SL_m$-local system on $\mathring{S}$. Then $\mathcal L_{\mathcal A^*}
:=
\mathcal L\times_{SL_m}(GL_m/U)$ is a $*$-decorated $SL_m$-local system associated with $\mathcal L$, with fiber $GL_m/U$.
A $*$-decoration on $\mathcal L_{\mathcal A^*}$ is a flat section of the restriction of $\mathcal L_{\mathcal A^*}$
to $T'(\partial_{\mathrm{th}}\mathring{S})$. 
\end{definition}

\begin{definition}
Let $\Gamma$ be a group, and $C$ be a $\Gamma$-set. Let $H$ be an algebraic group, and $Y$ be a $H$-space. The \emph{equivariant configuration space} is
\begin{equation}
\mathrm{Conf}_{\Gamma,H}(C,Y)
:=
\left\{
(\rho,f)
\ \middle|\
\begin{array}{l}
\rho:\Gamma\to H,\\[2pt]
f:C\to Y,\\[2pt]
f(\gamma x)=\rho(\gamma)f(x)
\quad\text{for all }\gamma\in\Gamma,\ x\in C
\end{array}
\right\}/H,
\end{equation}
where $\rho$ is a group homomorphism, and $H$ acts by $h\cdot(\rho,f)
=
(h\rho h^{-1},\,hf)$.
\end{definition}
Thus $\mathrm{Conf}_{\Gamma,H}(C,Y)$ means: a monodromy representation $\rho$ together
with a $\rho$-equivariant configuration of points of $Y$, modulo
simultaneous $H$-action.

\begin{definition}
The moduli space of decorated $G$-local systems on $S$ is given by
\begin{equation}
\mathcal A_{G,S}
=
\mathrm{Conf}_{\pi_1(T' \mathring{S}),G}
\left(
\widetilde{T'(\partial_{\mathrm{th}} \mathring{S})},
G/U
\right).
\end{equation}
Expanded, this means
\begin{equation}
\mathcal A_{G,S}
=
\left\{
\begin{array}{c}
(\rho,\alpha)\\[2pt]
\rho:\pi_1(T' \mathring{S})\to G,\\[2pt]
\alpha:\widetilde{T'(\partial_{\mathrm{th}} \mathring{S})}\to G/U
\end{array}
\ \middle|\
\begin{array}{l}
\alpha(\gamma x)=\rho(\gamma)\alpha(x),\forall\gamma\in\pi_1(T' \mathring{S}),\ x\in \widetilde{T'(\partial_{\mathrm{th}} \mathring{S})}\\[2pt]
\rho(\tau)\alpha(x)=\alpha(x)
\end{array}
\right\}/G.
\end{equation}
Here $\tau$ is a loop around any puncture.

The map $\alpha$ assigns a decorated flag to each lifted boundary point, and the equivariance condition makes it
descend to a flat section of the associated decorated flag bundle downstairs.

For $G=SL_m$, one also imposes the twisting condition $\rho(\sigma)=(-1)^{m-1}I$, where $\sigma$ is the loop rotating a non-zero tangent vector by
$2\pi$.
\end{definition}

\begin{definition}
The moduli space of $*$-decorated $SL_m$-local
systems on $S$ is given by
\begin{equation}
\mathcal A^*_{SL_m,S}
=
\mathrm{Conf}_{\pi_1(T' \mathring{S}),SL_m}
\left(
\widetilde{T'(\partial_{\mathrm{th}} \mathring{S})},
GL_m/U
\right).
\end{equation}
Expanded, this means
\begin{equation}
\mathcal A^*_{SL_m,S}
=
\left\{
\begin{array}{c}
(\rho,\alpha^*)\\[2pt]
\rho:\pi_1(T' \mathring{S})\to SL_m,\\[2pt]
\alpha^*:\widetilde{T'(\partial_{\mathrm{th}} \mathring{S})}\to GL_m/U
\end{array}
\ \middle|\
\begin{array}{l}
\alpha^*(\gamma x)=\rho(\gamma)\alpha^*(x),\forall\gamma\in\pi_1(T' \mathring{S}),\ x\in \widetilde{T'(\partial_{\mathrm{th}} \mathring{S})}\\[2pt]
\rho(\tau)\alpha^*(x)=\alpha^*(x),\\[2pt]
\rho(\sigma)=(-1)^{m-1}I
\end{array}
\right\}/SL_m.
\end{equation}
Here $\sigma$ is the loop rotating a nonzero tangent vector by $2\pi$, $\tau$ is a loop around any puncture.
\end{definition}

\subsection{Cluster \texorpdfstring{$\mathcal{A}$}{A}-Mutations on Moduli Spaces of Decorated Local Systems}

\subsubsection{Cluster Seeds and Cluster \texorpdfstring{$\mathcal{A}$}{A}-Mutations}

Let $I$ be a finite set of vertices, and let $I_{\mathrm{f}}\subset I$ be the subset of frozen vertices. The mutable vertices are $I_{\mathrm{mut}}:=I\setminus I_{\mathrm{f}}$. A quiver without loops or two-cycles is described by $I$, $I_{\mathrm{f}}\subset I$, and a skew-symmetric matrix $\varepsilon=(\varepsilon_{ij})_{i,j\in I}$, where $\varepsilon_{ij}
=
\#\{i\to j\}
-
\#\{j\to i\}$.

Equivalently, one may encode the quiver by a lattice $\Lambda=\bigoplus_{i\in I}\mathbb Z e_i$, a basis $\{e_i\}_{i\in I}$ of $\Lambda$, together with a
skew-symmetric $\mathbb Z$-valued bilinear form $(*,*):\Lambda\times\Lambda\to \mathbb Z$ such that $(e_i,e_j)=\varepsilon_{ij}$. The frozen sublattice is $\Lambda_{\mathrm{f}}
=
\bigoplus_{i\in I_{\mathrm{f}}}\mathbb Z e_i$. The non-frozen
basis vectors generate a subquiver, called the non-frozen part of the initial quiver.

\begin{definition}[\cite{Goncharov}]
A quiver seed $Q$ consists of data
\begin{equation}
Q=(I,I_{\mathrm{f}},\Lambda,\{e_i\}_{i\in I},(*,*)),
\end{equation}
where $\{e_i\}$ is a basis of $\Lambda$, $I_{\mathrm{f}}\subset I$ is the frozen subset,
and $(*,*)$ is a skew-symmetric $\mathbb Z$-valued form.
\end{definition}

Let $Q=(I,I_{\mathrm{f}},\Lambda,(*,*),\{e_i\}_{i\in I})$ be a quiver seed. The \emph{mutable part} is $I_{\mathrm{mut}}:=I\setminus I_{\mathrm{fr}}$, and the \emph{mutable lattice} is $\Lambda_{\mathrm{mut}}
=
\bigoplus_{i\in I_{\mathrm{mut}}}\mathbb Z e_i
\subset \Lambda$. A \emph{compatible $\mathcal A$-lattice} is a lattice $\Lambda_{\mathcal A}$ with basis $\{\theta_i\}_{i\in I}$ such that, for mutable vertices, $(e_i,\theta_j)=\delta_{ij}$, $i\in I_{\mathrm{mut}},\ j\in I$, where $(*,*)$ denotes the pairing between the quiver lattice and the $\mathcal A$-lattice. Therefore the $\{\theta_i\}_{i\in I_{\mathrm{mut}}}$ are dual to the mutable exchange directions $\{e_i\}_{i\in I_{\mathrm{mut}}}$, i.e., $\Lambda_{\mathcal A,\mathrm{mut}}
:=
\bigoplus_{i\in I_{\mathrm{mut}}}\mathbb Z\theta_i
\cong
\Lambda_{\mathrm{mut}}^\vee$ \cite{SS25}.

\begin{definition}[\cite{SS25}]
A cluster $\mathcal A$-seed is a quiver seed together with a compatible $\mathcal A$-lattice
and a distinguished set of algebraically independent coordinate functions $\{A_i:=z^{\theta_i}\}_{i\in I}$. Thus a cluster $\mathcal A$-seed determines an algebraic torus $T_{\mathcal A,q}
=
\operatorname{Hom}(\Lambda_{\mathcal A},\mathbb G_m)$ with coordinates $\{A_i\}_{i\in I}$.
\end{definition}

A mutation of a quiver $Q$ in the direction of a non-frozen basis vector $e_k$ is a new quiver $Q'=\mu_{e_k}(Q)$. The lattice $\Lambda$ and the form $(*,*)$ for $Q'$ are the same as for $Q$. 

The basis
$\{e_i'\}$ for $Q'$ is given by the \emph{tropical mutation formula} \cite{Goncharov}
\begin{equation}
e_i' :=
\begin{cases}
-e_k,
& \text{if } i=k,\\[4pt]
e_i+(e_i,e_k)_+e_k,
& \text{if } i\neq k,
\end{cases}
\qquad
a_+ := \max(a,0).
\end{equation}

The composition of mutations $\mu_{e_k'}\circ \mu_{e_k}$ gives $e_k''=e_k$,
$e_i''=e_i+(e_i,e_k)e_k$, $i\neq k$.
It is the reflection of the
basis $\{e_i\}$ at $e_k$.

Equivalently, at the level of the exchange matrix $\varepsilon_{ij}:=(e_i,e_j)$, the mutation is \cite{Goncharov}
\begin{equation}
\varepsilon'_{ij}
=
\begin{cases}
-\varepsilon_{ij}, & i=k \text{ or } j=k,\\[4pt]
\varepsilon_{ij}, & \varepsilon_{ik}\varepsilon_{kj}\leq 0,\ i,j\neq k,\\[4pt]
\varepsilon_{ij}+|\varepsilon_{ik}|\varepsilon_{kj},
& \varepsilon_{ik}\varepsilon_{kj}>0,\ i,j\neq k.
\end{cases}
\end{equation}

Let $\mathbf A=\{A_i\}_{i\in I}$ be a cluster $\mathcal A$-seed with quiver $Q$, and let $B=(\varepsilon_{ij})$ be its exchange matrix. A \emph{cluster $\mathcal A$-mutation} at a mutable
vertex $k$ changes the $\mathcal A$-coordinates by $A_i'=A_i$, $i\neq k$, and
\begin{equation}
A_kA_k'
=
\prod_i A_i^{[\varepsilon_{ik}]_+}
+
\prod_i A_i^{[-\varepsilon_{ik}]_+}
=\prod_{\varepsilon_{ik}>0}A_i^{\varepsilon_{ik}}
+
\prod_{\varepsilon_{ik}<0}A_i^{-\varepsilon_{ik}}, \qquad i\neq k,
\end{equation}
where $a_+:=\max(a,0)$. Equivalently,
\begin{equation}
A_k'
=
\frac{\prod_i A_i^{\#(i\to k)}}{A_k}
+
\frac{\prod_i A_i^{\#(k\to i)}}{A_k}, \qquad i\neq k.
\end{equation}
This is a birational coordinate change.
The two terms are determined by the incoming and outgoing
arrows at $k$.
Frozen vertices are not mutated, although their variables may occur as coefficients in exchange relations \cite{SS25}.

\subsubsection{Cluster Coordinates, Charts, and Atlases}

Let $X$ be an irreducible variety with function field $K=k(X)$.

\begin{definition}
A \emph{cluster coordinate} is a single rational function $x_i:X\dashrightarrow \mathbb G_m$ appearing as a variable in a cluster seed.
\end{definition}

\begin{definition}[\cite{SS25}]
A \emph{cluster coordinate system} is a collection $\mathbf x=(x_1,\ldots,x_N)$ of cluster coordinates belonging to one seed and defining a birational identification
\begin{equation}
\phi_{\mathbf x}:
X\dashrightarrow
T_{\mathbf x}:=
\operatorname{Spec}
k[x_1^{\pm1},\ldots,x_N^{\pm1}]
\cong \mathbb G_m^N,
\qquad
p\mapsto (x_1(p),\ldots,x_N(p)).
\end{equation}
The seed torus $T_{\mathbf x}$, together with $\phi_{\mathbf x}$, is called the
\emph{cluster coordinate chart} associated with $\mathbf x$.
\end{definition}

\begin{remark}
The seed torus $T_{\mathbf x}$ should not be identified with the whole $X$. Rather, birational means that there is a dense open subset $U_{\mathbf x}\subset X$ on which the coordinates give an isomorphism onto a dense open subset of the torus.
\end{remark}

\begin{definition}[\cite{SS25}]
A \emph{cluster atlas} on $X$ is a collection of birational torus charts $X\dashrightarrow\mathbb G_m^N$ whose transition maps are finite compositions of cluster mutations.
\end{definition}

\begin{remark}
A cluster atlas is not necessarily an ordinary atlas covering every point of $X$. Rather, it is a mutation-connected collection of birational torus charts.
\end{remark}

\begin{proposition}
Suppose geometric objects $W$ determine seeds $\mathbf s_W$ and coordinate systems $\mathbf x_W$. Assume:
\begin{enumerate}
    \item each $\mathbf x_W$ gives a birational torus chart on $X$;
    \item each elementary geometric move $W\to W'$
    induces a seed mutation
    $\mathbf s_W\to\mathbf s_{W'}$;
    \item any two relevant geometric objects are connected by finitely many elementary moves.
\end{enumerate}
Then the charts associated with the $\mathbf x_W$ form a \emph{cluster atlas} on $X$.
\end{proposition}

\begin{theorem}[\cite{BFZ05}]
Let $\{\mathbf x_\alpha\}_{\alpha\in I}$
be cluster coordinate systems associated with all seeds in a fixed mutation class, viewed in a
common function field $K$, and let $T_\alpha
=
\operatorname{Spec}
k[x_{\alpha,1}^{\pm1},\ldots,x_{\alpha,N}^{\pm1}]$
be the corresponding seed tori. Suppose the $T_\alpha$ are glued along the open
subsets on which the corresponding
cluster transformations and their inverses are regular, producing $X_{\mathrm{cl}}
=
\bigcup_{\alpha\in I}T_\alpha.$
Then
\begin{equation}
\mathcal O(X_{\mathrm{cl}})
=
\bigcap_{\alpha\in I}
\mathcal O(T_\alpha)
=
\bigcap_{\alpha\in I}
k[x_{\alpha,1}^{\pm1},\ldots,x_{\alpha,N}^{\pm1}]
\subset K.
\end{equation}
The intersection $\bigcap_{\alpha\in I}
\mathcal O(T_\alpha)$ is called the \emph{upper cluster algebra} associated with the mutation class.
\end{theorem}

\subsubsection{Regular Functions on Moduli Spaces of Decorated Local Systems}

Let $\widetilde{\mathcal L_{\mathcal A}}$ be the lift of a decorated $G$-local system $\mathcal L_{\mathcal A}$
to $\widetilde S$. For each special point $s_i$ on $S$, let $\widetilde s_i$ be one lift of $s_i$. Providing a decoration $A(s_i)$ around $s_i$, and a zig-zag strand $\gamma$ of $\widetilde{W}$ associated with $\widetilde s_i$, we can extend $A(s_i)$ to a flat section of $\widetilde{\mathcal L}_{\mathcal A}$ on the punctured disc
$\widetilde S_{\gamma}-\widetilde s_i$ \cite{Goncharov}.

Let $F$ be a face of $\widetilde{W}$.
Choose a non-zero tangent vector
$v$ at a point inside $F$. Let $p$ be a path in $T'\widetilde S$ connecting $v$ to a non-zero tangent vector $v_i$ at a point near the special point $s_i$ on $\widetilde{S}$. Transport the decoration $A(s_i)$ along $p$ from $v_i$ to $v$. The resulting decoration $A_v(s_i)$ over $v$ is well-defined since $A(s_i)$ is a flat section of $\widetilde{\mathcal L}_{\mathcal A}$ over $\widetilde S_{\gamma}-s_i$. We then get decorations $A_v(s_1),\ldots,A_v(s_n)$ over $v$ \cite{Goncharov}.

We can similarly get $*$-decorations on a given face $F$ for $*$-decorated $SL_m$-local systems on $S$.

\begin{definition}[\cite{Goncharov}]
Given an $A_{m-1}$ (resp. $A_{m-1}^*$)-web $W$ on $S$ and a face $F$ of
$\widetilde{W}$, a regular function $A_F^{W}$ on
the moduli space $\mathcal A_{SL_m,S}$ (resp. $\mathcal A^*_{SL_m,S}$) is
\begin{equation}
A_F^{W}
:=
\bigl\langle
\Omega_m,
\alpha_{d_1(F)}(A_v(s_1))\wedge \cdots \wedge
\alpha_{d_n(F)}(A_v(s_n))
\bigr\rangle.
\end{equation}
\end{definition}

This expression is well-defined because $d_1(F)+\cdots+d_n(F)=m$. The regular functions $A_F^{W}$ are $\pi_1(S)$-invariant, so we may assume that $F$ stands for the $\pi_1(S)$-orbit in the space of
faces of $\widetilde{W}$.

\begin{lemma}[\cite{Goncharov}]
The regular functions $A_F^{W}$ are invariant under the twisted cyclic shift
\begin{equation}
t:
(A_v(s_1),\ldots,A_v(s_n))
\mapsto
\bigl((-1)^{m-1}A_v(s_n),A_v(s_1),\ldots,A_v(s_{n-1})\bigr).
\end{equation}
\end{lemma}

\begin{remark}
A cyclic shift of the wedge product gives an extra sign $(-1)^{d_n(m-d_n)}$.

The twisted cyclic shift replaces $A_n$ by $(-1)^{m-1}A_n$. Since the central element $(-I)$ acts on $\Lambda^{d_n}V_m$ by $(-1)^{d_n}$, the twist contributes the additional factor $(-1)^{(m-1)d_n}$. Hence the total sign is
\begin{equation}
(-1)^{d_n(m-d_n)}(-1)^{(m-1)d_n}
=
(-1)^{d_n(2m-d_n-1)}
=
1,
\end{equation}
Therefore the coordinate $A_F^W$ is independent of the choice of
starting point in the cyclic order.
\end{remark}

\medskip
\noindent
Recall, in Section 3.1, the quiver $Q_{\Gamma}
:=
\bigl(\Lambda_{\Gamma},\{e_F\},(*,*)_{\Gamma}\bigr)$ and the non-frozen subquiver $Q_{\Gamma}^{\circ}
:=
\bigl(
\Lambda_{\Gamma}^{\circ},
\{e_G\},
\operatorname{Res}_{\Lambda_{\Gamma}^{\circ}}(*,*)_{\Gamma}
\bigr)$ assigned to a bipartite graph $\Gamma$ on $S$, with the vectors $\{e_G\}$ assigned to the internal faces $G$ of $\Gamma$ the
non-frozen basic vectors \cite{Goncharov}.

The face $F$ also labels a
regular cluster $\mathcal A$-coordinate $A_F^W$.
After choosing a compatible $\mathcal A$-lattice with basis
$\{\theta_F\}$, we may write
$A_F^W=z^{\theta_F}$.
Thus the set of regular functions $\{A_F^W\}_{F\in\mathcal F(W)}$ assigned to faces of a given web $W$ together with the associated quiver seed $Q=(I,I_{\mathrm{f}},\Lambda,(*,*),\{e_i\}_{i\in I})$ form a cluster $\mathcal{A}$-seed \cite{Goncharov}.

For an ideal triangulation $T$ and the associated standard $A_{m-1}$-web $W_m(T)$, it is proven that the functions $\{A_F^{W_m(T)}\}_{F\in\mathcal F(W_m(T))}$ form a regular cluster $\mathcal{A}$-coordinate system on $\mathcal A_{SL_m,S}$, hence they are algebraically independent. \cite{FG06,Goncharov}

Given an $A_{m-1}$-web $W$ on $S$, we will show later in Proposition~\ref{prop:22} that a \ttt move induces a cluster $\mathcal{A}$-mutation at $F$, which is birational. Thus algebraic independence is preserved under every \ttt move. 

By Proposition~\ref{prop:connect}, any $A_{m-1}$-web on $S$ is connected by a finite sequence of \ttt moves to the standard $A_{m-1}$-web $W_m(T)$ associated with an ideal triangulation $T$. Then the corresponding $A$-functions $\{A_F^W\}_{F\in\mathcal F(W)}$ are algebraically independent. Thus the regular functions $\{A_F^W\}_{F\in\mathcal F(W)}$ form a regular cluster $\mathcal A$-coordinate system on $\mathcal A_{SL_m,S}$.

For an $A_{m-1}$ or $A_{m-1}^*$ web, the ordinary $\mathcal A$-coordinates are face functions $A_F$, where $F$ runs over faces of a bipartite graph or web. For a strict ideal web, these are
ordinary faces downstairs on the surface. For a non-strict ideal web, one should interpret
the faces upstairs on the universal cover and then take deck-transformation orbits. Thus one writes $\mathcal F(\Gamma)
=
\pi_1(S)\backslash
\{\text{faces of } \widetilde{\Gamma}\}$.

\subsubsection{Cluster \texorpdfstring{$\mathcal A$}{A}-Mutations Induced by Two-by-two Moves}

Let $W$ be an ideal web of rank-$m$ on a decorated surface $S$. Let $F$ be a face of $W$.
Consider a regular function
$A_F^{W}
=
\bigl\langle
\Omega_m,
\alpha_{d_1(F)}(A_v(s_1))\wedge \cdots \wedge
\alpha_{d_n(F)}(A_v(s_n))
\bigr\rangle$.
Denote $d_i(F)$ by $d_i$, and $A_v(s_i)$ by $A_i$.
We may isolate two degree-one pieces $v_{d_i}$ and $v_{d_j}$ from the wedge product $\bigwedge_k\alpha_{d_k}(A_k)=\alpha_{d_1}(A_1)\wedge \cdots\wedge \alpha_{d_i}(A_i)\wedge\cdots\wedge \alpha_{d_j}(A_j) \wedge\cdots\wedge
\alpha_{d_n}(A_n)$ to the right.

\medskip
\noindent
\textbf{Case 1: $i\neq j$, $d_i,d_j\ge 1$.}
Write
\begin{equation}
\alpha_{d_i}(A_i)
=
\alpha_{d_i-1}(A_i)\wedge v_{d_i},
\qquad
\alpha_{d_j}(A_j)=\alpha_{d_j-1}(A_j)\wedge v_{d_j},
\end{equation}
with
$\bar v_{d_i}\in F_{d_i}^{(i)}/F_{d_i-1}^{(i)},
$ $
\bar v_{d_j}\in F_{d_j}^{(j)}/F_{d_j-1}^{(j)}.$
Define
\begin{equation}
\varrho_{ij}
=
\alpha_{d_1}(A_1)\wedge\cdots
\wedge\alpha_{d_i-1}(A_i)
\wedge\cdots
\wedge\alpha_{d_j-1}(A_j)
\wedge\cdots
\wedge\alpha_{d_n}(A_n).
\end{equation}
Then
$\varrho_{ij}\in\bigwedge^{m-2}V,$
and, up to a sign $\varepsilon_{ij}$,
$
\bigwedge_k\alpha_{d_k}(A_k)
=
\varepsilon_{ij}\,
\varrho_{ij}\wedge v_{d_i}\wedge v_{d_j}.
$
Hence
\begin{equation}
A_{F}^{W}
=
\varepsilon_{ij}
\bigl\langle
\iota_{\varrho_{pq}}\Omega_m, v_{d_i}\wedge v_{d_j}\bigr\rangle.
\end{equation}

\medskip
\noindent
\textbf{Case 2: $i=j$, $d_i\ge 2$.}
Write
\begin{equation}
\alpha_{d_i}(A_i)
=
\alpha_{d_i-2}(A_i)
\wedge
v_{d_i-1}
\wedge
v_{d_i},
\end{equation}
with
$\bar v_{d_i-1}
\in
F_{d_i-1}^{(i)}/F_{d_i-2}^{(i)},$
and
$\bar v_{d_i}
\in
F_{d_i}^{(i)}/F_{d_i-1}^{(i)}.$
Define
\begin{equation}
\varrho_{ii}
=
\alpha_{d_1}(A_1)\wedge\cdots
\wedge
\alpha_{d_i-2}(A_i)
\wedge\cdots
\wedge
\alpha_{d_n}(A_n)
\in\bigwedge^{m-2}V.
\end{equation}
Then
\begin{equation}
A_{F}^{W}
=
\varepsilon_{ii}
\bigl\langle
\iota_{\varrho_{ii}}\Omega_m, v_{d_i-1}\wedge v_{d_i}\bigr\rangle.
\end{equation}

\begin{remark}
If a lift is changed by
$v_k\mapsto v_k+x$,
$x\in F_{k-1},$
then
$\alpha_{k-1}(A)\wedge x=0.$
Hence the wedge-volume scalar is independent of the chosen lift.
Therefore the contracted form
$\iota_{\varrho}\Omega_m
\in\bigwedge^2V^*$
descends canonically to a pairing of the relevant quotients.

For $i\neq j$,
$\iota_{\varrho}\Omega_m:
\left(F_{d_i}^{(i)}/F_{d_i-1}^{(i)}\right)
\times
\left(F_{d_j}^{(j)}/F_{d_j-1}^{(j)}\right)
\to\mathbb C.$

For $i=j$,
$\iota_{\varrho}\Omega_m:
\left(F_{d_i-1}^{(i)}/F_{d_i-2}^{(i)}\right)
\times
\left(F_{d_i}^{(i)}/F_{d_i-1}^{(i)}\right)
\to\mathbb C.$

Thus isolating two vector representatives amounts to removing two degree-one quotient decorations from the original wedge product and contracting the remaining degree-$(m-2)$ decorated-flag data into the volume form.
\end{remark}

\begin{proposition}
\label{prop:22}
Let $W$ be an ideal web of rank-$m$ on a decorated surface $S$.
A \ttt move $\mu_F:W\to W'$ centred at an internal face $F$ induces a cluster $\mathcal{A}$-mutation at $F$. In particular, it is a birational isomorphism between the two cluster seed tori associated with the two webs before and after the mutation.
\end{proposition}

\begin{proof}
Let $W$ be an $A_{m-1}$-web on a decorated surface $S$. Now we perform a \ttt move on the web $W$ at face 1 and change it to a new web $W'$.

We know from Section 2 that a \ttt move changes the web as following:

\begin{figure}[htbp]
\centering
\includegraphics[width=0.5\textwidth]{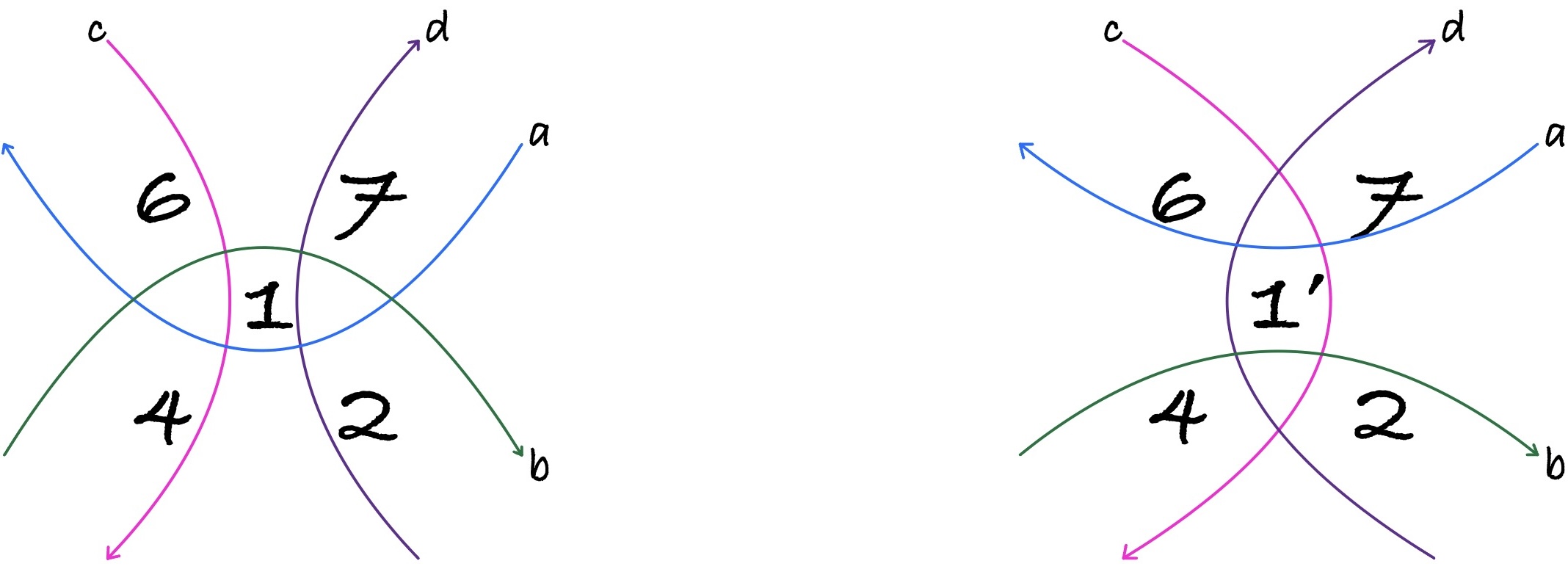}
\caption{Change of strands under a two-by-two move at face 1.}
\label{fig:webreg}
\end{figure}

Recall that the regular function assigned to an internal face $F$ of $W$ is defined as:
\begin{equation}
A_F^{W}
:=
\bigl\langle
\Omega_m,
\alpha_{d_1(F)}(A_v(s_1))\wedge \cdots \wedge
\alpha_{d_n(F)}(A_v(s_n))
\bigr\rangle .
\end{equation}

We can conclude from Figure~\ref{fig:webreg} and the construction of the regular functions, namely first extending the decorations from the special points on the universal cover $\widetilde{S}$ to discs bounded by strands by parallel transport and then taking the volume form of the wedge product of the decorations, that the regular functions assigned to face 2, 4, 6, 7, and all the other faces beside face 1 remain the same under the cluster mutation induced by the \ttt move at face 1. For instance, face 6 is always bounded by strands a and c. The only regular function that changes is the one assigned to face 1, since face 1 was originally bounded by strands a and b, but is now bounded by strands c and d after the \ttt move.

A regular function $A_i$ assigned to face $i$ can be expressed as $A_i =
\bigl\langle
\Omega_m,
\alpha_{d_1(s_1,F_i)}(A_v(s_1))\wedge \cdots \wedge
\alpha_{d_n(s_n,F_i)}(A_v(s_n))
\bigr\rangle$, where $d_j(s_j,F_i)$ is the codistance between the special point $s_j$ and $\widetilde{S}$ and a lift of face $i$. Notice that $A_i$ is $\pi_1(S)$-invariant, so the face $i$ can be viewed as the $\pi_1(S)$-orbit of lifts of face $i$ on $\widetilde{S}$.

Recall the cover of the spectral cover $\widetilde{\Sigma}_{\widetilde{W}}$ used in the construction of the spectral cover $\Sigma_W$ in section~\ref{sssec:231}, where $\Sigma_W = \widetilde{\Sigma}_{\widetilde{W}}/\pi_1(S)$. Now we consider the extended decorations on $\widetilde{\pi}: \widetilde{\Sigma}_{\widetilde{W}}\to \widetilde{S}$. For each $\alpha_{d_j(s_j,F_i)}(A_v(s_j))$ in the expression, it means that there are $d_j(s_j,F_i)$ discs on $\widetilde{\Sigma}_{\widetilde{W}}$, each bounded by a strand associated with the special point $s_j$, that cover a lift of face $i$ on $\widetilde{S}$, and $\alpha_{d_j(s_j,F_i)}(A_v(s_j))$ contributes into the wedge product in defining the regular function $A_i$.

Under a \ttt move at face 1, the strands that bound all the faces except face 1 remain unchanged, thus the discs bounded by strands on $\widetilde{\Sigma}_{\widetilde{W}}$ that cover any lift of those faces on $\widetilde{S}$ remain unchanged, hence the regular functions remain the same. For face 1, since there are only two strands that changes, there are only two discs on $\widetilde{\Sigma}_{\widetilde{W}}$, covering a lift of face 1 on $\widetilde{S}$, are affected by the \ttt move. These are the two discs having the two relevant branch points on their boundaries, and the two discs are glued on two edges incident to the two branch points that are not the boundaries of face 1 when quotienting $\pi_1(S)$.

The regular function assigned to face 1 before the mutation is $A_1 =
\bigl\langle
\Omega_m,
\alpha_{d_1(s_1,F_1)}(A_v(s_1))\wedge \cdots \wedge
\alpha_{d_n(s_n,F_1)}(A_v(s_n))
\bigr\rangle$. We rearrange the wedge product such that the contributions of the two affected discs $v_a$ and $v_b$ are moved to the right.

We write $\rho$ for the unchanged part of wedge product coming from all the discs that are unramified around face 1. Now we have
\begin{equation}
\label{eq:A1}
A_1 = (-1)^r
\bigl\langle
\Omega_m, \rho\wedge v_a\wedge v_b
\bigr\rangle=(-1)^r\bigl\langle
\iota_\rho \Omega_m, v_a\wedge v_b
\bigr\rangle,
\end{equation}
where $(-1)^r$ is the extra sign coming from rearranging the wedge product.

Now similarly, we rearrange the wedge product in all the other regular functions. Specifically, we write the rearranged regular functions assigned to face 2, 4, 6, and 7:
\begin{equation}
\label{eq:regfunc}
\begin{array}{cc}
A_2 = (-1)^r
\bigl\langle
\Omega_m, \rho \wedge v_b\wedge v_d
\bigr\rangle=(-1)^r\bigl\langle
\iota_\rho \Omega_m,v_b\wedge v_d
\bigr\rangle,\\[2pt]
A_4 = (-1)^r
\bigl\langle
\Omega_m, \rho\wedge v_b\wedge v_c
\bigr\rangle=(-1)^r\bigl\langle
\iota_\rho \Omega_m,v_b\wedge v_c
\bigr\rangle,\\[2pt]
A_6 = (-1)^r
\bigl\langle
\Omega_m, \rho\wedge v_a\wedge v_c
\bigr\rangle=(-1)^r\bigl\langle
\iota_\rho \Omega_m,v_a\wedge v_c
\bigr\rangle,\\[2pt]
A_7 = (-1)^r
\bigl\langle
\Omega_m, \rho\wedge v_a\wedge v_d
\bigr\rangle=(-1)^r\bigl\langle
\iota_\rho \Omega_m,v_a\wedge v_d
\bigr\rangle,
\end{array}
\end{equation}
where the subscribes correspond to the strands bounding the faces.

We also have the rearranged regular function assigned to face $1'$ after the mutation:
\begin{equation}
A_1' = (-1)^r
\bigl\langle
\Omega_m, \rho\wedge v_c\wedge v_d
\bigr\rangle=(-1)^r\bigl\langle
\iota_\rho \Omega_m,v_c\wedge v_d
\bigr\rangle.
\end{equation}

Recall that the Pl\"ucker coordinates on $Gr(2,4)$ satisfy
\begin{equation}
\Delta_{ab}\Delta_{cd}
=
\Delta_{ac}\Delta_{bd}
+
\Delta_{ad}\Delta_{bc}.
\end{equation}
Set $A_k=\Delta_{ab}$, $A_k'=\Delta_{cd}$. Then $A_kA_k'
=
\Delta_{ac}\Delta_{bd}
+
\Delta_{ad}\Delta_{bc}$, and therefore
\begin{equation}
A_k'
=
\frac{
\Delta_{ac}\Delta_{bd}
+
\Delta_{ad}\Delta_{bc}
}{
A_k
}.
\end{equation}

This is the cluster mutation replacing the diagonal minor $\Delta_{ab}$ by the other diagonal minor $\Delta_{cd}$. The boundary minors $\Delta_{ac}$, $\Delta_{bc}$, $\Delta_{bd}$, $\Delta_{ad}$ remain fixed. Geometrically, the mutation corresponds to the flip of a diagonal in a quadrilateral: $ab\longleftrightarrow cd$.

Therefore, we conclude that
\begin{equation}
\label{eq:clustermutation}
A_1'
=\frac{
A_6 A_2
+
A_7 A_4}{A_1},
\end{equation}
and $A_j'=A_j$ for $j\neq 1$.
This is the cluster mutation formula for general cluster coordinate systems formed by regular functions on the moduli spaces of decorated local systems.
\end{proof}

We now use the \bbb operation here to illustrate the above formula. Let us consider a affected lift of face 1 $F_{1_{1}}$ originally glued to lifts of the other four faces $F_{2_{1}}$, $F_{6_{1}}$, $F_{4_{1}}$, and $F_{7_{1}}$; another affected lift of face 1 $F_{1_{2}}$ originally glued to lifts of the other four faces $F_{2_{2}}$, $F_{6_{2}}$, $F_{4_{2}}$, and $F_{7_{2}}$. During the \bbb operation, we cut the edges between $F_{1_{1}}$ and $F_{4_{1}}$ / $F_{7_{1}}$, and between $F_{1_{2}}$ and $F_{4_{2}}$ / $F_{7_{2}}$. Then we glue $F_{1_{1}}$ to $F_{4_{2}}$ / $F_{7_{2}}$; $F_{1_{2}}$ to $F_{4_{1}}$ / $F_{7_{1}}$. After the \bbb operation, two lifts of face 4 / 7 switched places, while the lifts of other faces stay unchanged.

If the original regular functions are the same as Equation~\ref{eq:A1} and Equation~\ref{eq:regfunc}, then the above means that the $v$ in $A_1 = (-1)^r
\bigl\langle
\Omega_m, \rho\wedge v_a\wedge v_b
\bigr\rangle$, that appears in $A_4 = (-1)^r
\bigl\langle
\Omega_m, \rho\wedge v_b\wedge v_c
\bigr\rangle$, namely $v_b$ here, changes to the other $v$ in $A_4$, namely $v_c$ here. Similarly, $v_a$ in $A_1$ changes to $v_d$ which appears in $A_7$ but not in $A_1$. And the rest remains the same. Therefore, we again arrive at the new regular function assigned to face $1'$: $A_1' = (-1)^r
\bigl\langle
\Omega_m, \rho\wedge v_c\wedge v_d
\bigr\rangle$, and the same cluster mutation by the Pl\"ucker relation on $Gr(2,4)$.

\subsection{Cluster Atlases on Moduli Spaces of Decorated Local Systems}

\begin{lemma}[\cite{Goncharov}]
Let $S$ be a decorated surface of genus-$g$ with $n>0$ special points, and assume $2g-2+n>0$. Then the moduli space of decorated $SL_m$-local systems $\mathcal A_{SL_m,S}$ on $S$ has complex dimension
\begin{equation}
\dim_{\mathbb C}\mathcal A_{SL_m,S}
=
-\chi(S)\dim SL_m
=
(2g-2+n)(m^2-1).
\end{equation}
\end{lemma}

\begin{proof}
Using the presentation
\begin{equation}
\pi_1(S)
=
\left\langle
a_1,b_1,\ldots,a_g,B_{G,C,P},c_1,\ldots,c_n
\ \middle|\
\prod_{i=1}^{g}[a_i,b_i]\prod_{j=1}^{n}c_j=1
\right\rangle,
\end{equation}
the $2g+n$ monodromies contribute $(2g+n)\dim SL_m$.
The fundamental-group relation removes one copy of $\dim SL_m$, and quotienting by
global conjugation removes another.

The decoration at each puncture does not contribute an additional independent
$\dim(SL_m/U)$ parameters. A decorated flag is a point of $SL_m/U$, and its stabilizer is $U$. Hence the space of a decorated flag together with a puncture monodromy preserving it has dimension
\begin{equation}
\dim(SL_m/U)+\dim U
=
\dim SL_m.
\end{equation}
Thus the puncture decoration is already incorporated into the above dimension count.

Therefore, $\dim_{\mathbb C}\mathcal A_{SL_m,S}
=
-\chi(S)\dim SL_m
=
(2g-2+n)(m^2-1)$.
\end{proof}

\begin{remark}
Since a cluster coordinate chart on $\mathcal A_{SL_m,S}$
is birational to an algebraic torus of the same dimension, any cluster coordinate system
has $(2g-2+n)(m^2-1)$ independent coordinates.
\end{remark}

\begin{lemma}
Let $S$ be a decorated surface of genus-$g$ with $n>0$ special points satisfying $2g-2+n>0$.
Let $W$ be an ideal web of rank-$m$ on $S$, $m\ge 2$.
Then $W$ has $(2g-2+n)(m^2-1)$ internal faces, and hence there are $(2g-2+n)(m^2-1)$ independent mutable cluster $\mathcal{A}$-coordinates in one cluster $\mathcal{A}$-coordinate system on $\mathcal A_{SL_m,S}$.
\end{lemma}

\begin{proof}
Given an $(m-1)$-ideal triangulation $T$ on $S$ compatible with $W$, the number of ideal triangles in $T$ is $2(2g-2+n)$. The number of faces of $W$ that lie inside each ideal triangle is $\frac{(m-2)(m-1)}{2}$, and the number of faces of $W$ that lie on the boundaries of an ideal triangle is $3(m-1)$.

Each face that lies on a common boundary of two ideal triangles contributes a half face in the face counting of one ideal triangle. Thus there are $\frac{(m-2)(m-1)}{2}+\frac{3(m-1)}{2}=\frac{m^2-1}{2}$ faces in each ideal triangle of $T$. Hence $W$ has $(2g-2+n)(m^2-1)$ internal faces, providing $(2g-2+n)(m^2-1)$ independent mutable cluster $\mathcal{A}$-coordinates in one cluster $\mathcal{A}$-coordinate system on $\mathcal A_{SL_m,S}$.
\end{proof}

\begin{remark}
More general spectral webs may also produce functions $A_F^W$, but the number of faces can
exceed $\dim\mathcal A_{SL_m,S}$.
Such functions therefore need not form a coordinate system.
\end{remark}

\begin{theorem}
\label{thm:general}
Let $S$ be a decorated surface, and let $p:\widetilde S\to S$ be its universal cover. Let $W$ be an ideal web on $S$, Denote by $\widetilde W:=p^{-1}(W)$ its lift to $\widetilde S$. Define $\mathcal F(W)
:=
\{\text{faces of }\widetilde W\}/\pi_1(S)$, and let $\mathcal F_{\mathrm{int}}(W)
:=
\{\text{internal faces of }\widetilde W\}/\pi_1(S)$. In the non-strict ideal case, a face of $W$ means a
$\pi_1(S)$-orbit of faces of the lifted web $\widetilde W$.

\begin{enumerate}
\item[(i)]
An $A_{m-1}$-web $W$ on $S$ provides a cluster $A$-coordinate chart on $\mathcal A_{SL_m,S}$:

The regular functions $\{A_F^W\}_{F\in \mathcal F(W)}$ form cluster $A$-coordinate system on $\mathcal A_{SL_m,S}$. Equivalently, they define a birational coordinate map
\begin{equation}
\mathcal A_{SL_m,S}
\dashrightarrow
(\mathbb G_m)^{\mathcal F(W)},
\qquad
L\mapsto
\bigl(A_F^W(L)\bigr)_{F\in\mathcal F(W)}.
\end{equation}

\item[(ii)]
An $A_{m-1}^*$-web $W$ on $S$ provides a cluster $A$-coordinate chart on $\mathcal A_{SL_m,S}^*$:

The regular functions $\{A_F^W\}_{F\in \mathcal F(W)}$ form cluster $A$-coordinate system on $\mathcal A_{SL_m,S}^*$. The coordinates corresponding to external face-orbits are frozen. Equivalently, they define a birational coordinate map
\begin{equation}
\mathcal A_{SL_m,S}^*
\dashrightarrow
(\mathbb G_m)^{\mathcal F(W)},
\qquad
L\mapsto
\bigl(A_F^W(L)\bigr)_{F\in\mathcal F(W)}.
\end{equation}

\item[(iii)]
The collection of such cluster $A$-coordinate charts gives a cluster atlas:

Any \ttt move $\mu_F:W\to W'$ centred at an internal face $F$ amounts to a cluster mutation at the coordinate $A_F^W$.

The shrink/expand moves do not affect the coordinates.

Any two $A_m$-webs $($or $A_m^*$-webs$)$ are connected by a finite sequence of \ttt and shrink/expand moves.
\end{enumerate}
\end{theorem}

\begin{remark}
When $W$ is a strict ideal web, the lifted face-orbits are identified with
ordinary faces of $W$ downstairs: $\mathcal F(W)
=\{\text{faces of }W\}$.
\end{remark}

Let $\mathcal A_{\mathbb G_m,S}
:=\operatorname{Loc}_{GL(1)}(S)
=\operatorname{Hom}\bigl(\pi_1(S),\mathbb G_m\bigr)
=\operatorname{Hom}\bigl(H_1(S;\mathbb Z),\mathbb G_m\bigr)$ and $\mathcal A_{\mathbb G_m,\Gamma}
:=\operatorname{Loc}_{GL(1)}(\Gamma)
=\operatorname{Hom}\bigl(\pi_1(\Gamma),\mathbb G_m\bigr)
=\operatorname{Hom}\bigl(H_1(\Gamma;\mathbb Z),\mathbb G_m\bigr)$.

\begin{remark}
Goncharov introduces the moduli space
$\mathcal A^{*}_{SL_m,S}$ of twisted $SL_m$-local systems equipped
with an $*$-decoration on a decorated surface $S$, whose decorated flags take values in $GL_m/U$.

There is a scalar action $\mathbb G_m\curvearrowright \mathcal A^*_{SL_m,S}$, given by $\lambda\cdot(\rho,\alpha^*)
=
(\rho,\lambda\alpha^*)$, $\lambda\in\mathbb G_m$,
which rescales the decoration while leaving the underlying
$SL_m$-monodromy unchanged. Thus the quotient
$\mathcal A^{*}_{SL_m,S}/\mathbb G_m$
forgets the common scalar in the decoration, i.e., $\mathcal A^*_{SL_m,S}/\mathbb G_m$ is the part of the $GL_m$-decorated local-system moduli space with trivial determinant local system.

On the other hand, a $GL_m$-local system has an additional determinant local system $\det\rho:\pi_1(S)\to\mathbb G_m$, which induces a determinant map $\det:\mathcal A_{GL_m,S}
\to
\mathcal A_{\mathbb G_m,S}$. The quotient by the scalar action gives the trivial-determinant part. Goncharov describes this as \cite{Goncharov}
\begin{equation}
\mathcal A_{GL_m,S}
=
\mathcal A_{SL_m,S}^*/\mathbb G_m
\times
\mathcal A_{\mathbb G_m,S}.
\end{equation}
\end{remark}

\begin{corollary}
Let $W$ be an $A_{m-1}^{*}$-web with associated bipartite $A_{m-1}^*$-graph $\Gamma$ on a decorated surface $S$ without boundary.
There is a birational isomorphism
\begin{equation}
\mathcal A_{SL_m,S}^*/\mathbb G_m
\dashrightarrow
\operatorname{Hom}\left(
\frac{\mathbb Z[\mathcal F(\Gamma)]}{H_2(S;\mathbb Z)},
\mathbb G_m
\right)
\end{equation}
Hence there is a birational isomorphism
\begin{equation}
\label{eq:AA}
\mathcal A_{GL_m,S}
\dashrightarrow
\operatorname{Hom}\left(
\frac{\mathbb Z[\mathcal F(\Gamma)]}{H_2(S;\mathbb Z)},
\mathbb G_m
\right)
\times
\mathcal A_{\mathbb G_m,S}.
\end{equation}
\end{corollary}

\begin{corollary}
Given an $A_{m-1}^*$-web $W$ with associated bipartite $A_{m-1}^*$-graph $\Gamma$ on a decorated surface $S$ without boundary, there is a birational isomorphism
\begin{equation}
\mathcal A_{GL_m,S}
\dashrightarrow
\mathcal A_{\mathbb G_m,\Gamma}.
\end{equation}
\end{corollary}

\begin{proof}
Let $p:\widetilde S\to S$ be the universal cover, and set $\widetilde\Gamma:=p^{-1}(\Gamma)$. Denote by $\mathcal F(\Gamma)
:=\{\text{faces of }\widetilde\Gamma\}/\pi_1(S)$ the set of deck-transformation orbits of faces of the lifted graph.

Recall the splitting short exact sequence introduced in Section 3.1
\begin{equation}
0
\to
\frac{\mathbb Z[\mathcal F(\Gamma)]}{H_2(S;\mathbb Z)}
\to
H_1(\Gamma;\mathbb Z)
\to
H_1(S;\mathbb Z)
\to
0.
\end{equation}
After choosing a splitting, one obtains an algebraic-torus isomorphism
\begin{equation}
\mathcal A_{\mathbb G_m,\Gamma}
\cong
\operatorname{Hom}
\left(
\frac{\mathbb Z[\mathcal F(\Gamma)]}{H_2(S;\mathbb Z)},
\mathbb G_m
\right)
\times
\mathcal A_{\mathbb G_m,S}.
\end{equation}

Combining this with the
previous birational isomorphism of $\mathcal A_{GL_m,S}$ in Equation~\ref{eq:AA} gives a birational isomorphism
\begin{equation}
\mathcal A_{GL_m,S}
\dashrightarrow
\mathcal A_{\mathbb G_m,\Gamma}.
\end{equation}
\end{proof}

\newpage
\mysection{Mutations on Grothendieck Groups Assigned to Quivers with Potential}

In this section, we will study how mutations, especially \ttt moves, on quivers with potential induce isomorphisms between corresponding $K_0$ groups. We will also discuss that tropicalizing the cluster $\mathcal A$-mutation relation gives two linear pieces corresponding to the transformations associated with the two signed $K_0$-dual basis mutations.

We will see at the end of this section, that under a \ttt move on the quiver assigned to a bipartite graph on a decorated surface $S$ without boundary, the induced isomorphism between $K_0$ groups coincides with the isomorphism between the quiver lattices induced by the ura bijection, which amounts to the \ttt move at the level of the spectral cover. The relevant theorems are Theorem~\ref{thm:ttt} and Theorem~\ref{thm:7com}.

We will reach to the conclusion in Theroem~\ref{thm:GM} at the end of this section, that the automorphisms on the $SL_m$ Hitchin fiber lattices given by Gauss--Manin monodromies agree with the induced mutations on the $K_0$ groups of the quiver with potential assigned to certain generic $A_{m-1}$-web on a smooth complex genus-$g$ curve $C$ with a set of $n$ special points satisfying $2g-2+n>0$.

\subsection{Mutations on Quivers with Potential}

\begin{definition}[\cite{DWZ}]
Let $Q$ be a finite quiver, with $Q_l$ its set of paths of length $l$. The \emph{path algebra} $kQ$ of $Q$ is the $k$-vector space spanned by all oriented paths in $Q$,
with multiplication given by concatenation of paths whenever possible and zero otherwise:
\begin{equation}
kQ
:=
\bigoplus_{n\geq 2}
\bigoplus_{i_1,\ldots,i_n\in I}
A(i_1,i_2)\otimes \cdots \otimes A(i_{n-1},i_n).
\end{equation}

Equivalently, one may write
\begin{equation}
kQ=\bigoplus_{\ell\geq 0} kQ_\ell.
\end{equation}
An element of $kQ$ is a finite formal sum $\sum_{p} c_p p$, where $p$ runs over paths in $Q$, and $c_p\in k$.
\end{definition}

\begin{definition}[\cite{DWZ}]
Let $Q$ be a finite quiver, with $Q_l$ its set of paths of length $l$. Let
$kQ$ be its path algebra, and let $\mathfrak m
\subset kQ$ be the ideal generated by all arrows of $Q$. Equivalently, $\mathfrak m$ is spanned by
all paths of positive length. The \emph{completed path algebra} $\widehat{kQ}$ of $Q$ is the $\mathfrak m$-adic completion
\begin{equation}
\widehat{kQ}
:=
\varprojlim_{n} kQ/\mathfrak m^n.
\end{equation}

Equivalently, one may write
\begin{equation}
\widehat{kQ}
=
\prod_{\ell\geq 0} kQ_\ell,
\end{equation}
An element of $\widehat{kQ}$ is a possibly infinite formal sum $\sum_{p} c_p p$, where $p$ runs over paths in $Q$, and $c_p\in k$.
\end{definition}

$\operatorname{HH}_0(kQ)$ is the cyclic envelope of $kQ$:
\begin{equation}
\operatorname{HH}_0(kQ)
:=
\bigoplus_{n\geq 2}
\left(
\bigoplus_{i_1,\ldots,i_n\in I}
A(i_1,i_2)\otimes \cdots \otimes A(i_n,i_1)
\right)_{\mathbb Z/n\mathbb Z}.
\end{equation}
Here the subscript $\mathbb Z/n\mathbb Z$ denotes the coinvariants of the
cyclic shift \cite{Goncharov}.

\begin{definition}[\cite{Goncharov}]
A potential is a linear functional
\begin{equation}
\mathcal P:\operatorname{HH}_0(kQ)\to \mathbf k.
\end{equation}
\end{definition}

The vector space $\operatorname{HH}_0(kQ)$ has a natural basis given by oriented cycles of the quiver. So a potential can be equivalently described as a formal linear combination of oriented cycles in $Q$, considered
up to cyclic rotation of each cycle. We write it as
\begin{equation}
W=\sum_{\alpha} c_{\alpha} C_{\alpha},
\end{equation}
where each $C_{\alpha}$ is an oriented cycle in $Q$. Thus a quiver with potential is a pair $(Q,W)$.

\begin{remark}
The potential is not needed for ordinary cluster mutation of variables, but it is essential for
constructing Jacobian algebras, Ginzburg dg algebras, and $3$-Calabi--Yau categories.
\end{remark}

\medskip
A bipartite graph $\Gamma$ on a decorated surface $S$ gives a quiver with potential $(Q_\Gamma,P_\Gamma)$. The quiver $Q_\Gamma$ is the dual graph of $\Gamma$ with edges oriented such that the $\bullet$-vertex of each corresponding edge of $\Gamma$ is on the right. The oriented edges of $Q_\Gamma$ are called quiver arrows \cite{Goncharov}.

Each $\bullet$-vertex $b$ of a bipartite graph $\Gamma$ gives a simple cyclic path $C_b$ on $Q_\Gamma$: the arrows go clockwise around the vertex $b$. Each $\circ$-vertex $r$ of $\Gamma$ gives a simple
cyclic path $C_r$ on $Q_\Gamma$: the arrows go anticlockwise around the
$\circ$-vertex $r$. The potential is given by \cite{Goncharov}
\begin{equation}
P_\Gamma
:=
\sum_{b:\,\bullet\text{-vertices of }\Gamma} C_b
-
\sum_{r:\,\circ\text{-vertices of }\Gamma} C_r
\in
\operatorname{HH}_0(\mathcal P(Q_\Gamma)).
\end{equation}

\begin{lemma}[\cite{Goncharov}]
A \ttt move $\Gamma \to \Gamma'$ centred at an internal face $G$ gives a mutation of quivers with potential $(Q_\Gamma,P_\Gamma) \to (Q_{\Gamma'},P_{\Gamma'})$ in the direction of the vector $e_G$.
\end{lemma}

For the non-frozen subquiver
$Q=
\bigl(
\Lambda_{\Gamma}^{\circ},
\{e_G\},
\operatorname{Res}_{\Lambda_{\Gamma}^{\circ}}(*,*)_{\Gamma}
\bigr),$
write $h(a)$ and $t(a)$ for the head and tail of an arrow $a$.
We use the convention that $ab$ means that $b$ is followed by $a$.

Recall that the DWZ mutation of a quiver with potential $(Q,W)$ at a
vertex $v$ is obtained by first:
\begin{enumerate}
    \item reversing all arrows incident to $v$;
    \item introducing an arrow $[qp]$ for each length-two path
    $p:i\to v$, $q:v\to j$;
    \item replacing each occurrence of $qp$ in $W$ by $[qp]$ and
    adding the cubic term $[qp]p^*q^*$.
\end{enumerate}
The mutation $\mu_v(Q,W)$ is the reduced part of this premutation,
obtained by removing the trivial quadratic $2$-cycle summand
\cite{DWZ}. Quivers with potential related by an automorphism of the
completed path algebra fixing the vertices are called
\emph{right-equivalent}.

Let $\Gamma\to\Gamma'$ be a two-by-two move at $v$. In the local notation of \cite{Goncharov}, the potential of $Q_\Gamma$ is
\begin{equation}
W_\Gamma
=xab+Ax+ycd+yC+Bad+Dcb+W_0,
\end{equation}
where $W_0$ contains none of the arrows $a,b,c,d,x,y$.

Applying the DWZ mutation procedure at $v$, before reduction, gives
\begin{equation}
\widetilde W
=x[ab]+[cd]y+Ax+B[ad]+yC+D[cb]
+[cd]d^*c^*
+[ad]d^*a^*
+[cb]b^*c^*
+[ab]b^*a^*
+W_0.
\end{equation}

The right-equivalence
$[ab]\mapsto[ab]-A,$
$x\mapsto x-b^*a^*,$
$[cd]\mapsto[cd]-C,$
$y\mapsto y-d^*c^*$
gives
\begin{equation}
\widetilde W
\sim_{\mathrm{r.e.}}
x[ab]+[cd]y
+B[ad]+D[cb]
-Cd^*c^*-Ab^*a^*
+[ad]d^*a^*
+[cb]b^*c^*
+W_0.
\end{equation}
The terms
$x[ab]+[cd]y$
form a trivial quadratic $2$-cycle summand. Splitting off this
summand and applying
$b^*\mapsto-b^*$, $d^*\mapsto-d^*$
gives the reduced potential
\begin{equation}
W_{\Gamma'}
=B[ad]+D[cb]
+Cd^*c^*+Ab^*a^*
-[ad]d^*a^*
-[cb]b^*c^*
+W_0.
\end{equation}

This is the canonical potential of the quiver
$Q_{\Gamma'}$ assigned to the bipartite graph obtained from
$\Gamma$ by the two-by-two move \cite{Goncharov}. Hence
$\mu_v(Q_\Gamma,W_\Gamma)
\simeq
(Q_{\Gamma'},W_{\Gamma'})$
up to right-equivalence.

\subsection{Ginzburg Algebras and Derived Categories}
Let $Q$ be a finite quiver, with $Q_l$ its set of paths of length $l$. Let $a\in Q_1$ be an arrow. The cyclic derivative of a cycle $C=a_1a_2\cdots a_r$ with respect to $a$ is defined by cutting the cycle at every occurrence of $a$:
\begin{equation}
\frac{\partial C}{\partial a}
=
\sum_{a_s=a}
a_{s+1}a_{s+2}\cdots a_r a_1\cdots a_{s-1}.
\end{equation}
Extending linearly gives $\frac{\partial W}{\partial a}$.

\begin{definition}[\cite{DWZ}]
The \emph{Jacobian algebra} of a quiver with potential $(Q,W)$ is
\begin{equation}
J(Q,W)
=
\widehat{kQ}\big/\overline{\left\langle
\frac{\partial W}{\partial a}
\;\middle|\;
\forall a\in Q_1
\right\rangle},
\end{equation}
where $\widehat{kQ}$ is the completed path algebra and the bar denotes
closure in the $\mathfrak m$-adic topology.
\end{definition}

\begin{definition}[\cite{Ginzburg}]
Let $(Q,W)$ be a quiver with potential. The \emph{Ginzburg dg algebra} $\Gamma(Q,W)$ is generated as a differential graded algebra by the following data:
\begin{equation}
\begin{array}{c|c|c|c}
\text{generator} & \text{meaning} & \text{degree} & \text{differential} \\ \hline
e_i & \text{idempotent at vertex } i & 0 & d(e_i)=0 \\[2pt]
a:i\to j & \text{original arrow of }Q & 0 & d(a)=0 \\[2pt]
a^\ast:j\to i & \text{opposite arrow} & -1 & d(a^\ast)=\frac{\partial W}{\partial a} \\[2pt]
t_i:i\to i & \text{loop at vertex }i & -2 & d(t_i)
=\sum_{a:j\to i} aa^*
-\sum_{a:i\to j} a^*a.
\end{array}
\end{equation}
\end{definition}

The zeroth cohomology of the Ginzburg dg algebra is the Jacobian algebra:
\begin{equation}
H^0(\Gamma(Q,W))\cong J(Q,W).
\end{equation}

\medskip
Let $\mathcal D(\Gamma(Q,W))$ be the derived category of right dg $\Gamma(Q,W)$-modules. For a dg $\Gamma$-module $M$, write $H^\bullet(M)=\bigoplus_{p\in\mathbb Z}H^p(M)$ for its total cohomology.

\begin{definition}[\cite{KellerYang}]
The \emph{finite-dimensional derived category} of $\mathcal D(\Gamma(Q,W))$ is
\begin{equation}
\mathcal D_{\mathrm{fd}}(\Gamma(Q,W))
=
\left\{
M\in \mathcal D(\Gamma(Q,W)):
\dim_k H^\bullet(M)<\infty
\right\}.
\end{equation}
Denote $\mathcal D_{\mathrm{fd}}(\Gamma(Q,W))$ as $\mathcal D(\Gamma)$.
\end{definition}

\begin{itemize}
    \item The category $\mathcal D(\Gamma)$ is \emph{triangulated}. A distinguished triangle is the derived-category analogue of a short exact sequence. For objects $X,Y\in\mathcal D(\Gamma)$,
    \begin{equation}
    \operatorname{Ext}^i_{\mathcal D(\Gamma)}(X,Y)
    :=
    \operatorname{Hom}_{\mathcal D(\Gamma)}(X,Y[i]).
    \end{equation}

    \item The category $\mathcal D(\Gamma)$ is a \emph{$3D$-Calabi--Yau category}. This means that For objects $X,Y\in\mathcal D(\Gamma)$, there is a functorial duality
    \begin{equation}
    \operatorname{Hom}_{\mathcal D}(X,Y)
    \cong
    \operatorname{Hom}_{\mathcal D}(Y,X[3])^\vee.
    \end{equation}
    More generally, for every $i\in\mathbb Z$, there are bifunctorial isomorphisms
    \begin{equation}
    \operatorname{Ext}^i_{\mathcal D(\Gamma)}(X,Y)
    \cong
    \operatorname{Ext}^{3-i}_{\mathcal D(\Gamma)}(Y,X)^\vee.
    \end{equation}
\end{itemize}

\begin{definition}[\cite{KellerYang}]
The \emph{standard heart} of $\mathcal D_{\mathrm{fd}}(\Gamma(Q,W))$ is
\begin{equation}
\mathcal A(Q,W)
=
\left\{
M\in\mathcal D_{\mathrm{fd}}(\Gamma(Q,W))
\;\middle|\;
H^p(M)=0 \text{ for } p\neq0
\right\}.
\end{equation}
\end{definition}

The category $\mathcal A(Q,W)$ is abelian. It is the extra choice that extracts a cluster seed from the 3CY category. The zeroth
cohomology functor induces an equivalence
\begin{equation}
H^0:
\mathcal A(Q,W)
\xrightarrow{\sim}
\operatorname{mod}_{\mathrm{fd}}H^0(\Gamma(Q,W))
\simeq
\operatorname{mod}_{\mathrm{fd}}J(Q,W).
\end{equation}

The simple objects of the standard heart determine the vertices of the
associated exchange quiver, and their $\operatorname{Ext}^1$-groups
determine its arrows.

Let $Q_0=\{1,\dots,n\}$ be the vertex set of $Q$. For each vertex $i$, let $e_i\in J(Q,W)$ be the primitive idempotent represented by the stationary path at $i$. These satisfy $e_i^2=e_i$, $e_ie_j=0$ $(i\neq j)$, $1_A=\sum_{i\in Q_0}e_i$.

Let $\mathfrak m\subset J(Q,W)$ be the closed ideal generated by the arrows of $Q$. The indecomposable projective right module associated with $i$ is $P_i:=e_iJ(Q,W)$, and the associated simple modules $S_i
:=\frac{e_iJ(Q,W)}{e_i\mathfrak m}$ is the representation of $Q$ supported only at vertex $i$:
\begin{equation}
(S_i)_j=
\begin{cases}
k, & j=i,\\
0, & j\neq i,
\end{cases}
\end{equation}
and every arrow acts by zero.

The simple $J(Q,W)$-modules $\{S_1,\dots,S_n\}$ form a complete set of pairwise non-isomorphic simple objects of the standard heart $\mathcal A(Q,W)$. Hence, there is a bijection between $Q_0$ and simple objects of $\mathcal A(Q,W)$. The heart therefore provides a distinguished collection of simple objects inside the triangulated category:
\begin{equation}
\mathcal D(\Gamma)
\supset
\mathcal A(Q,W)
=
\langle S_1,\dots,S_n\rangle_{\mathrm{ext}},
\end{equation}
where $\langle S_1,\dots,S_n\rangle_{\mathrm{ext}}$ denotes the smallest full subcategory containing the $S_i$ and closed under finite extensions \cite{KellerYang}.

\subsection{Tropical Mutations on the Grothendieck Groups}

\subsubsection{Triangulated Equivalences and Tropical Mutations}

\begin{definition}[\cite{KellerYang}]
Let $\mathcal D$ be a triangulated category. The \emph{Grothendieck group} $K_0(\mathcal D)$ is
the abelian group generated by $[X]$, one for each object $X\in \mathcal D$,
modulo the relations $[Y]=[X]+[Z]$ whenever there is a distinguished triangle $X\to Y\to Z\to X[1]$.
\end{definition}

The classes of the vertex simple modules give a natural basis of the Grothendieck group:
\begin{equation}
K_0(\mathcal D(\Gamma))
\cong
\bigoplus_{i\in Q_0}\mathbb Z[S_i].
\end{equation}
Therefore, the standard heart determines the quiver seed lattice $\Lambda_\Gamma
\cong
K_0\bigl(\mathcal D_{\mathrm{fd}}(\Gamma)\bigr)$ and basis $e_i=[S_i]$.

\begin{lemma}
Inside the standard heart, $\operatorname{Ext}^1$ has its usual abelian meaning: $\operatorname{Ext}^1_{\mathcal A(Q,W)}(S_i,S_j)$ classifies short
exact sequences $0\to S_j\to E\to S_i\to 0$. With our convention of right $J(Q,W)$-modules and the path convention
that $ab$ denotes the path in which $b$ is followed by $a$, one has
\begin{equation}
\dim_k
\operatorname{Ext}^1_{\mathcal A(Q,W)}(S_j,S_i)
=
\#\{\text{arrows }i\to j\}.
\end{equation}
\end{lemma}

Thus the quiver is recovered from the simple objects and their $\operatorname{Ext}^1$-groups.

\begin{proof}
Let $P_i=e_i A$ be the indecomposable projective right $A$-module corresponding to the
vertex $i$. With our convention that $ab$ denotes the path in which
$b$ is followed by $a$, an arrow
$a:j\to i$
lies in $e_i A e_j$. Hence
\begin{equation}
\operatorname{rad}(P_i)/\operatorname{rad}^2(P_i)
\cong
\bigoplus_{a:j\to i} S_j.
\end{equation}
Therefore the beginning of a projective resolution of $S_i$ is
$\bigoplus_{a:j\to i}P_j
\to
P_i
\to
S_i
\to0$.
Applying $\operatorname{Hom}_A(-,S_j)$ gives
$\dim_k\operatorname{Ext}^1_A(S_j,S_i)
=\#\{a:i\to j\}$.
\end{proof}

\medskip
\begin{definition}[\cite{KellerYang}]
Let $\mathcal C$ be a $k$-linear triangulated category. Assume that for
all $X,Y\in\mathcal C$,
$\sum_{i\in\mathbb Z}
\dim_k\operatorname{Hom}_{\mathcal C}(X,Y[i])<\infty$.
The \emph{Euler form} on $K_0(\mathcal C)$ is
\begin{equation}
\chi(X,Y)
=
\sum_{i\in\mathbb Z}
(-1)^i
\dim_k\operatorname{Hom}_{\mathcal C}(X,Y[i]).
\end{equation}
\end{definition}

\begin{remark}
If $\mathcal C$ is 3-Calabi--Yau, this form is skew-symmetric: $\chi(X,Y)=-\chi(Y,X)$. For $K_0(\mathcal D(\Gamma))$, on the basis of simple objects, the exchange matrix is $\varepsilon_{ij}
=\#\{i\to j\}
-\#\{j\to i\}
=\chi(S_i,S_j)$.
\end{remark}

\begin{definition}
Let $\mathcal D_1,\mathcal D_2$ be triangulated categories. An \emph{exact equivalence} is a functor $F:\mathcal D_1\to \mathcal D_2$, such that:
\begin{enumerate}
    \item $F$ is an equivalence of categories,
    \item $F$ is equipped with a natural isomorphism
    $\xi_X:F(X[1])\xrightarrow{\sim}F(X)[1]$,
    \item $F$ sends every distinguished triangle
    $X\to Y\to Z\to X[1]$ in $\mathcal D_1$ to a distinguished triangle
    $F(X)\to F(Y)\to F(Z)
    \to F(X)[1]$ in
    $\mathcal D_2$,
    where the last morphism is composed with $\xi_X$.
\end{enumerate}
\end{definition}

\begin{lemma}
Every exact equivalence induces an isomorphism on Grothendieck groups: $F_\ast:
K_0(\mathcal D_1)
\overset{\sim}{\to}
K_0(\mathcal D_2)$, defined by $F_\ast([X])=[F(X)]$.
This is well-defined because exact functors preserve distinguished triangles.
\end{lemma}

Let $(Q,W)$ be a quiver with potential such that $Q$ has no
loops or oriented $2$-cycles, and let $k\in Q_0$ be a vertex. Mutation
at $k$ gives a new quiver with potential $(Q',W')=\mu_k(Q,W)$. At the level of cluster combinatorics, this is quiver mutation. At the categorical level, it is
realized by an equivalence of 3CY categories.

Starting with $\mathcal A(Q,W)\subset \mathcal D(\Gamma)$, mutation at a simple object $S_k$ produces a tilted heart $\mu_k(\mathcal A(Q,W))$, while the ambient 3CY category $\mathcal D(\Gamma)$ remains equivalent. Write $S_i\in\mathcal A(Q,W)$, $S_i'\in\mathcal A(Q',W')$ for the simple objects of the original and mutated hearts. The classes of $S_i'$ give the mutated basis of $K_0$: $[S_i']\in K_0(\mathcal D(\Gamma'))$. Thus a cluster mutation amounts to a tilting of hearts \cite{BS15}.

\medskip
\begin{theorem}[\cite{KellerYang}]
The Keller--Yang theorem gives triangulated equivalences
\begin{equation}
\Phi^\pm:
\mathcal D(\Gamma')
\overset{\sim}{\to}
\mathcal D(\Gamma).
\end{equation}
$\Phi^+$ identifies the standard heart of the mutated quiver with
the right tilt of $\mathcal A$, while $\Phi^-$ identifies it with the
left tilt:
\begin{equation}
\Phi^\pm(\mathcal A(Q',W'))
=
\mu^\pm_{k}(\mathcal A(Q,W)).
\end{equation}
\end{theorem}

For the right tilt, one may choose conventions such that
\begin{equation}
\Phi^+(S_i')
=
\begin{cases}
S_k[-1], & i=k,\\[4pt]
\operatorname{Cone}\!\left(
\operatorname{Ext}^1(S_k,S_i)\otimes S_k[-1]
\xrightarrow{\mathrm{ev}}
S_i
\right), & i\neq k.
\end{cases}
\end{equation}
Here $\mathrm{ev}$ is the universal evaluation morphism. In particular, if $\operatorname{Ext}^1(S_k,S_i)=0$, then $\Phi^+(S_i')\cong S_i$.

Dually, for the left tilt,
\begin{equation}
\Phi^-(S_i')
=
\begin{cases}
S_k[1], & i=k,\\[4pt]
\operatorname{Cone}\!\left(
S_i
\xrightarrow{\mathrm{coev}}
\operatorname{Ext}^1(S_i,S_k)^\vee\otimes S_k[1]
\right)[-1], & i\neq k.
\end{cases}
\end{equation}

Thus mutation replaces the simple $S_k$ by a shift and replaces each
neighbouring simple by a universal extension or coextension involving
$S_k$.

\medskip
Let $e_i=[S_i]\in K_0(\mathcal D(\Gamma))$ be the basis given by the simple objects of the original heart. Let $e_i'=[S_i']\in K_0(\mathcal D(Q',W'))$ be the basis given by the simple objects of the mutated heart.

\begin{lemma}[\cite{BS15}]
\label{le:k0}
The equivalences $\Phi^\pm:
\mathcal D(\Gamma')
\to
\mathcal D(\Gamma)$ induces isomorphisms
\begin{equation}
(\Phi^\pm)_\ast:
K_0(\mathcal D(\Gamma'))
\overset{\sim}{\to}
K_0(\mathcal D(\Gamma)).
\end{equation}
The two signed mutations
of the basis of $K_0$ are given by:
\begin{equation}
\begin{aligned}
(\Phi^-)_*:\quad
e_i'
&\longmapsto
\begin{cases}
-e_k, & i=k,\\[2pt]
e_i+[-\varepsilon_{ik}]_+e_k, & i\neq k.
\end{cases}
\\[6pt]
(\Phi^+)_*:\quad
e_i'
&\longmapsto
\begin{cases}
-e_k, & i=k,\\[2pt]
e_i+[\varepsilon_{ik}]_+e_k, & i\neq k,
\end{cases}
\end{aligned}
\end{equation}
\end{lemma}

Thus mutation changes the basis of $K_0$ by a piecewise-linear rule. Therefore, quiver mutation is categorified by derived equivalence, and tropical cluster mutation is the induced change of basis on $K_0$.

\subsubsection{Dual Lattices and Cluster \texorpdfstring{$\mathcal{A}$}{A}-Mutations}

let $K_0(\mathcal{D}_\Gamma):=K_0(\mathcal D(\Gamma(Q,W)))=\bigoplus_i\mathbb Ze_i,$ $e_i:=[S_i].$ $K_0(\mathcal{D}_\Gamma)$
carries the skew Euler form
$\varepsilon_{ij}:=\langle e_i,e_j\rangle.$
The positive and negative parts of $\varepsilon_{ik}
=
\#(i\to k)-\#(k\to i)$ record respectively the arrows entering
and leaving $k$:
$[\varepsilon_{ik}]_+=\#(i\to k),$
$[-\varepsilon_{ik}]_+=\#(k\to i).$
Let $K_0(\mathcal{D}_\Gamma)^\vee:=\operatorname{Hom}(K_0(\mathcal{D}_\Gamma),\mathbb Z)$ be the dual lattice with dual basis $f_i$: $f_i(e_j)=\delta_{ij}.$

\medskip
Let $\mathbf A=\{A_i\}_{i\in I}$ be a cluster $\mathcal A$-seed with quiver $Q$, and let $B=(\varepsilon_{ij})$ be the exchange matrix of $\mathbf A=\{A_i\}_{i\in I}$. A cluster $\mathcal A$-mutation at a mutable vertex $k$ changes the $\mathcal A$-coordinates by $A_j'=A_j$ for $j\neq k$, and \cite{SS25}
\begin{equation}
A_k'
=
\frac{\prod_i A_i^{[-\varepsilon_{ik}]_+}}{A_k}
+\frac{\prod_i A_i^{[\varepsilon_{ik}]_+}}{A_k}
=
\frac{\prod_{k\to i}A_i}{A_k}
+\frac{\prod_{i\to k}A_i}{A_k}, \qquad i\neq k.
\end{equation}

\medskip
For the mutation $K_0(\mathcal{D}_\Gamma')\to K_0(\mathcal{D}_\Gamma)$ of the $K_0$-basis at $k$, one convention is
\begin{equation}
\begin{array}{ll}
\widetilde e_k\mapsto-e_k,\\
\widetilde e_i
\mapsto
e_i+[-\varepsilon_{ik}]_+e_k,\quad i\neq k.
\end{array}
\end{equation}
Hence only vertices with arrows $k\to i$ contribute to this basis change.

The corresponding mutation on the dual basis in the same direction is
the inverse-dual map $K_0(\mathcal{D}_\Gamma')^\vee\to K_0(\mathcal{D}_\Gamma)^\vee$
\begin{equation}
\begin{array}{ll}
\widetilde f_k\mapsto-f_k+\sum_i[-\varepsilon_{ik}]_+f_i,\\[1mm]
\widetilde f_j\mapsto f_j,\qquad j\neq k.
\end{array}
\end{equation}

If $A^{f}$ denotes the monomial with exponent vector $f$, then
\begin{equation}
A^{\widetilde f_k}
=A_k^{-1}\prod_iA_i^{[-\varepsilon_{ik}]_+}
=
\frac{\prod_{k\to i}A_i}{A_k}.
\end{equation}
Thus the outgoing-arrow monomial in the cluster $\mathcal A$-mutation relation is the monomial associated with the $k$-th element of the dual of the mutated $K_0$-basis vector with $[-\varepsilon_{ik}]_+$ convention.

\medskip
The mutation $K_0(\mathcal{D}_\Gamma')\to K_0(\mathcal{D}_\Gamma)$ of the $K_0$-basis at $k$ with the opposite convention uses $[\varepsilon_{ik}]_+$:
\begin{equation}
\begin{array}{ll}
e_k'\mapsto -e_k,\\
e_i'\mapsto e_i+[\varepsilon_{ik}]_+ e_k,\quad i\neq k.
\end{array}
\end{equation}
The lattice and skew bilinear form themselves remain fixed; only the distinguished basis changes.

Similarly, the corresponding mutation on the dual basis in the same direction under this convention is
the inverse-dual map $K_0(\mathcal{D}_\Gamma')^\vee\to K_0(\mathcal{D}_\Gamma)^\vee$
\begin{equation}
\begin{array}{ll}
f_k'\mapsto -f_k+\sum_i[\varepsilon_{ik}]_+f_i,\\[1mm]
f_j'\mapsto f_j,\qquad j\neq k.
\end{array}
\end{equation}
Hence the incoming-arrow monomial in the cluster $\mathcal A$-mutation relation is the monomial associated with the dual of the mutated $K_0$-basis vector with $[\varepsilon_{ik}]_+$ convention:
\begin{equation}
A^{f_k'}
=
A_k^{-1}\prod_iA_i^{[\varepsilon_{ik}]_+}
=
\frac{\prod_{i\to k}A_i}{A_k}.
\end{equation}

\begin{proposition}[\cite{SS25}]
Let $\{e_i'\}$ and $\{\widetilde e_i\}$ be the two signed mutations
of the $K_0$-basis at $k$, determined respectively by the positive and
negative parts of the Euler pairing, and let
$\{f_i'\}$ and $\{\widetilde f_i\}$ be their dual bases.
Then the cluster $\mathcal A$-mutation relation at $k$ can be written as
\begin{equation}
\label{eq:A}
A_k'
=
A^{f_k'}+A^{\widetilde f_k}.
\end{equation}
The two monomials in the cluster $\mathcal A$-mutation relation
are precisely the monomials whose exponent vectors are the $k$-th
elements of the dual bases associated with the two signed
$K_0$-basis mutations.
\end{proposition}

\begin{proposition}[\cite{SS25}]
Let $a_i=\operatorname{trop}(A_i)$.
Using the min-plus tropical convention,
$\operatorname{trop}(XY)
=\operatorname{trop}(X)+\operatorname{trop}(Y)$,
$\operatorname{trop}(X+Y)
=\min\{\operatorname{trop}(X),\operatorname{trop}(Y)\}$.
Tropicalizing the cluster $\mathcal A$-mutation relation $(\ref{eq:A})$ gives $a_j'=a_j$ for $j\neq k$ and
\begin{equation}
\begin{aligned}
a_k'
&=
-a_k+
\min\left\{
\sum_i[-\varepsilon_{ik}]_+a_i,\,
\sum_i[\varepsilon_{ik}]_+a_i
\right\}
\\
&=
\begin{cases}
-a_k+\displaystyle\sum_{k\to i}a_i
=
-a_k+\displaystyle\sum_i[-\varepsilon_{ik}]_+a_i,
&
\text{if the outgoing linear form is smaller},
\\[2mm]
-a_k+\displaystyle\sum_{i\to k}a_i
=
-a_k+\displaystyle\sum_i[\varepsilon_{ik}]_+a_i,
&
\text{if the incoming linear form is smaller}.
\end{cases}
\end{aligned}
\end{equation}
The two linear pieces of the tropical mutation are the transformations associated with the two signed $K_0$-dual basis mutations,
determined respectively by the negative and positive parts of the Euler form.
\end{proposition}

\begin{remark}
A signed $K_0$-basis mutation uses one of the two linear conventions,
whereas the cluster $\mathcal A$-mutation relation contains both
corresponding monomials simultaneously. After tropicalization, their
sum becomes a minimum (or a maximum under the opposite tropical
convention). Hence the cluster $\mathcal A$-mutation becomes piecewise linear.
\end{remark}

The cluster $\mathcal A$-mutation relation under a \ttt move at face 1 in Figure~\ref{fig:localquiver} is given by Equation~\ref{eq:clustermutation}:
$A_j'=A_j$ for $j\neq 1$ and $A_1 A_1'=A_6 A_2+A_7 A_4$. Tropicalization gives
\begin{equation}
\begin{aligned}
a_j'&=a_j, \quad j\neq 1,\\[2pt]
a_1'
&=
\begin{cases}
-a_1+\displaystyle\sum_{1\to i}a_i
=-a_1+a_2+a_6,
&
\text{if the outgoing linear form is smaller},
\\[2mm]
-a_1+\displaystyle\sum_{i\to 1}a_i
=-a_1+a_4+a_7,
&
\text{if the incoming linear form is smaller}.
\end{cases}
\end{aligned}
\end{equation}
The two linear pieces are the transformations of the two signed $K_0$-dual basis mutations. Hence the corresponding two signed $K_0$-basis mutations are
\begin{equation}
\label{eq:trop}
\begin{aligned}
e_1'&\mapsto -e_1, \\[2pt]
e_i'
&\mapsto
\begin{cases}
e_i+\displaystyle\sum_{1\to i}e_1,
& i\neq 1,
\text{ left tilt convention},
\\[2mm]
e_i+\displaystyle\sum_{i\to 1}e_1,
& i\neq 1,
\text{ right tilt convention}.
\end{cases}
\end{aligned}
\end{equation}
This matches Equation~\ref{eq:basis} and Equation~\ref{eq:basis1}.

\subsection{Tropical Mutations on \texorpdfstring{$\mathrm{K_0(Q,W)}$}{} Induced by Two-by-two Moves}
\label{sec:7.4}
Let $S$ be a genus-$g$ decorated surface without boundary with $n$ special points satisfying $2g-2+n>0$.

For an $A_m^*$-bipartite graph $\Gamma^*$ on $S$, the full quiver assigned to $\Gamma^*$ is
$Q_{\Gamma^*}
:=
\bigl(\Lambda_{\Gamma^*},\{e_F\},(*,*)_{\Gamma^*}\bigr)$.
Its non-frozen part is
$Q_{\Gamma^*}^{\circ}
:=
\bigl(
\Lambda_{\Gamma^*}^{\circ},
\{e_G\},
\operatorname{Res}_{\Lambda_{\Gamma^*}^{\circ}}(*,*)_{\Gamma^*}
\bigr)$,
where the non-frozen basic vectors correspond to the internal faces
$G$ of $\Gamma$ \cite{Goncharov}.
For an $A_m$-bipartite graph $\Gamma$ on $S$, the full quiver assigned to $\Gamma$ is $Q_{\Gamma}^{\circ}
:=\bigl(\Lambda_{\Gamma}^{\circ},\{e_F\},(*,*)_{\Gamma}\bigr)$.

\subsubsection{Quiver Lattices and Bases of $\mathrm{K}_0(\mathcal{D}_{\Gamma})$ Groups}

Now let $\Gamma$ be an $A_m$ or $A_m^*$-bipartite graph on $S$. The full quiver $Q_{\Gamma}$ is equipped with the canonical potential
\begin{equation}
W_\Gamma
:=\sum_{b:\,\bullet\text{-vertices of }\Gamma} C_b
-\sum_{r:\,\circ\text{-vertices of }\Gamma} C_r,
\end{equation}
where each $C_b$ is a clockwise cycle in $Q_{\Gamma}$ around a $\bullet$-vertex $b$ in $\Gamma$ and each $C_r$ is a anticlockwise cycle in $Q_{\Gamma}$ around a $\circ$-vertex $r$ in $\Gamma$ \cite{Goncharov}. If $\Gamma$ has external faces, then the non-frozen subquiver $Q_{\Gamma}^{\circ}$ is equipped with the same potential $W_{\Gamma}$, as there is no arrow between external vertices with other vertices in $Q_{\Gamma}$, according to Proposition~\ref{prop:3.3}.

The quivers with potential $(Q_\Gamma,W_\Gamma)$ and $(Q_\Gamma^{\circ},W_\Gamma)$ determine 3CY categories
$\mathcal D_\Gamma
:=
\mathcal D_{\mathrm{fd}}
\bigl(\Gamma(Q_\Gamma,W_\Gamma)\bigr)$ and $\mathcal D_\Gamma^{\circ}
:=\mathcal D_{\mathrm{fd}}
\bigl(\Gamma(Q_\Gamma^{\circ},W_\Gamma)\bigr)$ respectively. Denote by $\mathcal F(\Gamma)$ the set of faces of $\Gamma$, and by $\mathcal F^\circ(\Gamma)$ the set of internal faces of $\Gamma$. Set $\widetilde\Lambda_\Gamma
:=\mathbb Z[\mathcal F(\Gamma)]$.

\begin{proposition}
\label{prop:quiver-K0-lattice}
Let $\Gamma$ be an $A_m^*$-bipartite graph on a decorated surface $S$ without boundary.
The classes of the vertex simple objects
$\{S_F\}_{F\in \mathcal F(\Gamma)}$
form a basis of the Grothendieck group $K_0(\mathcal D_\Gamma)$, and hence there is a canonical isomorphism of abelian groups
\begin{equation}
k:
\widetilde\Lambda_\Gamma
=\mathbb Z[\mathcal F(\Gamma)]
\xrightarrow{\sim}
K_0(\mathcal D_\Gamma),
\qquad
e_F\longmapsto [S_F].
\end{equation}

The full quiver lattice is the quotient
$\Lambda_\Gamma
\cong
\mathbb Z[\mathcal F(\Gamma)]/r(H_2(S;\mathbb Z))$ \cite{Goncharov}.
Consequently there is a quotient map
\begin{equation}
\widetilde\Lambda_\Gamma
\to
\frac{K_0(\mathcal D_\Gamma)}
     {r(H_2(S;\mathbb Z))},
\end{equation}
where $\sum_{F\in \mathcal F(\Gamma)} e_F$ is sent to $0$.
Since $S$ is connected, compact, and has no boundary,
$H_2(S;\mathbb Z)\cong\mathbb Z$.
\begin{equation}
\Lambda_\Gamma
\cong
\frac{K_0(\mathcal D_\Gamma)}
     {\mathbb Z\left\langle
       \sum_{F\in\mathcal F(\Gamma)}[S_F]
     \right\rangle}.
\end{equation}

The classes of vertex simple objects
$\{[S_G]\}_{G\in\mathcal F^\circ(\Gamma)}$
form a basis of the Grothendieck group $K_0(\mathcal D_\Gamma^\circ)$, and there is a canonical isomorphism of abelian groups
\begin{equation}
k|_{\Lambda_\Gamma^\circ}:\Lambda_\Gamma^\circ
\cong
\mathbb Z[\mathcal F^\circ(\Gamma)]
\xrightarrow{\sim}
K_0(\mathcal D_\Gamma^\circ),
\qquad
e_G\longmapsto[S_G].
\end{equation}
\end{proposition}

\begin{proposition}
\label{prop:quiverK0lattice}
Let $\Gamma$ be an $A_m$-bipartite graph on a decorated surface $S$ without boundary. Then $\mathcal F(\Gamma)=\mathcal F^\circ(\Gamma)$, and hence the full quiver lattice is $\Lambda_\Gamma^\circ\cong\mathbb Z[\mathcal F^\circ(\Gamma)]=Z[\mathcal F(\Gamma)]=\widetilde\Lambda_\Gamma$.

The classes of the vertex simple objects
$\{S_G\}_{G\in \mathcal F(\Gamma)}$
form a basis of the Grothendieck group $K_0(\mathcal D_\Gamma)=K_0(\mathcal D_\Gamma^\circ)$, and hence there is a canonical isomorphism of abelian groups
\begin{equation}
k:
\Lambda_\Gamma^\circ
\xrightarrow{\sim}
K_0(\mathcal D_\Gamma^\circ),
\qquad
e_G\longmapsto [S_G].
\end{equation}
\end{proposition}

\begin{figure}[htbp]
\centering
\includegraphics[width=0.4\textwidth]{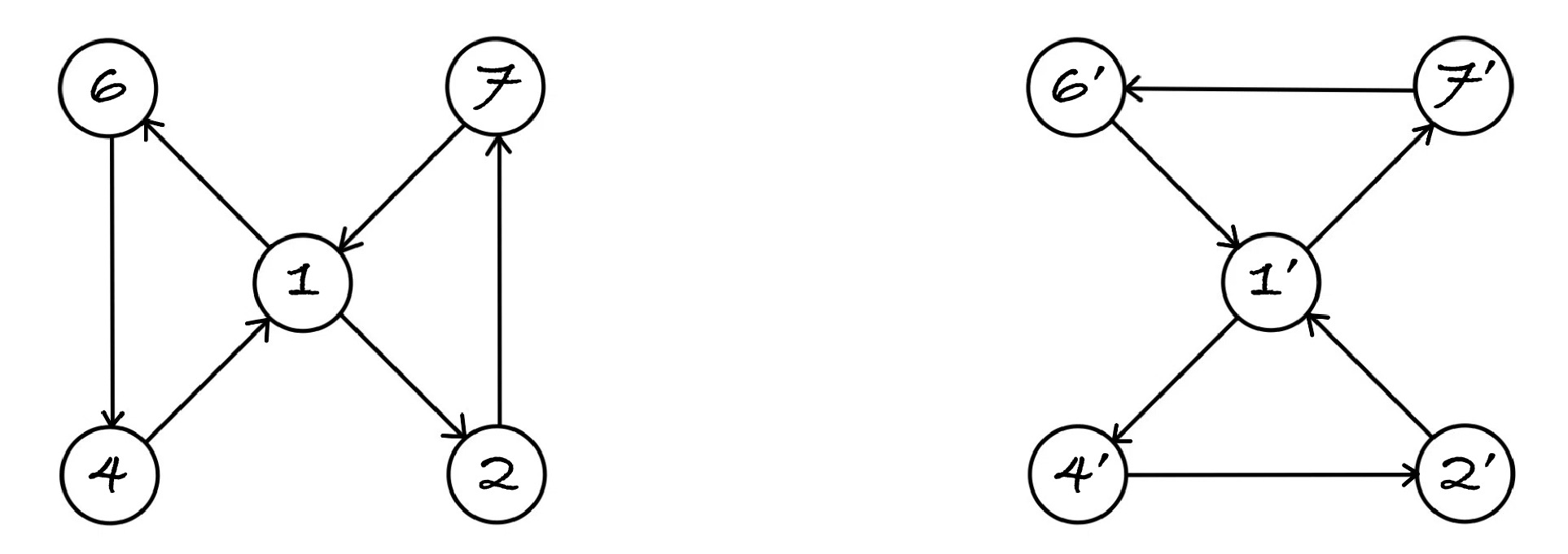}
\caption{The local change of the quiver under a \ttt move at face 1.\\ (This figure only shows the part of the quiver that changes under the mutation.\\ The rest of the quiver remains unchanged.)}
\label{fig:localquiver}
\end{figure}

\subsubsection{Mutations on $\mathrm{K}_0(\mathcal{D}_{\Gamma})$ Induced by Two-by-two Moves}
Let $\mu_F:\Gamma\to\Gamma'$
be a two-by-two move centred at an internal face $F$ whose corresponding vertex $v_F$ in $Q_\Gamma$ and $Q_\Gamma^\circ$ is 4-valent.
Then $\mu_{v_F}(Q_\Gamma,W_\Gamma)
\sim_{\mathrm{r.e.}}
(Q_{\Gamma'},W_{\Gamma'})$ and $\mu_{v_F}(Q_\Gamma^\circ,W_\Gamma)
\sim_{\mathrm{r.e.}}
(Q_{\Gamma'}^\circ,W_{\Gamma'})$,
after reduction of the trivial $2$-cycle part \cite{Goncharov}.

Consequently, Keller--Yang mutation gives exact equivalences
$\mathcal D_{\Gamma'}
\simeq
\mathcal D_\Gamma$ and $\mathcal D_{\Gamma'}^{\circ}
\simeq
\mathcal D_{\Gamma}^{\circ}$.
The induced isomorphisms on Grothendieck groups
$K_0(\mathcal D_{\Gamma'})
\xrightarrow{\sim}
K_0(\mathcal D_\Gamma)$
and
$K_0(\mathcal D_{\Gamma'}^{\circ})
\xrightarrow{\sim}
K_0(\mathcal D_\Gamma^{\circ})$
carry the basis of vertex simple objects to the corresponding mutated basis of vertex simple objects.

Recall the change of the local part of $Q_{\Gamma}$ or $Q_{\Gamma}^{\circ}$ under a \ttt move at face 1, as shown in Figure~\ref{fig:localquiver}. A \ttt move at face 1 corresponds to the mutation at vertex 1. The orientation of the \ttt move at the level of the spectral cover and the orientation of the induced isomorphism between $K_0$ groups are chosen in accordance with each other.

If we choose the induced isomorphism $(\Phi^-)_\ast:K_0(\mathcal D_{\Gamma'})
\xrightarrow{\sim}
K_0(\mathcal D_\Gamma)$, and accordingly $(\Phi^-)_\ast|_{K_0(\mathcal D_{\Gamma'}^{\circ})}:K_0(\mathcal D_{\Gamma'}^{\circ})
\xrightarrow{\sim}
K_0(\mathcal D_\Gamma^{\circ})$ between the Grothendieck groups, then by Lemma~\ref{le:k0},
\begin{equation}
\label{eq:left}
e_1'\mapsto -e_1,\qquad 
\begin{array}{l}
e_7'\mapsto e_7,\\[1pt]
e_4'\mapsto e_4,
\end{array}
\qquad
\begin{array}{l}
e_6'\mapsto e_6+e_1,\\[1pt]
e_2'\mapsto e_2+e_1.
\end{array}
\end{equation}
The rest of the basis is mapped to the corresponding ones with the same indices.

If we choose the induced isomorphism $(\Phi^+)_\ast:K_0(\mathcal D_{\Gamma'})
\xrightarrow{\sim}
K_0(\mathcal D_\Gamma)$, and accordingly $(\Phi^+)_\ast|_{K_0(\mathcal D_{\Gamma'}^{\circ})}:K_0(\mathcal D_{\Gamma'}^{\circ})
\xrightarrow{\sim}
K_0(\mathcal D_\Gamma^{\circ})$ between the Grothendieck groups, then by Lemma~\ref{le:k0},
\begin{equation}
\label{eq:right}
e_1'\mapsto -e_1,\qquad 
\begin{array}{l}
e_7'\mapsto e_7+e_1,\\[1pt]
e_4'\mapsto e_4+e_1,
\end{array}
\qquad
\begin{array}{l}
e_6'\mapsto e_6,\\[1pt]
e_2'\mapsto e_2.
\end{array}
\end{equation}
The rest of the basis is mapped to the corresponding ones with the same indices.

The isomorphisms between Grothendieck groups in Equation~\ref{eq:left} and Equation~\ref{eq:right} have the same form as the isomorphisms between quiver lattices in Theorem~\ref{thm:ttt}, and Equation~\ref{eq:trop} as well.

Therefore, we arrive at the following theorem.

\begin{theorem}
\label{thm:7com}
Let $\Gamma$ be an $A_m$ or $A_m^*$-bipartite graph on a genus-$g$ decorated surface $S$ without boundary, with $n$ special points, satisfying $2g-2+n>0$.

Let $Q_{\Gamma}
:=\bigl(\Lambda_{\Gamma},\{e_F\},(*,*)_{\Gamma}\bigr)$ be the full quiver and let $Q_{\Gamma}^{\circ}
:=\bigl(\Lambda_{\Gamma}^{\circ},\{e_F\},\operatorname{Res}_{\Lambda_{\Gamma}^{\circ}}(*,*)_{\Gamma}\bigr)$ be the mutable subquiver.
Denote by
$\mathcal D_\Gamma
:=
\mathcal D_{\mathrm{fd}}
\bigl(\Gamma(Q_\Gamma,W_\Gamma)\bigr)$ and $\mathcal D_\Gamma^{\circ}
:=\mathcal D_{\mathrm{fd}}
\bigl(\Gamma(Q_\Gamma^{\circ},W_\Gamma)\bigr)$ the 3CY categories determined by the quivers with potential $(Q_\Gamma,W_\Gamma)$ and $(Q_\Gamma^{\circ},W_\Gamma)$ respectively. 

Let $\mu_F:\Gamma\to\Gamma'$
be a two-by-two move centred at an internal face $F$ whose corresponding vertex $v_F$ in $Q_\Gamma$ is 4-valent.
Then there is a commutative diagram of isomorphisms:
\begin{equation}
\label{eq:diagram}
\begin{tikzcd}
\widetilde\Lambda_{\Gamma'}
=
\mathbb Z[\mathcal F(\Gamma')]
\arrow[r, "\phi^\pm"]
\arrow[d, "k'"]
&
\widetilde\Lambda_\Gamma
=
\mathbb Z[\mathcal F(\Gamma)]
\arrow[d, "k"]
\\
K_0(\mathcal D_{\Gamma'})
\arrow[r, "(\Phi^\pm)_\ast"]
&
K_0(\mathcal D_\Gamma).
\end{tikzcd}
\end{equation}
The skew bilinear intersection form on the quiver lattices and the skew Euler form on the $K_0$ groups are compatible with the isomorphisms in the commutative diagram.
The restriction of the above commutative diagram to $\Lambda_{\Gamma'}^{\circ}
\cong\mathbb Z[\mathcal F^\circ(\Gamma')]$ and $K_0(\mathcal{D}_{\Gamma'}^{\circ})$ is also a commutative diagram.
\end{theorem}

\begin{remark}
Let $W$ be an $A_{m-1}$-web on a decorated surface $S$ without boundary, and let $\Gamma$ be the corresponding bipartite $A_{m-1}$-graph. Let $W'$ and $\Gamma'$ be the $A_{m-1}$-web and bipartite $A_{m-1}$-graph respectively after a \ttt move.

By Proposition~\ref{prop:quiverK0lattice},
the full quiver lattice is 
$\Lambda_{\Gamma}^{\circ}$ as $W$ only has internal faces.
Then one has $\Lambda_{\Gamma}^{\circ}\cong \ker (H_1(\Sigma_W;\mathbb Z)\xrightarrow{(\pi_{W})_*}H_1(\mathring{S}; \mathbb Z))\cong\ker (H_1(\Sigma_W;\mathbb Z)\xrightarrow{(\pi_{W})_*}H_1(S;\mathbb Z))$.

Hence the commutative diagram in Equation~\ref{eq:diagram} can be written as
\begin{equation}
\begin{tikzcd}
\boxed{
\ker (H_1(\Sigma_{W'};\mathbb Z)\xrightarrow{(\pi_{W'})_*}H_1(S;\mathbb Z))
}
\arrow[r, "\phi^\pm"]
\arrow[d, "k'"]
&
\boxed{
\ker (H_1(\Sigma_W;\mathbb Z)\xrightarrow{(\pi_{W})_*}H_1(S;\mathbb Z))
}
\arrow[d, "k"]
\\
K_0(\mathcal{D}_{\Gamma'}^{\circ})
\arrow[r, "(\Phi^\pm)_\ast"]
&
K_0(\mathcal{D}_{\Gamma}^{\circ}),
\end{tikzcd}
\end{equation}
\end{remark}

\begin{remark}
The \ttt move on $\Gamma$ is the only type of elementary move that change the quiver with potential $(Q_\Gamma, W_\Gamma)$.
Shrink/expand moves on $\Gamma$ leave the quiver with potential $(Q_\Gamma, W_\Gamma)$ unchanged. So they induce \emph{trivial automorphisms} on $K_0(\mathcal{D}_\Gamma)$.
\end{remark}

\subsubsection{Gauss--Manin Monodromies and Mutations on $\mathrm{K}_0(\mathcal{D}_{\Gamma})$}

Let $\mathring{C}=C\setminus P$
be a $n$-punctured smooth complex curve of genus-$g$ with $2g-2+n>0$. Let $G=SL_m$, $m\ge2$.
Let $a$ be a point in $B_{G,C,P}^{\mathrm{good}}$.

Recall that by Corollary~\ref{cor:iso1}, given any point $a\in B_{G,C,P}^{\mathrm{good}}$, we can find a generic $A_{m-1}$-web $W$ and its corresponding bipartite $A_{m-1}$-graph $\Gamma$, such that $\pi_a:\Sigma_a\to C$ and $\pi_W:\Sigma_W\to C$ are Hurwitz isomorphic. Hence a basis of $H_1(\Prym(\Sigma_a);\mathbb Z)\cong \ker (H_1(\Sigma_a;\mathbb Z)\xrightarrow{(\pi_{a})_*}H_1(C;\mathbb Z))$ can be imposed from the basis of $\Lambda_{\Gamma}^{\circ}\cong \ker (H_1(\Sigma_W;\mathbb Z)\xrightarrow{(\pi_{W})_*}H_1(S;\mathbb Z))$ by the induced isomorphism.

\begin{lemma}
There is an isomorphism $k_c:H_1(\Prym(\Sigma_a);\mathbb Z)\to K_0(\mathcal{D}_{\Gamma}^{\circ})$, where $K_0(\mathcal{D}_{\Gamma}^{\circ})$ is the Grothendieck group associated with the quiver with potential assigned to some generic $A_{m-1}$-web $W$ and its corresponding bipartite $A_{m-1}$-graph $\Gamma$ on $C$.
\end{lemma}

\begin{theorem}
\label{thm:GM}
Let $\mathring{C}=C\setminus P$
be a $n$-punctured smooth complex curve of genus-$g$ with $2g-2+n>0$. Let $G=SL_m$, $m\ge2$.
Let $a\in B_{G,C,P}^{\mathrm{good}}$ be close to a point $b_0\in \mathcal{D}_{bdry,C,P}^\circ$.

Then for any sufficiently small meridian in $B_{G,C,P}^{\mathrm{good}}$ around $\mathcal{D}_{bdry,C,P}$ starting and ending at $a$, there is a commutative diagram of isomorphisms:
\begin{equation}
\begin{tikzcd}
H_1(\Prym(\Sigma_a);\mathbb Z)
\arrow[r, "GM"]
\arrow[d, "k_c"]
&
H_1(\Prym(\Sigma_a);\mathbb Z)
\arrow[d, "k_c"]
\\
K_0(\mathcal{D}_{\Gamma}^{\circ})
\arrow[r, "(\Phi'^\pm)_\ast\circ(\Phi^\pm)_\ast"]
&
K_0(\mathcal{D}_{\Gamma}^{\circ}),
\end{tikzcd}
\end{equation} 
where the top automorphism is given by the Gauss--Manin monodromy, and $K_0(\mathcal{D}_{\Gamma}^{\circ})$ is the Grothendieck group associated with the quiver with potential assigned to certain generic $A_{m-1}$-web $W$ and its corresponding bipartite $A_{m-1}$-graph $\Gamma$ on $C$.

$(\Phi'^\pm)_\ast\circ(\Phi^\pm)_\ast$ is the composition of two successive mutations on $K_0(\mathcal{D}_{\Gamma}^{\circ})$-basis induced by two successive \ttt moves in the same orientation at certain internal face of $W$ in the same orientation.

The skew bilinear intersection form on the Hitchin fiber lattices and the skew Euler form on the $K_0$ groups are compatible with the isomorphisms in the commutative diagram.
\end{theorem}

\begin{proof}
By Corollary~\ref{cor:iso1}, Proposition~\ref{prop:hittoweb}, and Theorem~\ref{thm:7com}.
\end{proof}

\newpage
\appendix

\includepdf[
    pages=1,
    scale=1,
    pagecommand={
        \thispagestyle{plain}
        \mysection{Paper Model Instructions}
    }
]{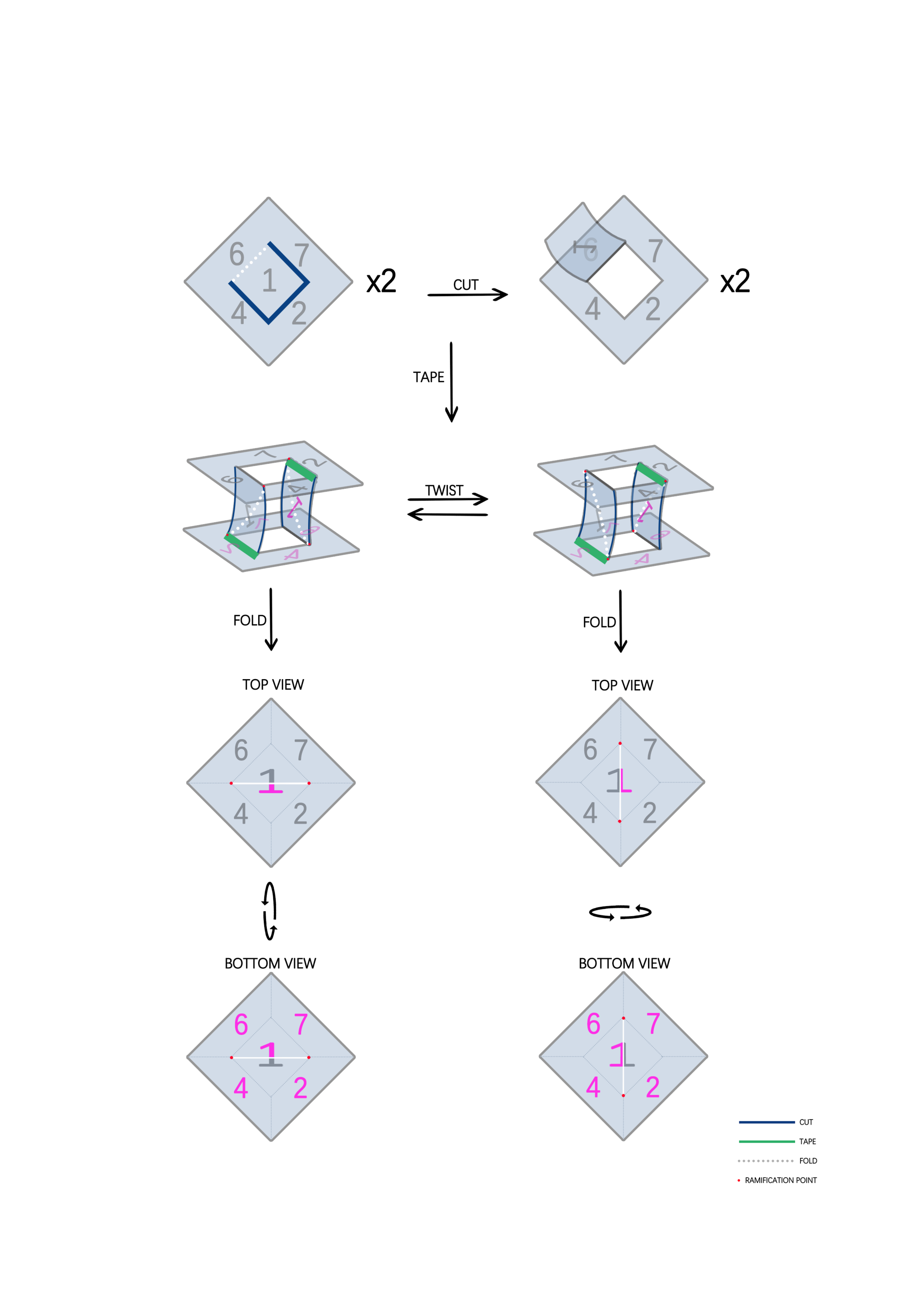}

\newpage
\mysection{Mathematica Animation Code}

\footnotesize
\begin{multicols}{2}
\begin{verbatim}
Clear["Global`*"]

A = 0.5;

s[t_] :=
 Piecewise[{
   {A (1 - 2 t), t <= 1/2},
   {I A (2 t - 1), t > 1/2}
 }];

f[z_, t_] := z^3 - 3 s[t]^2 z;

Manipulate[

 Module[
  {
   critZ,
   surface,
   points3D
  },

  critZ =
   If[Abs[s[t]] < 10^-6,
    {0},
    {s[t], -s[t]}
   ];

  surface =
   ParametricPlot3D[

    {
     Re[f[x + I y, t]],
     Im[f[x + I y, t]],
     x + y
    },

    {x, -1.2, 1.2},
    {y, -1.2, 1.2},

    PlotStyle ->
     Directive[
      LightBlue,
      Opacity[0.5]
     ],

    Mesh -> {15, 15},
    MeshStyle ->
     Directive[
      Gray,
      Thin,
      Opacity[0.5]
     ],

    PlotPoints ->
     ControlActive[25, 55],

    MaxRecursion ->
     ControlActive[0, 1],

    Axes -> False,
    Boxed -> False,

    PlotRange -> {
      {-3, 3},
      {-3, 3},
      {-3, 3}
    },

    BoxRatios -> {1, 1, 1},
    ImageSize -> Large
   ];

  points3D =
   Graphics3D[{

     Red,
     PointSize[0.03],

     Point[
      Table[
       {
        Re[f[z, t]],
        Im[f[z, t]],
        Re[z] + Im[z]
       },
       {z, critZ}
      ]
     ]
    }];

  Show[
   surface,
   points3D
  ]
 ],
 {{t, 0}, 0, 1, 0.005},
 TrackedSymbols :> {t}
]
\end{verbatim}
\end{multicols}

\newpage
\clearpage

\gdef\currentsectiontitle{References}
\thispagestyle{mainstyle}

\phantomsection
\addcontentsline{toc}{section}{References}

\bibliographystyle{amsalpha}
\bibliography{references}

\end{document}